\documentclass[automatic-bibliography,defaultmaths=false]{univsci} 
\usepackage{biblatex}
\usepackage{csquotes}
\let\Bbbk\relax
\usepackage{newtxtext, newtxmath}
\usepackage{enumitem} 
\usepackage{tikz-cd} 
\title{Some interactions between Hopf Galois extensions and noncommutative rings}
\shorttitle{Some interactions between Hopf Galois extensions and noncommutative rings} 
\author[1]{Fabio Calder\'on}
\author[1]{Armando Reyes*}
\authormail{mareyesv@unal.edu.co}
\shortauthor{Calder\'on \& Reyes}
\Affil{1}{Departamento de Matem\'aticas, Facultad de Ciencias, Universidad Nacional de Colombia - Sede Bogot\'a, Bogot\'a, D. C., Colombia.}
\funding{The authors were supported by the research fund of Faculty of Science, Code HERMES 52464, Universidad Nacional de Colombia - Sede Bogot\'{a}, Colombia. The first author was also supported by the Fulbright Visiting Student Researcher Program.}
\esmaterial{n.a.}
\receiveddate{05-24-2021}
\accepteddate{05-03-2022}
\publisheddate{10-08-2022}
\vol{27}
\issue{2}
\citat{
        Calderón F, Reyes A. Some interactions between Hopf Galois extensions
        and noncommutative rings, \emph{Universitas Scientiarum}, 27(2): 58--161,
        2022.
}
\doi{10.11144/Javeriana.SC271.sibh}
\begin{document}
\setcounter{page}{58}
\maketitle
\thispagestyle{firstpage}

\begin{abstract}
		In this paper, our objects of interest are Hopf Galois extensions (e.g., Hopf algebras, Galois field extensions, strongly graded algebras, crossed products, principal bundles, etc.) and families of noncommutative rings (e.g., skew polynomial rings, PBW extensions and skew PBW extensions, etc.). We collect and systematize questions, problems, properties and recent advances in both theories by explicitly developing examples and doing calculations that are usually omitted in the literature. In particular, for Hopf Galois extensions we consider approaches from the point of view of quantum torsors (also known as quantum heaps) and Hopf Galois systems, while for some families of noncommutative rings we present advances in the characterization of ring-theoretic and homological properties. Every developed topic is exemplified with abundant references to classic and current works, so this paper serves as a survey for those interested in either of the two theories. Throughout, interactions between both are presented.
		
        \end{abstract}
	
	\section{Introduction}
	
	In the last half-century, Hopf algebras turned out to be a great tool for studying a large number of problems in several contexts: from providing solutions for the Yang-Baxter equation and describing the so-called quantum groups --appearing in theoretical physics and algebraic theory--, to generalizing Galois theory. It is precisely this last instance what concerns us in this paper.
	
	Classically, Galois theory studies and classifies automorphism groups of fields. In 1965 the theory was generalized to groups acting on commutative rings \cite{CHR}, and in 1969 extended to commutative algebras by replacing the action of a group on the algebra for a coaction of a Hopf algebra on a commutative algebra \cite{CS}. The first general definition of \emph{Hopf Galois extensions} is due to Kreimer and Takeuchi \cite{KT}, although the authors restricted their study to the finite-dimensional case. In Part~\ref{ch1}, we will address the modern definition of such extensions, not without first recalling some basic notions regarding the theory of Hopf algebras. Our aim is to provide a large number of examples and properties, developing proofs and calculations that are usually omitted in the literature. We finish the section by giving two recent alternative approaches for Hopf Galois theory: \emph{quantum torsors} (or \emph{quantum heaps}), defined independently by Grunspan \cite{Gru2} and \v{S}koda \cite{Sko}, and \emph{Hopf Galois systems}, introduced by Bichon~\cite{Bic}. We address the equivalence between these three notions.
	
	Almost in parallel to the first appearance of Hopf algebras, Ore introduced in 1933 a new class of noncommutative rings, nowadays known as \emph{skew polynomial rings} (or \emph{Ore extensions}) \cite{Ore}. Although the aim of Ore was to find noncommutative algebras which could be embedded on division rings (e.g., \cite[Chapter 8]{Coh}), these structures belong, per se, to a branch of study in algebra used to describe many rings and algebras, mostly coming from mathematical physics and with broad applications in quantum mechanics. Therefore, some classic results such as the Hilbert's Basis Theorem or the Hilbert's Syzygy Theorem have been generalized to these objects (see Theorem~\ref{t2} or, e.g., \cite[Sections 2.9 and 3.1]{MR}), while many other properties are still being studied. Hence, at the start of Part~\ref{ch2} we review basic definitions and results on skew polynomial rings, along with some remarkable examples.
	
	However, this is not the only family of noncommutative rings (or algebras) that has been defined since then. Inspired by the Poincar\'e-Birkhoff-Witt (PBW) theorem for enveloping algebras of Lie algebras, Bell and Goodearl defined in 1988 the \emph{PBW extensions} \cite{BG}. These consist of polynomial-type rings having a PBW basis and specific commutation rules. Furthermore, Ga\-lle\-go and Lezama in 2011 generalized the notion to \emph{skew PBW extensions} so new examples of rings with polynomial behavior could be studied \cite{LezBook, GL}. Hence, our aim is to also address these types of rings and some examples of the theory. We also study a quite different collection of algebras, known as \emph{almost symmetric algebras} (or \emph{Sridharan enveloping algebras}). These generalize enveloping algebras via twisting by $2$-cocycles, without losing some nice properties \cite{Sri}.
	
	With these two overviews in mind, one could ask for possible relations between some of the mentioned families and Hopf algebras (e.g. \cite{Pan,Sal}), and in particular, with Hopf Galois theory. Therefore, as an original contribution, we study coactions of arbitrary Hopf algebras over skew polynomial rings. Also, following \cite{Gru}, we attach a Hopf Galois system to almost symmetric algebras, and elucidate the structure of quantum torsor present in Kashiwara algebras. 
	
	Therefore, the purpose of this paper is three-fold: as a survey of classic and current works on Hopf Galois theory, as a quick overview of several approaches to noncommutative algebras appearing in applications, and as a compilation (with some original developments) of the interactions between both theories. This work is the result of the first author's master's thesis, which had a Meritorious Mention at Universidad Nacional de Colombia and was written under the advice of the second author.
	
	\subsection*{Notations and conventions} Throughout this manuscript all rings and their morphisms are unitary. $K$ will denote an arbitrary commutative ring and $\Bbbk$ a field (if necessary, algebraically closed and of characteristic 0). Unless stated otherwise, tensor products are assumed to be over $K$ and every $K$-module is non-zero.
	
	Let $f,g,h$ be functions. We denote the composition of $f$ with $g$ by $fg$ and the composition of $h$ with itself $n$-times as $h^n$. The arrow $\operatorname{id}_X:X\rightarrow X$ will always denote the identity map of $X$.
	
	\emph{Arrow diagrams} will be constantly used, representing composition of functions as concatenation of arrows. A diagram is said to be \emph{commutative} if, no matter what path one follows, the composition of arrows (functions) returns always the same result.
	
	The symbols $\mathbb{N}$, $\mathbb{Z}$, $\mathbb{Q}$, $\mathbb{R}$, $\mathbb{C}$ denote the usual numerical systems, assuming that $0\in\mathbb{N}$.
	
	\section{Hopf Galois theory: preliminaries, definitions and examples}\label{ch1}
	
	Although we assume some familiarity with the theory of Hopf algebras, for the purpose of a self-contained document, this section will start addressing basic terminology of (co)algebras (Section~\ref{sec2.1} and~\ref{sec2.2}), some examples (Section~\ref{sec2.3}), and (co)modules, Hopf modules and (co)actions (Sections~\ref{sec2.4} to~\ref{sec2.6}). Then we introduce in Section~\ref{sec2.7} the concept of Hopf Galois extension, which is transversal to this work. A large amount of examples and properties are presented in Sections~\ref{sec2.8} and~\ref{sec2.9}. Finally, and following recent developments, we dedicate Sections~\ref{se4} and~\ref{sec2.11} to two alternative --and, under some conditions, equivalent-- approaches of Hopf Galois extensions, namely quantum torsors and Hopf Galois systems.
	
	Except for the last two sections, all definitions and results presented in this part are classical and can be consulted, for example, in \cite{CS,DNR,DT,Kas,KT,Mon1,Mon2,Sch,Sch4,Swe}.
	
	\subsection{Algebras and coalgebras}\label{sec2.1}
	
	Recall that a \emph{$K$-algebra} is a $K$-module $A$ together with two $K$-linear maps, $m:A\otimes A \rightarrow A$ and $u:K\rightarrow A$, such that the following diagrams are commutative:
	\begin{equation*}
		\begin{tikzcd}
			A\otimes A\otimes A \arrow[r,"m\otimes \operatorname{id}_A"] \arrow[d,swap,"\operatorname{id}_A\otimes m"] &
			A \otimes A \arrow[d,"m"] \\
			A \otimes A \arrow[r,"m"'] & A
		\end{tikzcd} \qquad \begin{tikzcd}[row sep=scriptsize,
			column sep=scriptsize]
			& A\otimes A \arrow[dd,"m"]&\\
			K\otimes A \arrow[ur,"u\otimes \operatorname{id}_A"] \arrow[dr,"\cong"'] & & A\otimes K \arrow[ul,"\operatorname{id}_A\otimes u"'] \arrow[dl,"\cong"] \\
			&A&
		\end{tikzcd}
	\end{equation*}
	The first diagram is known as the \emph{associativity property} while the second as the \emph{main unit property}. We write $ab=m(a\otimes b)$ and $1_A:=u(1_K)$. Similarly, a \emph{$K$-coalgebra} is a $K$-module $C$ together with two $K$-linear maps, $\Delta:C\rightarrow C\otimes C$ and $\varepsilon: C\rightarrow K$, such that the following diagrams are commutative:
	\begin{equation*}
		\begin{tikzcd}
			C \arrow[r,"\Delta"] \arrow[d,swap,"\Delta"] &
			C\otimes C \arrow[d,"\operatorname{id}_C\otimes\Delta"] \\
			C \otimes C \arrow[r,"\Delta\otimes \operatorname{id}_C"'] & C\otimes C \otimes C
		\end{tikzcd} \qquad \begin{tikzcd}[row sep=scriptsize,
			column sep=scriptsize]
			& C \arrow[dd,"\Delta"] \arrow[dl,"\cong"'] \arrow[dr,"\cong"] &\\
			K\otimes C  & & C\otimes K  \\
			&C\otimes C \arrow[ul,"\varepsilon\otimes \operatorname{id}_C"] \arrow[ur,"\operatorname{id}_C\otimes \varepsilon"'] &
		\end{tikzcd}
	\end{equation*}
	The map $\Delta$ is called the \emph{comultiplication} and $\varepsilon$ the \emph{counit}. The left diagram is known as the \emph{coassociativity property} while the second as the \emph{main counit property}. We use the widely accepted \emph{Heyneman--Sweedler notation} for the comultiplication, that is, for any $c\in C$ we write $\Delta(c)=c_{(1)}\otimes c_{(2)} \in C\otimes C$.

	Arrows between these structures are defined as those preserving the operations. Therefore we consider $K$-$\operatorname{Alg}$, the category of $K$-algebras, and $K$-$\operatorname{Cog}$, the category of $K$-coalgebras. Regarding substructures in $K$-$\operatorname{Cog}$, recall that for a coalgebra $C$, a submodule $D$ is called a \emph{subcoalgebra} if $\Delta(D)\subseteq D\otimes D$, and a submodule $I$ is called a \emph{left} (resp. \emph {right}) \emph{coideal} if $\Delta(I)\subseteq C\otimes I$ (resp. $\Delta(I)\subseteq I\otimes C$). Also, a submodule $J$ of $C$ is called a \emph{coideal} if $\Delta(J)\subseteq J\otimes C + C\otimes J$ and $\varepsilon(J)=0$.
	
	Given two $K$-modules $M, N$, the \emph{twist map} $\tau_{M,N}:M\otimes N\rightarrow N\otimes M$ is defined by $m\otimes n \mapsto n\otimes m$, for all $m\in M$ and $n\in N$. Sometimes this arrow is also denoted by $\tau_{(12)}$ when emphasis in the interchanged coordinates is needed. In the situation $M=N$, we simply write $\tau_M$. An algebra $A$ is \emph{commutative} if $m\tau_A=m$, while a coalgebra $C$ is \emph{cocommutative} if $\tau_C\Delta=\Delta$, that is, if $c_{(1)}\otimes c_{(2)}=c_{(2)} \otimes c_{(1)}$, for all $c\in C$.
	
	If $M$ is a $K$-module, for $n\geq 2$, we denote $M^{\otimes n}:= M\otimes \dotsb \otimes M$ ($n$ times). By definition, $M^{\otimes0}:=K$ and $M^{\otimes 1} := M$. This notation is useful when dealing with coalgebras and their maps. For any coalgebra $C$, we define the sequence of maps $\{\Delta_n : C\rightarrow C^{\otimes (n+1)} \}_{n\geq 1}$ recurrently as follows: $\Delta_1:=\Delta$, and $ \Delta_n:= (\Delta \otimes \operatorname{id}_C^{n-1}) \Delta_{n-1}$, for any $ n\geq 2$.
	
	The coassociativity of $\Delta$ states that, for any $c\in C$,
	\begin{equation*}
		( {c_{(1)}}_{(1)}\otimes {c_{(1)}}_{(2)} ) \otimes c_{(2)} =  c_{(1)} \otimes ( {c_{(2)}}_{(1)} \otimes {c_{(2)}}_{(2)} ),
	\end{equation*}
	and therefore we are able to just write $ \Delta_2(c)= c_{(1)}\otimes c_{(2)}\otimes c_{(3)}$. Moreover, we will have
	$\Delta_n(c)= c_{(1)}\otimes \cdots \otimes c_{(n+1)},$ for any $n\geq 1.$
	The main property of the counit $\varepsilon$ may be formulated as $\varepsilon(c_{(1)})c_{(2)}=c= c_{(1)}\varepsilon(c_{(2)})$. Also, the behavior of a coalgebra morphism $g:C\rightarrow D$ can be stated as $g(c)_{(1)} \otimes g(c)_{(2)} = g(c_{(1)}) \otimes g(c_{(2)})$.
	
	\subsection{Bialgebras and Hopf algebras}\label{sec2.2}
	
	Recall that a \emph{$K$-bialgebra} $H$ is a $K$-module simultaneously endowed with an algebra and a coalgebra structure, both over $K$, satisfying that $m$ and $u$ are morphisms of coalgebras (or equivalently, that $\Delta$ and $\varepsilon$ are algebra maps). For simplicity, and since it holds for most applications, throughout we will assume the hypothesis that every $K$-bialgebra $H$ is flat over $K$, meaning that the tensor product functor $-\otimes H$ is exact (i.e., preserves the exactness of sequences). The category of $K$-bialgebras is denoted by $K$-$\operatorname{Bialg}$.
	
	In the $K$-module $\operatorname{Hom}_K(C,A)$ of all $K$-linear maps between a coalgebra $C$ and an algebra $A$ we define the \emph{convolution product} as $	(f*g)(c):=(m_A(f\otimes g)\Delta_C) (c)$, for all $f,g\in \operatorname{Hom}_K(C,A)$, and $c\in C$. In Heyneman--Sweedler notation, $(f*g)(c)=f(c_{(1)})g(c_{(2)})$. Hence the convolution product endows $\operatorname{Hom}_K(C,A)$ with an algebra structure where the identity element is $u_A\varepsilon_C$. We say that $f\in \operatorname{Hom}_K(C,A)$ is \emph{convolution invertible} if there exists an element $g\in\operatorname{Hom}_K(C,A) $ such that $f*g=g*f=u_A\varepsilon_C$.
	
	\begin{definition}[Hopf algebra]
		A $K$-bialgebra $H$ is a \emph{$K$-Hopf algebra} if $\operatorname{id}_H$ is convolution invertible by an element $S\in\operatorname{Hom}_K(H,H)$. In this case, $S$ is called an \emph{antipode} for $H$.
	\end{definition}
	
	Since $\operatorname{Hom}_K(H,H)$ is an algebra and $S$ is defined as an inverse, the antipode is unique. Moreover, using Heyneman--Sweedler notation, $S$ satisfies $S(h_{(1)})h_{(2)} = \varepsilon(h)1=h_{(1)}S(h_{(2)})$, for all $h\in H$. This is known as the \emph{main property of the antipode}. It can be shown that $S$ is an anti-morphism of algebras and coalgebras (e.g. \cite[Proposition 4.2.6]{DNR}).
	
	As with previous structures, arrows can be considered between Hopf algebras and therefore one can define the corresponding category, denoted by $K$-$\operatorname{HopfAlg}$. It can be shown that any bialgebra morphism between two Hopf algebras is always a Hopf algebra morphism (e.g. \cite[Proposition 4.2.5]{DNR}), which in terms of categories means that $K$-$\operatorname{HopfAlg}$ is a full subcategory of $K$-$\operatorname{Bialg}$. Regarding substructures in $K$-$\operatorname{HopfAlg}$, for a Hopf algebra $H$, a submodule $L$ is said to be a \emph{Hopf subalgebra} if it is a subalgebra of $H$, a subcoalgebra of $H$ and $S(L)\subseteq L$. Also, a submodule $I$ of $H$ is said to be a \emph{Hopf ideal} if it is an ideal of $H$ (as algebra), a coideal of $H$ (as coalgebra) and $S(I)\subseteq I$.
	
	Finally, recall that in a coalgebra $C$, an element $c\in C$ is said to be \emph{group-like} if $\Delta(c)=c\otimes c$. The set of all group-like elements of $C$ is denoted by $G(C)$. Similarly, $x\in C$ is said to be a \emph{$(g,h)$-primitive element} (or simply a \emph{skew primitive element} when $g$ and $h$ are not specified), if $\Delta(x)=x\otimes g + h\otimes x$ with $g,h\in G(C)$. The set of all $(g,h)$-primitive elements of $C$ is denoted by $P_{g,h}(C)$. Notice that if $H$ is a Hopf algebra, $G(H)$ becomes a group with induced multiplication and inverses given by the antipode $S$.	
	
	\subsection{Examples of Hopf algebras}\label{sec2.3}
	
	Now we address some essential examples of Hopf algebras (for a larger amount see e.g. \cite{DNR,Kas,Mon1,Swe}).
	
	\begin{example}\label{ex3}
		Any commutative ring $K$ has structure of $K$-Hopf algebra by defining, for all $k\in K$,
		\begin{equation*}
			\Delta(k):=k\otimes 1, \quad \varepsilon(k):=k, \quad \quad S(k):=k.
		\end{equation*}
		 In particular, this holds if $K=\Bbbk$ is a field.
	\end{example}
	
	The next example establishes a methodical way of constructing new Hopf algebras.
	
	\begin{example}\label{ex1}
		Let $A,B$ be two $K$-algebras. Then the $K$-module $A\otimes B$ has also the structure of a $K$-algebra with multiplication $m_{A\otimes B}$ and unit $u_{A\otimes B}$ given by the compositions
		\begin{equation*}
			\begin{gathered}
				m_{A\otimes B}: \begin{tikzcd}[column sep=huge]
					(A\otimes B) \otimes (A\otimes B) \arrow[r,"\operatorname{id}_A \otimes \tau_{A,B}\otimes \operatorname{id}_B"] & (A\otimes A) \otimes (B\otimes B) \arrow[r,"m_A\otimes m_B"] & A\otimes B,
				\end{tikzcd}\\
				u_{A\otimes B}: \begin{tikzcd}
					K \arrow[r,"\cong"] & K\otimes K \arrow[r,"u_A\otimes u_B"] & A\otimes B.
				\end{tikzcd}
			\end{gathered}	
		\end{equation*}
		Notice that the multiplication can be stated as $(a\otimes b)(a'\otimes b'):=aa'\otimes bb'$, for all $a,a'\in A$, $b,b'\in B$. The unit element is $1_A\otimes 1_B$. Further results on this algebra can be found in \cite[Section II. 4]{Kas}. When $A=B$, the multiplication simplifies to $m_{A\otimes A}:= m_A\otimes m_A $.
		
		Similarly, if $C,D$ are two $K$-coalgebras, then $C\otimes D$ has also the structure of a $K$-coalgebra with comultiplication $\Delta_{C\otimes D}$ and counit $\varepsilon_{C\otimes D}$ given by the compositions
		\begin{equation*}
			\begin{gathered}
				\Delta_{C\otimes D}: \begin{tikzcd}[column sep=huge]
				C\otimes D \arrow[r,"\Delta_C\otimes\Delta_D"] &
                                (C\otimes C) \otimes (D\otimes D)
                                \arrow[r,"\operatorname{id}_C\otimes
                                \tau_{C,D}\otimes\operatorname{id}_D"] &
                                (C\otimes D) \otimes (C\otimes C),
			\end{tikzcd}\\
			\varepsilon_{C\otimes D}: \begin{tikzcd}
				C\otimes D \arrow[r,"\varepsilon_C\otimes\varepsilon_D"] & K\otimes K \arrow[r,"\cong"] &K.
			\end{tikzcd}
			\end{gathered}
		\end{equation*}
		In Heyneman--Sweedler notation,
		\begin{gather*}
			\Delta_{C\otimes D}(c\otimes d)=  (c\otimes d)_{(1)} \otimes (c\otimes d)_{(2)}= (c_{(1)} \otimes d_{(1)}) \otimes (c_{(2)} \otimes d_{(2)}),\\
			\varepsilon_{C\otimes D}(c\otimes d) =  \varepsilon_C(c)\varepsilon_D(d).
		\end{gather*}
		When $C=D$, the comultiplication simplifies to $\Delta_{C\otimes C}:= \Delta_C \otimes \Delta_C $.
		
		Furthermore, if $H, L$ are two $K$-Hopf algebras, then their tensor product is also a $K$-Hopf algebra with antipode $S_{H\otimes L}:=S_H\otimes S_L$.
	\end{example}

\begin{example}\label{ex20}
	Let $H$ be a Hopf algebra and $I$ a Hopf ideal of $H$. Since $I$ is a two-sided ideal of $H$, the quotient module $H/I$ already has algebra structure, by putting
	$\overline{h}\overline{g}:=\overline{hg}$, for all $h, g\in H$, where $\overline{h}:=h+I$. The identity element is $\overline{1}$. Moreover, $H/I$ has Hopf algebra structure given by
	\begin{equation*}
		\overline{\Delta}(\overline{h}):=\overline{h_{(1)}} \otimes \overline{h_{(2)}}, \qquad \overline{\varepsilon}(\overline{h}):=\varepsilon(h) \qquad \mbox{and} \qquad \overline{S}(\overline{h}):=\overline{S(h)}, \qquad \mbox{for all } h\in H.
	\end{equation*}
	One can easily check that the canonical projection $H\rightarrow H/I$ is a surjective morphism of Hopf algebras.
\end{example}

	In the following examples, the base ring is a field.
	
	\begin{example}\label{ex5}
		Recall that for any $\Bbbk$-vector space $V$ we denote by $V^*:=\operatorname{Hom}_\Bbbk(V,\Bbbk)$ its \emph{(linear) dual}, consisting of all $\Bbbk$-linear maps from $V$ to $\Bbbk$, together with the point-wise addition and scalar multiplication by constants so that $V^*$ is also a $\Bbbk$-vector space. In particular, we have $\Bbbk^* \cong \Bbbk$ via the identification $f\mapsto f(1_\Bbbk)$.
		
		Let $H$ be a finite dimensional $\Bbbk$-Hopf algebra (i.e., $H$ is finite dimensional as $\Bbbk$-vector space). Then its linear dual $H^*$ is also a $\Bbbk$-Hopf algebra with operations
		\begin{equation*}
			\begin{gathered}
				m^\circ: \begin{tikzcd}
					H^* \otimes H^* \arrow[r,"\cong"] & (H\otimes H)^* \arrow[r,"\Delta^*"] & H^*,
				\end{tikzcd}\\
				u^\circ: \begin{tikzcd}
					\Bbbk \arrow[r,"\cong"] & \Bbbk^* \arrow[r,"\varepsilon^*"] & H^*,
				\end{tikzcd}\\
			\Delta^\circ: \begin{tikzcd}
				H^* \arrow[r,"m^*"] & (H \otimes H)^* \arrow[r,"\cong"] & H^* \otimes H^*, 
			\end{tikzcd}\\
			\varepsilon^\circ: \begin{tikzcd}
				H^* \arrow[r,"u^*"]& \Bbbk^* \arrow[r,"\cong"] & \Bbbk,
			\end{tikzcd}\\
		S^\circ : \begin{tikzcd}
			H^* \arrow[r,"S^*"]& H^*.
		\end{tikzcd}
			\end{gathered}
		\end{equation*}
	In this example the condition of $H$ being finite dimensional cannot be easily dropped, for if $H$ is not finite dimensional, $H^* \otimes H^*$ could be a proper subspace of $(H\otimes H)^*$ and thus the image of $m^\circ: H^* \rightarrow (H\otimes H)^*$ might not lie in $H^* \otimes H^*$. Therefore, for the general case a certain subset $H^\circ$ of $H^*$ is considered, often called the \emph{finite dual} of $H$ (see e.g. \cite[Section 1.2]{Mon1}). On the other hand, duals for Hopf algebras defined over commutative rings constitute an open line of investigation, and some progress has been made when the base ring is a polynomial algebra (see e.g. \cite{Kur}).
	\end{example}
	
	\begin{example}[Group algebra]\label{ex6}
		Let $G$ be a (multiplicative) group. The \emph{group algebra}, denoted by $\Bbbk G$, is the $\Bbbk$-vector space with $G$ as a basis, and hence its elements are of the form $\sum_{g\in G} k_g g$, where only finite $k_g$ are non-zero scalars. $\Bbbk G$ is an algebra with multiplication given by
		\begin{equation*}
			\left( \sum_{i=1}^n k_i g_i \right)\left( \sum_{j=1}^m l_jh_j \right) = \sum_{i=1}^n \sum_{j=1}^m (k_il_j)(g_ih_j), \qquad \mbox{for all } k_i,l_j\in \Bbbk, \, g_i, h_j\in G,
		\end{equation*}
	and unit $1_{\Bbbk G} := 1_G$. Furthermore, $\Bbbk G$ becomes a Hopf algebra by linearly extending the following rules:
		\begin{equation*}
			\Delta(g):=g\otimes g, \qquad \varepsilon(g):=1_{\Bbbk} \qquad \mbox{and} \qquad S(g):=g^{-1}, \qquad \mbox{for all }g\in G.
		\end{equation*}
	\end{example}

	\begin{example}[Dual of group algebra]\label{ex10}
		Let $G$ be a finite group and $H=(\Bbbk G)^*$ the dual Hopf algebra of the group algebra. Even though Example~\ref{ex5} describes the operations of $H$, we want a more detailed description. So notice that the universal property of $\Bbbk G$ allows us to identify $H$ with $\Bbbk^G$, the algebra of functions from $G$ to $\Bbbk$. Hence we have for $f,g\in \Bbbk^G$ and $x,y\in G$:
		\begin{align*}
			(f\cdot g)(x)&:=[m^\circ(f\otimes g)](x)=[\Delta^*(f\otimes g)](x)\\
			&=[(f\otimes g)\Delta](x)=(f\otimes g)(x\otimes x)=f(x)g(x),\\
			[\Delta(f)](x \otimes y)&:=[\Delta^\circ(f)](x\otimes y) = [m^*(f)](x\otimes y) = f(xy),\\
			[S(f)](x)&:=[S^\circ(f)](x)=[S^*(f)](x)= [fS](x)=f(S(x)).
		\end{align*}
		Despite these formulas, a full description of $\Delta(f)$ is not given. Therefore we define for every $x\in G$ the map $p_x: G \rightarrow \Bbbk$ given by
		\begin{equation*}
			p_x(y):=\delta_{x,y}:=\begin{cases*}
				1 & \mbox{if } \mbox{$x=y$},\\
				0 & \mbox{if } \mbox{$x\neq y$}.
			\end{cases*}
		\end{equation*}
		Since $G$ is finite, $\{ p_x : x\in G \}$ is a basis for $\Bbbk^G$ (in correspondence with the dual basis of $(\Bbbk G)^*$). For those elements, we have
		\begin{equation*}
			\Delta(p_x)=\sum_{uv=x} p_u \otimes p_v = \sum_{ y\in G } p_y \otimes p_{y^{-1}x},
		\end{equation*}
		which describes the comultiplication for basis elements.
	\end{example}
	
	\begin{example}[Tensor algebra]\label{ex18}
		Let $V$ be a $\Bbbk$-vector space. An algebra $T(V)$ is said to be a \emph{tensor algebra of $V$} if there exists a linear map $\iota:V\rightarrow T(V)$ such that the following universal property is satisfied: \textit{for any algebra $A$ and any linear map $f:V\rightarrow A$ there exists an unique algebra morphism $\overline{f}:T(V)\rightarrow A$ such that the following diagram is commutative:}
		\begin{equation*}
			\begin{tikzcd}
				V \arrow[d,"f"'] \arrow[r,"\iota"] & T(V) \arrow[dl,dashed,"\overline{f}"]\\
				A
			\end{tikzcd}
		\end{equation*}
		The tensor algebra $T(V)$ is unique up to isomorphism and can be described as $T(V)=\bigoplus_{i\geq 0} V^{\otimes i}$, meaning that any element of $T(V)$ has the form $z=(z_i)_{i\geq 0}$, where $z_i\in V^{\otimes i}$ and almost all $z_i$ vanish. The multiplication is given by the rule
		\begin{equation*}
			(v_1\otimes \cdots \otimes v_i)(v_{i+1}\otimes \cdots \otimes \otimes v_{i+j})=v_1\otimes \cdots \otimes v_{i+j}, \qquad \mbox{for all } i,j\geq 0.
		\end{equation*}
		The identity element is $1\in V^{\otimes0}=\Bbbk$. $T(V)$ becomes a Hopf algebra by extending (via the universal property) the rules
		\begin{equation*}
			\Delta(v):=v\otimes 1 + 1\otimes v, \qquad \varepsilon(v):=0 \qquad \mbox{and} \qquad S(v):=-v, \qquad \mbox{for all } v\in V.
		\end{equation*}
		A complete proof of this fact can be found in \cite[Section 4.3.2]{DNR}.
	\end{example}
	
	\begin{example}[Symmetric algebra]\label{ex19}
		Let $V$ be a $\Bbbk$-vector space. A commutative algebra $S(V)$ is said to be a \emph{symmetric algebra of $V$} if there exists a linear map $l: V\rightarrow S(V)$ such that the following universal property is satisfied: \textit{for any commutative algebra $A$ and any linear map $h: V \rightarrow A$ there exists an unique algebra morphism $\overline{h}: S(V)\rightarrow A$ such that the following diagram is commutative:}
		\begin{equation*}
			\begin{tikzcd}
				V \arrow[d,"h"'] \arrow[r,"l"] & S(V) \arrow[dl,dashed,"\overline{h}"]\\
				A
			\end{tikzcd}
		\end{equation*}
		 A straightforward use of the universal property shows that $S(V)$ is unique up to isomorphism. The existence of symmetric algebras is shown by explicitly construction of $S(V)$ as the quotient $T(V)/I$, where $I=\langle u\otimes v - v\otimes u : u,v\in V \rangle$ \cite[Section 4.3.3]{DNR}. An alternative construction can be found in \cite[Section 15.1.18]{MR}. Since for all $u,v\in V$,
		\begin{align*}
			&\Delta_{T(V)}(u\otimes v - v\otimes u)\\
			 &= \Delta(u)\Delta(v)-\Delta(v)\Delta(u)\\& = (u\otimes 1 + 1\otimes u)(v\otimes 1 + 1\otimes v)-(v\otimes 1 + 1\otimes v)(u\otimes 1 - 1\otimes u)\\
			&= (u\otimes v - v\otimes u)\otimes 1 + 1\otimes (u\otimes v-v\otimes u),
		\end{align*}
		which is an element of $I \otimes T(V)+T(V)\otimes I$, and 
		\begin{gather*}
			\varepsilon_{T(V)}(u\otimes v - v\otimes u)=\varepsilon(u)\varepsilon(v)-\varepsilon(v)\varepsilon(u)=0,\\
			S_{T(V)}(u\otimes v - v\otimes u)=S(u)S(v)-S(v)S(u)=(-y)\otimes (-x)-(-x)\otimes(-y) \in I,
		\end{gather*}
		we have shown that $I$ is a Hopf ideal of $T(V)$. Hence, by Example~\ref{ex20}, $S(V)=T(V)/I$ is a (commutative) Hopf algebra with induced operations.
	\end{example}
	
	\begin{example}[Universal enveloping algebra of a Lie algebra]\label{ex17}
		A $\Bbbk$-vector space $\mathfrak{g}$ is a \emph{Lie algebra} if there exists a bilinear map $[-,-]: \mathfrak{g} \times \mathfrak{g} \rightarrow \mathfrak{g} $, called the \emph{Lie bracket}, such that the following conditions hold:
		\begin{enumerate}[label=(L\arabic*), align=parleft, leftmargin=*]
			\item\label{L1} (\emph{Antisymmetry}) $[x,y]=-[y,x]$, for all $x,y\in \mathfrak{g}$,
			\item\label{L2} (\emph{Jacobi identity}) $[[x,y],z]+[[z,x],y]+[[y,z],x]=0$, for all $x,y,z\in \mathfrak{g}$.
		\end{enumerate}
		The Lie algebra $\mathfrak{g}$ is called \emph{Abelian} if $[x,y]=0$ for every $x,y\in \mathfrak{g}$. In general the Lie bracket is not associative. Moreover,~\ref{L1} implies $[x,x]=0$, for all $x\in\mathfrak{g}$. For example, $\mathbb{R}^3$ equipped with the usual vector product is a $\mathbb{R}$-Lie algebra.
		
		Any (associative) algebra $A$ can be endowed with a $\Bbbk$-Lie algebra structure with $[a,b]:=ab-ba$, for all $a,b\in A$. In this example we shall consider the converse construction, i.e., an associative algebra rising from a given Lie algebra. The importance of this construction is well known, going from representation theory (e.g. \cite{Hum}), construction of Verma modules (e.g. \cite[Section 9.5]{Hall}) or characterization of left-invariant differential operators (e.g. \cite[Chapter II]{Hel}), to cocommutative cases of quantum groups (e.g. \cite{Kas}).
		
		If $\mathfrak{g}_1$ and $\mathfrak{g}_2$ are two Lie algebras, a linear map $f:\mathfrak{g}_1\rightarrow\mathfrak{g}_2$ is an \emph{morphism of Lie algebras} if $		f([x,y])=[f(x),f(y)]$, for all $x,y\in \mathfrak{g}_1$. In particular, if $\mathfrak{g}_2=A$ is an associative algebra endowed with the Lie bracket mentioned above, we say that the map $f:\mathfrak{g}_1\rightarrow A$ is a \emph{representation} (of $\mathfrak{g}_1$).
		
		Let $\mathfrak{g}$ be a Lie algebra. An associative algebra $U(\mathfrak{g})$ is an \emph{universal enveloping algebra of $\mathfrak{g}$}, if there exists a representation $f:\mathfrak{g}\rightarrow U(\mathfrak{g})$ such that the following universal property is satisfied: \textit{for any associative algebra $A$ and any representation $h: \mathfrak{g} \rightarrow A$ there exists an unique algebra morphism $\overline{h}: U(\mathfrak{g})\rightarrow A$ such that the following diagram is commutative:}
		\begin{equation*}
			\begin{tikzcd}
				\mathfrak{g} \arrow[d,"h"'] \arrow[r,"f"] & U(\mathfrak{g}) \arrow[dl,dashed,"\overline{h}"]\\
				A
			\end{tikzcd}
		\end{equation*}
		It follows that $\overline{h}$ is also a Lie algebra map and that $U(\mathfrak{g})$ is unique up to isomorphism. The existence of such enveloping algebra is given by explicitly constructing $U(\mathfrak{g})$ as the quotient algebra $\Bbbk\langle X\rangle/I$ , where $X=\{x_i\}_i$ is a basis of $\mathfrak{g}$, $\Bbbk\langle X\rangle$ is the free $\Bbbk$-algebra over $X$ and $I=\langle x_ix_j-x_jx_i-[x_i,x_j] : x_j,x_i\in X  \rangle$ \cite[Section V.2]{Kas}. An alternative construction \cite[Section 4.3.4]{DNR} is done by taking $U(\mathfrak{g})$ as the quotient $T(\mathfrak{g})/J$, where $T(\mathfrak{g})$ is the tensor algebra of $\mathfrak{g}$ and
		\begin{equation}\label{e39}
			J=\langle [x,y]-x\otimes y + y\otimes x : x,y\in \mathfrak{g} \rangle.
		\end{equation}
		In either case, the Poincar\'e--Birkhoff--Witt Theorem establishes that if there exists a total order $\preceq$ in $X$, then the set containing 1 and all elements of the form $x_{i_1}\cdots x_{i_n}$, with $ x_{i_1}\preceq \cdots \preceq x_{i_n}$, constitutes a $\Bbbk$-basis of $U(\mathfrak{g})$ (e.g. \cite[Theorem V.3]{Jac2}). The rules
		\begin{equation}\label{e40}
			\Delta(x)=x\otimes 1 + 1\otimes x, \quad \varepsilon(x)=0, \quad \varepsilon(x)=-x, \quad \mbox{for all } x\in \mathfrak{g},
		\end{equation}
		can be extended to $U(\mathfrak{g})$ by applying the universal property or, alternatively, verifying that $J$ as in \eqref{e39} is a Hopf ideal. Either way, \eqref{e40} makes $U(\mathfrak{g})$ a cocommutative Hopf algebra.
	\end{example}
	
	\begin{example}[Universal enveloping algebra of $\mathfrak{sl}_2(\Bbbk)$]\label{ex35}
		Denote by $\mathfrak{gl}_2(\Bbbk)$ the $\Bbbk$-algebra consisting of all $n\times n$ matrices with entries in $\Bbbk$ seen as a Lie algebra. One can easily check that the elements
		\begin{equation*}
			x=\left(\begin{array}{c c}
				0 & 1\\
				0 & 0
			\end{array}\right), \qquad
			y=\left(\begin{array}{c c}
				0 & 0\\
				1 & 0
			\end{array}\right), \qquad
			h=\left(\begin{array}{c c}
				1 & 0\\
				0 & -1
			\end{array}\right), \qquad
			i=\left(\begin{array}{c c}
				1 & 0\\
				0 & -1
			\end{array}\right),
		\end{equation*}
		form a basis for $\mathfrak{gl}_2(\Bbbk)$. Moreover, $[x,y]=h$, $[h,x]=2x$, $[h,y]=-2y$, and $[i,x]=[i,y]=[i,h]=0$. We denote by $\mathfrak{sl}_2(\Bbbk)$ the subspace of all matrices with null trace. A basis is $\{x,y,h\}$. For the particular case $\Bbbk=\mathbb{C}$ a detailed study of this Lie algebra can be found in \cite[Chapter V]{Kas}. By the previous example, $U(\mathfrak{sl}_2(\Bbbk))$ can be seen as the Hopf algebra generated by $x,y,h$ subject to the relation $[x,y]=h$, $[h,x]=2x$ and $[h,y]=-2y$. In Part~\ref{ch2}, we will endow this algebra with another structure, evidencing that a single object can be enriched with several distinct structures.
	\end{example}
	
	For the next example, recall that $\omega \in \Bbbk$ is said to be a \emph{$n$-th root of unity} ($n \in \mathbb{Z}^+$) if $\omega^n=1$. Furthermore, $\omega$ is \emph{primitive} if it is not a $k$-th root of unity for some $k<n$.
	
	\begin{example}[Taft Hopf algebra]\label{ex33}
		Given a $n$-th root of unity $\omega$ in $\Bbbk$, the $n^2$-dimensional \emph{Taft Hopf algebra} is given as an algebra by $T_{n^2}(\omega) = \Bbbk\langle g,x\rangle / \langle g^n-1,  x^n, xg-\omega gx \rangle$. $T_{n^2}(\omega)$ acquires structure of non-(co)commutative Hopf algebra via
		\begin{gather*}
			\Delta(g)=g\otimes g, \qquad \varepsilon(g)=1, \qquad S(g)=g^{-1},\\
			\Delta(x)=x\otimes 1 + g\otimes x,\qquad \varepsilon(x)=0, \qquad S(x)=-g^{-1}x.
		\end{gather*}
		Since the construction depends on the choice of $\omega$, there are $\Phi(n)$ non-isomorphic Taft Hopf algebras for each dimension $n^2$, where $\Phi$ denotes Euler's totient function. These Hopf algebras were constructed as examples of finite dimensional Hopf algebras having antipodes of arbitrarily high order, since in $T_{n^2}(\omega)$, $S$ has order $2n$ \cite{Taf}. The case $n=2$ is also known as the \emph{Sweedler Hopf algebra}.
	\end{example}
	
	\begin{example}[Quantum enveloping algebra of $\mathfrak{sl}_2(\Bbbk)$]\label{ex32}
		Let $q\in \Bbbk$ be an invertible element such that $q \neq \pm 1$. We define $U_q:=U_q(\mathfrak{sl}_2(\Bbbk))$ as the algebra generated by $e,f,k,k^{-1}$ subject to the relations
		\begin{gather*}
			kk^{-1}=k^{-1}k=1, \qquad kek^{-1}=q^2e, \qquad kfk^{-1}=q^{-2}f, \\
			\left[e,f\right]=ef-fe=\frac{k-k^{-1}}{q-q^{-1}}.
		\end{gather*}
		It can be shown that $\{ e^i f^j k^l : i,j\in\mathbb{N},\, l\in\mathbb{Z} \}$ is a basis \cite[Proposition VII.1.1]{Kas}. For simplicity, set $\Bbbk=\mathbb{C}$ and $q\in \mathbb{C}$ not being a root of unity. Hence $U_q$ is a $\mathbb{C}$-Hopf algebra with the operations induced by
		\begin{gather*}
			\Delta(e):=1\otimes e + e\otimes k, \qquad	\Delta(f):=k^{-1}\otimes f + f\otimes 1,\qquad	\Delta(k):=k\otimes k,\\
			\Delta(k^{-1}):= k^{-1} \otimes k^{-1}, \qquad \varepsilon(e)=\varepsilon(f)=0, \qquad	\varepsilon(k)=\varepsilon(k^{-1})=1,\\
			S(e):=-ek^{-1}, \qquad S(f):=-kf, \qquad S(k):=k^{-1}, \qquad S(k^{-1}):=k.
		\end{gather*}
		Moreover, if $q^2$ is a $n$-th primitive root of unity, the elements $k^n-1$, $e^n$ and $f^n$ are skew-primitive. Hence the ideal generated by them is a Hopf ideal \cite[Proposition 1.7]{Kha} and thus $U_q':=U_q/\langle k^n-1,e^n,f^n \rangle$ is a Hopf algebra, known as the \emph{Frobenius-Lusztig kernel}.
	\end{example}
	
	\begin{example}[Circle Hopf algebra]\label{ex30}
		Let $H_\Bbbk$ be the algebra defined as $H_\Bbbk:=\Bbbk\langle c,s \rangle /I$, with $I=\langle c^2+s^2-1,cs\rangle$. Then $H_\Bbbk$ is a Hopf algebra via
		\begin{gather*}
			\Delta(c)=c\otimes c - s\otimes s, \qquad \varepsilon(c)=1, \qquad S(c)=c,\\
			\Delta(s)=c\otimes s + s\otimes c,\qquad \varepsilon(s)=0, \qquad S(s)=-s.
		\end{gather*}
		As we shall see in Example~\ref{hge3}, this algebra naturally appears in some examples of separable field extensions not being Galois, but still satisfying the defining condition of a Hopf Galois extension.
	\end{example}
	
	\subsection{Modules and comodules}\label{sec2.4}
	
	The aim of this section is to introduce (co)actions of Hopf algebras over arbitrary $K$-algebras; these are of utmost importance for Hopf Galois extensions. Therefore, we review the notions of (co)module over an algebra and (co)module algebra. Throughout this section $A$ will denote an arbitrary $K$-algebra, while $C$ a $K$-coalgebra.
	
	Recall that a \emph{left $A$-module} is a $K$-module $M$ together with a $K$-linear map $\gamma: A\otimes M \rightarrow M$, called the \emph{scalar product map} of $M$, such that  $\gamma(m\otimes \operatorname{id}_M)=\gamma (\operatorname{id}_A \otimes \gamma)$ and $\gamma(u\otimes \operatorname{id}_M)(k \otimes x)=k x$, for all $k\in K$ and $x\in M$. We write $a\cdot x := \gamma(a\otimes x)$, so the above means that $(ab)\cdot x=a\cdot(b\cdot x)$ and $1\cdot x=x$, for all $x\in M$ and $a,b\in A$. Right modules over $A$ are defined similarly, the difference being that the scalar product map has the form $\gamma:M\otimes A\rightarrow M$.
	
	Similarly, a \emph{right $C$-comodule} is a $K$-module $N$ together with a $K$-linear map $\rho: N \rightarrow N\otimes C$, called the \emph{structure map} of $N$, such that $			(\operatorname{id}_N \otimes \Delta)\rho = (\rho \otimes \operatorname{id}_C)\rho$ and $(\operatorname{id}_N \otimes \varepsilon)\rho(n)=n\otimes 1$, for all $n\in N$. Left comodules over $C$ are defined in the same way, having structure map of the form $\rho:N\rightarrow C\otimes N$.
	
	We extend Heyneman--Sweedler notation to comodules. Let $N$ be a right $C$-comodule with structure map $\rho: N \rightarrow N\otimes C $. For any $n\in N$, the element $\rho(n)$ of $N\otimes C$ shall be written as $\rho(n)= n_{(0)} \otimes n_{(1)}$, fixing the convention that $n_{(j)}\in C$ for $j\neq 0$. With this, the defining properties of a right comodule may be written as
	\begin{gather}
		(n_{(0)})_{(0)} \otimes (n_{(0)})_{(1)} \otimes n_{(1)} = n_{(0)} \otimes (n_{(1)})_{(1)} \otimes (n_{(1)})_{(2)}=n_{(0)} \otimes n_{(1)} \otimes n_{(2)} , \nonumber\\
		\varepsilon(n_{(1)})n_{(0)}=n.\label{e32}
	\end{gather}
	Likewise, if $N$ is a left $C$-comodule with structure map $\rho:N\rightarrow C\otimes N$, preserving the convention for non-zero indexes, we write $\rho(n)= n_{(-1)}\otimes n_{(0)}$.
	
	Given two right $C$-comodules $N$ and $L$, with structure maps $\rho_N$ and $\rho_L$, respectively, a $K$-linear map $g:N\rightarrow Y$ is a \emph{comodule morphism} if $g(n)_{(0)} \otimes g(n)_{(1)} = g(n_{(0)}) \otimes n_{(1)}$, for all $n\in N$. We denote the  category of left (resp. right) $A$-modules by $ {}_A{\operatorname{Mod}}$ (resp. $\operatorname{Mod}_A$). Similarly, the category of right (resp. left) comodules over a coalgebra $C$ is denoted by $\operatorname{Mod}^C$ (resp. $ {}^C {\operatorname{Mod}}$).
	
	The definition of a bimodule over an algebra can also be dualized. Indeed, given two coalgebras $C,D$, a $K$-module $N$ is said to be a \emph{$(D,C)$-bicomodule} if $N$ is a left $D$-comodule with structure map $\mu: N \rightarrow D\otimes N$, $N$ is a right $C$-comodule with structure map $\rho:N\rightarrow N\otimes C$, and $(\mu\otimes \operatorname{id}_C)\rho=(\operatorname{id}_D \otimes \rho)\mu$.
	The later condition may be written in Heyneman--Sweedler notation as
	\begin{equation*}
		(n_{(0)})_{{(-1)}} \otimes (n_{(0)})_{(0)} \otimes n_{(1)} = n_{{(-1)}} \otimes (n_{(0)})_{(0)} \otimes (n_{(0)})_{(1)}, \qquad \mbox{for all } n\in N.
	\end{equation*}
	A \emph{morphism of bicomodules} is a linear map between two $(D,C)$-bicomodules which is both a morphism of left $D$-comodules and a morphism of right $C$-modules. Hence, we can define the correspondent category, which is denoted by ${}^D {\operatorname{Mod}}^C$. Similarly, given two algebras $A,B$, the category of $(B,A)$-bimodules is denoted by $ {}_B {\operatorname{Mod}}_A$.
	
	Let $H$ be a Hopf algebra. For a left $H$-module $M$, \emph{the set of invariants of $H$ on $M$} is
			\begin{equation*}
				M^H:=\{ m\in M : h\cdot m = \varepsilon(h)m ,\ \forall h\in H  \}.
			\end{equation*} 
	Similarly, for a right $H$-comodule $N$ with structure map $\rho : N \rightarrow N\otimes H$, \emph{the set of coinvariants of $H$ on $N$} is given by
			\begin{equation*}
				N^{\operatorname{co}H}:=\{ n\in N : \rho(n)=n\otimes 1  \}.
			\end{equation*}
	When the base ring is a field a natural question is whether there exists a relation between the comodules of $H$ and the modules of the dual Hopf algebra $H^*$. The following result shows that, at least in the finite-dimensional case, there is such a correspondence preserving (co)invariants.
	
	\begin{proposition}[{e.g. \cite[Lemma 1.7.2]{Mon1}}]\label{ex9}
		Let $H$ be a finite-dimensional $\Bbbk$-Hopf algebra and $H^*$ its dual Hopf algebra. Then, for any $\Bbbk$-vector space $N$, the following assertions are equivalent:
		\begin{enumerate}[label=\normalfont(\roman*)]
			\item $N$ is a right $H$-comodule.
			\item $N$ is a left $H^*$-module.
		\end{enumerate}
		Moreover, under these conditions, $N^{H^*}=N^{\operatorname{co}H}$.
	\end{proposition}
	
	\begin{proof}
	Let $\{ e_1,\ldots,e_n \}$ be a basis for $H$ and $\{ e_1^*,\ldots,e_n^* \}$ the corresponding dual basis for $H^*$ (i.e., $e_i^*(e_j)=\delta_{ij}$, the Kronecker delta). If $N$ is a right $H$-comodule, then $N$ becomes a $H^*$-module via
	\begin{equation}\label{e18}
		f\cdot n := f(n_{(1)})n_{(0)}, \qquad \mbox{for all } f\in H^*, \, n\in N.
	\end{equation}
	Reciprocally, if $N$ is a left $H^*$-module, then $N$ becomes a right $H$-comodule with structure map $\rho: N\rightarrow N\otimes H$ via
	\begin{equation}\label{e19}
		\rho(a):=\sum_{i=1}^n e_i^* \cdot a \otimes e_i.
	\end{equation}
		We omit the details in these two implications since these are a straightforward verification of the defining conditions. Finally, using the notation of Example~\ref{ex5}, we have
	\begin{align*}
		N^{H^{*}} &= \{ n\in N : f\cdot a = u^\circ(f) n, \, \forall f\in H^{*} \}\\
		&= \{ n\in N : f(n_{(1)})n_{(0)} = (fu)(1_\Bbbk)n, \forall f\in H^{*} \}\\
		&= \left\{ n\in N : n_{(0)} f(n_{(1)}) = f(1_H)n, \forall f\in H^{*} \right\}\\
		&= \left\{ n\in N : (\operatorname{id}_N\otimes f)(\rho(n)) = (\operatorname{id}_N \otimes f)(n\otimes 1), \forall f\in H^{*} \right\}\\
		&= \{ n\in N : \rho(n)=n\otimes 1 \}= N^{\operatorname{co}H},
	\end{align*}
which shows the desired equality.
	\end{proof}
	
	Now, we review some essential examples of modules and comodules.
	
	\begin{example}
		Any $K$-algebra $A$ is a left module over itself by taking $\gamma=m_A$ (in other words, $a\cdot b=ab$, for all $a,b\in A$). Similarly, any $K$-coalgebra $C$ is a right comodule over itself by taking $\rho=\Delta_C$.
	\end{example}
	
	\begin{example}\label{ex31}
		Let $H$ be a $K$-Hopf algebra and let $V,W$ be two left $H$-modules. Then $V\otimes W$ has also structure of left $H$-module via
		\begin{equation*}
			h\cdot (v\otimes w) := (h_{(1)}\cdot v)\otimes (h_{(2)}\cdot w), \qquad \mbox{for all } h\in H, v\in V, w\in W.
		\end{equation*}
		If $\gamma_V$ and $\gamma_W$ are the respective structure maps of $V$ and $W$, the above means that the structure map $\gamma_{V\otimes W}$ is defined as the composition
		\begin{equation*}
			\gamma_{V\otimes W} := (\gamma_V\otimes \gamma_W)(\operatorname{id}_H \otimes \tau_{H,V} \otimes \operatorname{id}_W)(\Delta \otimes \operatorname{id}_V \otimes \operatorname{id}_W).
		\end{equation*}
	\end{example}
	
	\begin{example}\label{ex4}
		Let $H$ be a $K$-Hopf algebra and let $V,W$ two right $H$-comodules with respective structure maps $\rho_V$ and $\rho_W$. Then $V\otimes W$ is also a right $H$-comodule by taking $\rho_{V\otimes W}$ as the composition
		\begin{equation*}
			\rho_{V\otimes W} = (\operatorname{id}_V \otimes \operatorname{id}_W \otimes m)(\operatorname{id}_V\otimes \tau_{H,W} \otimes \operatorname{id}_H)(\rho_V \otimes \rho_W),
		\end{equation*}
	i.e., $	\rho_{V\otimes W}( v \otimes w ) = v_{(0)} \otimes w_{(0)} \otimes v_{(1)}w_{(1)}$, for all $v\in V$ and $w\in W$.
	\end{example}

	The previous examples implicitly describe the structure of monoidal category that both ${}_H {\operatorname{Mod}}$ and $\operatorname{Mod}^H$ possess (see e.g. \cite[Section~10.4]{Mon1} for the definition).
	
	We end this section by discussing Hopf modules. Let $H$ be a $K$-Hopf algebra. Then, similarly to $H$ being both an algebra and a coalgebra with certain compatibility, a Hopf module over $H$ will be both an $H$-module and an $H$-comodule in which the structure map is a module map. Namely, a $K$-module $M$ is a \emph{right-right $H$-Hopf module} if $M$ is a right $H$-module, $M$ is a right $H$-comodule with structure map $\rho: M \rightarrow M\otimes H$, and $\rho(m\cdot h)= (m \cdot  h)_{(0)} \otimes (m \cdot  h)_{(1)} = m_{(0)} \cdot h_{(1)} \otimes m_{(1)} h_{(2)}$, for all $m\in M$ and $h\in H$. 
		
	In the first defining condition, $H$ may be replaced by a Hopf subalgebra $L$ of $H$. In this case we say that $M$ is a \emph{right-right $(H,L)$-Hopf module}. The category of all right-right $(H,L)$-Hopf modules is denoted by $\operatorname{Mod}_L^H$, in which morphisms are $K$-linear maps also being both morphisms of right $L$-modules and morphisms of right $H$-comodules. Clearly, by changing laterality and modifying the compatibility condition, we also obtain the categories ${}^H {\operatorname{Mod}}_L$, ${}_L {\operatorname{Mod}}^H$ and $ {}^H_L {\operatorname{Mod}}$.
	
	\begin{example}
		Any $K$-Hopf algebra $H$ is a $H$-Hopf module via $\rho=\Delta$.
	\end{example}
	
	\begin{example}[Trivial Hopf module]
		Let $M$ be any right $H$-module. Then $M\otimes H$ is a right-right $H$-Hopf module using $\rho=\operatorname{id}_M \otimes \Delta$. A special case of this is when $M$ is the \emph{trivial $H$-module}, that is, $m\cdot h = \varepsilon(h)m$, for all $m\in M$ and $h\in H$. In this situation, $M\otimes H$ is called the \emph{trivial Hopf module}.
	\end{example}
	
	Recall that the fundamental theorem of Hopf modules classify all Hopf modules as trivial, i.e., if $M$ is a right-right $H$-Hopf module, then $M \cong M^{\operatorname{co}H}\otimes H$ as right-right $H$-Hopf module, where $M^{\operatorname{co}H}\otimes H$ has the trivial structure of Hopf module (e.g. \cite[Theorem 1.9.4]{Mon1}).
	
	\subsection{(Co)module algebras}\label{sec2.6}
	
	In this section we present (co)actions of Hopf algebras over algebras. Although most results and definitions presented are valid over bialgebras, throughout this part $H$ will denote an arbitrary $K$-Hopf algebra (remember that we assume $H$ flat over $K$).
	
	\begin{definition}[Module algebra, comodule algebra]\label{d6}
	Let $A$ be a $K$-algebra.
	\begin{enumerate}[label=\normalfont(\roman*)]
		\item $A$ is a \emph{left $H$-module algebra} if the following conditions hold:
		\begin{enumerate}[label=\normalfont (MA\arabic*), align=parleft, leftmargin=*]
			\item\label{MA1} $A$ is a left $H$-module,
			\item\label{MA2} For all $h\in H$ and $a,b\in A$,
			\begin{equation}\label{e11}
				h\cdot (ab) = (h_{(1)} \cdot a)(h_{(2)}\cdot b) \qquad \mbox{and} \qquad h\cdot 1_A = \varepsilon(h) 1_A.
			\end{equation}
		\end{enumerate}
		\item $A$ is a \emph{right $H$-comodule algebra} if the following conditions hold:
		\begin{enumerate}[label=\normalfont (CA\arabic*), align=parleft, leftmargin=*]
			\item\label{CA1} $A$ is a right $H$-comodule with structure map $\rho:A\rightarrow A\otimes H$,
			\item\label{CA2} For all $a,b\in A$, 
			\begin{equation}\label{e12}
				\rho(ab)= a_{(0)}b_{(0)} \otimes a_{(1)}b_{(1)} \qquad \mbox{and} \qquad \rho(1_A)=1_A\otimes 1_H.
			\end{equation}
		\end{enumerate}
	\end{enumerate}
	\end{definition}

Condition~\ref{MA2} is equivalent to $m_A$ and $u_A$ being $H$-module maps. Dually, condition~\ref{CA2} is equivalent to $m_A$ and $u_A$ being $H$-comodule maps, which is also equivalent to $\rho$ being an algebra map (see e.g. \cite[Propositions~6.1.4 and 6.2.2]{DNR}). Also, as with previous concepts, we may define \emph{right $H$-module algebras}, \emph{left $H$-comodule algebras} and \emph{$(L,H)$-bi(co)module algebras} similarly.

\begin{remark}
	The set of coinvariants in a comodule algebra is always a subalgebra, called the \emph{subalgebra of coinvariants}. Indeed, if $A$ is a right $H$-comodule algebra, $k\in K$ and $a,b\in A^{\operatorname{co}H}$, then
	\begin{gather*}
		\rho(a+b)=\rho(a)+\rho(b)=a\otimes 1 + b\otimes 1 = (a+b)\otimes 1, \\
		\rho(ka)=k\rho(a)=k(a\otimes 1)=ka\otimes 1, \\
		\rho(ab)=ab\otimes 1.
	\end{gather*}
	In the first two rows we used the linearity of $\rho$, while for the last one we used \eqref{e12}.
\end{remark}
	
The next result extends Proposition~\ref{ex9}.
	
	\begin{proposition}[e.g. {\cite[Proposition 6.2.4]{DNR}}]\label{p2}
		Let $H$ be a finite-dimensional $\Bbbk$-Hopf algebra and $H^*$ its dual Hopf algebra. Then, for a $\Bbbk$-algebra $A$, the following assertions are equivalent:
		\begin{enumerate}[label=\normalfont(\roman*)]
			\item $A$ is a right $H$-comodule algebra.
			\item $A$ is a left $H^*$-module algebra.
		\end{enumerate}
		Moreover, under these conditions $A^{H^*}=A^{\operatorname{co}H}$.
	\end{proposition}
	
	\begin{proof}
		Let $\{ e_1,\ldots,e_n \}$ be a basis for $H$ and $\{ e_1^*,\ldots,e_n^* \}$ the corresponding dual basis for $H^*$. If $A$ is a right $H$-comodule algebra, then we already know that $A$ is a $H^*$-module via \eqref{e18}. Moreover, if $a,b\in A$ and $f\in H^*$, we have
		\begin{align*}
			f\cdot (ab) &= f((ab)_{(1)})(ab_{(0)}) = f(a_{(1)}b_{(1)})a_{(0)}b_{(0)} =  (fm_H)(a_{(1)} \otimes b_{(1)})a_{(0)}b_{(0)}\\
			&= \Delta^\circ(f)(a_{(1)}\otimes b_{(1)})a_{(0)}b_{(0)} = f_{(1)}(a_{(1)})f_{(2)}(b_{(1)})a_{(0)}b_{(0)}\\
			& = f_{(1)}(a_{(1)})a_{(0)} f_{(2)}(b_{(1)})b_{(0)} = (f_{(1)}\cdot a)(f_{(2)}\cdot b);\\
			f\cdot 1_A &= f(1_A)1_A= (fu_H)(1_H)1_A = u^\circ(f)1_A.
		\end{align*}
		Therefore, $A$ is a left $H^*$-module algebra.
		
		Conversely, let $A$ be a left $H^*$-module algebra. We already know that $A$ is a $H$-comodule via \eqref{e19}. Moreover, if $a,b\in A$ and $f\in H^*$, we have
		\begin{align*}
		&(\operatorname{id}_A\otimes f)(\rho(ab))\\
		&= \sum_{i=1}^n e_i^* \cdot (ab) \otimes f(e_i) = \sum_{i=1}^n (e_i^* \cdot (ab)) f(e_i) \otimes 1 = \sum_{i=1}^n (e_i^* f(e_i)) \cdot (ab) \otimes 1 \\
		&= f \cdot (ab) \otimes 1 = (f_{(1)}\cdot a)(f_{(2)}\cdot b) \otimes 1 = \sum_{i,j=1}^n ( (e_i^* f_{(1)}(e_i)) \cdot a )( (e_j^* f_{(2)}(e_j)) \cdot b ) \otimes 1\\
		& = \sum_{i,j=1}^n ( e_i^*  \cdot a )(e_j^* \cdot b ) \otimes f_{(1)}(e_i)f_{(2)}(e_j) =  \sum_{i,j=1}^n ( e_i^*  \cdot a )(e_j^* \cdot b ) \otimes f(e_ie_j)\\
		&= (\operatorname{id}_A \otimes f)\left( \sum_{i,j=1}^n (e_i^* \cdot a)(e_j^*\cdot b) \otimes e_ie_j \right) =(\operatorname{id}_A \otimes f)(\rho(a)\rho(b)),
		\end{align*}
		and hence $\rho(ab)=\rho(a)\rho(b)$. On the other hand,
		\begin{align*}
			\rho(1)&= \sum_{i=1}^n e_i^* \cdot 1_A \otimes e_i = \sum_{i=1}^n e_i^*(1_H) 1_A \otimes e_i \\
			&= \sum_{i=1}^n 1_A \otimes e_i^*(1_H) e_i = 1_A \otimes \sum_{i=1}^n e_i^*(1_H) e_i = 1_A \otimes 1_H.
		\end{align*}
		Thus, $A$ is a right $H$-comodule algebra.
		
		The equality $A^{H^*}=A^{\operatorname{co}H}$ follows as in the proof of Proposition~\ref{ex9}.
	\end{proof}
	
	Now, we generalize the notion of Hopf module by replacing the module structure over $H$ (or a Hopf subalgebra $L$) by a comodule algebra.
	
	\begin{definition}\label{d16}
		Let $A$ be a right $H$-comodule algebra. A $K$-module $M$ is said to be a \emph{left-right $(A,H)$-Hopf module} if the following conditions hold:
		\begin{enumerate}[label=\normalfont (HM\arabic*), align=parleft, leftmargin=*]
			\item\label{HM1} $M$ is a left $A$-module,
			\item\label{HM2} $M$ is a right $H$-comodule with structure map $\rho_M: M \rightarrow M \otimes H$,
			\item\label{HM3} For every $a\in A$ and $m\in M$, $\rho_M(a \cdot m)= a_{(0)} \cdot m_{(0)} \otimes a_{(1)} m_{(1)}$.
		\end{enumerate}
	\end{definition}
	
	The category of left-right $(A,H)$-Hopf modules, ${}_A{\operatorname{Mod}}^H$, has as morphisms the $K$-linear maps which are also $A$-linear and $H$-colinear. Similarly, objects in the category ${\operatorname{Mod}}^H_A$ can be define by replacing~\ref{HM3} with $\rho_M(m \cdot a)= m_{(0)} \cdot a_{(0)} \otimes m_{(1)}a_{(1)}$.
	
	We end this section by giving some examples characterizing actions and coactions of distinguished Hopf algebras. We assume that the base ring is a field $\Bbbk$.
	
	\begin{example}[(Co)actions of a Hopf algebra over itself]\label{ex34}
		It is clear that $H$ is a right $H$-comodule algebra using $\rho=\Delta$. By Proposition~\ref{p2}, when $H$ if finite dimensional, this dualizes to a left action (denoted by $\rightharpoonup$) of $H^*$ on $H$ given by
		\begin{equation*}
			f\rightharpoonup h = f(h_{(2)})h_{(1)}, \qquad \mbox{for all } h\in H \mbox{ and } f\in H^*.
		\end{equation*}
		The (co)invariants are given by $H^{H^*}=H^{\operatorname{co}H}=\Bbbk 1$.
	\end{example}
	
	\begin{example}[Actions of the group algebra]\label{ex8}
		Let $G$ be a group. We say that $G$ \emph{acts (from the left) as automorphisms} on a $\Bbbk$-algebra $A$ if there is a group morphism $\psi: G \rightarrow \operatorname{Aut}_{K-\operatorname{Alg}}(A)$. In this case, we write $\psi(g)(a)=g(a)$ (or $g^a$), for all $g\in G$ and $a\in A$. Moreover, if $\psi$ is injective, we say that $G$ acts \emph{faithfully}.
		
		If $G$ acts as automorphisms on a $A$ and $\Bbbk G$ is the group algebra of $G$ (see Example~\ref{ex6}), then $A$ is a left $\Bbbk G$-module algebra via $g\cdot a = g(a)$, for all $g\in G$ and $a\in A$. Indeed, for every $g\in G$ and $a,b\in A$, we have
		\begin{gather*}
			g\cdot(ab)=g(ab)=g(a)g(b)=(g\cdot a)(g\cdot b),\\
			g\cdot 1_A = g(1_A)=1_A = 1_A1_A= \varepsilon(g)1_A.
		\end{gather*}
		In this case, $A^{\Bbbk G}$ is the set of fixed points under the action of $G$,
		\begin{equation*}
			A^{\Bbbk G}=A^G:=\{ a\in A : g(a)=a,\, \forall g\in G \}.
		\end{equation*}
		Conversely, if $A$ is a $\Bbbk G$-module algebra, then $G$ acts as automorphisms on $A$ via the map $\psi: G \rightarrow \operatorname{Aut}_{K-\operatorname{Alg}}(A)$, given by $\psi(g)(a)= g\cdot a$.
		
		Thus, we have shown that $A$ is a $\Bbbk G$-module algebra if and only if $G$ acts as automorphisms on $A$.
	\end{example}
	
	\begin{example}[Coactions of the group algebra]\label{ex7}
		Let $G$ be a group and let $A$ be a $\Bbbk$-algebra. We say that $A$ is a \emph{$G$-graded algebra} if there exists a collection $\{ A_g \}_{g\in G}$ of $\Bbbk$-subspaces of $A$ such that
		$A=\bigoplus_{g\in G} A_g$ and $A_gA_h \subseteq A_{gh}$, for all $g,h\in G$.
		
		If $A$ is a $\Bbbk G$-comodule algebra with structure map $\rho: A \rightarrow A \otimes \Bbbk G$, then we replace the Heyneman--Sweedler notation by writing instead
		$\rho(a)=\sum a_g \otimes g$, for all $a\in A$. By definition, $[(\operatorname{id}_A \otimes \Delta)\rho](a) = [(\rho \otimes \operatorname{id}_{\Bbbk G})\rho](a)$, for all $a\in A$. Expanding the left-hand side we get
		\begin{equation*}
			[(\operatorname{id}_A \otimes \Delta)\rho](a) = (\operatorname{id}_A \otimes \Delta) \left( \sum a_g \otimes g \right) = \sum a_g \otimes g\otimes g,
		\end{equation*}
		while the right-hand side gives
		\begin{equation*}
			[(\rho \otimes \operatorname{id}_{\Bbbk G})\rho](a) = (\rho \otimes \operatorname{id}_{\Bbbk G}) \left( \sum a_g \otimes g \right) = \sum (a_g)_h \otimes h \otimes g.
		\end{equation*}
		Comparing both expressions we have that
		\begin{equation*}
			(a_g)_h = \begin{cases*}
				a_g & \mbox{if } \mbox{$g=h$},\\
				0 & \mbox{if } \mbox{$g\neq h$}.
			\end{cases*}
		\end{equation*}
		Thus $\rho(a_g)=a_g\otimes g$ and we can set $A_g=\{ a_g : a\in A \}$, for all $g\in G$. Since $\rho$ is $\Bbbk$-linear we have $\rho(a+b)=\rho(a)+\rho(b)$ and $\rho(ka)=k\rho(a)$, for all $a,b\in A$ and $k\in\Bbbk$. Thus,
		$(a+b)_g = a_g + b_g$ and $(ka)_g=ka_g$, and therefore $A_g$ is a $\Bbbk$-subspace of $A$, for all $g\in G$. On the other hand, since $G$ is a basis of $\Bbbk G$, the sum $\sum_{g\in G} A_g$ is direct. Indeed, if $0=a_{g_1}+\cdots+a_{g_t}$ for some $a_{g_k} \in A_{g_k}$, $1\leq k \leq t$, by applying $\rho$ we get
		\begin{equation*}
			0\otimes 0=\rho(0)=\rho(a_{g_1}+\cdots+a_{g_t}) = \rho(a_{g_1})+\cdots+\rho(a_{g_t}) = a_{g_1}\otimes g_1 + \cdots + a_{g_t} \otimes g_t,
		\end{equation*}
		and thus $a_{g_k}=0$, for $1\leq k \leq t$. Additionally, since $A$ is a $\Bbbk G$-comodule algebra,
		\begin{equation*}
			\rho(a_gb_h)= a_gb_h\otimes gh, \qquad \mbox{for all } a_g\in A_g, \, b_h\in A_h.
		\end{equation*}
		Therefore $a_gb_h \in A_{gh}$ and $A_gA_h \subseteq A_{gh}$. Finally, by definition $[(\operatorname{id}_A \otimes \varepsilon)\rho](a)= a\otimes 1$, for all $a\in A$. But the left hand side is
		\begin{equation*}
			[(\operatorname{id}_A \otimes \varepsilon)\rho](a) = (\operatorname{id}_A \otimes \varepsilon) \left( \sum a_g\otimes g \right) = \sum a_g \otimes 1,
		\end{equation*}
		and thus for each $a\in A$, we have $\sum a_g = a$. Therefore $A=\bigoplus_{g\in G} A_g$ and $A$ is a $G$-graded algebra. Notice that
		\begin{equation*}
			A^{\operatorname{co}H}=\{ a\in A : \rho(a)=a\otimes 1 \}=A_1,
		\end{equation*}
		the identity component of the $G$-graduation. In particular, $1_A\in A_1$.
		
		Conversely, if $A$ is a $G$-graded algebra define the structure map $\rho: A \rightarrow A\otimes \Bbbk G$ as $\rho(a)=a\otimes g$, for all $a\in A_g$ and $g\in G$.
		
		Thus, we have shown that $A$ is a $\Bbbk G$-comodule algebra if and only if $A$ is a $G$-graded algebra.
	\end{example}
	
	\begin{example}[Actions of the dual group algebra]\label{ex10a}
		Let $G$ be a finite group and recall from Example~\ref{ex10} that $(\Bbbk G)^*=\Bbbk^G$. By Proposition~\ref{p2}, the actions of $\Bbbk^G$ correspond to coactions of $\Bbbk G$, which by Example~\ref{ex7} are precisely $G$-gradings.
	\end{example}
	
	\begin{example}[Actions of the universal enveloping algebra of a Lie algebra]
		Let $A$ be a $\Bbbk$-algebra. A $\Bbbk$-linear map $\delta: A \to A$ is a \emph{derivation} if $\delta(ab)=a\delta(b)+\delta(a)b$, for all $a,b\in A$. Then the space of $\Bbbk$-derivations over $A$, $\operatorname{Der}_{\Bbbk}(A)$, becomes a Lie algebra by taking as Lie bracket the commutator, i.e., $[\delta,\pi]=\delta \pi - \pi\delta$, for all $\delta,\pi\in \operatorname{Der}_{\Bbbk}(A)$. Indeed, for $a,b\in A$:
		\begin{align*}
			(\delta \pi - \pi \delta)(ab) &= \delta(\pi(ab))- \pi(\delta(ab)) = \delta(a \pi(b)) + \delta(\pi(a)b) - \pi(a\delta(b))-\pi(\delta(a)b)\\
			&= a \delta(\pi(b)) + \delta(a)\pi(b) + \pi(a)\delta(b) + \delta(\pi(a))b \\
			& \quad - a\pi(\delta(b)) - \pi(a)\delta(b) - \delta(a)\pi(b)- \pi(\delta(a))b\\
			&= a (\delta \pi - \pi \delta)(b) + (\delta \pi - \pi \delta)(a)b.
		\end{align*}
		Let $\mathfrak{g}$ be a Lie algebra. We say that $\mathfrak{g}$ \emph{acts by derivations} on $A$ if there exists a Lie algebra morphism $\psi: \mathfrak{g} \to \operatorname{Der}_{\Bbbk}(A)$.
		
		If $\mathfrak{g}$ acts by derivations on $A$ and $U(\mathfrak{g})$ is the universal enveloping algebra of $G$ (see Example~\ref{ex17}), then $A$ is a left $U(\mathfrak{g})$-module algebra via $x\cdot a = \psi(x)(a)$, for all $x\in \mathfrak{g}$ and $a\in A$ (extended by the PBW theorem). In this case, \begin{equation*}
			A^{U(\mathfrak{g})}:=A^{\mathfrak{g}}=\{ a\in A : x\cdot a =0, \mbox{ for all } x\in \mathfrak{g} \}.
		\end{equation*}
		Conversely, if $A$ is a $U(\mathfrak{g})$-module algebra, using \eqref{e11} we get $x\cdot (ab)= x\cdot (a) b + a(x\cdot b)$ and $x\cdot 1 = 0$. This implies that the map $\psi: \mathfrak{g} \to \operatorname{Der}_{\Bbbk}(A)$ given by $\psi(x)(a):=x\cdot a$ is a Lie algebra morphism and thus $\mathfrak{g}$ acts by derivations on $A$.
		
		Hence, we have shown that $A$ is a $U(\mathfrak{g})$-module algebra if and only if $\mathfrak{g}$ acts by derivations on $A$.		
	\end{example}
	
	\subsection{Hopf Galois extensions}\label{sec2.7}
	
	Finally, we are able to introduce one of the transversal concepts appearing in this document. Hopf Galois extensions generalize the notion of Galois extensions over rings, replacing the action of a group on the algebra by the coaction of a Hopf algebra. The first general definition is due to Kreimer and Takeuchi \cite{KT}, although the commutative case was previously studied by Chase and Sweedler \cite{CS}. As in the previous section, throughout the discussion $H$ will denote an arbitrary $K$-Hopf algebra.
	
	\begin{definition}[$H$-Galois extension]\label{d7}
		Let $A$ be a right $H$-comodule algebra with structure map $\rho:A\rightarrow A\otimes H$. We say that the extension $A^{\operatorname{co}H} \subset A$ (also denoted $A/A^{\operatorname{co}H}$) is a \emph{right $H$-Galois extension} if the map $\beta: A\otimes_{A^{\operatorname{co}H}} A \rightarrow A \otimes H$, given by
		\begin{equation}\label{e62}
			\beta(a\otimes b) = (a\otimes 1)\rho(b)= ab_{(0)}\otimes b_{(1)}, \qquad \mbox{for all } a,b\in A,
		\end{equation}
		is bijective. In this case $\beta$ is called the \emph{Galois map}.
	\end{definition}
	
	Notice that we could have defined the Galois map as $\beta': A\otimes_{A^{\operatorname{co}H}} A \rightarrow A \otimes H$ given by 
	\begin{equation}\label{e63}
		\beta'(a\otimes b) = \rho(a)(b\otimes 1)= a_{(0)}b\otimes a_{(1)}, \qquad \mbox{for all } a,b\in A.
	\end{equation}
	A natural question is whether $\beta'$ can replace $\beta$. The following result gives a case in which the answer is affirmative.
	
	\begin{proposition}[e.g. {\cite[p. 372]{Mon2}}]
		Let $A$ be a right $H$-comodule algebra with structure map $\rho:A\rightarrow A\otimes H$. If the antipode $S$ is bijective, then the following assertions are equivalent:
		\begin{enumerate}[label=\normalfont(\roman*)]
			\item The map $\beta$, given by \eqref{e62}, is bijective (resp. injective, surjective).
			\item The map $\beta'$, given by \eqref{e63}, is bijective (resp. injective, surjective).
		\end{enumerate}
	\end{proposition}
	
	\begin{proof}
		Let $\Phi: A\otimes H\rightarrow A\otimes H $ the endomorphism given by
		\begin{equation*}
			\Phi(a\otimes h) := \rho(a)(1\otimes S(h)) = a_{(0)} \otimes a_{(1)}S(h), \qquad \mbox{for all } a\in A, \, h\in H.
		\end{equation*}
		We have, for all $a,b\in A$,
		\begin{align*}
			(\Phi\beta)(a\otimes b) &=\Phi\left( ab_{(0)} \otimes b_{(1)} \right) = \Phi(ab_{(0)}\otimes b_{(1)}) = \rho(ab_{(0)}) (1\otimes S(b_{(1)})) \\
			&= (a_{(0)}b_{(0)} \otimes a_{(1)}b_{(1)})(1\otimes S(b_{(2)})) = a_{(0)}b_{(0)} \otimes a_{(1)}b_{(1)}S(b_{(2)})\\
			&= a_{(0)}b_{(0)} \otimes a_{(1)} \varepsilon(b_{(1)}) = a_{(0)} b_{(0)} \varepsilon(b_{(1)}) \otimes a_{(1)}\\
			&= a_{(0)}b \otimes a_{(1)} = \beta'(a\otimes b).
		\end{align*}
		Hence $\Phi \beta=\beta' $. But notice that $\Phi$ has as inverse $\Phi^{-1}: A\otimes H\rightarrow A\otimes H $ given by 
		\begin{equation*}
			\Phi^{-1}(a\otimes h) := (1\otimes S^{-1}(h))\rho(a) = a_{(0)} \otimes S^{-1}(h)a_{(1)} , \qquad \mbox{for all } a\in A, \, h\in H.
		\end{equation*}
		Therefore, the assertions are equivalent.
	\end{proof}
	
	Similarly, \emph{left $H$-Galois extensions} can be defined with Galois map $\beta_l: A\otimes_{A^{\operatorname{co}H}} A \rightarrow H \otimes A$. If several Hopf Galois extensions of different laterality are involved, we add to their Galois maps an index indicating whether these are left or right sided; for example, for a right Hopf Galois extension we may write $\beta_r$.
	
	When a right $H$-Galois extension is of the form $K \subset A$ (i.e., $A^{\operatorname{co}H}=K$), we call $A$ a \emph{right $H$-Galois object}.	In this situation, the Galois map is given by the composition
	\begin{equation*}
		\beta : \begin{tikzcd}
			A\otimes A \arrow[r,"\operatorname{id}_A\otimes \rho_A"] & A\otimes A \otimes H \arrow[r,"m_A \otimes \operatorname{id}_H"] & A \otimes H.
		\end{tikzcd}
	\end{equation*}
	Moreover, we have the following useful result.
	
	\begin{proposition}[{\cite[Theorem 1.15]{CS}}]\label{p11}
		Let $A$ be a right $H$-Galois object. Then $A$ is a faithfully flat $K$-module.
	\end{proposition}
	
	As with bi(co)modules, the notion of Galois objects may be two-sided with certain compatibility.
	
	\begin{definition}[BiGalois object]\label{d11}
		Let $H,L$ be two $K$-Hopf algebras. A $K$-algebra $A$ is a \emph{$(L,H)$-biGalois object} if the following assertions hold:
		\begin{enumerate}[label=(BG\arabic*), align=parleft, leftmargin=*]
			\item\label{BG1} $A$ is a faithfully flat $(L,H)$-bicomodule algebra,
			\item\label{BG2} $A$ is a left $L$-Galois extension of $K$,
			\item\label{BG3} $A$ is a right $H$-Galois extension of $K$.
		\end{enumerate}
	\end{definition}
	
	\subsection{Examples of Hopf Galois extensions}\label{sec2.8}
	
	In this section we review a large amount of examples of Hopf Galois extensions. Unless otherwise stated, they are adapted from \cite{DNR,Mon1,Mon2,Sch4}.
	
	\subsubsection{Hopf algebras}\label{hge1}
	
	Every $K$-Hopf algebra $H$ can be seen as a $K$-Galois object. Indeed, since $H$ has a natural structure of $H$-comodule algebra via the comultiplication (i.e., the structure map is $\rho:=\Delta$), for any $h\in H^{\operatorname{co}H}$,
	\begin{equation*}
		\rho(h)=\Delta(h) = h_{(1)} \otimes h_{(2)} = h\otimes 1.
	\end{equation*}
	Applying $\varepsilon \otimes \operatorname{id}_H$ we have $$(\varepsilon\otimes \operatorname{id}_H)( h_{(1)} \otimes h_{(2)} ) = (\varepsilon \otimes \operatorname{id}_H)(h\otimes 1),$$ whence $ \varepsilon(h_{(1)})1 \otimes h_{(2)} = \varepsilon(h)1\otimes 1$, which via the isomorphism $K \otimes H \cong H$, gives $$h = \varepsilon(h_{(1)})h_{(2)} = \varepsilon(h)1,$$ that is, $h\in K$. Conversely, if $h\in K$ then $\Delta(h)=\Delta(h1)=h\Delta(1)=h(1\otimes 1)=h\otimes 1$ and therefore $h\in H^{\operatorname{co}H}$. Hence $H^{\operatorname{co}H}=K$.
	
	The map $\beta: H\otimes H \rightarrow H\otimes H$ is defined by
	\begin{equation*}
		\beta(a\otimes b)= (a\otimes 1) \rho(b) = (a\otimes 1)\Delta(b)= ab_{(1)} \otimes b_{(2)}, \qquad \mbox{for all }  a,b\in H.
	\end{equation*}
	
	Then, the inverse $\beta^{-1}: H\otimes H \rightarrow H\otimes H$ of the Galois map is given by
	\begin{equation*}
		\beta^{-1}(a\otimes b)= a S(b_{(1)})\otimes b_{(2)}, \qquad \mbox{for all }  a,b\in H.
	\end{equation*}
	Indeed, if $a,b\in H$, then
	\begin{align*}
		\beta(\beta^{-1}(a\otimes b))&=\beta\left( aS(b_{(1)})\otimes b_{(2)} \right) = \beta(aS(b_{(1)})\otimes b_{(2)}) = aS(b_{(1)})b_{(2)}\otimes b_{(3)} \\
		&= a \varepsilon(b_{(1)}) \otimes b_{(2)} = a \otimes \varepsilon(b_{(1)})b_{(2)} = a\otimes b;\\
		\beta^{-1}(\beta(a\otimes b)) &= \beta^{-1}\left( ab_{(1)} \otimes b_{(2)} \right) =  \beta^{-1}(ab_{(1)}\otimes b_{(2)}) = ab_{(1)}S(b_{(2)})\otimes b_{(3)}\\
		&= a\varepsilon(b_{(1)}) \otimes b_{(2)} = a \otimes \varepsilon(b_{(1)}) b_{(2)} = a\otimes b.
	\end{align*}
	Therefore, $\beta$ is bijective and we have shown the next result.
	
	\begin{proposition}\label{ex13}
		Let $H$ be a $K$-Hopf algebra. Then $K\subset H$ is a $H$-Galois extension.
	\end{proposition}
	
	\begin{remark}\label{r1}
		Since Hopf Galois extensions are definable over bialgebras, it can be shown for every bialgebra $H$ that, in fact, $H$ is a Hopf algebra if and only if $K \subset H$ is a $H$-Galois extension \cite[Example 2.1.2]{Sch4}. More generally, if $H$ is a $K$-flat bialgebra admitting a $H$-Galois extension $A^{\operatorname{co}H}\subset A$ that is faithfully flat as a $K$-module, $H$ must be a Hopf algebra \cite{Sch2}.
	\end{remark}
	
	\subsubsection{Classical field extensions}\label{hge2}
	
	Let $G$ be a finite group acting as automorphisms on a field $E\supset \Bbbk$ in the sense of Example~\ref{ex8}. Clearly $E$ is a left $\Bbbk G$-module algebra. Hence, by Proposition~\ref{p2}, it is a right $\Bbbk^G$-comodule algebra (recall that $(\Bbbk G)^*=\Bbbk^G$; see Examples~\ref{ex10} and~\ref{ex10a}). As the name suggest, the next result shows that classical Galois extensions of fields can be seen as Hopf Galois extensions.
	
	\begin{theorem}[e.g. {\cite[Section 8.1.2]{Mon1}}]\label{t9}
		Let $G$ be a finite group acting as automorphisms on a field $E\supset \Bbbk$, and $F:=E^G$. Then the following assertions are equivalent:
		\begin{enumerate}[label=\normalfont(\roman*)]
			\item The field extension $F \subset E$ is Galois with Galois group $G$.
			\item $G$ acts faithfully on $E$.
			\item $[E:F]=|G|$.
			\item $F \subset E$ is a right $\Bbbk^G$-Galois extension.
		\end{enumerate}
	\end{theorem}
	\begin{proof}
		The implications (i) $\Leftrightarrow$ (ii) $\Leftrightarrow$ (iii) are consequences of Artin's Lemma (e.g. \cite[Theorem 4.7]{Jac}). Hence, we will just show their equivalence with (iv). Suppose that the field extension $F \subset E$ is Galois. Set $n:=|G|$ and $G=\{ x_1,\ldots,x_n \}$. Let $\{ u_1,\ldots,u_n \}$ be a basis of $E/F$ and let $\{ p_1,\ldots,p_n \} \subset \Bbbk^G $ be the dual basis to $\{ x_i \} \subset \Bbbk G$. As it was said before, $E$ coacts on $\Bbbk^G$, where the structure map $\rho: E \rightarrow E\otimes \Bbbk^G $ is given by
		\begin{equation*}
			\rho(a)=\sum_{i=1}^n (x_i \cdot a) \otimes p_i, \qquad \mbox{for all } a\in E.
		\end{equation*}
		Therefore, the Galois map $\beta: E\otimes_F E \rightarrow E \otimes \Bbbk^G $ is given by
		\begin{equation*}
			\beta(a\otimes b) = (a\otimes 1) \rho(b) = \sum_{i=1}^n a(x_i \cdot b) \otimes p_i.
		\end{equation*}
		We must show that $\beta$ is bijective. For that, suppose $w=\sum_j a_j \otimes u_j \in \ker(\beta)$. Then
		\begin{equation*}
			\beta \left( w \right) = \sum_j a_j(x_i \cdot u_j) \otimes p_i =0\otimes 0.
		\end{equation*}
		Since the $p_i$ are linearly independent, we conclude that
		\begin{equation}\label{e23}
			\sum_j a_j(x_i \cdot u_j) = 0, \qquad \mbox{for all } 1\leq i \leq n.
		\end{equation}
		Using the faithfulness of the action of $G$, Dedekind independence theorem implies that the $n\times n$ matrix $C=[x_i \cdot u_j]$ associated to the system \eqref{e23} is invertible (e.g. \cite[p. 291]{Jac}). Hence, all $a_j$ are $0$ and $w=0$. Thus $\beta$ is injective. Since both $E\otimes_F E$ and $E \otimes \Bbbk^G$ are $F$-vector spaces of dimension $n^2$, it follows that $\beta$ is a bijection.
		
		Conversely, suppose that the Galois map $\beta$ is an isomorphism. Notice that
		\begin{equation*}
			\dim_F(E\otimes_F E) = [E:F]^2 \qquad \mbox{and} \qquad \dim_F (E\otimes \Bbbk^G)=[E:F]|G|.
		\end{equation*}
		Via the isomorphism, we get $[E:F]=|G|$ and therefore $F\subset E$ is a Galois extension of fields.
	\end{proof}
	
	\subsubsection{Separable field extensions}\label{hge3}
	
	For the circle Hopf algebra $H_\Bbbk$ (cf. Example~\ref{ex30}), it is possible for a finite separable field extension $F \subset F$ to be $H_\Bbbk$-Galois, although it is not Galois in the classical sense. This example is due to Greither and Pareigis \cite{PG}.
	
	Let $\Bbbk=F=\mathbb{Q}$ and $E=F(\omega)$, where $\omega$ is the real $4$-th root of 2. $F\subset E$ is not Galois for any group $G$, since the automorphism group of $E/F$ fixes $\mathbb{Q}(\sqrt{2})$ pointwise. However, if $H_\mathbb{Q}$ is the circle Hopf algebra, it can be shown that $F \subset E$ is $H_\mathbb{Q}^*$-Galois. In this case $H_\mathbb{Q}$ acts on $E$ as follows:
	\begin{center}
		\begin{tabular}{c|cccc}
			$\cdot$ & $1$ & $\omega$  & $\omega^2$  & $\omega^3$ \\ \hline
			$c$     & $1$ & $0$       & $-\omega^2$ & $0$        \\
			$s$     & $0$ & $-\omega$ & $0$         & $\omega^3$
		\end{tabular}
	\end{center}
	
	When we change $\Bbbk$ for $\mathbb{Q}(i)$, $H_\Bbbk \cong \Bbbk\mathbb{Z}_4$, the group algebra of $\mathbb{Z}_4$. This means that $\mathbb{Q}\mathbb{Z}_4$ and $H_\mathbb{Q}$ are $\mathbb{Q}(i)$-forms of each other. For a suitable Hopf algebra $H$ described in \cite[Section 3]{Par}, which is a $\mathbb{Q}(\sqrt{-2})$-form of $\mathbb{Q}[\mathbb{Z}_2 \times \mathbb{Z}_2]$, $\mathbb{Q}\subset E$ is also a $H^*$-Galois. Hence an extension can be Hopf Galois over two different Hopf algebras.
	
	\subsubsection{Strongly graded algebras}\label{hge4}
	
	This example is due to Ulbrich and Osterburg \cite{Ulb,Ulb2,OQ}. Let $A=\bigoplus_{g\in G} A_g$ be a $G$-graded algebra (see Example~\ref{ex7}). $A$ is \emph{strongly graded} if $A_gA_h=A_{gh}$, for all $g,h\in G$.
	
	\begin{lemma}[e.g. {\cite[Proposition 1.1.1]{NV}}]\label{l3}
		Let $A$ be a $G$-graded algebra. Then the following assertions are equivalent:
		\begin{enumerate}[label=\normalfont(\roman*)]
			\item $A$ is strongly graded,
			\item $A_gA_{g^{-1}}=A_1$ for all $g\in G$.
		\end{enumerate}
	\end{lemma}
	
	\begin{proof}
		(i)$\Rightarrow$(ii) is trivial. (ii)$\Rightarrow$(i): We have
		\begin{equation*}
			A_{gh}=A_{gh}A_1 = A_{gh}(A_{h^{-1}}A_h) \subseteq (A_{gh}A_h^{-1})A_h \subseteq A_{ghh^{-1}}A_h =A_gA_h \subseteq A_{gh},
		\end{equation*}
		for all $g,h\in G$.
	\end{proof}
	
	Recall that $A$ is a $\Bbbk G$-comodule algebra with structure map $\rho: A \rightarrow A\otimes \Bbbk G$ given by $\rho(a)=a\otimes g$, for all $a\in A_g$, and that $A^{co\Bbbk G}=A_1$. Then $\beta: A \otimes_{A_1} A \rightarrow A \otimes \Bbbk G $ is given by $\beta( a\otimes b ) = (a\otimes 1)\rho(b)= \sum_{g\in G} ab_g \otimes g$, for all $a, b\in A$ with $b=\sum_{g\in G} b_g$.
	
	\begin{theorem}[e.g. {\cite[Theorem 8.1.7]{Mon1}}]\label{t0}
		Let $G$ be a group and $A$ be a $G$-graded algebra. Then the following assertions are equivalent:
		\begin{enumerate}[label=\normalfont(\roman*)]
			\item $A_1 \subset A$ is a $\Bbbk G$-extension.
			\item $A$ is strongly graded.
		\end{enumerate}
	\end{theorem}
	
	\begin{proof}
		(i)$\Rightarrow$(ii): First assume that $\beta$ is bijective and, in particular, surjective. Whence for every $g\in G$ there exist finitely many $a_i,b_i\in A$ such that
		\begin{equation*}
			\beta\left( \sum_i a_i \otimes b_i \right)=\sum_{i} \sum_{h\in G} a_i (b_i)_h \otimes h = \sum_{i,h} a_i (b_i)_h \otimes h = 1\otimes g.
		\end{equation*}
		Since $G$ is a basis for $\Bbbk G$, it follows that $\sum_i a_i(b_i)_g =1 $ and $\sum_i a_i(b_i)_h=0$ for all $h\neq g$.  The inclusion $A_gA_{g^{-1}} \subseteq A_1$ always holds in $G$-graded algebras, so now suppose $a\in A_1$. Fixing $g$, we have $a = a1 = a\sum_i a_i(b_i)_g =\sum_i aa_i(b_i)_g$. Since the $(b_i)_g$ are homogeneous of degree $g$, $a$ is homogeneous of degree 1, and since the sum of homogeneous components is direct, we get $a_i\in A_{g^{-1}}$. Hence $A_{g}A_{g^{-1}}=A_1$ and by Lemma~\ref{l3} it follows that $A$ is strongly graded.
		
		(ii)$\Rightarrow$(i): Conversely, suppose now that $A$ is strongly graded. If $g\in G$, then by Lemma~\ref{l3}, $1\in A_1=A_{g^{-1}}A_{g}$ and we may write $1=\sum_i a_ib_i$ for some $a_i\in A_{g^{-1}}$ and $b_i\in A_g$. Therefore let $\alpha: A\otimes\Bbbk G \rightarrow A\otimes_{A_1} A$ be defined as $\alpha(a\otimes g)= \sum_i a a_i \otimes b_i$.
		Since all $g_i\in A_g$, we get
		\begin{equation*}
			(\beta\alpha)(a\otimes g)=\beta \left( \sum_i a a_i \otimes b_i \right) = \sum_i aa_ib_i \otimes g = a \left( a_ib_i \right) \otimes g = a1\otimes g =a\otimes g.
		\end{equation*}
		On the other hand,
		\begin{equation*}
			(\alpha\beta)(a\otimes b)= \alpha\left( \sum_{g\in G} ab_g \otimes g \right) = \sum_{g\in G} \alpha( ab_g \otimes g ) = \sum_{g\in G} \sum_i ab_g a_i \otimes b_i,
		\end{equation*}
		but $b_ga_i \in A_g A_{g^{-1}}= A_1$ and hence
		\begin{align*}
			(\alpha\beta)(a\otimes b)&= \sum_{g\in G} \sum_i ab_g a_i \otimes b_i = \sum_{g\in G} \sum_i a \otimes b_g a_ib_i\\
			&= \sum_{g\in G} a \otimes b_g \left( \sum_i a_ib_i \right) = \sum_{g\in G} a \otimes b_g 1 = a\otimes b.
		\end{align*}
		Therefore, $\alpha=\beta^{-1}$ and $\beta$ is bijective.
	\end{proof}
	
	\subsubsection{Crossed products over groups}\label{hge5}
	
	We now give a major example that can be deduced from the previous one. Let $G$ be a (multiplicative) group acting as automorphisms on a $\Bbbk$-algebra $R$. Recall that we denoted $g(r)=g\cdot r$. The action is said to be \emph{twisted} if there exists a map $\sigma: G \times G \rightarrow R$ such that the following conditions hold:
	\begin{enumerate}[label=\normalfont(\roman*)]
		\item (\emph{Cocycle condition}) For all $g,h,k\in G$, $[g\cdot \sigma(h,k)]\sigma(h,hk)=\sigma(g,h)\sigma(gh,k)$, and $\sigma(g,1)=\sigma(1,g)=1$
		\item (\emph{Twisted module condition}) For all $g,h\in G$ and $r\in R$, $$[g\cdot(h\cdot r)]\sigma(g,h)=\sigma(g,h)(gh\cdot r).$$
	\end{enumerate}
	In this case, we say that $\sigma$ is a \emph{2-cocycle}.
	
	We define a new different structure over $R \otimes \Bbbk G$. In this context, an arbitrary element $r\otimes g \in R \otimes \Bbbk G$ will be denoted by $r*g$.
	
	\begin{definition}[Crossed product]\label{crossprod}
		Let $G$ be a group acting as automorphisms on a $\Bbbk$-algebra $R$. If the action is twisted, we define the \emph{crossed product of $R$ and $G$}, denoted by $R*G$, as the $\Bbbk$-space $R \otimes \Bbbk G$ together with the multiplication
		\begin{equation*}
			(r*g)(s*h)=r(g\cdot s)\sigma(g,h)*gh, \qquad \mbox{for all } r,s\in R \mbox{ and } g,h\in G,
		\end{equation*}
		and unit element $1_R*1_G$.
	\end{definition}
	
	\begin{proposition}[e.g. {\cite[Example 2.7]{Mon2}}]
		Let $G$ be a group acting as automorphisms on a $\Bbbk$-algebra $R$. Suppose that the action is twisted. Then $R*G$ is a $G$-graded algebra. Moreover, $R\subset A$ is a $\Bbbk G$-Galois extension.
	\end{proposition}
	
	\begin{proof}
		For all $g,h,k\in G$ and $r,s,t\in R$ we have
		\begin{align*}
			(r*g)[(s*h)(t*k)] &= (r*g)(s(h\cdot t)\sigma(h,k)*hk)\\&= r\left[ g\cdot (s(h\cdot t)\sigma(h,k)) \right]\sigma(g,hk) * g(hk)\\
			&= r \left[ (g\cdot s)(g\cdot(h\cdot t)) (g\cdot \sigma(h,k)) \right]\sigma(g,hk) * g(hk) \\
			&=	r (g\cdot s)(g\cdot(h\cdot t)) \left[ (g\cdot \sigma(h,k))\sigma(g,hk) \right] * (gh)k\\
			&=r (g\cdot s)(g\cdot(h\cdot t)) \left[ \sigma(g,h)\sigma(gh,k) \right] * (gh)k\\
			&=r (g\cdot s)\left[(g\cdot(h\cdot t))  \sigma(g,h)\right]\sigma(gh,k)  * (gh)k\\
			&= r (g\cdot s)\left[ \sigma(g,h)(gh\cdot t) \right]\sigma(gh,k)  * (gh)k\\
			&=\left[r(g\cdot s)\sigma(g,h)*gh  \right](t*k)=\left[ (r*g)(s*h) \right](t*k).
		\end{align*}
		On the other hand,
		\begin{align*}
			(r*g)(1*1)&=r(g\cdot 1)\sigma(g,1)*g1=r1*g1=r*g,\\
			(1*1)(r*g)&=1(1\cdot r)\sigma(1,g)*1g=1r1*1g=r*g.
		\end{align*}
		Hence, $R*G$ is a $\Bbbk$-algebra. The homogeneous components are given by
		\begin{equation*}
			(R*G)_1 = R\otimes 1\cong R	\qquad \mbox{and} \qquad  (R*G)_g = R\otimes g, \quad \mbox{for all } g\in G. 
		\end{equation*}
		Finally, since $(R*G)_g(R*G)_h=R\otimes gh = (R*G)_{gh}$, for all $g,h\in G$, the algebra is strongly graded and therefore, by Theorem~\ref{t0}, $R\subset A$ is $\Bbbk G$-Galois.
	\end{proof}
	
	The relation between crossed products over groups and group extensions is explored in \cite[Example 2.8]{Mon2} by showing that for any group $G$ with normal subgroup $N$ and quotient $L=G/N$, $\Bbbk G=\Bbbk N * \Bbbk L$. Also, in \cite[Examples 2.9 and 2.10]{Mon2} counterexamples of strongly graded algebras that are not crossed products are given.
	
	\subsubsection{Hopf crossed products and smash products}\label{hge6}
	
	This example attempts to generalize the previous one to the context of Hopf algebras. Let $H$ be a $K$-Hopf algebra and $R$ an $H$-module algebra. The action is said to be \emph{twisted} if there exists a map $\sigma: H \times H \rightarrow R$ such that the following conditions hold:
	\begin{enumerate}[label=\normalfont(\roman*)]
		\item (\emph{Hopf 2-cocycle condition}) For all $g,h,k\in H$,
		\begin{gather}
			[g_{(1)} \cdot \sigma(h_{(1)},k_{(1)})] \sigma(g_{(2)},h_{(2)}k_{(2)})= \sigma(g_{(1)},h_{(1)})\sigma(g_{(2)}h_{(2)},k), \label{e81} \\
			\sigma(g,1)=\sigma(1,g)=\varepsilon(g)1.\nonumber
		\end{gather}
		\item (\emph{Hopf twisted module condition}) For all $g,h\in H$ and $r\in R$,
		\begin{equation}
			[g\cdot(h\cdot r)] = \sigma(g_{(1)},h_{(1)}) (g_{(2)}h_{(2)} \cdot r)\sigma(g_{(3)},h_{(3)})^{-1}. \label{e82}
		\end{equation}
	\end{enumerate}
	In this case, we say that $\sigma$ is a \emph{2-cocycle}.
	
	\begin{definition}[Hopf crossed product]
		Let $H$ be a $K$-Hopf algebra and $R$ an $H$-module algebra. If the action is twisted, we define the \emph{crossed product of $R$ and $H$}, denoted by $R\#_{\sigma}H$, as the $K$-module $R \otimes H$ together with the multiplication
		\begin{equation}\label{e83}
			(r\#g)(s\#h)= r(g_{(1)}\cdot s)\sigma(g_{(2)},h_{(1)})\#g_{(3)}h_{(2)}, \quad \mbox{for all } r,s\in R \mbox{ and } g,h\in H,
		\end{equation}
		and unit element $1_R\#1_H$.
	\end{definition}
	
	For any group $G$, if $H=\Bbbk G$, then this definition coincides with that of a crossed product over $G$ (see Definition~\ref{crossprod}).
	
	\begin{theorem}[e.g. {\cite[Example 3.6]{Mon2}}]\label{p10}
		Let $R$ be an $H$-module algebra. Suppose that the action is twisted. If $A:=R\#_\sigma H$, then:
		\begin{enumerate}[label=\normalfont(\roman*)]
			\item  $A$ is an algebra.
			\item $A$ is a right $H$-comodule algebra via $\rho: A \rightarrow A \otimes H$ given by
			\begin{equation*}
				\rho(r\# h) = (r\# h_{(1)}) \otimes h_{(2)}, \qquad \mbox{for all } r\#h\in A.
			\end{equation*}
			\item $A^{\operatorname{co}H} \cong R$.
		\end{enumerate}
	\end{theorem}
	
	\begin{proof}
		(i) For all $r,s,t\in R$ and $g,h,f\in H$, we have 
		\begin{align*}
			&(r\#g)\left[ (s\# h) (t\#f) \right] = (r\#g) \left[ s(h_{(1)}\cdot t)\sigma(h_{(2)},f_{(1)})\#h_{(3)}f_{(2)} \right]\\
			&= r(g_{(1)} \cdot (s(h_{(1)}\cdot t)\sigma(h_{(2)},f_{(1)}))) \sigma(g_{(2)}, (h_{(3)}f_{(2)})_{(2)} ) \# g_{(3)}(h_{(3)}f_{(2)})_{(2)}  \\
			&=  r(g_{(1)} \cdot (s(h_{(1)}\cdot t)\sigma(h_{(2)},f_{(1)}))) \sigma(g_{(2)}, h_{(4)}f_{(3)} ) \# g_{(3)}h_{(4)}f_{(3)}\\
			&\overset{\mbox{\eqref{e11}}}{=} r(g_{(1)} \cdot s) ( g_{(2)} \cdot (h_{(1)}\cdot t)) (g_{(3)} \cdot \sigma(h_{(2)},f_{(1)})) \sigma(g_{(4)}, h_{(4)}f_{(3)} ) \# g_{(5)}h_{(4)}f_{(3)}\\
			&\overset{\mbox{\eqref{e81}}}{=}  r(g_{(1)} \cdot s) ( g_{(2)} \cdot (h_{(1)}\cdot t)) \sigma(g_{(3)},h_{(2)})\sigma(g_{(4)}h_{(3)},f_{(1)})  \# g_{(5)}h_{(4)}f_{(2)} \\
			&\overset{\mbox{\eqref{e82}}}{=} r(g_{(1)}\cdot s)\sigma(g_{(2)},h_{(1)}) ( g_{(3)}h_{(2)} \cdot t) \sigma(g_{(4)}h_{(3)},f_{(1)}) \# g_{(5)}h_{(4)}f_{(2)}\\
			&= r(g_{(1)}\cdot s)\sigma(g_{(2)},h_{(1)}) ( (g_{(3)}h_{(2)})_{(1)} \cdot t) \sigma((g_{(3)}h_{(2)})_{(2)},f_{(2)}) \# (g_{(3)}h_{(2)})_{(3)} f_{(2)}\\
			&= \left[ r(g_{(1)}\cdot s)\sigma(g_{(2)},h_{(1)})\#g_{(3)}h_{(2)} \right] (t\# f) =\left[ (r\# g)(s\#h) \right] (t\# f)
		\end{align*}
		and
		\begin{align*}
			(r\#g)(1\#1)&= r(g_{(1)}\cdot 1) \sigma(g_{(2)},{(1)}) \# g_{(3)} \overset{\mbox{\eqref{e82}}}{=} (g_1\cdot 1) \varepsilon(g_{(2)}) \# g_{(3)}\\
			&\overset{\mbox{\eqref{e11}}}{=} \varepsilon(g_{(1)}) \varepsilon(g_{(2)}) \# g_{(3)} = r \# \varepsilon(g_{(1)}g_{(2)})g_{(3)} = r \# g,\\
			(1\#1)(r\#g) &= 1(1\cdot r) \sigma(1,g_{(1)}) \# 1g_{(2)} \overset{\mbox{\eqref{e82}}}{=} r \varepsilon(g_{(1)}) \# g_{(2)} = r \# \varepsilon(g_{(1)})g_{(2)} = r\# g.
		\end{align*}
		Therefore, $R\#_\sigma H$ is an algebra.
		
		(ii) Since for all $r\in R$ and $h\in H$ we have
		\begin{align*}
			[(\operatorname{id}_{R\#_\sigma H}\otimes \Delta)\rho](r\# h) &= (\operatorname{id}_{R\#_\sigma H}\otimes \Delta)\left( (r\# h_{(1)})\otimes h_{(2)} \right) = (a\# h_{(1)}) \otimes h_{(2)} \otimes h_{(3)} \\
			& = (\rho\otimes \operatorname{id}_H)\left(  (r\# h_{(1)})\otimes h_{(2)} \right)=[(\rho\otimes \operatorname{id}_H)\rho](r\# h),\\
			[(\operatorname{id}_{R\#_\sigma H} \otimes \varepsilon)\rho](r\#h)&=(\operatorname{id}_{R\#_\sigma H} \otimes \varepsilon) \left( (r\# h_{(1)})\otimes h_{(2)} \right) = (r\# h_{(1)}) \otimes \varepsilon(h_{(2)})1 \\&= r\# h_{(1)}\varepsilon(h_{(2)}) \otimes 1= r\# h \otimes 1,
		\end{align*}
		it follows that $R\#_\sigma H$ is an $H$-module. Moreover, for all $r,s\in R$ and $g,h\in H$,
		\begin{align*}
			\rho((r\#g)(s\#h)) &= \rho\left( r(g_{(1)}\cdot s) \sigma(g_{(2)},h_{(1)}) \# g_{(3)}h_{(2)} \right) \\
			&= (r(g_{(1)}\cdot s) \sigma(g_{(2)},h_{(1)}) \# (g_{(3)}h_{(2)})_{(1)}) \otimes (g_{(3)}h_{(2)})_{(2)}\\
			&= (r(g_{(1)}\cdot s) \sigma(g_{(2)},h_{(1)}) \# g_{(3)}h_{(2)} \otimes g_{(4)}h_{(3)} \\
			&= (r\# g_{(1)})(s\# h_{(1)}) \otimes g_{(2)}h_{(2)},\\
			\rho(1\otimes 1) &= (1 \# 1) \otimes 1,
		\end{align*} 
		meaning that $R\#_\sigma H$ is an $H$-comodule algebra.
		
		(iii) We have
		\begin{equation*}
			(R\#_\sigma H)^{\operatorname{co}H} = \left\{ z\in R\#_\sigma H : \rho(z)=z\otimes 1 \right\}.
		\end{equation*} 
		Clearly, $R\cong R \#_\sigma 1 \subseteq (R\#_\sigma H)^{\operatorname{co}H}$. Reciprocally, if $r\#h\in(R\#_\sigma H)^{\operatorname{co}H}$, applying the map $\operatorname{id}_R\otimes\varepsilon\otimes \operatorname{id}_H$ to the equality $\rho(r\#h)=(r\#h)\otimes 1$ we get at the left-hand side
		\begin{align*}
			(\operatorname{id}_R\otimes\varepsilon\otimes \operatorname{id}_H)(\rho(r\#h)) &= 	(\operatorname{id}_R\otimes\varepsilon\otimes \operatorname{id}_H) \left( (r\# h_{(1)})\otimes h_{(2)} \right)\\
			&= (r \# \varepsilon(h_{(1)})1) \otimes h_{(2)} = (r\# 1)\otimes h,
		\end{align*} 
		while at the right-hand side
		\begin{equation*}
			(\operatorname{id}_R\otimes\varepsilon\otimes \operatorname{id}_H)((r\#h)\otimes 1) = (r \# \varepsilon(h)1) \otimes 1 = (r \# 1) \otimes \varepsilon(h)1.
		\end{equation*}
		Comparing, we get $h=\varepsilon(h)1 \in K$, and hence, $r\# h= r\varepsilon(h)\#1 \in R\#_\sigma 1 \cong R$. Thus $A^{\operatorname{co}H}\cong R$, as desired.
	\end{proof}
	
	Hopf crossed products are relevant for characterizing Hopf Galois extensions having the normal basis property (see Theorem~\ref{t8}).
	
	A particular case of Hopf crossed products is when the cocycle $\sigma$ is \emph{trivial}, that is, when $\sigma(g,h)=\varepsilon(g)\varepsilon(h)$ for all $g,h\in H$. In this case, we write $R\# H$ and the multiplication defined in \eqref{e83} is simply $(r\#g)(s\#h)= r(g_{(1)}\cdot s) \# g_{(2)}h$, for all $r,s\in R$ and $g,h\in H$.
	$R\# H$ is called the \emph{smash product} of $R$ and $H$.
	
	Under these conditions, Theorem~\ref{p10} can be restated as follows.
	
	\begin{theorem}
		Let $R$ be an $H$-module algebra. Then $A:=R\#H$ is a $H$-comodule algebra with structure map $\rho: A \rightarrow A \otimes H$ given by
		\begin{equation*}
			\rho(r\#h) = (r\# h_{(1)})\otimes h_{(2)}, \qquad \mbox{for all } r\in R \mbox{ and } h\in H.
		\end{equation*}
		Moreover, $A^{\operatorname{co}H} \cong R$.
	\end{theorem}
	
	\begin{corollary}[e.g. {\cite[Example 6.4.3]{DNR}}]
		Let $R$ be an $H$-module algebra. Then $R \subset R\# H$ is an $H$-Galois extension.
	\end{corollary}
	
	\begin{proof}
		We have $\beta: (R\# H) \otimes_R (R\#H) \rightarrow (R\#H) \otimes H $ defined as $\beta(z\otimes w)=(z\otimes 1)\rho(w)$, for all $z,w\in R\# H$. But since the first tensor product is taken over $R\# 1 \cong R$, for all $r,s\in A$ and $h,g\in H$ we have $(r\# h)(s\# 1)\otimes (1\# g) = (r\# h) \otimes (s\# g)$. Thus, it is enough to study $\beta$ on elements of the form $(r\#h) \otimes (1\# g)$. That is,
		\begin{align*}
			\beta((r\#h) \otimes (1\# g)) &= ((r\# h)\otimes 1) \rho(1\# g) = ((r\# h)\otimes 1)\left( (1\# g_{(1)}) \otimes g_{(2)} \right)\\
			&= ((r\# h)(1\# g_{(1)})) \otimes g_{(2)} = (r(h_{(1)}\cdot 1)\# h_{(2)}g_{(1)}) \otimes g_{(2)}\\
			& =(r \varepsilon(h_{(1)}) \# h_{(2)}g_{(1)}) \otimes g_{(2)} = (r\# hg_{(1)}) \otimes g_{(2)}.
		\end{align*}
		As in the proof of Proposition~\ref{ex13}, $\beta^{-1}:(R\# H) \otimes H \rightarrow (R\# H) \otimes_R (R\# H)$ is given by
		\begin{equation*}
			\beta^{-1}( (r\# h) \otimes g) = (r\#hS(g_{(1)})) \otimes (1\#g_{(2)}), \qquad \mbox{for all } r\in R \mbox{ and } h,g\in H.
		\end{equation*}
		Indeed, if $r\in R$ and $h,g\in H$, we have
		\begin{align*}
			\beta^{-1}(\beta((r\#h) \otimes (1\# g))) &= \beta^{-1}\left( (r\# hg_{(1)}) \otimes g_{(2)} \right) \\&= (r\# hg_{(1)}S((g_{(2)})_{(1)}) ) \otimes (1\#(g_{(2)})_{(2)})\\
			&= (r\# hg_{(1)}S(g_{(2)}) ) \otimes (1\#g_{(3)}) = (r\# h\varepsilon(g_{(1)}))  \otimes (1\#g_{(2)})\\
			&= (r\#h)\otimes (1\# g),\\
			\beta(\beta^{-1}( (r\# h) \otimes g)) &= \beta\left( (r\#hS(g_{(1)})) \otimes (1\#g_{(2)}) \right) \\&= (r\#  hS(g_{(1)})(g_{(2)})_{(1)})\otimes (g_{(2)})_{(2)}\\
			&= (r\#  hS(g_{(1)})g_{(2)})\otimes g_{(3)} = (r\# h\varepsilon(g_{(1)})) \otimes g_{(2)}\\
			&= (r\# h) \otimes g.
		\end{align*}
		Therefore, $R \subset R\# H$ is a $H$-Galois extension.
	\end{proof}
	
	Smash products are relevant for Hopf Galois extensions since, when $H$ is finite dimensional, a generator of $\operatorname{Mod}_H$ describes all $H^*$-Galois extensions (see Theorem~\ref{t7}).
	
	\subsubsection{Groups acting on sets}\label{hge7}
	
	The Galois map $\beta$ can be seen as the dual of a natural map arising whenever a group acts on a set, as this example shows. Recall that if $G$ is a (multiplicative) group and $X$ is a non-empty set, a function $\mu: X \times G \rightarrow X$ is called a \emph{(right) group action of $G$ on $X$}, which we denote by $(x,g)\mapsto x\cdot g$, if
	\begin{equation*}
		x\cdot 1 = x \qquad \mbox{and} \qquad x\cdot (gh)=(x\cdot g)\cdot h, \qquad \mbox{for all } x\in X \mbox{ and } g,h\in G.
	\end{equation*}
	The action is said to be \emph{free} if for a given $g\in G$ such that $x\cdot g=x$, for some $x\in X$, it follows that $g= 1$; in other words, no element in $G$, besides 1, has fixed points. It is not hard to check that an action is free if and only if, given $g,h\in G$, the existence of an $x \in X$ such that $x\cdot g = x\cdot h$ implies $g = h$.
	
	Also, recall that for any $x\in X$, its \emph{orbit} is defined by $x\cdot G := \{ x\cdot g : g\in G \}$. The set of all orbits of $X$ under the action of $G$ is denoted by $X/G$ and is called the \emph{quotient} of the action.
	
	Consider the map $\alpha: X \times G \rightarrow X\times X$ given by $(x,g) \mapsto (x,x\cdot g)$. Notice that $\alpha$ is injective if only if the action is free. Moreover, the image of this map can bee seen as a pull-back. Indeed,
	\begin{align*}
		\operatorname{Im}(\alpha)&=\{ (x,y)\in X\times X : y=x\cdot g \mbox{ for some } g\in G \}\\
		&= \{ (x,y)\in X\times X : x\cdot G = y\cdot G \}\\
		&= X \times_{X/G} X,
	\end{align*}
	called \emph{the fiber product} of $X$ with itself over $X/G$ via the canonical map $x\mapsto x\cdot G$.
	
	We want to dualize this scenario; for simplicity assume that $X$ and $G$ are finite. We denote by $A:=\Bbbk^X$ the algebra of functions from $X$ to $\Bbbk$ endowed with the pointwise addition and multiplication. The unit of this algebra is the map $1_A: X \rightarrow \Bbbk$ given by $x \mapsto 1_\Bbbk$. We say that $a\in A$ is \emph{constant on $G$-orbits} if $a(x\cdot g)=a(x)$, for all $x\in X$ and $g\in G$. The set of functions constant on $G$-orbits is denoted by $\Bbbk^{X/G}$. Finally, recall that $H:=(\Bbbk G)^*=\Bbbk^G$ is the Hopf algebra of functions from $G$ to $\Bbbk$ (see Example~\ref{ex10}). Hence, we have the following.
	
	\begin{lemma}[e.g. {\cite[Example 8.1.9]{Mon1}}]\label{l21}
		Let $X$ be a finite non-empty set, $G$ a finite group and $\mu: X \times G \rightarrow X$ a right group action. If $A=\Bbbk^X$ and $H=\Bbbk^G$, then:
		\begin{enumerate}[label=\normalfont(\roman*)]
			\item The right $G$-action on $X$ induces a left $G$-action on $A$, given by $(g\cdot a)(x):=a(x\cdot g)$,
			\item $A$ is a right $H$-comodule algebra with induced structure map $\mu^* : A\rightarrow A\otimes H$. Moreover, $A^{\operatorname{co}H}=\Bbbk^{X/G}$.
		\end{enumerate}	 
	\end{lemma}
	
	\begin{proof}
		(i) For every $a\in A$, $g,h\in G$ and $x\in X$ we have
		\begin{gather*}
			(1\cdot a)(x)=a(x\cdot 1)=a(x),\\
			((gh)\cdot a)(x) = a(x\cdot (gh))=a((x\cdot g)\cdot h) = (h\cdot a)(x\cdot g)=(g\cdot (h\cdot a))(x).
		\end{gather*}
		
		(ii) Since for any non-empty sets $U$ and $V$, $\Bbbk^{U\times V}\cong \Bbbk^U \otimes \Bbbk^V$, we have $A\otimes H \cong \Bbbk^{X\otimes G}$. Thus, we can define $(\mu^*(a))(x,g)=a\mu(x,g)=a(x\cdot g)$. The verification of $A$ being a right $H$-comodule algebra is straightforward. It is evident that $A^{\operatorname{co}H}=\Bbbk^{X/G}$.
	\end{proof}
	
	The map $\alpha$ defined above dualizes to $\alpha^*: A \otimes_B A \rightarrow A \otimes H$. By transposition, it is given by
	\begin{equation}\label{e85}
		\alpha^*(a\otimes b)=(a\otimes 1)\mu^*(b), \qquad \mbox{for all } a,b\in A,
	\end{equation}
	that is, $\alpha^*=\beta$, the Galois map. By remarks in previous paragraphs, the freeness of the action is equivalent to $\Bbbk^{X/G} \subset \Bbbk^X$ being a $\Bbbk^G$-Galois extension. In other words, we have the following.
	
	\begin{theorem}
		Let $X$ be a finite non-empty set, $G$ a finite group and $\mu: X \times G \rightarrow X$ a right group action. The Galois map $\beta=\alpha^*$ given by \eqref{e85} is bijective if and only if $\alpha: X\times G \rightarrow X \times_{X/G} X$ is bijective, and this holds if and only if the $G$-action is free.
	\end{theorem}
	
	\subsubsection{Algebraic group schemes}\label{hge8}
	
	Recall that a $\Bbbk$-algebra $A$ is called \emph{affine} if it is finitely generated as $\Bbbk$-algebra, i.e., there exist finitely many elements $a_1,\ldots,a_n \in A$ such that every element of $A$ can be expressed as a polynomial in $a_1,\ldots,a_n$ with coefficients in $\Bbbk$. This definition we use is the one given in \cite[Definition 4.2.3]{Mon1}. However, nowadays in most contexts affine algebras are also required to be commutative and reduced (i.e., without nilpotent elements).
	
	We say that $X$ is an \emph{affine scheme} if $X=\operatorname{Spec}(A)$ for a commutative affine $\Bbbk$-algebra $A$. Similarly, $G$ is said to be an \emph{affine algebraic group scheme} if $G=\operatorname{Spec}(H)$ for some commutative affine $\Bbbk$-Hopf algebra $H$. As in Lemma~\ref{l21}, any action $\mu: X \times G \rightarrow X$ is determined by a coaction $\rho=\mu^*: A \rightarrow A\otimes H$.
	
	\begin{lemma}[e.g. {\cite[Example 2.12]{Mon2}}]
		Let $X=\operatorname{Spec}(A)$ be an affine scheme, $G=\operatorname{Spec}(H)$ an affine algebraic group scheme and $\rho: A \rightarrow A\otimes H$ a coaction. The map $\alpha: X \times G \rightarrow X\times X$ given by $\alpha(x,g)=(x,x\cdot g)$, for all $x\in X$ and $g\in G$, is a closed embedding if and only if $\alpha^*: A\otimes A \rightarrow A \otimes H$ given by $\alpha^*(a,b)=(a\otimes 1)\rho(b)$ is surjective. Under these conditions, we say that the coaction $\rho$ is \emph{free}.
	\end{lemma}
	
	However, in contrast with the previous example, the Galois map cannot be $\alpha^*$, since its domain is not $A\otimes_{A^{\operatorname{co}H}} A$.  Instead, we shall proceed differently by applying the $\operatorname{Spec}$ functor to the exact sequence
	\begin{equation*}
		\begin{tikzcd}
			A^{\operatorname{co}H} \arrow[r,hook] & A \arrow[r, yshift=0.7ex,"\rho"] \arrow[r, yshift=-0.7ex,"-\otimes 1"'] & A\otimes_{A^{\operatorname{co}H}} A,
		\end{tikzcd}
	\end{equation*}
	and getting an exact sequence of affine schemes
	\begin{equation*}
		\begin{tikzcd}
			X \times G \arrow[r, yshift=0.7ex,"\mu"] \arrow[r, yshift=-0.7ex,"\pi"'] & X \arrow[r]& \operatorname{Spec}(A^{\operatorname{co}H}),
		\end{tikzcd}
	\end{equation*}
	where $\pi$ is the projection on the first coordinate. $\operatorname{Spec}(A^{\operatorname{co}H})$ is called the \emph{affine quotient} of $X$ by $G$.
	
	In general, $Y:=\operatorname{Spec}(A^{\operatorname{co}H})$ does not necessarily coincide with $X/G$, the set of $G$-orbits on $X$. However, if the coaction is free, it will happen that $Y=X/G$. Thus the map $X \times G \rightarrow X \times_Y X$ is an isomorphism and $X\rightarrow Y$ is faithfully flat. In an algebraic language, we have the following result.
	
	\begin{theorem}[e.g. {\cite[Example 2.12]{Mon2}}]
		Let $X=\operatorname{Spec}(A)$ be an affine scheme, $G=\operatorname{Spec}(H)$ an affine algebraic group scheme and $\rho: A \rightarrow A\otimes H$ a free coaction. Then
		\begin{enumerate}[label=\normalfont(\roman*)]
			\item $\beta: A \otimes_{A^{\operatorname{co}H}} A \rightarrow A \otimes H$ is bijective and so $B \subset A$ is $H$-Galois,
			\item $A$ is a faithfully flat $A^{\operatorname{co}H}$-module.
		\end{enumerate}
	\end{theorem}
	
	\subsubsection{Principal bundles}\label{hge9}
	
	In this example we discuss why, in noncommutative algebraic geometry, faithfully flat Hopf Galois extensions are considered a generalization of classical affine principal bundles. Our main reference for this example is the work of Brzezi{\'n}ski and Fairfax \cite{BF}. Let us first shortly recall some basic terminology related to topological bundles.
	
	\begin{definition}
		A \emph{bundle} is a triple $(E,\pi,M)$ where $E$ and $M$ are topological spaces an $\pi: E\rightarrow M$ is a continuous surjective map.
	\end{definition}
	
	Here $M$ is called the \emph{base space}, $E$ the \emph{total space} and $\pi$ the \emph{projection} of the bundle.	For each $m\in M$, the \emph{fiber} over $m$ is the topological space $E_m:=\pi^{-1}(m)$. A \emph{local section} of a bundle is a continuous map $s : U \rightarrow E$ with $\pi s = \operatorname{id}_M|_U$, where $U$ is an open subset of $M$. If each fiber of a bundle is endowed with a vector space structure such that the addition and scalar multiplication are continuous, we call it a \emph{vector bundle}.
	
	When the fibers of a bundle are all homeomorphic to a common space $F$, it is known as a \emph{fiber bundle}. An intuitive example of fiber bundle is the M\"obius strip, since it has a circle that runs lengthwise through the center of the strip as a base $M$ and a line segment running vertically for the fiber $F$. The line segments are, in fact, copies of the real line, so $F=\mathbb{R}$.
	
	\begin{remark}
		Commonly, in the definition of fiber bundle a condition of local triviality is required for $\pi$, which means that, for each $x\in E$, there is a open neighborhood $U_x \subset M$ and a homeomorphism $\phi_x: \pi^{-1}(U_x)\rightarrow U_x\times F$ such that the following diagram commutes:
		\begin{equation*}
			\begin{tikzcd}
				\pi^{-1}(U_x) \arrow[r,"\phi_x"] \arrow[d,"\pi"'] & U_x\times F \arrow[dl,"p_1"]\\
				U_x
			\end{tikzcd}
		\end{equation*}
		Here $p_1$ denotes the first projection.
	\end{remark}
	
	In the most general sense, a \emph{bundle} over an object $M$ in a category $\mathcal{C}$ is a morphism $\pi:E\rightarrow M$ in $\mathcal{C}$. For $m:1\rightarrow M$, a generalized element of $M$, the \emph{fiber} $E_m$ is defined as the pullback of $E$ along $m$. Moreover, given an object $F$ in $\mathcal{C}$, $p:E\rightarrow M$ is called a \emph{fiber bundle} with standard fiber $F$, if given any $m:1\rightarrow M$, $E_m$ is isomorphic to $F$. Locally trivial fiber bundles can be defined over sites (see e.g. \cite[p. 20]{Hu}).
	
	Let $X$ be a topological space and $G$ a topological group. Suppose there is a right action $\mu: X\times G\rightarrow X$ and write $\mu(x,g)=x\cdot g$. We had seen in previous examples that, even without structure, $G$ acts freely on $X$ if and only if the map $\alpha: X\times G \rightarrow X\times X$, given by $(x,g)\mapsto (x,x\cdot g)$, is injective if and only if $\alpha: X\times G \rightarrow X\times_{X/G} X$ is bijective. Since $(X,\pi,X/G)$ is a bundle, where $\pi$ is the natural projection, our goal is to characterize the freeness of the action in this topological context.
	
	Recall that a continuous map $f: Y\rightarrow W$ is \emph{proper} if the map $f\times \operatorname{id}_Z: Y \times Z \rightarrow W\times Z$ is closed, for any topological space $Z$. The action $\mu$ is said to be \emph{proper} if it is continuous and $\alpha: X\times G \rightarrow X\times X$ is proper.
	
	\begin{definition}\label{d17}
		A \emph{$G$-principal bundle} is a quadruple $(X,\pi,M,G)$ such that
		\begin{enumerate}[label=\normalfont(\roman*)]
			\item $(X,\pi,M)$ is a bundle and $G$ is a topological group acting continuously on $X$ from the right via $\mu: X\times G\rightarrow X$,
			\item $\mu$ is principal (i.e., free and proper),
			\item $\pi(x)=\pi(y)$ if and only if there exists $g\in G$ such that $y=x\cdot g$,
			\item[(iv)] The induced map $X/G\rightarrow M$ is a homeomorphism.
		\end{enumerate}
	\end{definition}
	
	The first two properties tell us that principal bundles are bundles admitting a principal action of a group $G$ on the total space $X$, i.e., principal bundles correspond to principal actions. But as desired, those are just continuous free actions (i.e., continuous actions such that $\alpha: X\times G \rightarrow X\times X$ is injective). The third property ensures that the fibers of the bundle correspond to the orbits coming from the action and the final one implies that the quotient space can topologically be viewed as the base space of the bundle.
	
	The following two examples are due to \cite[p. 13]{BHMS} and \cite[p. 4]{BF}.
	
	\begin{example}
		Clearly, a principal right action of $G$ on $X$ automatically makes the bundle $(X, \pi, X/G)$ a $G$-principal bundle. However, not every principal bundle has to be of this form. If we replace $X/G$ by a homeomorphic space, not only we are formally defining a different bundle, but also it might happen that such a new bundle is not equivalent to $(X, \pi, X/G)$ \cite[p. 157]{Fri}.
	\end{example}
	
	\begin{example}
		Any vector bundle can be understood as a bundle associated to a principal bundle in the following way: consider a $G$-principal bundle $(X,\pi,M,G)$ and let $V$ be a \emph{representation space} of $G$, that is, a (topological) vector space with a (continuous) left $G$-action $G\times V \rightarrow V$, denoted $(g,v)\mapsto g\cdot v$. Then $G$ acts from the right on $X\times V$ by $(x,v) \cdot g := (xg,g^{-1}\cdot v)$, for all $x\in X$, $v\in V$ and $g\in G$. Hence, we define $E=(X\times V)/G$ and a surjective (continuous) map $\pi_E: E \rightarrow M$ given by $(x,v)G \mapsto \pi(x)$, for all $x\in X, \ v\in V$. Thus, we have the fiber bundle $(E,\pi_E,M,V)$.
	\end{example}
	
	\begin{remark}\label{rem1}
	In a category $\mathcal{C}$, given a group object $G$, a \emph{$G$-principal bundle} (also called a \emph{$G$-torsor}) is a bundle $\pi: E\rightarrow X$ equipped with a $G$-action $\mu: E\times G \rightarrow G$ on $E$ over $X$ such that the canonical morphism $\alpha: E\times G \rightarrow E \times_X E$ is an isomorphism, which in turn means that the action is free and transitive over $X$ and hence each fiber of the bundle looks like $G$ once we choose a base point. In other words, this says that, after picking any point of $X$ as the identity, $X$ ``acquires a group structure'' isomorphic to $G$. Hence, colloquially, a torsor is understood as a group that has forgotten its identity.	 
	
	In specific contexts (as that of the category of topological spaces, in Definition~\ref{d17}), several perspectives for torsors come up. For example:
	\begin{itemize}
		\item In the category of sets, a $G$-torsor over a set $X$ becomes a group action $X\times G \rightarrow G$ (denoted by $(x,g)\rightarrow x\cdot g$) such that for any $x_1,x_2 \in X$, there exists a unique $g\in G$ such that $x_1\cdot g  = x_2$. Usually $g$ is denoted by $x_2/x_1$ and is called the \emph{ratio} of $x_1$ and $x_2$. Due to the above, while in an multiplicative (resp. additive) group $G$ one can multiply and divide (resp. add and subtract) elements, in a $G$-torsor one can multiply (resp. add) an element of $G$ to an element of $X$ and get a result in $X$, or one can divide (resp. subtract) two elements of $X$ and get a result in $G$. Basic examples of torsors are the anti-derivatives of a function or the euclidean plane without the origin (see the discussion at the start of Section~\ref{se4}). 
		\item In algebraic geometry, given a smooth algebraic group $G$, a $G$-torsor over a scheme $X$ is a scheme with an action of $G$ that is locally trivial in the given Grothendieck topology \cite{BCE}.
		\item Recent works in measurement theory attempt to understand the algebraic structure underlying the quantity calculus using topological bundles (see e.g. \cite{Dom,Rap}).
	\end{itemize}
	\end{remark}
	
	Focusing on topological bundles, we want to dualize the setup previously described in order to obtain a noncommutative version. For simplicity, assume that $X$ is a complex affine variety with an action of an affine algebraic group $G$ and set $Y=X/G$ with the usual Euclidean topology. Let $A:=\mathcal{O}(X)$, $B:=\mathcal{O}(Y)$ and $H:=\mathcal{O}(G)$ be the corresponding coordinate rings. Since $\mathcal{O}(G\times G) \cong \mathcal{O}(G) \otimes \mathcal{O}(G)$, $H$ is a Hopf algebra with operations given by $\Delta(f)(g,h)=f(gh)$, $\varepsilon(f)=f(e)$, and $(Sf)(g)=f(g^{-1})$.
	
	Using the fact that $G$ acts on $X$ from the right, $A$ is a right $H$-algebra comodule with structure map $\rho: A \rightarrow A\otimes H$ given by $\rho(a)(x,g):=a(x\cdot g)$ \cite[p. 5]{BF}. Also, $B$ can be viewed as a subalgebra of $A$ via $\pi^*:B\rightarrow A$ given by $b \mapsto b\circ \pi$, where $\pi$ is the canonical surjection $\pi: X\rightarrow X/G$. Indeed, the map $\pi^*$ is injective since $b\neq b'$ in $B=\mathcal{O}(X/Y)$ means that there exists at least one orbit $x\cdot G$ such that $b(x\cdot G)\neq b'(x\cdot G)$. But, since $\pi(x)=x\cdot G$, it follows $\pi^*(b)\neq \pi^*(b')$. 	Furthermore, $a\in \pi^{*}(B)$ if and only if $a(x\cdot g)=a(x)$, for all $x\in X$ and $g\in G$, meaning that $\rho(a)(x,g)=(a\otimes 1)(x,g)$, where $1:G\rightarrow \mathbb{C}$ is the unit constant function $1(g)=1$. Hence $a\in A^{\operatorname{co}H}$. The other inclusion is obvious, so $A^{\operatorname{co}H}=\pi^*(B) \cong B$.
	
	Finally, notice that we can identify $\mathcal{O}(X\times_Y X)$ with $\mathcal{O}(X) \otimes_{\mathcal{O}(Y)} \mathcal{O}(X) = A \otimes_B A$ via the map $\theta(a\otimes a')(x,y)=a(x) a'(y)$, where $\pi(x)=\pi(y)$. This last condition implies the well-definition of $\theta$. With this, we have the following result.
	
	\begin{theorem}[{\cite[Proposition 4]{BF}}]
		Let $X$ be a complex affine variety with a right action of an affine algebraic group $G$ and put $Y=X/G$. Let $A:=\mathcal{O}(X)$, $B:=\mathcal{O}(Y)$ and $H:=\mathcal{O}(G)$ be the corresponding coordinate rings. Then the following assertions are equivalent:
		\begin{enumerate}[label=\normalfont(\roman*)]
			\item The action of $G$ on $X$ is free,
			\item The map $\alpha^*: \mathcal{O}(X \times_Y X)\rightarrow \mathcal{O}(X\times G)$ given by $f\mapsto f\circ \alpha$ is bijective, where $\alpha: X\times G \rightarrow X \times_Y X$ is defined as $(x,g)\mapsto (x,xg)$.
			\item The map $\beta: A\otimes_B A \rightarrow A \otimes H$ given by $\beta(a\otimes a')=a\rho(a')$ is bijective and thus $B\subset A$ is a right $H$-Galois extension.
		\end{enumerate}
	\end{theorem}
	
	Basically this theorem states that in bundles the freeness condition is equivalent to the Galois map being bijective. Hence, the Hopf Galois extension condition is a necessary condition to ensure that a bundle is principal.
	
	However, not all information about the topological spaces involved is encoded in their coordinate rings, so to make a transparent reflection on the richness of principal bundles, we require an additional notion.
	
	\begin{definition}\label{d19}
		Let $H$ be a $K$-Hopf algebra with bijective antipode and let $A$ be a right $H$-comodule algebra with structure map $\rho: A \rightarrow A\otimes H$. We say that $A$ is a \emph{principal $H$-comodule algebra}, if it satisfies the following conditions:
		\begin{enumerate}[label=(PCA\arabic*), align=parleft, leftmargin=*]
			\item\label{PCA1} $A^{\operatorname{co}H} \subset A$ is a right $H$-Galois extension.
			\item\label{PCA2} (\emph{Equivariant projectivity condition}) The map $B\otimes A \rightarrow A$, given by $b\otimes a \rightarrow ba$, splits as a left $A^{\operatorname{co}H}$-module and right $H$-comodule morphisms.
		\end{enumerate}
	\end{definition}
	
	The concept of equivariant projectivity here replaces that of faithful flatness used in general Hopf Galois theory. Under the hypothesis of bijective antipode, in a Hopf Galois extension these two concepts are equivalent, while in general only the implication ``equivariant projectivity'' $\Rightarrow$ ``faithful flatness'' holds \cite{Sch7}.
	
	The next result is due to \cite{DGHH} and \cite{BH}.
	
	\begin{theorem}
		Let $H$ be a $\mathbb{C}$-Hopf algebra with bijective antipode and let $A$ be a right $H$-comodule algebra with structure map $\rho: A \rightarrow A\otimes H$. $A$ is a principal $H$-comodule algebra if and only if it admits a \emph{strong connection form}, that is, if there exists a map $\omega: H \rightarrow A\otimes A$ such that $\omega(1)=1\otimes 1$, $m_H\omega=u_H\varepsilon_H$, $	(\omega \otimes \operatorname{id}_H)\Delta = (\operatorname{id}_A \otimes \rho)\omega$, and $(S\otimes \omega)\Delta = (\tau_{A,H} \otimes \operatorname{id}_A)(\rho \otimes \operatorname{id}_A) \omega$.
	\end{theorem}
	
	This theorem provides an effective method for the verification of the principality of a comodule algebra. For example, every cleft comodule algebra (see Definition~\ref{d18} below) is a principal comodule algebra \cite[Example 3]{BF}.
	
	\subsubsection{Other examples}\label{hge10}
	
	Several other examples of Hopf Galois extensions (HGE) are treated in the literature. The list here is in chronological order and is non-exhaustive.
	
	\begin{itemize}
		\item HGE for Azumaya algebras \cite[Theorem 6.20]{DT}. For a fixed Hopf algebra $H$, let $E$ be an Azumaya algebra and let $C \subset E$ be a subalgebra such that the right $C$-module $E$ is a progenerator. Doi and Takeuchi showed that there is a one-to-one correspondence between the right $H$-Galois extensions $B^{\operatorname{co}H} \subset B$ (such that there exists an algebra morphism $B\rightarrow E$ and $B^{\operatorname{co}H}\cong C$) and the measuring actions of the form $E^C \otimes H \rightarrow E^C$, where $E^C:=\{ x\in E : cx=xc, \mbox{ for all } c\in C \}$.
		\item Differential Galois theory \cite[Section 8.1.3]{Mon1}. Let $E \supset \Bbbk$ be a field of characteristic $p>0$ and let $\mathfrak{g} \subset \operatorname{Der}_\Bbbk(E)$ be a restricted Lie algebra of $\Bbbk$-derivations of $E$, which is finite-dimensional over $\Bbbk$. If the restricted enveloping algebra is denoted $u(\mathfrak{g})$ and it acts on $E$ via $\mathfrak{g}$ acting as derivations, then for $H=u(\mathfrak{g})^*$, $E^{\operatorname{co}H}=E^{\mathfrak{g}}=\{ a\in E : x\cdot a = 0, \mbox{for all } x\in \mathfrak{g} \}$. $E^{\mathfrak{g}} \subset E$ is $u(\mathfrak{g})^*$-Galois if and only if $E \otimes \mathfrak{g} \rightarrow \operatorname{Der}_\Bbbk(E)$ is injective.
		\item HGE for Taft Hopf algebras \cite{Mas}. Masoka classified cleft extensions for certain Hopf algebras generated by skew primitive elements. A particular case is the Taft Hopf algebra of Example~\ref{ex33}.
		\item Quantum principal bundles with a compact structure group \cite{Dur}. Dur{\dj}evi{\'c} gave another generalization of classical principal bundles, via the so-called \textit{quantum principal bundles}, which are defined in terms of $*$-algebras. In his work it is shown that every quantum principal bundle with a compact structure group is a HGE.
		\item HGE for the Drinfel'd double of the Taft algebra \cite{Sch8}. In general, the Drinfel'd double of a finite dimensional $\Bbbk$-Hopf algebra $H$ is the tensor product algebra $D(H)=H\otimes H^*$. This construction is a quasi-triangular Hopf algebra. Roughly speaking, the Drinfel'd double of the Taft algebra can be seen as $U_q(\mathfrak{sl}_2(\Bbbk))'$ with two copies of the group-like generators. Schauenburg classifies all the Galois objects over $U_q(\mathfrak{sl}_2(\Bbbk))'$ and $D(T_{n^2}(\omega))$.
		\item HGE for pointed Hopf algebras \cite{Gun}. G{\"u}nther developed a systematic method to calculate cleft extensions for pointed Hopf algebras. In particular, cleft extensions for the quantum enveloping algebra $U_q(\mathfrak{sl}_2(\Bbbk))$ (see Example~\ref{ex32}) and the Frobenius-Lusztig kernel $U_q(\mathfrak{sl}_2(\Bbbk))'$ were classified.
		\item Reduced enveloping algebras \cite[Section 6]{Skr}. Let $\Bbbk$ be a field of characteristic $p > 0$, and let $\mathfrak{g}$ be a finite dimensional $p$-Lie algebra over $\Bbbk$. For $\xi\in \mathfrak{g}^*$, denote by $u_\xi(\mathfrak{g})$ the corresponding reduced enveloping algebra of $\mathfrak{g}$. In other words $u_\xi(\mathfrak{g})$ is the factor algebra of the universal enveloping algebra $U(\mathfrak{g})$ by its ideal generated by central elements $x^p - x^{[p]} - \xi(x)^p1$, with $x\in \mathfrak{g}$. Skryabin shows that $u_\xi(\mathfrak{g})$ is a $u_0(\mathfrak{g})^*$-Galois extension if and only if $u_\xi(\mathfrak{g})$ is central simple.
		\item HGE for Calabi-Yau Hopf algebras \cite{Yu}. Yu showed that Hopf Galois objects of a twisted Calabi-Yau Hopf algebra with bijective antipode are also twisted Calabi-Yau, and described explicitly their Nakayama automorphism. As an application, Yu shows that cleft objects (see Definition~\ref{d18}) of twisted Calabi-Yau Hopf algebras and Hopf Galois objects of the quantum automorphism groups of non-degenerate bilinear forms are twisted Calabi-Yau.
		\item HGE for monoidal Hom-Hopf algebras \cite{CZ}. The concept of $\operatorname{Hom}$-Hopf algebra is similar to that of monoid object in a monoidal category, and have appearances in some physical contexts. Chen and Zhang generalized the Schneider’s affineness theorem (\cite[Theorem 8.5.6]{Mon1}) for monoidal $\operatorname{Hom}$-Hopf algebras in terms of total integrals and $\operatorname{Hom}$-Hopf Galois extensions.
		\item Zhang twists \cite[to appear]{HNUVVW}. The notion of a Zhang twist of a graded algebra $A$ was introduced as a deformation of the original graded product by a graded authomorphism $\phi$ of $A$. It is denoted by $A^{\phi}$. Among other results, the mentioned paper shows that the Zhang twist $H^{\phi}$ of a graded Hopf algebra $H$ is left (resp. right) $H$-cleft (see Definition~\ref{d18}) by realizing $H^{\phi} \cong {}_{\sigma} H$ (resp. $H^{\phi} \cong H_{\sigma^{-1}}$) for a suitable 2-cocycle $\sigma$ (see Example~\ref{twist}).
	\end{itemize}
	
	\subsection{Properties of Hopf Galois extensions}\label{sec2.9}
	
	Since their first appearance Hopf Galois extensions have been intensively studied, and used as a tool in the investigation and classification of Hopf algebras themselves. In this section, we review some of those studies by following~\cite{DNR,Mon1,Mon2,Schn}.
	
	The first result we present is the structure of Hopf module morphisms for the Galois map. The proof is a routine check of~\ref{HM3} for each case.
	
	\begin{proposition}[{\cite[Remark 1.1]{Schn}}]\label{p9}
		Let $A^{\operatorname{co}H} \subset A$ be a right $H$-Galois extension. Then
		\begin{enumerate}[label=\normalfont(\roman*)]
			\item $A\otimes_{A^{\operatorname{co}H}} A$ is a left-right $(A,H)$-Hopf module with left $A$-module structure given by
			\begin{equation*}
				a(x\otimes y)=ax \otimes y, \qquad \mbox{for all } a,x,y\in A,
			\end{equation*}
			and right $H$-comodule structure given by
			\begin{equation}\label{e51}
				x\otimes y \mapsto x_{(0)} \otimes y \otimes x_{(1)}, \qquad \mbox{for all } x,y\in A.
			\end{equation}
			\item $A\otimes H$ is a left-right $(A,H)$-Hopf module with left $A$-module structure given by
			\begin{equation*}
				a(x\otimes h)=ax \otimes h, \qquad \mbox{for all } a,x\in A \mbox{ and } h\in H,
			\end{equation*}
			and right $H$-comodule structure given by
			\begin{equation}\label{e53}
				x\otimes h \mapsto  x_{(0)} \otimes h_{(2)} \otimes x_{(1)}S(h_{(1)}), \qquad \mbox{for all } x\in A \mbox{ and } h\in H.
			\end{equation}
			\item $A\otimes_{A^{\operatorname{co}H}} A$ is a right-right $(A,H)$-Hopf module with right $A$-module structure given by
			\begin{equation*}
				(x\otimes y)a=x \otimes ya, \qquad \mbox{for all } a,x,y\in A,
			\end{equation*}
			and right $H$-comodule structure given by
			\begin{equation}\label{e55}
				x\otimes y \mapsto x \otimes y_{(0)} \otimes y_{(1)}, \qquad \mbox{for all } x,y\in A.
			\end{equation}
			\item $A\otimes H$ is a right-right $(A,H)$-Hopf module with right $A$-module structure given by
			\begin{equation*}
				(x\otimes h)a= xa_{(0)} \otimes ha_{(1)} , \qquad \mbox{for all } a,x\in A \mbox{ and } h\in H,
			\end{equation*}
			and right $H$-comodule structure given by
			\begin{equation}\label{e57}
				x\otimes h \mapsto x \otimes h_{(1)} \otimes h_{(2)}, \qquad \mbox{for all } x\in A \mbox{ and } h\in H.
			\end{equation}
		\end{enumerate}
		Hence, the Galois map $\beta: A\otimes_{A^{\operatorname{co}H}} A \rightarrow A \otimes H$ is a morphism in both ${}_A{\operatorname{Mod}}^H$ and ${\operatorname{Mod}}^H_A$.
	\end{proposition}

	A classical theorem in Galois theory says that, if $F\subset E$ is a finite Galois extension of fields with Galois group $G$, then there exists $a\in E$ such that $\{ x\cdot a : x\in G \}$ is basis for $E$ over $F$. Such feature is known as the \emph{normal basis property}. An important question in Hopf Galois theory is whether such property is satisfied.
	
	\begin{definition}[Normal basis property]
		Let $A$ be a right $H$-comodule algebra and consider the algebra extension $A^{\operatorname{co}H}\subset A$ (not necessarily being Galois). We say that the extension has the \emph{right normal basis property}, if $A \cong A^{\operatorname{co}H} \otimes H$ as left $A^{\operatorname{co}H}$-modules and right $H$-comodules.
	\end{definition}
	
	We review some basic facts about finite-dimensional Hopf algebras in order to show that, in this case, the normal basis property is equivalent to the classical notion.
	
	\begin{definition}[Integrals]
		Let $H$ be a $K$-Hopf algebra.
		\begin{enumerate}[label=\normalfont(\roman*)]
			\item A \emph{left integral} in $H$ is an element $\lambda\in H$ such that $h\lambda=\varepsilon(h)\lambda$, for all $h\in H$.
			\item A \emph{right integral} in $H$ is an element $\lambda'\in H$ such that $\lambda'h=\varepsilon(h)\lambda'$, for all $h\in H$.
		\end{enumerate}
		We denote the space of left (resp. right) integral by $\int_H^l$ (resp. $\int_H^r$).
	\end{definition}
	
	If $H$ is such that $\int_H^l=\int_H^r$, it is called \emph{unimodular}. (Counter)examples of such situation can be found in \cite[Examples 2.1.2]{Mon1}.
	
	Using the Fundamental Theorem of Hopf modules, Larson and Sweedler proved that for a finite-dimensional $\Bbbk$-Hopf algebra, both $\int_H^l$ and $\int_H^r$ are one-dimensional. Moreover, if $\lambda \in \int_H^l$ and $\lambda\neq 0$, $H$ is a cyclic left $H^*$-module with generator $\lambda$ for the action $\rightharpoonup$ described in Example~\ref{ex34}. That is, we can identify $H$ with $H^*\rightharpoonup \lambda$. As a consequence, Maschke’s theorem for Hopf algebras states that $H$ is semisimple if and only if $\varepsilon(\int_H^l)\neq 0$ if and only if $\varepsilon(\in_H^r)\neq 0$. In this case, $\int_H^l = \int_H^r$ and we may choose $\lambda$ such that $\varepsilon(\lambda)=1$ (see e.g. \cite[Theorems 2.1.3 and 2.2.1]{Mon1}). We use these results to prove the following.
	
	\begin{proposition}[e.g. {\cite[Lemma 3.5]{Mon2}}]
		Let $H$ be a finite-dimensional $\Bbbk$-Hopf algebra such that $\dim_\Bbbk(H)=n$ and let $A$ be a right $H$-comodule algebra. Consider $H^*$ acting on $A$ with $A^{H^*}=A^{\operatorname{co}H}$. Then the following assertions are equivalent:
		\begin{enumerate}[label=\normalfont(\roman*)]
			\item There exists $u\in A$ and $\{f_i\} \subset H^*$ such that $\{ f_1 \cdot u, \ldots, f_n \cdot u \}$ is a basis for the free left $A^{\operatorname{co}H}$-module $A$ (i.e., $A$ has a normal basis over $A^{\operatorname{co}H}$ in the classical sense).
			\item The algebra extension $A^{\operatorname{co}H}\subset A$ has the right normal basis property.
		\end{enumerate}
	\end{proposition}
	
	\begin{proof}
		(i)$\Rightarrow$(ii): Assume that $A$ has a normal basis over $A^{\operatorname{co}H}$ in the classical sense. Using the identification $H=H^*\rightharpoonup \lambda$ of the previous paragraph, we may consider the left $H^*$-module map $\phi: A^{\operatorname{co}H} \otimes H \rightarrow A$ given by
		\begin{equation*}
			\phi(b\otimes(f\rightharpoonup \lambda)):=b(f\cdot u), \qquad \mbox{for all } b\in A^{\operatorname{co}H} \mbox{ and } f\in H^*.
		\end{equation*}
		$A^{\operatorname{co}H}\otimes H$ is a left $H^*$ module via $f\cdot(b\otimes h)=b\otimes(f\rightharpoonup h)$, since this is the dual of the right comodule structure given by $\operatorname{id} \otimes \Delta$. Thus, $\phi$ is a right $H$-comodule map. Since this is a left $A^{\operatorname{co}H}$-module isomorphism, $A^{\operatorname{co}H}$ has the normal basis property.
		
		(ii)$\Rightarrow$(i): Now, assume that there exists an isomorphism $\phi: A^{\operatorname{co}H} \otimes H \rightarrow A$ of left $A^{\operatorname{co}H}$-modules and right $H$-comodules. Given a $\Bbbk$-basis $\{ f_1, \ldots, f_n \}$ for $H^*$, the set $$\{ 1\otimes (f_1\rightharpoonup \lambda), \ldots, 1\otimes (f_n\rightharpoonup \lambda) \}$$ is an $A^{\operatorname{co}H}$-basis for $A^{\operatorname{co}H}\otimes H$.  Hence, it follows that $\{ f_1 \cdot u, \ldots, f_n \cdot u \}$ is an $A^{\operatorname{co}H}$-basis for $A$, where $u=\phi(1\otimes \lambda)$. Thus $A^{\operatorname{co}H}\subset A$ has a normal basis in the usual sense.
	\end{proof}
	
	In general, not all Hopf Galois extensions have the normal basis property \cite[Example 8.2.3]{Mon1}. However, we mention the work of Doi and Takeuchi that characterizes extensions having the normal basis property.
	
	\begin{definition}[Cleft extension]\label{d18}
		Let A be a right $H$-co\-module algebra. The algebra extension $A^{\operatorname{co}H}\subset A$ is said to be \emph{$H$-cleft} if there exists a right $H$-comodule map $\gamma: H \rightarrow B$ which is convolution invertible.
	\end{definition}
	
	\begin{remark}
		If $A^{\operatorname{co}H}\subset A$ is a $H$-cleft extension, then the map $\gamma$ can always be chosen \emph{normalized}, in the sense that $\gamma(1)=1$. Indeed, if it is not normalized we can replace $\gamma$ by $\gamma'=\gamma^{-1}(1)\gamma$. This is possible since $1$ is a group-like element and hence $\gamma(1)$ is invertible with inverse $(\gamma(1))^{-1}=\gamma^{-1}(1)$.
	\end{remark}

	\begin{theorem}[e.g. {\cite[Theorem 6.4.12]{DNR}}]\label{t8}
		Let $A$ be a right $H$-comodule algebra. Then the following assertions are equivalent:
		\begin{enumerate}[label=\normalfont(\roman*)]
			\item There exists an invertible $2$-cocycle and an action of $H$ on $A^{\operatorname{co}H}$ such that $A\cong A^{\operatorname{co}H} \#_\sigma H$. 
			\item The extension $A^{\operatorname{co}H} \subset A$ is $H$-cleft.
			\item $A^{\operatorname{co}H} \subset A$ is a $H$-Galois extension and has the normal basis property.
		\end{enumerate}
	\end{theorem}
	
	Roughly speaking, this result establishes that every Hopf crossed product is ``cleft'' using the convolution invertible map $\gamma : H\rightarrow A$ given by $\gamma(h)=1\#h$.
	
	Now, we discuss some conditions equivalent to $A^H \subset H$ being $H^*$-Galois.
	
	\begin{lemma}[e.g. {\cite[Lemma 8.3.2]{Mon1}}]\label{l22}
		Let $H$ be an arbitrary Hopf algebra and $A$ a left $H$-module algebra. Then the following assertions hold:
		\begin{enumerate}[label=\normalfont(\roman*)]
			\item $A$ is a left $A\# H$-module, via $(a\#h) \cdot b := a(h\cdot b)$, for all $a,b\in A$ and $h\in H$.
			\item $A$ is a right $A^H$-module, via right multiplication.
			\item $A$ is a $(A\#H,A^H)$-bimodule.
			\item $A^H \cong \operatorname{End}( {}_{A\# H}A )^{\operatorname{op}}$ as algebras.
		\end{enumerate}
	\end{lemma}
	
	\begin{proof}
	(i) and (ii) are straightforward.
	
	(iii) Let $a\in A^H$, $b\# h \in A\# H$ and $c\in A$. Hence, using the associativity of $A$ we have
	\begin{align*}
		(b\# h) \cdot (ca) &=  b(h\cdot (ca)) \overset{\mbox{\eqref{e11}}}{=} b\left[ (h_{(1)} \cdot c) (h_{(2)} \cdot a) \right] = b\left[ (h_{(1)} \cdot c) \varepsilon(h_{(2)})a \right] \\
		&=b[(h\cdot c)a] = [ b(h\cdot c) ]a = 	[(b\#h)\cdot c]a.
	\end{align*}
	 (iv) Let $\psi: A^H \rightarrow \operatorname{End}({}_{A\# H}A)$ be given by $a\mapsto a_r$, where $a_r$ is the right multiplication by $a$ in $A^H$. Then $a_r$ is indeed a ${A\# H}$-map since, for all $a\in A^H$, $b\# h \in A\# H$ and $c\in A$, we have
	\begin{align*}
		a_r((b\#h)\cdot c) = [(b\#h)\cdot c]a = (b\#h)\cdot (ca)= (b\#h)\cdot a_r(c).
	\end{align*}
	On the other hand, if $a,b\in A^H$ and $x\in {}_{A\# H}A$, then
	\begin{align*}
		\psi(a+b)(x)&=(a+b)_r(x)=x(a+b)=xa+xb=a_r(x)+b_r(x)=\psi(a)(x)+\psi(b)(x),\\
		\psi(ab)(x)&= (ab)_r(x)=x(ab)=(xa)b=(a_r(x))b=b_r(a_r(x))=b_ra_r(x),\\
		\psi(1_A)(x) &= (1_A)_r(x)=x1_A = x = \operatorname{id}_A(x).
	\end{align*}
	Hence, $\psi$ is an algebra anti-morphism. Clearly $a_r=0$ if and only if $a1=a=0$, so $\psi$ is injective. Moreover, it is surjective since for any $\sigma\in \operatorname{End}({}_{A\# H}A)$ and $a\in A$, we have
	\begin{align*}
		\sigma(a)&=\sigma(a1_A)=\sigma( a(1_H\cdot 1_A) )=\sigma( (a\#1_H)\cdot 1_A )\\&=(a\#1_H) \cdot \sigma(1_A) = a (1_H\cdot \sigma(1_A))=a\sigma(1_A),
	\end{align*}
	and hence $\sigma=\sigma(1_A)_r=\psi(\sigma(1_A))$. Then $\sigma(1_A)$ is indeed an element of $A^H$ since, for any $h\in H$,
	\begin{equation*}
		h\cdot \sigma(1_A)=(1_A\# h)\cdot \sigma(1_A) = \sigma((1_A\# h)\cdot 1) = \sigma(h \cdot 1_A ) \overset{\mbox{\eqref{e11}}}{=} \varepsilon(h) \sigma(1_A). \qedhere
	\end{equation*}
\end{proof}
	
	When the antipode is bijective, the laterality in this result can be interchanged.
	
	\begin{lemma}[{\cite[Lemma 0.3]{CFM}}]\label{l23}
		Let $H$ be a Hopf algebra such that the antipode $S$ is bijective, and let $A$ be a left $H$-module algebra. Then the following assertions hold:
		\begin{enumerate}[label=\normalfont(\roman*)]
			\item $A$ is a right $A\# H$-module, via
			\begin{equation*}
				b \cdot (a\# h) = S^{-1}(h) \cdot (ba), \qquad \mbox{for all } a,b\in A \mbox{ and } h\in H,
			\end{equation*}
			\item $A^H \cong \operatorname{End}( A_{A\# H})$ as algebras.
		\end{enumerate}
	\end{lemma}
	
	\begin{proof}
	(i) The proof is similar to that of Lemma~\ref{l22}.
	
	(ii) Let $\psi: A^H \rightarrow \operatorname{End}( A_{A\# H})$ be given by $a \mapsto a_l$, where $a_l$ is the left multiplication by $a$ in $A^H$. Then $a_l$ is indeed an ${A\# H}$-map since, for all $a\in A^H$, $b\in A$ and $c\# h \in A\# H$, we have
	\begin{align*}
		a_l(b \cdot (c\# h)) &= a_l( S^{-1}(h) \cdot (bc) )= a( S^{-1}(h) \cdot (bc) ) \\&=  S^{-1}(h) \cdot (abc) = (ab) \cdot (c\# h)=a_l(b) \cdot (c\# h).
	\end{align*}
	Here we used that the multiplication of $A$ is a $H$-module map. On the other hand, if $a,b\in A^H$ and $x\in {A}_{A\# H}$, then
	\begin{align*}
		\psi(a+b)(x)&=(a+b)_lr(x)=(a+b)x=ax+bx=a_l(x)+b_l(x)=\psi(a)(x)+\psi(b)(x),\\
		\psi(ab)(x)&= (ab)_l(x)=(ab)x=a(bx)=a(b_l(x))=a_l(b_l(x))=a_lb_l(x),\\
		\psi(1_A)(x) &= (1_A)_l(x)=1_Ax=x=\operatorname{id}_A(x).
	\end{align*}
	Hence, $\psi$ is an algebra morphism. Clearly $a_l=0$ if and only if $1a=a=0$, so $\psi$ is injective. Moreover, it is surjective since for any $\sigma\in \operatorname{End}({A}_{A\# H})$ and $a\in A$, we have
	\begin{align*}
		\sigma(a)&=\sigma( 1_H\cdot a )= \sigma( S^{-1}(1_H)\cdot (1_Aa) ) = \sigma(1_A \cdot (a\# 1_H)) = \sigma(1_A)\cdot (a\# 1_H) \\
		&= S^{-1}(1_H) \cdot (\sigma(1_A)a) 	= \sigma(1_A)a
	\end{align*}
	and hence $\sigma=\sigma(1_A)_l=\psi(\sigma(1_A))$. Then $\sigma(1_A)$ is indeed an element of $A^H$ since, for any $h\in H$,
	\begin{equation*}
		h\cdot \sigma(1_A)= \sigma(1_A) \cdot (1_A \# S(h)) = \sigma( 1_A \cdot (1_A \# S(h)) ) = \sigma(h \cdot 1_A)\overset{\mbox{\eqref{e11}}}{=} \varepsilon(h) \sigma(1). \qedhere
	\end{equation*}
\end{proof}
	
	Using Lemma~\ref{l22} and the notion of \emph{left trace function} (i.e., maps of the form $\hat{\lambda}:A\rightarrow A^H$ such that $\hat{\lambda}(a)=\lambda a$, for some left integral $\lambda \neq 0$ in $H$), works of Doi, Kreimer, Takeuchi and Ulbrich lead to the following characterization of $H^*$-Galois extensions \cite{DT2,KT,Ulb2}.
	
	\begin{theorem}[e.g. {\cite[Theorem 8.3.3]{Mon1}}]\label{t7}
		Let $H$ be a finite-dimensional $\Bbbk$-Hopf algebra and $A$ a left $H$-module algebra (and thus, a right $H^*$-comodule). Then the following assertions are equivalent:
		\begin{enumerate}[label=\normalfont(\roman*)]
			\item $A^H \subset A$ is a right $H^*$-Galois extension.
			\item The map $\pi: A\# H \rightarrow \operatorname{End}(A_{A^{H}})$ is an algebra morphism and $A$ is finitely generated projective as a right $A^H$-module.
			\item $A$ is a generator for the category of left $A\# H$-modules.
			\item If $0\neq\lambda \in \int_H^l $, then the map $[-,-]: A\otimes_{A^H} A \rightarrow A\# H$, given by $[a,b]=a\lambda b$, is surjective.
			\item For any left $A\# H$-module $M$, consider $A \otimes_{A^H} M^H$ as a left $A\# H$-module by letting $A\# H$ act on $A$ via $\pi$. Then the map $\Phi: A \otimes_{A^H} M^H \rightarrow M$, given by $a \otimes m \mapsto a\cdot m$, is a left $A\# H$-module isomorphism.
		\end{enumerate}
	\end{theorem}
	
In \cite[Example 4.6]{Mon2} it is shown that, even for group actions, Theorem~\ref{t7} does not necessarily holds.

Now, we discuss Morita contexts. Two rings $R,S$ are \emph{connected by a Morita context} if there exist an $(R,S)$-bimodule ${}_{R} V_S$, a $(S,R)$-bimodule ${}_S W_R$ and bimodule morphisms
	\begin{equation*}
		[-,-]:W \otimes_R V \rightarrow S \qquad \mbox{and} \qquad (-,-): V \otimes_S W \rightarrow R
	\end{equation*}
	such that, for all $v,v'\in V$ and $w,w'\in W$, the relations
	\begin{equation*}
		v' \cdot [w,v]=(v',w)\cdot v \in V \qquad \mbox{and} \qquad [w,v]\cdot w' = w\cdot (v,w') \in W
	\end{equation*}
	hold. This is equivalent to saying that the array
	\begin{equation*}
		T=\begin{bmatrix}
			R & V\\
			W & S
		\end{bmatrix}
	\end{equation*}
	becomes an associative ring, where the formal operations are those of $2\times 2$ matrices, using $[-,-]$ and $(-,-)$ to compute the multiplication.
	
	When the two maps $[-,-]$ and $(-,-)$ are surjective, we say that the rings $R$ and $S$ are \emph{Morita equivalent}. This is equivalent to saying that the four functors
	\begin{align*}
		- \otimes_R V&: \operatorname{Mod}_R \rightarrow \operatorname{Mod}_S, & - \otimes_S W &: \operatorname{Mod}_S \rightarrow \operatorname{Mod}_R\\
		W \otimes_R - &: {}_R{\operatorname{Mod}}\rightarrow {}_S{\operatorname{Mod}} & V \otimes_S - &: {}_S{\operatorname{Mod}} \rightarrow {}_R{\operatorname{Mod}}
	\end{align*}
	are equivalences of categories \cite[Section 3.5.5]{MR}.
	
	We shall see that, for any finite-dimensional Hopf algebra $H$ and any left $H$-module algebra $A$, such a set-up exists for the rings $R=A^H$ and $S=A\#H$, using $V=W=A$. By Lemma~\ref{l22}, we already guaranteed that $A$ is an $(A\# H,A^H)$-bimodule. However, the structure described in Lemma~\ref{l23} is not enough for the other laterality to work. Hence, we proceed as follows: recall that if $H$ is finite-dimensional, then $S$ is bijective and both $\int_H^r$ and $\int_H^l$ are one-dimensional. Also, notice that if $0\neq \lambda \in \int_H^l$, then $\lambda h\in \int_H^l$, for any $h\in H$. Thus, there exists $\alpha \in H^*$ such that $\lambda h = \alpha(h)\lambda$, for all $h\in H$.
	
	With the notation of Example~\ref{ex34}, we define $h^\alpha:= \alpha \rightharpoonup h$. Since $\alpha$ is multiplicative, it is a group-like element of $H^*$ and thus the map $h\mapsto h^\alpha$ is an automorphism of $H$. By \cite{Rad}, $\lambda^\alpha=S(\alpha)$. We define our new right action of $A\# H$ on $A$ by
	\begin{equation}\label{e88a}
		a\cdot (b\# h)=S^{-1}h^\alpha \cdot (ab), \qquad \mbox{for all } a,b\in A \mbox{ and } h\in H.
	\end{equation}
	Notice that this is the action of Lemma~\ref{l23} but ``twisted'' by $\alpha$.
	
	\begin{theorem}[{\cite[Theorem 2.10]{CFM}}]
	Let $H$ be a finite-dimensional Hopf algebra and $A$ a left $H$-module algebra (and hence, a right $H^*$-comodule algebra). Consider $A$ in ${}_{A\# H}{\operatorname{Mod}}_{A^H} $ as in Lemma~\ref{l22}, and in ${}_{A^H}{\operatorname{Mod}}_{A\# H}$ with the right action of $A\# H$ given by \eqref{e88a}. Then $V= {}_{A^H} A_{A\# H}$ and $W={}_{A\# H} A_{A^H}$, together with the maps
		\begin{align*}
			[-,-] &: A \otimes_{A^H} A \rightarrow A \# H  & \mbox{given by } [a,b]&:=a\lambda b,\\
			(-,-) &: A \otimes_{A\# H} A \rightarrow A^H & \mbox{given by } (a,b)&:= \lambda\cdot(ab),
		\end{align*}
		give a Morita context for $A^H$ and $A\# H$.
	\end{theorem}
	
	Using the left trace function, $\hat{\lambda}(A)=\lambda \cdot A = (A,A)$. If we also consider the ideal $A\lambda A = [A,A]$, the next result is immediate.
	
	\begin{corollary}[{\cite[Corollary 2.9]{Mon2}}]
		Let $H$ be a finite-dimensional Hopf algebra and $A$ a left $H$-module algebra. If $0\neq \lambda \in \int_H^l$ is such that the left trace function $\hat{\lambda}: A \rightarrow A^H$ is surjective and $A\lambda A = A\# H$, then $A\# H$ is Morita equivalent to $A^H$.
	\end{corollary}
	
	Since simplicity of $A\# H$ implies that $A\lambda A= A\# H$ and semisimplicity of $H$ implies that the trace function is surjective \cite[Corollary 1.3]{CFM}, the next result follows from Theorem~\ref{t7}.
	
	\begin{corollary}
		Let $H$ be a semisimple finite-dimensional Hopf algebra and $A$ a left $H$-module algebra such that $A\# H$ is a simple algebra. Then $A^H \subset A$ is $H^*$-Galois and $A^H$ is Morita equivalent to $A\# H$.
	\end{corollary}
	
	Finally, we mention a relevant result also involving equivalences of categories of modules.
	
	\begin{theorem}[{\cite[Theorem I]{Schn2}}]
		Let $H$ be an arbitrary Hopf algebra with bijective antipode and $A$ a right $H$-comodule algebra. Then the following assertions are equivalent:
		\begin{enumerate}[label=\normalfont(\roman*)]
			\item $A^{\operatorname{co}H} \subset A$ is right $H$-Galois and $A$ is a faithfully flat left (or right) $A^{\operatorname{co}H}$-module.
			\item The Galois map $\beta : A \otimes_{A^{\operatorname{co}H}} A \rightarrow A \otimes H$ is surjective and $A$ is an injective $H$-comodule.
			\item The functor $\Phi: \operatorname{Mod}_{A^{\operatorname{co}H}} \rightarrow \operatorname{Mod}_A^H $ given by $M \mapsto M \otimes_{A^{\operatorname{co}H}} A$ is an equivalence.
			\item The functor $\Phi': {}_{A^{\operatorname{co}H}}{\operatorname{Mod}} \rightarrow {}_A{\operatorname{Mod}}^H$ given by $M \mapsto A \otimes_{A^{\operatorname{co}H}} M$ is an equivalence.
		\end{enumerate}
	\end{theorem}
	
	Although in this section we tried to cover several properties of Hopf Galois extensions (HGE), the list is large and it would be impossible to address them all. Nevertheless, we shall mention some remarkable recent advances in Hopf Galois theory. The presented list appears in chronological order and is non-exhaustive.
	
	\begin{itemize}
		\item Representation theory of HGE \cite{Schn}. Schneider adressed some questions of representation theory for HGE, such as induction and restriction of simple or indecomposable modules.  In particular, generalizations of classical results on representations of groups and Lie algebras were given.
		\item Maschke's theorem for HGE \cite{Doi}. Classically, Maschke's Theorem states that if $G$ is a finite group and $\Bbbk$ a field whose characteristic does not divide the order of $G$, then the group algebra $\Bbbk G$ is semisimple. For a finitely generated projective Hopf algebra $H$ and an $H$-Galois extension $A^{\operatorname{co}H} \subset A$, Doi proved an analogous of Maschke's Theorem, stating that if $A^{\operatorname{co}H}$ is semisimple Artinian, then so is $A$.
		\item Hopf biGalois objects and Galois theory for HGE \cite{OZ,Sch}. Van Oystaeyen and Zhang proved that if $\Bbbk \subset A$ are fields such that $\Bbbk \subset A$ is a $(L,H)$-biGalois extension, then there is a one-to-one correspondence between the Hopf ideal of $L$ and the $H$-costable intermediate fields $F\subset A$. This correspondence theorem is a generalization of the classical Galois connection in the theory of field extensions. On the other hand, Schauenburg proved that the existence of a $(L,H)$-biGalois object is equivalent to $H$ and $L$ being monoidally co-Morita $\Bbbk$-equivalent, i.e, their monoidal categories of comodules are equivalent as monoidal $\Bbbk$-linear categories. This leads to another Galois correspondence that is studied in \cite{Sch9}. More recently, the existence of a Galois connection between subalgebras of an $H$-comodule algebra and generalized quotients of the Hopf algebra $H$ has been studied in unpublished works of Marciniak and Szamotulski \cite{MSz}.
		\item Hopf Galois coextensions \cite{CWW,DMR}. We studied above the relation between HGE and the normal basis property; however, there exists a coalgebra version of the normal basis property involving the notion of crossed coproduct and cleft coextension, introduced by D{\u a}sc{\u a}lescu, Militaru and Raianu. This is further studied by Caenepeel, Wang and Wang by addressing the notion of Hopf Galois coextension and twisting techniques. More recently, this theory was used in \cite{Hass} to show that Hopf Galois coextensions of coalgebras are the sources of stable anti Yetter-Drinfeld modules.
		\item Hochschild cohomology on HGE \cite{Ste}. {\c{S}}tefan constructed a spectral sequence for $H$-Galois extensions $A^{\operatorname{co}H}\subset A$, and used it to connect the Hochschild cohomologies and homologies of $A$ and $A^{\operatorname{co}H}$.  This is further studied in \cite{MS}.
		\item HGE with central invariants \cite{Rum}. A $H$-Galois extension $A^{\operatorname{co}H}\subset A$ is said to have \emph{central invariants} if $A^{\operatorname{co}H} \subset Z(A)$. Rumynin studied these HGE, addressing some geometric properties which are close to those of principal bundles and Frobenius manifolds.
		\item Prime ideals in HGE \cite{MoS}. Let $H$ be a finite-dimensional Hopf algebra and $A^{\operatorname{co}H}\subset A$ a $H$-Galois extension such that $A$ is a faithfully flat as a left $A^{\operatorname{co}H}$-module. Montgomery and Schneider gave a comparison between the prime ideals of $A^{\operatorname{co}H}$ and of $A$, studying in particular the classical Krull relations. Since Hopf crossed products are examples of faithfully flat Galois extensions, those results were applied to crossed products. These also show that if $H$ is semisimple and semisolvable, then $A$ is semiprime, provided $A^{\operatorname{co}H}$ is $H$-semiprime.
		\item Cyclic homology of HGE \cite{JS}. For a Hopf algebra $H$, the category $\mathcal{CM}_m(H)$ of modular crossed modules over $H$ was introduced by Jara and {\c{S}}tefan. If $M \in \mathcal{CM}_m(H)$ and $L$ is a Hopf subalgebra, these allow the computing of the cyclic homology of $H\otimes_L M$ under certain restrictions for $L$ and $M$. In particular, this was used to calculate the cyclic homology of group algebras, quantum tori and almost symmetric algebras. This is further studied in \cite{HR}.
		\item Algebraic $K$-theory of HGE \cite{AW,BH}. Principal extensions given in Definition~\ref{d19} were firstly introduced by Brzezi{\'n}ski and Hajac, and along the applications given in Section~\ref{hge9}, principal extensions also were used to construct an explicit formula for the Chern-Galois character (which is a homomorphism of Abelian groups that assigns the homology class of an even cyclic cycle to the isomorphism class of a finite-dimensional corepresentation). Later, Ardakov and Wadsley showed that the Cartan map from $K$-theory to G-theory of HGE is a rational isomorphism, provided the subalgebra of coinvariants is regular, the base Hopf algebra is finite-dimensional and its Cartan map is injective in degree zero. In particular, this covers the case of a crossed product of a regular ring with a finite group, and was applied to the study of Iwasawa modules.
		\item Generalized HGE \cite{Sch7}. By definition, an extension $A^{\operatorname{co}H}\subset A$ is $H$-Galois if the Galois map $\beta$ is bijective. Criteria under which surjectivity of $\beta$ (which is usually much easier to verify) is sufficient were studied by Schauenburg. Such conditions were used to investigate the structure of $A$ as an $A^{\operatorname{co}H}$-module and as $H$-comodule. In particular, equivariant projectivity of extensions in several important cases was proven. Moreover, these reconstructed the theory when the Hopf algebra $H$ is interchanged for a quotient coalgebra or an one-sided module of a Hopf algebra.
		\item Homotopy theory of HGE \cite{Hess,KSch}. As we have remarked before, HGE can be viewed as the noncommutative analogues of principal fiber bundles where the role of the 	structural group is played by a Hopf algebra. It is therefore natural to adapt the concept of homotopy to them. Such construction was made by Kassel and Schneider. They classified HGE up to homotopy equivalence and some of their homotopy invariants were studied. Later, Hess developed a theory of homotopic HGE in monoidal categories (with compatible model category structure), which generalizes the case of structured ring spectra.
		\item HGE in braided tensor categories \cite{ZZ}. Braided Hopf algebras have attracted much attention in both mathematics and mathematical physics for playing an important role in the classification of finite-dimensional pointed Hopf algebras (see e.g. \cite{AnS}). The immediate generalization of such setup is the concept of braided tensor categories (BTC). Hence, unpublished work of Zhang and Zhang attempts to generalize Galois theory to BTC by showing that if the category $\mathcal{C}$ is BTC and has (co)equalizers, $A=B \#_\sigma H$ is a crossed product algebra if and only if the extension $B\subset A$ is Galois, the canonical map $q : A\otimes A \rightarrow A \otimes_B A$ is split, and $A$ is isomorphic as a left $B$-module and as a right $H$-comodule to $B \otimes H$ in $\mathcal{C}$ (compare with Theorem~\ref{t8}).
		\item HGE for weak Hopf algebras and Hopf algebroids \cite{Boh,CD}. Both theories of weak Hopf algebras and Hopf algebroids are natural generalizations of the classical notion (and, under some conditions, equivalent). Caenepeel and De Groot have studied the Galois theory for weak Hopf algebras objects, while Böhm for Hopf algebroids.
		\item Morita (auto)equivalences of HGE \cite{CCMT,CaM}. Let $A^{\operatorname{co}H} \subset A$ and $B^{\operatorname{co}H} \subset B$  be two $H$-Galois extensions. Caenepeel, Crivei, Marcus and Takeuchi investigated the category ${}_A{\operatorname{Mod}}_B^H$ of relative Hopf bimodules, and therefore the Morita equivalences between $A$ and $B$ induced by them. More recently, the first and third mentioned authors addressed $H$-Morita autoequivalences of HGE, introduced the concept of $H$-Picard group, and established an exact sequence linking the $H$-Picard group of the comodule algebra $A$ and the Picard group of $A^{\operatorname{co}H}$.
		\item Generic HGE \cite{Kas4,Kas3}. Aljadeff and Kassel studied twisted algebras, which are associative algebras ${}^{\alpha} H$ obtained from a given Hopf algebra $H$ by twisting its product by a cocycle $\alpha$. These coincide with the class of cleft objects (as we saw in the examples, classical Galois extensions and strongly graded algebras belong to this class). The authors attached two universal algebras $U(H)^{\alpha}$ and $A(H)^{\alpha}$ to each twisted algebra ${}^{\alpha}{H}$, so the second author later studied $A(H)^{\alpha}$, which is a generic version of ${}^{\alpha}{H}$. Then, he calculated the generic cocycle cohomologous to the original cocycle $\alpha$, and considered the commutative algebra $B(H)^{\alpha}$ generated by the values of the generic cocycle and of its convolution inverse. It was shown that $A(H)^{\alpha}$ is a cleft $H$-Galois extension of $B(H)^{\alpha}$, called a generic $H$-Galois extension. Some theory regarding versal deformation spaces was also developed.
		\item Cohen-Macaulay invariant subalgebra of dense HGE \cite{JZh}. Let $H$ be a finite-dimensional Hopf algebra and $A$ a left $H$-module (and hence, a right $H^*$-comodule). The algebra extension $A^H \subset A$ is called a $H^*$-dense Galois extension if the cokernel of the Galois map $\beta: A\otimes_{A^H} A \rightarrow A \otimes H^*$ is finite-dimensional (no bijectivity required). Obviously the concept of Hopf dense Galois extension is a weaker version of that of HGE. When $H$ is semisimple and $A$ is left $H$-Noetherian, He and Zhang studied (unpublished work) whether $A^H$ inherits the AS-Cohen-Macaulay property from $A$ under some mild conditions, and whether $A$, when viewed as a right $A^H$-module, is a Cohen-Macaulay module.
		\item HGE for Hopf categories \cite{BCV,CaF}. The concept of $\Bbbk$-algebra can be translated to category theory using the notion of $\Bbbk$-linear category. Similarly, categorical notions of bialgebras and Hopf algebras have been introduced (respectively, $\Bbbk$-linear semi-Hopf categories and $\Bbbk$-linear Hopf categories). Batista, Caenepeel and Vercruysse proved that several classical properties of Hopf algebras can be generalized to Hopf categories. Later, the second mentioned author and Fieremans introduced the notion of a Hopf--Galois category extension. Also, the concept of descent datum (see Definition~\ref{d14}) was categorized.
		\item Galois cowreaths \cite{BT}. Bulacu and Torrecillas studied pre-Galois cowreath, which are a generalization of HGE to monoidal categories via the language of cowreaths (coalgebras in a suitable associative monoidal category associated to an algebra in a monoidal category). 
		\item Discriminant of HGE \cite{Zhu}. The discriminant of an algebra over a commutative ring is a well known construction that has been used to solve several algebraic problems such as the Zariski cancellation problem (see e.g. \cite{LWZ}). In work of Zhu (to appear), the discriminant of the smash product was calculated by finding a formula for the discriminant of a HGE.
		\item Partial HGE \cite{CFPQT}. Using the language of partial actions, a notion of partial (co)module algebra can be given. Hence, the mentioned paper (to appear) introduced Hopf--Galois partial extensions and studied several properties analogous to those presented in this section.
	\end{itemize}
	
	\subsection{Quantum torsors}\label{se4}
	
	We mentioned in Remark~\ref{rem1} that the notion of (classical) $G$-torsor is present in many algebraic formulations of different contexts, such as vector bundle, affine scheme, categorical bundles, etc.; furthermore, dualizing such setup, we motivated the notion of Hopf Galois extension based on the bijectivity of the map $\alpha^*$ (see Section~\ref{hge9}). However, in recent years different approaches to noncommutative torsors have been achieved. In this section, we review the one given by Grunspan which, instead of working with $\alpha^*$ and the freeness of the action, uses a ``parallelogram'' property of torsors \cite{Gru2}. As Zoran \v{S}koda has pointed to us, this same notion was discovered independently by him \cite{Sko}. We shall give a rough motivation to this ``parallelogram'' property with one simple example (see also \cite[Section 1.2]{Gru2} and \cite[Section 1.1]{Sko}).
	
	Recall that, when working with vectors on the Euclidean plane, a point is fixed and called the \emph{origin}. Thus, any point in the plane is identified with the arrow going from the origin to that point. This lets us add points in the plane by adding their arrows (in other words, the parallelogram property), making $\mathbb{R}^2$ a group. However, if we forget the origin, we lost the identification of points with arrows. In this case we cannot longer add them, but we can still subtract two of them and get an arrow. Thus, the plane (without origin) is a $\mathbb{R}^2$-torsor. The moral of this is that, although we are not longer able to explicitly apply the parallelogram property with arrows, we can still associate to three points $a,b,c$ a fourth point $d$ such that $a,b,c,d$ is a parallelogram. In multiplicative notation for an arbitrary $G$-torsor, we have identified the assignation $(a,b,c)\mapsto d:=ab^{-1}c$.
	
	The following axioms dualize this setup to the noncommutative case.
	
	\begin{definition}\label{d13}
		A \emph{quantum $K$-torsor} (or \emph{quantum $K$-heap}) is a $K$-algebra $T$ together with an algebra morphism $\mu: T \rightarrow T\otimes T^{\operatorname{op}} \otimes T $ such that the following relations hold:
		\begin{align}
			(\operatorname{id}_T \otimes \operatorname{id}_{T^{\operatorname{op}}} \otimes \mu)\mu &= (\mu\otimes \operatorname{id}_{T^{\operatorname{op}}} \otimes \operatorname{id}_T)\mu, \label{neq2} \\
			(m\otimes \operatorname{id}_T)u &= u\otimes \operatorname{id}_T.\label{neq4}\\
			(\operatorname{id}_T\otimes m)u &= \operatorname{id}_T\otimes u, \label{neq3}
		\end{align}
	\end{definition}
	
	We extend Heyneman--Sweedler notation by writing $\mu(x)=x^{(1)}\otimes x^{(2)} \otimes x^{(3)}$ for $x\in T$. Hence, condition~\eqref{neq2} can be written as
	\begin{equation}\label{e42}
		\mu(x^{(1)}) \otimes x^{(2)} \otimes x^{(3)}  = x^{(1)} \otimes x^{(2)} \otimes \mu(x^{(3)}),
	\end{equation}
	while \eqref{neq4} and \eqref{neq3} are
	$x^{(1)} \otimes x^{(2)}x^{(3)} = x \otimes 1_T$ and $x^{(1)} x^{(2)} \otimes x^{(3)} = 1_T \otimes x$, respectively.
	
	The torsor is said to be \emph{commutative} if $T$ is a commutative algebra. If $\mu=\mu^{\operatorname{op}}$, where $\mu^{\operatorname{op}}(x)=x^{(3)}\otimes x^{(2)} \otimes x^{(1)}$, the torsor is said to be \emph{equipped with a commutative law}. A morphism of quantum torsors is an algebra map $\phi: T \to T'$ such that $\mu_{T'}(\phi(x)) = (\phi \otimes \phi \otimes \phi) \mu_{T}(x)$, for all $x\in T$. The category of quantum torsors is denoted by $\operatorname{QHeaps}$. The main theorem of \cite{Sko} established a categorical isomorphism between certain subcategory of $\operatorname{QHeaps}$ and the category $K$-$\operatorname{HopfAlg}$.
	
	In order to relate quantum torsors with Hopf Galois extensions, we briefly recall the mechanism of \emph{faithfully flat descent} for extensions of noncommutative rings and a related result.
	
	\begin{definition}[Descent datum]\label{d14}
		Let $R$ be a subring of the ring $S$, with the inclusion map denoted by $\eta: R\rightarrow S$, and let $M$ be a left $S$-module with structure map $\gamma: S \otimes M \rightarrow M$. An \emph{$S/R$-descent datum on $M$} is a left $S$-module map $D:M \rightarrow S\otimes_R M$ such that
		\begin{equation}\label{dd}
			(\operatorname{id}_S\otimes_R D)D = (\operatorname{id}_S\otimes_R \eta \otimes_R \operatorname{id}_M)D \qquad \mbox{and} \qquad \gamma D = \operatorname{id}_M.
		\end{equation}
	\end{definition}
	
	Consider the pairs $(M,D)$, consisting of a $S$-module $M$ and a $S/R$-descent datum $D$ on $M$, together with arrows $f: (M,D)\rightarrow (N,E)$, where $f:M\rightarrow N$ is a $S$-module morphism such that $Ef=(\operatorname{id}_S\otimes f)D$. This category is denoted by $\operatorname{DD}(S/R)$. When $S$ is faithfully flat as a right $R$-module, there exists an equivalence between the category of left $R$-modules and $\operatorname{DD}(S/R)$. We present a formulation of this result due to Schauenburg.
	
	\begin{lemma}[{\cite[Section 1.3]{BR}}]\label{l6}
		Let $R$ be a subring of the ring $S$ with the inclusion map denoted by $\eta: R\rightarrow S$. If the category of left $R$-modules is denoted by $ {}_R{\operatorname{Mod}}$, then the assignation ${}_R{\operatorname{Mod}} \rightarrow \operatorname{DD}(S/R) $ given by $N \mapsto (S\otimes_R N,D)$, where
		\begin{equation*}
			D(s\otimes n)= s\otimes 1 \otimes n \in S\otimes_R\otimes S\otimes_R N,
		\end{equation*}
		induces a functor. Moreover, if $S$ is faithfully flat as a right $R$-module, then this functor is an equivalence. The inverse equivalence is given in objects as
		\begin{equation*}
			(M,D) \mapsto {}^D M:=\{m\in M: D(m)=1\otimes m \}.
		\end{equation*}
		In particular, for every descent datum $(M,D)$, the map $f:S\otimes_R {}^D M\rightarrow M$ given by $s\otimes m \mapsto sm$ is an isomorphism with inverse induced by $D$, i.e., $f^{-1}:M\rightarrow S \otimes_R  {}^D{M}$ is given by $f^{-1}(m)=D(m)$.
	\end{lemma}
	
	The next result shows that every torsor induces a descent datum.
	
	\begin{lemma}[{\cite[Lemma 3.3]{Sch6}}]\label{l7}
		Let $T$ be a quantum $K$-torsor. If $\mu: T \rightarrow T\otimes T^{\operatorname{op}} \otimes T $ is the associated map to $T$, then
		\begin{equation*}
			D:=(m\otimes \operatorname{id}_T \otimes \operatorname{id}_T)(\operatorname{id}_T\otimes \mu) : T\otimes T \rightarrow T\otimes T \otimes T
		\end{equation*}
		is a $T/K$-descent datum on the left $T$-module $T\otimes T$. Moreover, it satisfies
		\begin{equation*}
			(\operatorname{id}_T\otimes D)\mu(x)= x^{(1)} \otimes 1 \otimes x^{(2)} \otimes x^{(3)} = (\operatorname{id}_T\otimes 1_T \otimes \operatorname{id}_T\otimes \operatorname{id}_T )\mu(x).
		\end{equation*}
	\end{lemma}
	
	\begin{proof}
		In Heyneman--Sweedler notation, for every $x\otimes y\in T\otimes T$, we have
		\begin{equation*}
			D(x\otimes y) = xy^{(1)} \otimes y^{(2)} \otimes y^{(3)}.
		\end{equation*}
		Left $T$-linearity of this map is obvious. Additionally, for every $x\otimes y \in T\otimes T$ we have
		\begin{align*}
			(\operatorname{id}_T \otimes D)D(x\otimes y) &= (\operatorname{id}_T \otimes D)(xy^{(1)} \otimes y^{(2)} \otimes y^{(3)}) = xy^{(1)} \otimes D( y^{(2)} \otimes y^{(3)})\\
			&= xy^{(1)} \otimes (m\otimes \operatorname{id}_T \otimes \operatorname{id}_T)( y^{(2)} \otimes \mu(y^{(3)}) ) \\
			&=  ( m \otimes m \otimes \operatorname{id}_T \otimes \operatorname{id}_T ) ( x \otimes y^{(1)} \otimes y^{(2)} \otimes  \mu(y^{(3)}) )\\
			&=  ( m \otimes m \otimes \operatorname{id}_T \otimes \operatorname{id}_T ) ( x \otimes \mu (y^{(1)}) \otimes y^{(2)} \otimes  y^{(3)})\\
			&= ( m \otimes \operatorname{id}_T \otimes \operatorname{id}_T \otimes \operatorname{id}_T ) ( x \otimes {y^{(1)}}^{(1)} \otimes {y^{(1)}}^{(2)}{y^{(1)}}^{(3)} \otimes y^{(2)} \otimes  y^{(3)})\\
			&= ( m \otimes \operatorname{id}_T \otimes \operatorname{id}_T \otimes \operatorname{id}_T ) ( x \otimes y^{(1)} \otimes 1_T \otimes y^{(2)} \otimes  y^{(3)})\\
			&= xy^{(1)} \otimes  1_T \otimes y^{(2)} \otimes  y^{(3)}\\
			&= (\operatorname{id}_T \otimes 1_T \otimes \operatorname{id}_{T\otimes T}) (xy^{(1)} \otimes y^{(2)} \otimes y^{(3)})\\
			&=(\operatorname{id}_T \otimes 1_T \otimes \operatorname{id}_{T\otimes T})D(x\otimes y),
		\end{align*}
		which is the left condition of \eqref{dd}. For the right one, we have
		\begin{align*}
			\gamma_{T\otimes T} D(x\otimes y) &=\gamma_{T\otimes T}(xy^{(1)} \otimes y^{(2)} \otimes y^{(3)}) = xy^{(1)} y^{(2)} \otimes y^{(3)}\\
			&= (m \otimes \operatorname{id}_T ) (x \otimes y^{(1)}y^{(2)} \otimes y^{(3)}) = (m \otimes \operatorname{id}_T ) (x \otimes 1_T \otimes y) = x\otimes y.
		\end{align*}
		Hence, $D$ is a $T/K$-descent datum on the left $T$-module $T\otimes T$.
	\end{proof}
	
	Now we are ready to prove that every faithfully flat quantum $K$-torsor $T$ is a right $H$-Galois object, for a suitable Hopf algebra $H$ provided by the categorical equivalence of Lemma~\ref{l7}.
	
	\begin{theorem}[{\cite[Theorem 3.4]{Sch6}}]\label{t6}
		Let $T$ be a faithfully flat quantum $K$-torsor and
		\begin{equation*}
			H:={}^D (T\otimes T)= \{ x\otimes y \in T\otimes T : xy^{(1)} \otimes y^{(2)} \otimes y^{(3)} = 1\otimes x \otimes y \}.
		\end{equation*}
		Then the following assertions hold:
		\begin{enumerate}[label=\normalfont(\roman*)]
			\item $H$ is a Hopf algebra. The algebra structure is that of a subalgebra of $T^{\operatorname{op}}\otimes T$; comultiplication and counit are given by
			\begin{equation*}
				\Delta(x\otimes y)=x\otimes y^{(1)} \otimes y^{(2)} \otimes y^{(3)} \qquad \mbox{and} \qquad \varepsilon(x\otimes y )=xy.
			\end{equation*}
			\item $T$ is a right $H$-comodule algebra with structure map given by
			\begin{equation*}
				\rho(x)=\mu(x)=x^{(1)} \otimes x^{(2)} \otimes x^{(3)}.
			\end{equation*}
			Moreover, $T^{\operatorname{co}H}=K$.
			\item $T$ is a right $H$-Galois object.
		\end{enumerate}	
	\end{theorem}
	
	\begin{proof}
		(i) Since $\mu$ is a morphism of algebras, given $x\otimes y, a\otimes b\in H$, we have
		\begin{align*}
			D[(x\otimes y)(a\otimes b)]&=D(xa\otimes yb)=xa(yb)^{(1)} \otimes (yb)^{(2)} \otimes (yb)^{(3)} \\
			&= xa y^{(1)} b^{(1)} \otimes b^{(2)} y^{(2)}  \otimes y^{(3)} b^{(3)}= ab^{(1)}\otimes b^{(2)}x \otimes yb^{(3)} \\&= 1 \otimes ax \otimes yb = 1 \otimes xa \otimes yb  = 1 \otimes (x\otimes y)(a\otimes b).
		\end{align*}
		Here we used that the second coordinate of the tensor is in $T^{\operatorname{op}}$. The above proves that $H$ is a subalgebra of $T^{\operatorname{op}}\otimes T$. Now, since $H$ is faithfully flat, it is the equalizer of
		\begin{equation*}
			\begin{tikzcd}[column sep=6em]
				T\otimes T \otimes H \arrow[r, shift left=-1ex, "D\otimes \operatorname{id}_H"'] \arrow[r,shift left=1ex, "u_T \otimes \operatorname{id}_T  \otimes \operatorname{id}_T \otimes \operatorname{id}_H"] & T \otimes T \otimes T \otimes H
			\end{tikzcd}
		\end{equation*}
		and thus the image of $\Delta$ is contained in $H\otimes H$, showing that $\Delta$ is well defined. Now, if $x\otimes y \in H$, then
		\begin{align*}
			[(\operatorname{id}_H \otimes \Delta)\Delta](x\otimes y)&=(\operatorname{id}_H\otimes \Delta)(x\otimes y^{(1)} \otimes y^{(2)} \otimes y^{(3)})\\&= x\otimes y^{(1)} \otimes y^{(2)} \otimes (y^{(3)})^{(1)} \otimes (y^{(3)})^{(2)} \otimes (y^{(3)})^{(3)},
		\end{align*}
		while
		\begin{align*}
			[(\Delta\otimes \operatorname{id}_H)\Delta](x\otimes y)&= (\Delta\otimes \operatorname{id}_H)(x\otimes y^{(1)} \otimes y^{(2)} \otimes y^{(3)})\\&= x\otimes (y^{(1)})^{(1)} \otimes (y^{(1)})^{(2)} \otimes (y^{(1)})^{(3)} \otimes y^{(2)} \otimes y^{(3)}.
		\end{align*}
		By \eqref{e42} these two expressions are equivalent, proving the coassociativity of $\Delta$. Moreover, $\Delta$ is an algebra map since $\mu$ is so. On the other hand, if $x\otimes y \in H$, we have
		\begin{equation*}
			xy\otimes 1 = xy^{(1)} \otimes y^{(2)}y^{(3)}=1\otimes xy,
		\end{equation*}
		whence $xy\in K$ by faithful flatness of $T$. This proves that $\varepsilon$ is well defined. Moreover,
		\begin{align*}
			[(\varepsilon \otimes \operatorname{id}_H)\Delta](x\otimes y)&=  (\varepsilon\otimes\operatorname{id}_H)(x\otimes y^{(1)} \otimes y^{(2)} \otimes y^{(3)}) \\&= xy^{(1)} \otimes y^{(2)} \otimes y^{(3)} = 1\otimes x \otimes y,\\
			[(\operatorname{id}_H \otimes \varepsilon)\Delta](x\otimes y)&=(\operatorname{id}_H \otimes \varepsilon) (x\otimes y^{(1)} \otimes y^{(2)} \otimes y^{(3)}) \\&= x \otimes y^{(1)} \otimes y^{(2)} y^{(3)} = x \otimes y \otimes 1,
		\end{align*}
		which proves the main property of the counit. It is easy to check that $\varepsilon$ is an algebra morphism. Therefore $H$ is a $K$-bialgebra. Once (ii)-(iii) are proven below, we can invoke Remark~\ref{r1} to guarantee that $H$ is in fact a Hopf algebra.
		
		(ii) In order to prove that $\rho: T\rightarrow T\otimes H$ is well defined, we have to check that the image of $\mu$ is contained in $T\otimes H$, which is, by faithful flatness of $T$, the equalizer of
		\begin{equation*}
			\begin{tikzcd}[column sep=6em]
				T\otimes T \otimes T \arrow[r, shift left=-1ex, "\operatorname{id}_T \otimes D"'] \arrow[r,shift left=1ex, "\operatorname{id}_T \otimes u_T  \otimes \operatorname{id}_T \otimes \operatorname{id}_T"] & T \otimes T \otimes T \otimes T
			\end{tikzcd}
		\end{equation*} 
		But in the proof of Lemma~\ref{l7} was shown that $(\operatorname{id}_T\otimes D) = (\operatorname{id}_T\otimes 1_T \otimes \operatorname{id}_T\otimes\operatorname{id}_T )$. On the other hand, since $\mu$ is an algebra morphism, so is $\rho$, which implies that $T$ is a right $H$-comodule. Moreover, if $x\in T^{\operatorname{co}H}$ then
		\begin{equation*}
			x\otimes 1 = x^{(1)}x^{(2)}\otimes x^{(3)}=1\otimes x \in T \otimes T,
		\end{equation*}
		whence $x\in K$ by faithful flatness of $T$; the other inclusion is straightforward.
		
		(iii) The Galois map $\beta:T \otimes T \rightarrow T \otimes H$ is given by
		\begin{equation*}
			\beta(x\otimes y)=(x\otimes 1 \otimes 1)\rho(y)=xy^{(1)} \otimes y^{(2)} \otimes y^{(3)} = D(x\otimes y).
		\end{equation*}
		By Lemma~\ref{l6}, $\beta$ is an isomorphism. It follows that $H$ is faithfully flat over $K$.
	\end{proof}
	
	In the original definition of quantum torsor an algebra map $\theta: T\rightarrow T$ (called the \emph{Grunspan map}) satisfying
	\begin{gather}
		(\operatorname{id}_T\otimes \operatorname{id}_{T^{\operatorname{op}}}\otimes \theta \otimes \operatorname{id}_{T^{\operatorname{op}}} \otimes \operatorname{id}_T)(\mu \otimes \operatorname{id}_{T^{\operatorname{op}}} \otimes \operatorname{id}_T)\mu=(\operatorname{id}_T \otimes \mu^{\operatorname{op}}\otimes \operatorname{id}_T)\mu, \label{e59} \\
		(\theta\otimes\theta \otimes \theta )\mu = \mu\theta, \label{e60}
	\end{gather}
	was additionally required. However, $\theta$ is fully determined by the multiplication of $T$ and $\mu$, via
	\begin{equation*}
		(m \otimes \operatorname{id}_T \otimes m)(\operatorname{id}_T \otimes \mu^{\operatorname{op}}\otimes \operatorname{id}_T)\mu(x)=1_T \otimes \theta(x) \otimes 1_T,
	\end{equation*}
 i.e., $\theta(x) = x^{(1)} {x^{(2)}}^{(3)} {x^{(2)}}^{(2)} {x^{(2)}}^{(1)} x^{(3)}$ (see \cite[Note 2.3]{Gru2}). However, the existence of the Grunspan map is demonstrable (see Corollary~\ref{cor1}), so it is not a necessary condition in Definition~\ref{d13}. If $T$ is either commutative or equipped with a commutative law, then $\theta=\operatorname{id}_T$. If $\theta$ is bijective, the quantum torsor is said to be \emph{autonomous}.
 
 We prove some preliminaries in order to enunciate the converse of Theorem~\ref{t6}. Let $H$ be an arbitrary $K$-Hopf algebra and $T$ a right $H$-Galois object with bijective Galois map $\beta: T \otimes T\rightarrow T \otimes H$. We define $\gamma: H\rightarrow T\otimes T$ by
	\begin{equation}\label{e43}
		\gamma(h):=\beta^{-1}(1\otimes h), \qquad \mbox{for all } h\in H,
	\end{equation}
	and write $\gamma(h)=h^{[1]} \otimes h^{[2]} \in T\otimes T$. Notice that $h^{[1]} \otimes h^{[2]}$ is not necessary a simple tensor. With this notation, we obtain the following formulas.
	
	\begin{lemma}[{\cite[Remark 3.4]{Schn}}]
		Let $H$ be an arbitrary $K$-Hopf algebra and $T$ a $H$-Galois object. Then $\gamma: H\rightarrow T\otimes T$ defined by \eqref{e43} satisfies the following relations:
		\begin{align}
			x_{(0)} {x_{(1)}}^{[1]} \otimes {x_{(1)}}^{[2]}&=1\otimes x, & \mbox{for all } x\in T, \label{e44} \\
			h^{[1]}h^{[2]}&=\varepsilon(h), & \mbox{for all } h\in H, \label{e45}\\
			h^{[1]}{h^{[2]}}_{(0)} \otimes {h^{[2]}}_{(1)} &= 1\otimes h, & \mbox{for all } h\in H, \label{e107}\\
			h^{[1]} \otimes {h^{[2]}}_{(0)} \otimes {h^{[2]}}_{(1)} &= {h_{(1)}}^{[1]} \otimes {h_{(1)}}^{[2]} \otimes h_{(2)}, & \mbox{for all } h\in H,\label{e46}\\
			{h^{[1]}}_{(0)} \otimes h^{[2]} \otimes {h^{[1]}}_{(1)}&={h_{(2)}}^{[1]}\otimes {h_{(2)}}^{[2]} \otimes S(h_{(1)}), & \mbox{for all } h\in H, \label{e47}\\
			(gh)^{[1]} \otimes (gh)^{[2]} &= h^{[1]} g^{[1]} \otimes g^{[2]} h^{[2]}, & \mbox{for all } h,g\in H,\label{e48}\\
			1^{[1]} \otimes 1^{[2]} &= 1\otimes 1.&\label{e49}
		\end{align}
	\end{lemma}
	
	\begin{proof}
		Since $T$ is a $H$-Galois object, there exists an structure map $\rho: T\rightarrow T\otimes H$ endowing $T$ with a right $H$-comodule algebra structure. Hence, $\beta=(m_T \otimes \operatorname{id}_H)(\operatorname{id}_T \otimes \rho)$ is an algebra morphism. Then, for any $x\in T$ we have
		\begin{align*}
			\beta( x_{(0)} {x_{(1)}}^{[1]} \otimes {x_{(1)}}^{[2]}) &= \beta(x_{(0)}\otimes 1) \beta(\gamma(x_{(1)})) = (x_{(0)} \otimes 1)\beta(\beta^{-1}(1\otimes x_{(1)} ))\\
			&= (x_{(0)} \otimes 1) (1\otimes x_{(1)})= x_{(0)} \otimes x_{(1)}=1x_{(0)} \otimes x_{(1)} \\
			&= (1\otimes 1) \rho(x) = \beta(1\otimes x),
		\end{align*}
		and hence, applying $\beta^{-1}$, we get \eqref{e44}.
		
		Now, since for every $x\otimes y \in T\otimes T$,
		\begin{align*}
			[(\operatorname{id}_T \otimes \varepsilon)\beta](x\otimes y)=(\operatorname{id}_T \otimes \varepsilon)(xy_0 \otimes y_1) = xy_0 \otimes \varepsilon(y_1) = xy_0 \varepsilon(y_1) \otimes 1_H = xy \otimes 1_H,
		\end{align*}
		i.e., $(\operatorname{id}_T \otimes \varepsilon)\beta=(m\otimes 1_H)$. Applying $\beta^{-1}$ we get \eqref{e45}.
		
		Relation \eqref{e107} immediately follows from
		\begin{align*}
			h^{[1]}{h^{[2]}}_{(0)} \otimes {h^{[2]}}_{(1)} = (h^{[1]} \otimes 1)\rho(h^{[2]}) = \beta(h^{[1]} \otimes h^{[2]})=\beta(\beta^{-1})(1\otimes h)=1\otimes h.
		\end{align*}
		 Now, if $\rho_{T\otimes T}$ (resp. $\rho_{T\otimes H}$) denotes the structure map of $T\otimes T$ (resp. $T\otimes H$) as a right $H$-comodule via \eqref{e55} (resp. \eqref{e57}), the $H$-collinearity of $\beta$ (see Proposition~\ref{p9}) implies that $\rho_{T\otimes H}\beta = (\beta\otimes \operatorname{id}_H)\rho_{T\otimes T}$. Hence, for all $h\in H$ we have
		\begin{align*}
			(\beta \otimes \operatorname{id}_T)(h^{[1]} \otimes {h^{[2]}}_{(0)} \otimes {h^{[2]}}_{(1)}) &=[(\beta \otimes \operatorname{id}_H)\rho_{T\otimes T}](h^{[1]} \otimes h^{[2]})\\
			&= \rho_{T\otimes H} \beta(\gamma(h)) = \rho_{T\otimes H} (1\otimes h) = 1\otimes h_{(1)} \otimes h_{(2)} \\&=\beta(\gamma(h_{(1)})) \otimes h_{(2)} \\
			&= (\beta \otimes \operatorname{id}_T)({h_{(1)}}^{[1]} \otimes {h_{(1)}}^{[2]} \otimes h_{(2)}),
		\end{align*}
		which proves \eqref{e46}.
		
		Similarly, \eqref{e47} follows from the collinearity of $\beta$ in the sense of \eqref{e51} and \eqref{e53}. Finally, to prove \eqref{e48}, apply $\beta$ and use \eqref{e46}.
	\end{proof}
	
	Notice that \eqref{e48}-\eqref{e49} basically say that $\gamma: H \rightarrow T^{\operatorname{op}}\otimes T$ is an algebra morphism.
	
	\begin{lemma}[{\cite[Lemma 3.1]{Sch5}}]
		Let $T$ be a faithfully flat $H$-Galois object. Then
		\begin{equation}\label{e50}
			S(x_{(1)})^{[1]}\otimes x_{(0)} S(x_{(1)})^{[2]} \in T \otimes K \subset T\otimes T, \qquad \mbox{for all }x\in T,
		\end{equation}
		and
		\begin{equation}\label{e52}
			h_{(1)}^{[1]} \otimes S(h_{(2)})^{[1]} \otimes h_{(1)}^{[2]} S(h_{(2)})^{[2]} \in T\otimes T \otimes K \subset T\otimes T \otimes T, \qquad \mbox{for all } h\in H.
		\end{equation}
	\end{lemma}
	
	\begin{proof}
		If $\rho: T \rightarrow T \otimes H$ is the structure map of $T$ as $H$-comodule algebra, for $x\in T$ we have
		\begin{align*}
			S(x_{(1)})^{[1]} \otimes \rho(x_{(0)}S(x_{(1)})^{[2]}) &= S(x_{(2)})^{[1]} \otimes x_{(0)} {S(x_{(2)})^{[2]}}_{(0)} \otimes x_{(1)} {S(x_{(2)})^{[2]}}_{(1)}\\
			&\overset{\mbox{\eqref{e46}}}{=} {S(x_{(2)})_{(1)}}^{[1]} \otimes x_{(0)} {S(x_{(2)})_{(1)}}^{[2]} \otimes x_{(1)} S(x_{(2)})_{(2)}\\
			&= S(x_{(3)})^{[1]} \otimes x_{(0)} S(x_{(3)})^{[2]} \otimes x_{(1)} S(x_{(2)})\\
			&= S(x_{(1)})^{[1]} \otimes x_{(0)} S(x_{(1)})^{[2]} \otimes 1,
		\end{align*}
		in $T\otimes T \otimes H$. Since $T^{\operatorname{co}H}=K$ and $T$ is flat over $K$, this proves the first claim. Similarly, for $h\in H$ we have
		\begin{align*}
			&{h_{(1)}}^{[1]} \otimes S(h_{2})^{[1]} \otimes \rho({h_{(1)}}^{[2]} {S(h_{(2)})}^{[2]}) \\
			&\overset{\mbox{\eqref{e46}}}{=} h_{(1)}^{[1]} \otimes {S(h_{(3)})_{(1)}}^{[1]} \otimes {h_{(1)}}^{[2]} {S(h_{(3)})_{(1)}}^{[2]} \otimes h_{(2)} S(h_{(3)})_{(2)}\\
			&= {h_{(1)}}^{[1]} \otimes S(h_{(4)})^{[1]} \otimes {h_{(1)}}^{[2]} S(h_{(4)})^{[2]} \otimes h_{(2)} S(h_{(3)})\\
			&= {h_{(1)}}^{[1]} \otimes S(h_{(2)})^{[1]} \otimes {h_{(1)}}^{[2]} S(h_{(2)})^{[2]} \otimes 1,
		\end{align*}
		proving the second claim, again by flatness of $T$.
	\end{proof}
	
	Roughly speaking, the previous result says that the elements $x_{(0)}S(x_{(1)})^{[2]}$ and ${h_{(1)}}^{[2]} S(h_{(2)})^{[2]}$ behave like scalars and hence, in the calculation below, we will be able to move these around freely in any $K$-multilinear expression.
	
	\begin{theorem}[{\cite[Theorem 3.2]{Sch5}}]\label{t10}
		Let $T$ be a faithfully flat $H$-Galois object. Then $T$ is a quantum $K$-torsor with associated map $\mu:T \rightarrow T \otimes T^{\operatorname{op}} \otimes T$ defined by
		\begin{equation*}
			\mu(x)=(\operatorname{id}_T\otimes \gamma) \rho(x) = x_{(0)} \otimes {x_{(1)}}^{[1]} \otimes {x_{(1)}}^{[2]}, \qquad \mbox{for all } x\in T.
		\end{equation*}
		Moreover, $T$ has a Grunspan map $\theta: T\rightarrow T$ given by
		\begin{equation*}
			\theta(x) = ( x_{(0)} S(x_{(1)})^{[2]} ) S(x_{(1)})^{[1]} =  S(x_{(1)})^{[1]} ( x_{(0)} S(x_{(1)})^{[2]} ).
		\end{equation*}
	\end{theorem}
	
	\begin{proof}
		For all calculations, we take $x,y\in T$ and $h\in H$. Since $\rho$ and $\gamma$ are algebra morphisms, so is $\mu$. We have
		\begin{align*}
			(\operatorname{id}_T\otimes \operatorname{id}_{T^{\operatorname{op}}} \otimes  \mu )\mu(x)&=(\operatorname{id}_T\otimes \operatorname{id}_{T^{\operatorname{op}}} \otimes  \mu ) ( x_{(0)} \otimes {x_{(1)}}^{[1]} \otimes {x_{(1)}}^{[2]}) \\
			&= x_{(0)} \otimes {x_{(1)}}^{[1]} \otimes \mu({x_{(1)}}^{[2]})  \\
			&= x_{(0)} \otimes {x_{(1)}}^{[1]} \otimes {x_{(1)}}^{[2]} \otimes \gamma( {{x_{(1)}}^{[2]}}_{(1)} ) \\
			&\overset{\mbox{\eqref{e46}}}{=} x_{(0)} \otimes {x_{(1)}}^{[1]}  \otimes {x_{(1)}}^{[2]} \otimes \gamma(x_{(2)})\\
			& = \mu(x_{(0)}) \otimes \gamma(x_{(1)}) = (\mu \otimes \operatorname{id}_{T^{\operatorname{op}}} \otimes \it_T )\mu(x),
		\end{align*}
		which proves \eqref{neq2}. Additionally,
		\begin{align*}
			(\operatorname{id}_T\otimes m) \mu(x)&= (\operatorname{id}_T \otimes m)(x_{(0)} \otimes {x_{(1)}}^{[1]} \otimes {x_{(1)}}^{[2]})\\
			&= x_{(0)} \otimes {x_{(1)}}^{[1]} {x_{(1)}}^{[2]} \overset{\mbox{\eqref{e45}}}{=} x_{(0)} \otimes \varepsilon(x_{(1)}) = x\otimes 1.
		\end{align*}
		On the other hand,
		\begin{equation*}
			(m\otimes \operatorname{id}_T)\mu(x)= (m\otimes \operatorname{id}_T)(x_{(0)} \otimes {x_{(1)}}^{[1]} \otimes {x_{(1)}}^{[2]}) = x_{(0)} {x_{(1)}}^{[1]} \otimes {x_{(1)}}^{[2]} \overset{\mbox{\eqref{e44}}}{=} 1 \otimes x.
		\end{equation*}
		These two relations prove \eqref{neq4} and \eqref{neq3}. Hence, $T$ is a quantum $K$-torsor. Now, since
		\begin{align*}
			\theta(xy)&=( (xy)_{(0)} S((xy)_{(1)})^{[2]}) S((xy)_{(1)})^{[1]} = x_{(0)}y_{(0)} S(x_{(1)}y_{(1)})^{[2]} S(x_{(1)}y_{(1)})^{[1]} \\
			&= x_{(0)}y_{(0)} (S(y_{(1)})S(x_{(1)}))^{[2]} (S(y_{(1)})S(x_{(1)}))^{[1]} \\
			&\overset{\mbox{\eqref{e48}}}{=}  x_{(0)} [y_{(0)} S(y_{(1)})^{[2]} ]S(x_{(1)})^{[2]} S(x_{(1)})^{[1]} S(y_{(1)})^{[1]}\\
			&\overset{\mbox{\eqref{e50}}}{=} x_{(0)} S(x_{(1)})^{[2]} S(x_{(1)})^{[1]} [y_{(0)} S(y_{(1)})^{[2]} ] S(y_{(1)})^{[1]} = \theta(x)\theta(y),
		\end{align*}
		$\theta$ is an algebra morphism. Moreover,
		\begin{align}
			h^{[1]} \otimes \theta(h^{[2]}) &= h^{[1]} \otimes {h^{[2]}}_{(0)} S({h^{[2]}}_{(1)})^{[2]} S({h^{[2]}}_{(1)})^{[1]}\\
			&\overset{\mbox{\eqref{e46}}}{=} {h_{(1)}}^{[1]} \otimes [{h_{(1)}}^{[2]} S(h_{(2)})^{[2]}] S(h_{(2)})^{[1]} \nonumber \\
			&\overset{\mbox{\eqref{e52}}}{=} {h_{(1)}}^{[1]} [{h_{(1)}}^{[2]} S(h_{(2)})^{[2]}] \otimes S(h_{(2)})^{[1]} \\
			&\overset{\mbox{\eqref{e45}}}{=} \varepsilon(h_{(1)}) S(h_{(2)})^{[2]} \otimes S(h_{(2)})^{[1]} \nonumber \\
			&= S(h)^{[2]} \otimes S(h)^{[1]}, \label{e86}
		\end{align}
		so we conclude that
		\begin{align*}
			(\operatorname{id}_T \otimes \operatorname{id}_{T^{\operatorname{op}}} \otimes \theta) \mu(x)&= (\operatorname{id}_T \otimes \operatorname{id}_{T^{\operatorname{op}}} \otimes \theta) ( x_{(0)} \otimes {x_{(1)}}^{[1]} \otimes {x_{(1)}}^{[2]} )\\
			& = x_{(0)} \otimes {x_{(1)}}^{[1]} \otimes \theta({x_{(1)}}^{[2]})\\
			&= x_{(0)} \otimes S(x_{(1)})^{[2]} \otimes S(x_{(1)})^{[1]}. 
		\end{align*}
		Hence,
		\begin{align*}
			&(\operatorname{id}_T\otimes \operatorname{id}_{T^{\operatorname{op}}}\otimes \theta \otimes \operatorname{id}_{T^{\operatorname{op}}} \otimes \operatorname{id}_T)(\mu \otimes \operatorname{id}_{T^{\operatorname{op}}} \otimes \operatorname{id}_T)\mu(x)\\
			&= (\operatorname{id}_T\otimes \operatorname{id}_{T^{\operatorname{op}}}\otimes \theta \otimes \operatorname{id}_{T^{\operatorname{op}}} \otimes \operatorname{id}_T)(\mu \otimes \operatorname{id}_{T^{\operatorname{op}}} \otimes \operatorname{id}_T)(\operatorname{id}_T\otimes \gamma)(x_{(0)} \otimes x_{(1)})\\
			&= (\operatorname{id}_T \otimes \operatorname{id}_{T^{\operatorname{op}}} \otimes \theta)\mu(x_{(0)}) \otimes \gamma(x_{(1)}) = x_{(0)} \otimes S(x_{(1)})^{[2]} \otimes S(x_{(1)})^{[1]} \otimes \gamma(x_{(2)}).
		\end{align*}
	On the other hand,
		\begin{align*}
			(\operatorname{id}_T \otimes \mu^{\operatorname{op}}\otimes \operatorname{id}_T)\mu(x) &= x_{(0)} \otimes  \mu^{\operatorname{op}}({x_{(1)}}^{[1]}) \otimes {x_{(1)}}^{[2]}\\
			&= x_{(0)} \otimes {{{x_{(1)}}^{[1]}}_{(1)}}^{[2]} \otimes {{{x_{(1)}}^{[1]}}_{(1)}}^{[1]} \otimes {{x_{(1)}}^{[1]}}_{(0)} \otimes {x_{(1)}}^{[2]}\\
			& \overset{\mbox{\eqref{e47}}}{=} x_{(0)} \otimes S(x_{(1)})^{[2]} \otimes S(x_{(1)})^{[1]} \otimes {x_{(2)}}^{[1]} \otimes {x_{(2)}}^{[2]}.
		\end{align*}
		Comparing these two equalities, we get \eqref{e59}. To prove \eqref{e60}, we first see that
		\begin{align}
			\rho\theta(x) & = \rho( [x_{(0)} S(x_{1})^{[2]}] S(x_{(1)})^{[1]} ) \overset{\mbox{\eqref{e50}}}{=} x_{(0)} S(x_{(1)})^{[2]} \rho( S(x_{(1)})^{[1]} ) \nonumber \\
			& = x_{(0)} S(x_{(1)})^{[2]}{S(x_{(1)})^{[1]}}_{(0)} \otimes {S(x_{(1)})^{[1]}}_{(1)}\nonumber\\
			& \overset{\mbox{\eqref{e47}}}{=} x_{(0)} {S(x_{(1)})_{(2)}}^{[2]} {S(x_{(1)})_{(2)}}^{[1]} \otimes S(S(x_{(1)})_{(1)})\nonumber\\
			& = x_{(0)} S(x_{(1)})^{[2]} S(x_{(1)})^{[1]} \otimes S^{2}(x_{(2)}) = \theta( x_{(0)} ) \otimes S^2(x_{(2)}) \label{e87}
		\end{align}
		and therefore
		\begin{align*}
			(\theta \otimes \theta \otimes \theta)\mu(x) &= \theta(x_{(0)}) \otimes \theta({x_{(1)}}^{[1]}) \otimes \theta( {x_{(1)}}^{[2]} )\\
			&\overset{\mbox{\eqref{e86}}}{=} \theta(x_{(0)}) \otimes \theta(S(x_{(1)})^{[2]}) \otimes S(x_{(1)})^{[1]}\\
			&\overset{\mbox{\eqref{e86}}}{=} \theta(x_{(0)}) \otimes S^2(x_{(1)})^{[1]} \otimes S^2(x_{(1)})^{[2]} \\
			&= \theta(x_{(0)}) \otimes \gamma(S^2(x_{(1)})) \\
			&\overset{\mbox{\eqref{e87}}}{=} \theta(x)_{(0)} \otimes \gamma(\theta(x)_{(1)}) = \mu \theta(x).
		\end{align*}
		This proves that $\theta$ is a Grunspan map.
	\end{proof}
	
	Under conditions of faithful flatness, we have thus shown that
	\begin{equation*}
		\mbox{Quantum torsors} \Leftrightarrow \mbox{Hopf Galois objects}.
	\end{equation*}
	
	Moreover, if $T$ is a quantum torsor, by Theorem~\ref{t6}, it is a faithfully flat right $H$-Galois object. But Theorem~\ref{t10} says that every faithfully flat right $H$-Galois object is a quantum torsor with Grunspan map. Hence, we have the following.
	
	\begin{corollary}\label{cor1}
		Every quantum torsor has a Grunspan map.
	\end{corollary}
	
	\subsubsection{Examples of quantum torsors}\label{sec2.10.1}
	
	In this section we present examples of quantum torsors. These were adapted from \cite{Gru2}. Some of them evidence that this new perspective of Hopf Galois extensions may include examples not studied as such before.
	
	\begin{example}[Hopf algebras]
		Let $H$ be a $K$-Hopf algebra. $H$ becomes a $K$-torsor by taking $\mu=(\operatorname{id}_H \otimes S \otimes \operatorname{id}_H)\Delta_2$. Indeed, for every $h\in H$,
		\begin{align*}
			[ (\mu \otimes \operatorname{id}_{H^{\operatorname{op}}} \otimes \operatorname{id}_{H} )\mu ](h) &= (\mu \otimes \operatorname{id}_{H^{\operatorname{op}}} \otimes \operatorname{id}_{H} )(h_{(1)} \otimes S(h_{(2)}) \otimes h_{(3)} )\\
			& = {h_{(1)}}_{(1)}  \otimes S({h_{(1)}}_{(2)}) \otimes {h_{(1)}}_{(3)} \otimes S(h_{(2)}) \otimes h_{(3)} \\
			&= h_{(1)} \otimes S(h_{(2)}) \otimes h_{(3)} \otimes S(h_{(4)}) \otimes h_{(5)} \\
			&= h_{(1)} \otimes S(h_{(2)}) \otimes {h_{(3)}}_{(1)} \otimes S({h_{(3)}}_{(2)}) \otimes {h_{(3)}}_{(3)} \\
			&=	(\operatorname{id}_{H}\otimes \operatorname{id}_{H^{\operatorname{op}}} \otimes \mu )(h_{(1)} \otimes S(h_{(2)}) \otimes h_{(3)}) \\&= [(\operatorname{id}_{H}\otimes \operatorname{id}_{H^{\operatorname{op}}} \otimes \mu )\mu](h),
		\end{align*}
		which proves \eqref{neq2}. Similarly,
		\begin{align*}
			[(\operatorname{id}_H\otimes m)\mu](h)&=(\operatorname{id}_H\otimes m)(h_{(1)} \otimes S(h_{(2)}) \otimes h_{(3)} ) = h_{(1)}  \otimes S(h_{(2)})h_{(3)}\\
			&= h_{(1)} \otimes \varepsilon(h_{(2)}) = h_{(1)} \varepsilon(h_{(2)}) \otimes 1 = h\otimes 1,\\
			[(m\otimes \operatorname{id}_H)\mu](h)&= (m\otimes \operatorname{id}_H)(h_{(1)} \otimes S(h_{(2)}) \otimes h_{(3)} ) =h_{(1)} S(h_{(2)}) \otimes h_{(3)} \\
			&= \varepsilon(h_{(1)}) \otimes h_{(2)} = 1\otimes \varepsilon(h_{(1)}) h_{(2)} = 1\otimes h,
		\end{align*}
		showing \eqref{neq4} and \eqref{neq3}. Moreover, since $S$ is an anti-morphism of coalgebras,
		\begin{align*}
			\mu(S(h))&=(\operatorname{id}_H \otimes S \otimes \operatorname{id}_H)\Delta_2(S(h))\\&=(\operatorname{id}_H \otimes S \otimes \operatorname{id}_H)(S(h_{(3)})\otimes S(h_{(2)})\otimes S(h_{(1)}))\\
			&=S(h_{(3)})\otimes S^2(h_{(2)}) \otimes S(h_{(1)}).
		\end{align*}
 Hence
		\begin{align*}
			\theta(h) &= h^{(1)} {h^{(2)}}^{(3)} {h^{(2)}}^{(2)} {h^{(2)}}^{(1)} h^{(3)} = h_{(1)} S(h_{(2)})S^2(h_{(3)})S(h_{(4)})h_{(5)}\\
			&= \varepsilon(h_{(1)})S(S(h_{(2)}))\varepsilon(h_{(3)}) = S^2(\varepsilon(h_{(1)})h_{(2)}) \varepsilon(h_{(3)}) \\
			&= S^2( h_1 )\varepsilon(h_{(2)})= S^2(h).
		\end{align*}
		Thus, $\theta=S^2$.
	\end{example}
	
	\begin{example}[Simple algebraic field extensions]
		Recall that a simple field extension $E$ of a field $\Bbbk$ is one obtained by the adjunction of a single  element, i.e., $E=\Bbbk(\alpha)$. In this case $\alpha$ is called primitive. It is known that if $\alpha$ is algebraic over $\Bbbk$, then
		\begin{equation*}
			E=\Bbbk(\alpha) = \left\{ \dfrac{f(\alpha)}{g(\alpha)} : f,g\in \Bbbk[x], \, g(\alpha) \neq 0 \right\} \cong \Bbbk[x]/\langle q_\alpha(x) \rangle ,
		\end{equation*}
		where $\Bbbk[x]$ is the classical univariate polynomial algebra over $\Bbbk$ and $q_\alpha(x)$ is the minimal polynomial of $\alpha$ over $\Bbbk$, i.e., the unique monic $\Bbbk$-polynomial of smallest degree satisfied by $\alpha$ (see e.g. \cite[Theorem 2.4.1]{Rom}). Moreover, if $d=\operatorname{dg}(q_\alpha(x))$, then the set $\{1,\alpha,\ldots,\alpha^{d-1}\}$ is a $\Bbbk$-basis for $E$.
		
		Suppose that $\Bbbk \subset E$ is a Galois extension of fields. Then, if $G=\{g_1,\ldots,g_n\}$ is the associated Galois group and $\{ p_1,\ldots,p_n \} \subset \Bbbk^G $ is the dual basis, by Theorem~\ref{t9}, we known that the Galois map $\beta : E \otimes E \rightarrow E \otimes \Bbbk^G$, given by
		\begin{equation*}
			F\otimes G\mapsto \sum_{i=1}^n Fg_i(G)\otimes p_i, \qquad \mbox{for all } F,G\in E,
		\end{equation*}
		is bijective. Notice that the simple extension
		\begin{equation*}
			E(\gamma)=\Bbbk(\alpha)(\gamma)=\Bbbk(\alpha,\gamma)=\left\{ \dfrac{f(\alpha,\gamma)}{g(\alpha,\gamma)} : f,g\in \Bbbk[x,y], \, g(\alpha,\beta) \neq 0 \right\}
		\end{equation*}
		can be identified with the algebra $E \otimes E$ in such a way that $\alpha \mapsto \alpha \otimes 1$ and $\gamma \mapsto 1\otimes \alpha$. Then, the inverse of $1_E \otimes p_k$ by $\beta$ is
		\begin{equation*}
			P_k:= \prod_{j\neq k} \dfrac{\beta-g_j(\alpha)}{g_k(\alpha)-g_j(\alpha)}, \qquad \mbox{for all } k=1,\ldots,n.
		\end{equation*}
		Indeed, having in mind the identification, for all $1\leq k \leq n$ we have
		\begin{align*}
			\beta(P_k) &= \sum_{i=1}^n \prod_{j\neq k} \dfrac{1g_i(\alpha)-g_j(\alpha)g_i(1)}{g_k(\alpha)g_i(1)-g_j(\alpha)g_i(1)} \otimes p_i = \sum_{i=1}^n \prod_{j\neq k} \dfrac{g_i(\alpha)-g_j(\alpha)}{g_k(\alpha)-g_j(\alpha)} \otimes p_i\\
			&= \prod_{j\neq k} \dfrac{g_k(\alpha)-g_j(\alpha)}{g_k(\alpha)-g_j(\alpha)} \otimes p_k = 1_E \otimes p_k.
		\end{align*}
		Hence, following the proof of Theorem~\ref{t10}, we know that the map $\mu: E \rightarrow E \otimes E \otimes E$, given by
		\begin{equation*}
			F \mapsto \sum_{i=1}^n g_i(F) \otimes P_i, \qquad \mbox{for all } F\in E,
		\end{equation*}
		makes $E$ into a $\Bbbk$-torsor.
	\end{example}
	
	\begin{example}[Noncommutative quantum torsor with no character]
		For a fixed non-negative integer $n$, suppose that the field $\Bbbk$ contains a $n$-th primitive root of unity  $q\neq1$. For any $\alpha,\beta\in \Bbbk^\times$, the $\Bbbk$-algebra generated by the elements $x$ and $y$ together with the relations $x^n = \alpha$, $y^n=\beta$, and $xy=qyx$, is called the \emph{noncommutative algebra without character $A_{\alpha,\beta}^{(n)}$}. It is known that this algebra is a non-trivial cyclic algebra and $\dim_\Bbbk A_{\alpha,\beta}^{(n)} =n^2$. If $n=2$, it is an algebra of quaternions (cf. Grunspan, 2003, Example 2.8).
		
		Taking $T=A_{\alpha,\beta}^{(n)}$, we define $\mu: T\rightarrow T \otimes T^{\operatorname{op}} \otimes T$ as
		\begin{equation*}
			\mu(x)=x\otimes x^{-1} \otimes x \qquad \mbox{and} \qquad \mu(y)=y\otimes y^{-1} \otimes y.
		\end{equation*}
		Then
		\begin{align*}
			(\operatorname{id}_T \otimes \operatorname{id}_{T^{\operatorname{op}}} \otimes \mu)\mu(x)&=(\operatorname{id}_T \otimes \operatorname{id}_{T^{\operatorname{op}}} \otimes \mu)(x\otimes x^{-1}\otimes x)=x\otimes x^{-1} \otimes x \otimes x^{-1} \otimes x \\&= (\mu \otimes \operatorname{id}_{T^{\operatorname{op}}} \otimes \operatorname{id}_T)(x\otimes x^{-1}\otimes x)= (\mu \otimes \operatorname{id}_{T^{\operatorname{op}}} \otimes \operatorname{id}_T)\mu(x),\\
			(\operatorname{id}_T \otimes m)\mu(x)&=(\operatorname{id}_T \otimes m)(x\otimes x^{-1}\otimes x)=x\otimes 1,\\
			(m\otimes \operatorname{id}_T)\mu(x)&=(m\otimes \operatorname{id}_T)(x\otimes x^{-1} \otimes x) = 1\otimes x.
		\end{align*}
		The same calculations are valid for $y$, so $T=A_{\alpha,\beta}^{(n)}$ together with $\mu$ is a quantum $\Bbbk$-torsor.
	\end{example}
	
	\subsubsection{More on quantum torsors}\label{sec2.10.2}
	
	Now we discuss some further results on quantum torsors. For instance, Grunspan found that if the base ring is a field then we can attach two Hopf algebras to any quantum torsor.
	
	\begin{theorem}[Reconstruction Theorem, {\cite[Theorem 2.10]{Gru2}}]\label{t11}
		Let $T$ be a quantum $\Bbbk$-torsor with associated map $\mu:T \rightarrow T \otimes T^{\operatorname{op}} \otimes T$ and Grunspan map $\theta: T\rightarrow T$. Take
		\begin{equation}\label{e108}
			H_l(T):=\{ z\in T\otimes T^{\operatorname{op}} : (\operatorname{id}_T \otimes \operatorname{id}_{T^{\operatorname{op}}} \otimes \theta \otimes \operatorname{id}_{T^{\operatorname{op}}})(\mu \otimes \operatorname{id}_{T^{\operatorname{op}}})(z)=(\operatorname{id}_T \otimes \mu^{\operatorname{op}})(z) \}.
		\end{equation}
		Then, the following assertions hold:
		\begin{enumerate}[label=\normalfont(\roman*)]
			\item If $z\in H_l(T)$, then both $m_T(z)$ and $m_{T^{\operatorname{op}}}(\theta\otimes \operatorname{id}_{T^{\operatorname{op}}})(z)$ are equal to a common scalar denoted by $\varepsilon(z)1_T$.
			\item If $z\in H_l(T)$, then $\Delta(z):=(\mu \otimes \operatorname{id}_{T^{\operatorname{op}}})(z)\in H_l(T) \otimes H_l(T)$.
			\item By defining $m_{H_{l}(T)}$ as the restriction of $m_T \otimes m_T^{\operatorname{op}}$ to $H_l(T)$ and $u_{H_l(T)}:\Bbbk \rightarrow H_l(T)$ as $u_{H_l(T)}(1)=1_T\otimes 1_T$, $H_l(T)$ becomes a bialgebra.
			\item $\operatorname{Im}(u_T) \subset H_l(T)\otimes T$ and $\gamma_T:= \mu: T \rightarrow H_l(T)\otimes T $ embeds $T$ with a left $H_l(T)$-comodule structure.
		\end{enumerate}
		Moreover, if we set $S_{H_l(T)}(z):=\tau_{T}(\theta \otimes \operatorname{id}_{T^{\operatorname{op}}})(z)$, for all $z\in H_l(T)$, then we obtain $\operatorname{Im}(S_{H_l(T)}) \subset H_l(T)$ and hence $H_l(T)$ is a Hopf algebra.
	\end{theorem}
	
	In fact, the result is also valid for $\Bbbk[[x]]$-torsors, providing $T$ is topologically free over $\Bbbk[[x]]$ \cite{Gru2}.
	
	Similarly, for every $\Bbbk$-torsor $T$, we can define a Hopf algebra structure on the set
	\begin{equation}\label{e109}
		H_r(T)=\{ z\in T^{\operatorname{op}} \otimes T : (\operatorname{id}_{T^{\operatorname{op}}} \otimes \theta \otimes \operatorname{id}_{T^{\operatorname{op}}} \otimes \operatorname{id}_T)(\operatorname{id}_{T^{\operatorname{op}}} \otimes \mu)(z)=(\mu^{\operatorname{op}}\otimes \operatorname{id}_T)(z) \}.
	\end{equation}
	We also have $\operatorname{Im}(\mu)\subset T\otimes H_r(T)$ and hence the map $\delta_T:= \mu=T \rightarrow T \otimes H_r(T)$ equips $T$ with a right $H_r(T)$-comodule algebra structure. Moreover, relation~\ref{neq2} implies that the two structures of left $H_l(T)$-comodule algebra and right $H_r(T)$-comodule algebra are compatible. Therefore, we get the following result.
	
	\begin{corollary}[{\cite[Corollary 4.13]{Gru2}}]
		Let $T$ be a quantum $\Bbbk$-torsor. Then $T$ is a $(H_l(T),H_r(T))$-bicomodule algebra of $\Bbbk$.
	\end{corollary}
	
	Grunspan proved that, in fact, $T$ is a $(H_l(T),H_r(T))$-biGalois object, while also providing properties of $H_l(T)$ and $H_r(T)$ and their explicit calculation for examples.
	
	Roughly speaking, the next result establishes that these Hopf algebras attached to quantum torsors and the ones found in Hopf biGalois extensions (see e.g. \cite{Sch}) are the same. Since the proof uses either Hopf biGalois theory, or Miyashita--Ulbrich techniques, and those topics are not covered in this document, we shall omit it (see \cite[Section 4]{Sch5}).
	
	\begin{proposition}[{\cite[Proposition 3.4]{Sch5}}]
		\begin{enumerate}[label=\normalfont(\roman*)]
			\item Let $T$ be a faithfully flat $H$-Galois object and consider the quantum torsor structure associated in Theorem~\ref{t10}. Then $H_r(T)\cong H$ and $H_l(T)=(T \otimes T)^{\operatorname{co}H}$.
			\item If $T$ is a quantum torsor, then the quantum torsor associated in Theorem~\ref{t10} to the $H_r(T)$-Galois object $T$ coincides with $T$.
		\end{enumerate}
	\end{proposition}
	
	We end this section by discussing a generalization of quantum torsors proposed by Schauenburg~\cite{Sch4}. For that, let $B$ be a $K$-algebra and $B \subset T$ an algebra extension such that $T$ is a faithfully flat $K$-module. Since $T \otimes_B T$ is obviously a $(B,B)$-bimodule, we can consider the \emph{centralizer}
	\begin{equation*}
		C_B(T\otimes_B T):= \{ t\in T\otimes_B T : b\cdot t=t\cdot b, \, \mbox{for all } b\in B \}.
	\end{equation*}
	This module is endowed with an algebra structure.
	
	\begin{lemma}
		The centralizer $C_B(T\otimes_B T)$ is a $K$-algebra with multiplication given by $(x\otimes y)(z\otimes w):=zx\otimes yw$, for all $x\otimes y$ and $z\otimes w \in C_B(T\otimes_B T)$, and unit $1_T\otimes 1_T$.
	\end{lemma}
	
	\begin{proof}
		If $x\otimes y, z\otimes w \in C_B(T\otimes_B T)$, then
		\begin{align*}
			b[(x\otimes y)(z\otimes w)]&=b(zx\otimes yw)= bzx \otimes yw = (x\otimes y)(bz\otimes w)=(x\otimes y)[b(x\otimes w)]\\&=(x\otimes y)[(z\otimes w)b]=(x\otimes y)(z\otimes wb)=zx\otimes ywb=(zx\otimes yw)b\\&=[(x\otimes y)(z\otimes w)]b.
		\end{align*}
		Thus, $C_B(T\otimes_B T)$ is indeed closed under this operation. The other properties are immediate to check.
	\end{proof}
	
	\begin{definition}[Generalized quantum torsor]\label{d20}
		Let $B$ be a $K$-algebra and $B \subset T$ an algebra extension such that $T$ is a faithfully flat $K$-module. A \emph{generalized quantum $B$-torsor} structure on $T$ is a map $\mu: T \rightarrow T \otimes C_B(T\otimes_B T)$ such that if the induced map $\mu_0: T\rightarrow T \otimes T \otimes_B T$ is denoted by $\mu_0(x):=x^{(1)} \otimes x^{(2)} \otimes x^{(3)}$ for any $x\in T$, then the following relations hold:
		\begin{align}
			x^{(1)}x^{(2)} \otimes x^{(3)} = &\ 1\otimes x \in T \otimes_B T, \nonumber\\
			x^{(1)} \otimes x^{(2)}x^{(3)} = &\ x\otimes 1 \in T\otimes T \nonumber,\\
			\mu(b) = &\ b\otimes 1\otimes 1, \, \forall b\in B,\label{e97}\\
			\mu(x^{(1)}) \otimes x^{(2)} \otimes x^{(3)} = &\ x^{(1)} \otimes x^{(2)} \otimes \mu(x^{(3)}) \in T \otimes T \otimes_B T \otimes_B T.\label{e98}
		\end{align}
	\end{definition}
	
	Note that \eqref{e97} implies that $\mu$ is a left $B$-module map and hence the relation \eqref{e98} actually makes sense. This generalization of quantum torsor also induces a descend datum.
	
	\begin{lemma}[{\cite[Lemma 2.8.3]{Sch4}}]
		Let $T$ be a generalized quantum $B$-torsor with associated map $\mu$. Then $D(x\otimes y)=xy^{(1)} \otimes y^{(2)} \otimes y^{(3)}$ defines a $T/K$-descent datum on $T \otimes_B T$. Moreover, it satisfies $(T\otimes D)\mu(x)=x^{(1)} \otimes 1 \otimes x^{(2)} \otimes x^{(3)}$, and $D(T\otimes_B T) \subset T \otimes C_B(T\otimes_B T)$.
	\end{lemma}
	\begin{proof}
		
		The calculations are completely similar to those made for Lemma~\ref{l7}. The only new part is the last inclusion, which is immediately obtained from noticing that, by definition, $y^{(2)} \otimes y^{(3)}$ is in the centralizer.
	\end{proof}
	
	Recall that for an arbitrary descent datum $D$ on $\operatorname{DD}(S/R)$ over $M$, Lemma~\ref{l6} gives an associated $R$-module $ {}^D M=\{m\in M: D(m)=1\otimes m \}$. From the above, we conclude that in our setup ${}^D (T \otimes_B T)\subset C_B(T\otimes_B T)$. Hence, the following result generalizes Theorem~\ref{t6}.
	
	\begin{theorem}[{\cite[Theorem 2.8.4]{Sch4}}]
		Let $T$ be a generalized quantum $B$-torsor with associated map $\mu$ such that $T$ is faithfully flat as right $B$-module. If
		\begin{equation*}
			H:={}^D (T \otimes_B T) = \{ x\otimes y \in T \otimes_B T : xy^{(1)} \otimes y^{(2)} \otimes y^{(3)} = 1\otimes x \otimes y \},
		\end{equation*}
		then the following assertions hold:
		\begin{enumerate}[label=\normalfont(\roman*)]
			\item $H$ is a $K$-flat Hopf algebra. The algebra structure is that of a subalgebra of the algebra $C_B(T\otimes_B T)$; comultiplication and counit are given by
			\begin{equation*}
				\Delta(x\otimes y)=x\otimes y^{(1)} \otimes y^{(2)} \otimes y^{(3)} \qquad \mbox{and} \qquad \varepsilon(x\otimes y)=xy.
			\end{equation*}
			\item $T$ is a right $H$-comodule algebra with structure map $\rho:=\mu$. Moreover, $T^{\operatorname{co}H}=B$.
			\item $B\subset T$ is a right $H$-Galois extension.
		\end{enumerate}
	\end{theorem}
	
	\begin{proof}
		Again, the calculations are not essentially different from the ones for Theorem~\ref{t6}, the only difference being that, in this case, the assumption of faithful flatness of $T_B$ is used to deduce, along the bijectivity of the Galois map $\beta: T \otimes_B T \rightarrow T \otimes H$, that $H$ is a faithful flat $K$-module.
	\end{proof}
	
	Finally, we have the following result which, together with the previous one, establishes the equivalence
	\begin{equation*}
		\mbox{Generalized quantum torsors} \Leftrightarrow \mbox{Hopf Galois extensions},
	\end{equation*}
	provided conditions of faithful flatness.
	
	\begin{theorem}[{\cite[Lemma 2.8.5]{Sch4}}]
		Let $H$ be a faithfully flat $K$-Hopf algebra and $T^{\operatorname{co}H}\subset T$ a faithfully flat right $H$-Galois extension. If $B:=T^{\operatorname{co}H}$, then $T$ is a generalized quantum $B$-torsor with associated map $\mu:T \rightarrow T \otimes C_B(T \otimes_B T)$ defined by
		\begin{equation}
			\mu(x)=x_{(0)} \otimes {x_{(1)}}^{[1]} \otimes {x_{(1)}}^{[2]} , \qquad \mbox{for all } x\in T,
		\end{equation}
		where $h^{[1]} \otimes h^{[2]}:=\beta^{-1}(1\otimes h) \in T \otimes_B T$, whit $\beta: T\otimes_B T \rightarrow T\otimes H$ the Galois map.
	\end{theorem}
	
	\subsection{Hopf Galois systems}\label{sec2.11}
	
	Almost parallel to the development of quantum torsors, Bichon presented another formulation for noncommutative torsors \cite{Bic}. That approach was made by noticing that a classical torsor naturally gives rise to a grupoid with two objects. Although the axiomatisation looks slightly complicated, it is also more natural an easier to handle.
	
	\begin{definition}[Hopf Galois system]\label{d12}
		A \emph{$K$-Hopf Galois system} consists of four $K$-algebras $(A,B,Z,T)$ satisfying the following axioms:
		\begin{enumerate}[label=(HGS\arabic*), align=parleft, leftmargin=*]
			\item\label{HGS1} $A$ and $B$ are $K$-bialgebras,
			\item\label{HGS2} $Z$ is an $(A,B)$-bicomodule algebra with respective structure maps $\alpha_Z : Z \rightarrow A\otimes Z$ and $\beta_Z : Z \rightarrow Z\otimes B$.
			\item\label{HGS3} There exist algebra morphisms $\gamma: A \rightarrow Z\otimes T$ and $\delta: B \rightarrow T\otimes Z$ such that the following relations hold:
			\begin{align}
				(\gamma \otimes \operatorname{id}_Z)\alpha_Z &= (\operatorname{id}_Z \otimes \delta)\beta_Z,\label{e100}\\
				(\operatorname{id}_A \otimes \gamma)\Delta_A &= (\alpha_Z \otimes \operatorname{id}_T)\gamma,\label{e101}\\
				(\delta \otimes \operatorname{id}_B) \Delta_B &= (\operatorname{id}_T \otimes \beta_Z) \delta.\label{e102}
			\end{align}
			\item\label{HGS4} There exists a $K$-linear map $S:T\rightarrow Z$ such that the following relations hold:
			\begin{align}
				m_Z(\operatorname{id}_Z\otimes S)\gamma &= u_Z \varepsilon_A,\label{e103}\\
				m_Z(S\otimes \operatorname{id}_Z)\delta &= u_Z \varepsilon_B.\label{e104}
			\end{align}
		\end{enumerate}
	\end{definition}
	
	We may extend Heyneman--Sweedler notation to Hopf Galois systems by writing
	\begin{equation*}
		\gamma(a)=a_Z \otimes a_T \qquad \mbox{and} \qquad \delta(b)=b_T \otimes b_Z, \qquad \mbox{for all } a\in A \mbox{ and } b\in B.
	\end{equation*}
	
	It can be proved that in any Hopf Galois system $(A,B,Z,T)$ the bialgebras $A$ and $B$ are in fact Hopf algebras, and $S:T\rightarrow Z^{\operatorname{op}}$ is an algebra morphism \cite[Corollaries 1.3 and 1.10]{Bic}.
	
	The following result relates Hopf Galois systems with Hopf biGalois objects.
	
	\begin{theorem}[{\cite[Theorem 1.2]{Bic}}] \label{th1}
		Let $(A,B,Z,T)$ be a $K$-Hopf Galois system with $Z$ faithfully flat over $K$. Then $Z$ is an $(A,B)$-biGalois object.
	\end{theorem}
	
	\begin{proof}
		First, we have to prove that the composition
		\begin{equation*}
			\beta_l : \begin{tikzcd}[column sep=huge]
				Z\otimes Z \arrow[r,"\alpha_Z\otimes \operatorname{id}_Z"] & A\otimes Z \otimes Z \arrow[r,"\operatorname{id}_A \otimes m_Z"] & A \otimes Z
			\end{tikzcd}
		\end{equation*}
		is bijective. For that, let $\eta_l: A\otimes Z \rightarrow Z\otimes Z$ be the map defined as the composition
		\begin{equation*}
			\eta_l : \begin{tikzcd}[column sep=huge]
				A\otimes Z \arrow[r,"\gamma \otimes \operatorname{id}_Z"] & Z\otimes T \otimes Z \arrow[r,"\operatorname{id}_Z \otimes S \otimes \operatorname{id}_Z"] & Z \otimes Z \otimes Z \arrow[r,"\operatorname{id}_Z \otimes m_Z"] & Z \otimes Z.
			\end{tikzcd}
		\end{equation*}
		For all $a\in A$ and $z,w\in Z$ we have the equalities
		\begin{align*}
			&(\operatorname{id}_Z \otimes S \otimes \operatorname{id}_Z)(\gamma \otimes \operatorname{id}_Z)(\operatorname{id}_A \otimes m_Z)(a\otimes z \otimes w) \\&= (\operatorname{id}_Z \otimes S \otimes \operatorname{id}_Z)(\gamma \otimes \operatorname{id}_Z)(a\otimes zw) \\
			&= (\operatorname{id}_Z \otimes S \otimes \operatorname{id}_Z)(a_Z \otimes a_T \otimes zw) = a_Z \otimes S(a_T) \otimes zw \\&= (\operatorname{id}_Z \otimes \operatorname{id}_Z \otimes m_Z)(a_Z \otimes S(a_T) \otimes z\otimes w)\\
			&= (\operatorname{id}_Z \otimes \operatorname{id}_Z \otimes m_Z)(\operatorname{id}_Z \otimes S \otimes \operatorname{id}_Z \otimes \operatorname{id}_Z)(a_Z \otimes a_T \otimes z \otimes w)\\
			&= (\operatorname{id}_Z \otimes \operatorname{id}_Z \otimes m_Z)(\operatorname{id}_Z \otimes S \otimes \operatorname{id}_Z \otimes \operatorname{id}_Z)(\gamma \otimes \operatorname{id}_Z \otimes \operatorname{id}_Z)(a\otimes z \otimes w),
		\end{align*}
		so,
		\begin{align}
			&(\operatorname{id}_Z \otimes S \otimes \operatorname{id}_Z)(\gamma \otimes \operatorname{id}_Z)(\operatorname{id}_A \otimes m_Z) \nonumber\\
			&=(\operatorname{id}_Z \otimes \operatorname{id}_Z \otimes m_Z)(\operatorname{id}_Z \otimes S \otimes \operatorname{id}_Z \otimes \operatorname{id}_Z)(\gamma \otimes \operatorname{id}_Z \otimes \operatorname{id}_Z). \label{e105}
		\end{align}
		Hence,
		\begin{align*}
			\eta_l \beta_l &= (\operatorname{id}_Z \otimes m_Z)(\operatorname{id}_Z \otimes S \otimes \operatorname{id}_Z)(\gamma \otimes \operatorname{id}_Z)(\operatorname{id}_A \otimes m_Z)(\alpha_Z \otimes\operatorname{id}_Z )\\
			&\overset{\mbox{\eqref{e105}}}{=} (\operatorname{id}_Z \otimes m_Z)(\operatorname{id}_Z \otimes m_Z \otimes \operatorname{id}_Z)(\operatorname{id}_Z \otimes S \otimes \operatorname{id}_Z \otimes \operatorname{id}_Z)(\operatorname{id}_Z \otimes \delta \otimes \operatorname{id}_Z) (\beta \otimes \operatorname{id}_Z)\\
			&\overset{\mbox{\eqref{e104}}}{=}  (\operatorname{id}_Z \otimes m_Z)(\operatorname{id}_Z \otimes u_Z \varepsilon_B \otimes \operatorname{id}_Z)(\beta_Z \otimes \operatorname{id}_Z) \\
			&\overset{\mbox{\eqref{e32}}}{=} \operatorname{id}_{Z\otimes Z}.
		\end{align*}
		Similarly, one can show that $\beta_l \eta_l = \operatorname{id}_{A\otimes Z}$ and thus $\beta_l$ is bijective. On the other hand, we also have to prove that the composition 
		\begin{equation*}
			\beta_r : \begin{tikzcd}[column sep=huge]
				Z\otimes Z \arrow[r,"\operatorname{id}_Z\otimes \beta_Z"] & Z\otimes Z \otimes B \arrow[r,"m_Z \otimes \operatorname{id}_B"] & Z \otimes B
			\end{tikzcd}
		\end{equation*}
		is bijective. For that, we define $\eta_r: Z\otimes B \rightarrow Z\otimes Z$ as the composition
		\begin{equation*}
			\eta_r : \begin{tikzcd}[column sep=huge]
				Z \otimes B \arrow[r,"\operatorname{id}_Z \otimes \delta"] & Z \otimes T\otimes Z \arrow[r,"\operatorname{id}_Z \otimes S \otimes \operatorname{id}_Z"] & Z \otimes Z \otimes Z \arrow[r,"m_Z \otimes \operatorname{id}_Z"] & Z \otimes Z,
			\end{tikzcd}
		\end{equation*}
		and similarly to the first part, one can show that $\eta_r$ is the inverse of $\beta_r$. Finally, since $Z$ is $K$-faithfully flat, by Proposition~\ref{p11}, it is $(A,B)$-faithfully flat.
	\end{proof}
	
	The converse is proven using techniques of Tannaka duality (see \cite{Sch3} and \cite[Remark 1.9]{Bic}) that are not covered in this document.
	
	\begin{theorem}[{\cite[Corollary 1.8]{Bic} and \cite[Corollary 1]{Gru}}]\label{t12}
		Let $A$ be a faithfully flat $K$-Hopf algebra and $Z$ a faithfully flat left $A$-Galois object. Then there exists a Hopf algebra $B$ and an algebra $T$ such that $(A,B,Z,T)$ is a Hopf Galois system.
	\end{theorem}
	
	Hence with certain assumptions of faithful flatness, if we glue together Theorem~\ref{t6} (quantum torsors $\Rightarrow$ Galois objects), Theorem~\ref{t10} (Galois objects $\Rightarrow$ quantum torsors), Theorem~\ref{th1} (Hopf Galois systems $\Rightarrow$ (bi)Galois objects) and Theorem~\ref{t12} (Galois objects $\Rightarrow$ Hopf Galois systems), we have the following equivalences:
	\begin{equation}\label{e110}
		\begin{tikzcd}[row sep=scriptsize,column sep=-4mm]
			\begin{matrix}
				\mbox{Hopf Galois}\\
				\mbox{systems}
			\end{matrix}  && \begin{matrix}
				\mbox{Hopf Galois}\\
				\mbox{objects}
			\end{matrix} \arrow[ll,Leftrightarrow]\\
			& \begin{matrix}
				\mbox{Quantum}\\
				\mbox{torsors}
			\end{matrix} \arrow[ur,Leftrightarrow]
		\end{tikzcd}
	\end{equation}
	Clearly, this means that every Hopf Galois system gives rise to a quantum torsor and vice versa. However, for the sake of completeness, and because some interesting interactions arise, we will study the explicit equivalence between quantum torsors and Hopf Galois systems in Section~\ref{sec2.11.2}. Also, notice that we are not saying that the correspondences of \eqref{e110} are necessarily bijective or functorial.
	
	\subsubsection{Examples of Hopf Galois systems}\label{sec2.11.1}
	
	In this section we address some examples of Hopf Galois systems, which are adapted from \cite{Bic,Sch}. As with quantum torsors, the novelty of this new approach is the inclusion of new examples.
	
	\begin{example}[Hopf algebras]
		Let $H$ be a $K$-Hopf algebra. If we put $A=B=Z=T=H$, $\alpha_Z=\beta_Z=\gamma=\delta=\Delta_H$ and $S=S_H$, then $(A,B,Z,T)$ is a Hopf Galois system. Indeed, \eqref{e100}-\eqref{e102} correspond to the coassociativity and \eqref{e103}-\eqref{e104} to the main property of the antipode.
	\end{example}
	
	\begin{example}[Hopf algebras twisted by $2$-cocycles]\label{twist}
		Let $H$ be a faithfully flat $K$-Hopf algebra. Using the universal property of the tensor product, a given $\sigma \in \operatorname{Hom}_K(H\otimes H,K)$ corresponds to an unique bilinear $K$-form, so we shall write $\sigma(h\otimes k)=\sigma(h,k)$, for all $h, k\in H$. Since $H\otimes H$ is a coalgebra and $K$ is an algebra, $\operatorname{Hom}_K(H\otimes H,K)$ is also an algebra with the convolution product. We say that $\sigma: H \otimes H \rightarrow K$ is a \emph{$2$-cocycle} if $\sigma$ is a convolution invertible $K$-linear map satisfying
		\begin{equation*}
			\sigma(g_{(1)},h_{(1)} ) \sigma(g_{(2)}h_{(2)},k)= \sigma(h_{(1)},k_{(1)})\sigma(g,h_{(2)}k_{(2)}) \quad \mbox{and} \quad \sigma(h,1)=\sigma(1,h)=\varepsilon(h)1,
		\end{equation*}
		for all $g,h,k\in H$. By \cite[Theorem 1.6]{Doi2}, if $\overline{\sigma}$ denotes the convolution inverse of $\sigma$, then it satisfies
		\begin{equation*}
			\overline{\sigma}(f_{(1)}g_{(1)},h)\overline{\sigma}(f_{(2)},g_{(2)}) = \overline{\sigma}(f,g_{(1)}h_{(1)})\overline{\sigma}(g_{(2)}h_{(2)}) \quad \mbox{and} \quad \overline{\sigma}(h,1)=\overline{\sigma}(1,h)=\varepsilon(h)1,
		\end{equation*}
	for all $f,g,h\in H$.
		
		For a fixed $2$-cocycle $\sigma$ over $H$, we consider a new multiplication over the $K$-module $H$, given by
		\begin{equation*}
			h \cdot_\sigma k:=\sigma(h_{(1)},k_{(1)}) h_{(2)}k_{(2)}, \qquad \mbox{for all } h, k \in H.
		\end{equation*}
		This new algebra is denoted by ${}_{\sigma}{H}$. Similarly, another possible product for $H$ is
		\begin{equation*}
			h \cdot_{\overline{\sigma}} k := \overline{\sigma}(h_{(2)},k_{(2)})h_{(1)}k_{(1)}, \qquad \mbox{for all } h, k\in H,
		\end{equation*}
		 and the induced algebra is denoted by $H_{\overline{\sigma}}$. Notice that $H_{\overline{\sigma}}$ is a $H$-comodule algebra with structure map $\rho_l=\Delta_H$.
		
		Finally, we can also define a new Hopf algebra ${}_{\sigma}{H}_{\overline{\sigma}}$, which is isomorphic to $H$ as coalgebra, with multiplication
		\begin{equation*}
			k\cdot h := \sigma(h_{(1)},k_{(1)}) \overline{\sigma}(h_{(3)},k_{(3)})h_{(2)}k_{(2)}, \qquad \mbox{for all } h, k\in H,
		\end{equation*}
		and antipode $S^{\sigma}(h):=\sigma(h_{(1)},S(h_{(2)}))\overline{\sigma}(S(h_{(4)}),h_{(5)})S(h_{(3)})$. It can be shown that ${H}_{\overline{\sigma}}$ is a right ${}_{\sigma}{H}_{\overline{\sigma}}$-comodule algebra with structure map $\rho_r=\Delta$, and that it is a $(H,{}_{\sigma}{H}_{\overline{\sigma}})$-bicomodule algebra. The details of these constructions can be found in \cite[Section 2]{Doi2} and \cite[Section 3]{Sch}. 		From \cite[Proposition 2.1]{Bic} and \cite[Theorem 1.6.(a5)]{Doi} we know that $(H,{}_{\sigma}{H}_{\overline{\sigma}}, {H}_{\overline{\sigma}},{}_{\sigma}{H})$ is a Hopf Galois system.
	\end{example}
	
	\begin{example}[Hopf algebras of a non-degenerate bilinear form]
		Let $\Bbbk$ be an algebraically closed field and $n,m>1$. For two fixed invertible matrices $E_{m\times m}$ and $F_{n\times n}$, we denote by $\mathcal{B}(E,F)$ the $\Bbbk$-algebra generated by $\{x_{ij} : 1\leq i \leq m, \, 1\leq j \leq n\}$ together with the relations $F^{-1t}XEX=I_n$ and $ XF^{-1t}XE=I_m$, where $X$ is the matrix $(x_{ij})$ and $I_n$ and $I_m$ are the identity matrices of size $n$ and $m$, respectively.  For the particular case $n=m$ and $E=F$ we simply write $\mathcal{B}(E)$. This (Hopf) algebra was introduced by Dubois-Violette and Launer \cite{DL}, and it turns out to be the function algebra on the quantum (symmetry) group of a non-degenerate bilinear form \cite[Section 2]{Bic2}. For any matrix $A=(a_{ij})$ over $\mathcal{B}(E)$, the comultiplication, counit and antipode are given by
		\begin{equation*}
			\Delta(a_{ij})=\sum_{k=1}^n a_{ik}\otimes a_{ik}, \qquad \varepsilon(a_{ij})=\delta_{ij} \qquad \mbox{and} \qquad S(A)=E^{-1t}AE,
		\end{equation*}
		where $\delta_{ij}$ denotes the Kronecker delta.
		
		If $\operatorname{tr}(E^tE^{-1})=\operatorname{tr}(F^tF^{-1})$, Bichon proved that $$(\mathcal{B}(E), \mathcal{B}(F), \mathcal{B}(E,F), \mathcal{B}(F,E))$$ is a Hopf Galois system \cite[Proposition 3.1]{Bic}. Even without the assumption on the traces it can be shown that $\mathcal{B}(E,F)$ is a $(\mathcal{B}(E),\mathcal{B}(F))$-biGalois object \cite[Proposition 3.3]{Bic2}.
	\end{example}
	
	\begin{example}[Free Hopf algebras generated by dual matrix coalgebras]
		Let $C$ be a $K$-coalgebra.  We say that a $K$-Hopf algebra $H(C)$ is a \emph{free Hopf algebra generated by $C$} if there exists a coalgebra map $i:C\rightarrow H(C)$ such that the following universal property is satisfied: \textit{for any $K$-Hopf algebra $H$ and any coalgebra morphism $f: C \rightarrow H$ there exists an unique Hopf algebra morphism $\overline{f}: H(C)\rightarrow H$ such that the following diagram is commutative:}
		\begin{equation*}
			\begin{tikzcd}
				C \arrow[d,"f"'] \arrow[r,"i"] & H(C) \arrow[dl,dashed,"\overline{f}"]\\
				H
			\end{tikzcd}
		\end{equation*}
		From the above, it follows that $H(C)$ is unique up to isomorphism. The existence of such free Hopf algebra is explicitly shown in \cite[Section 1]{Tak} by constructing $H(C)$ as follows. Let $\{ V_i \}_{i\geq 0}$ be the sequence of coalgebras	$V_0:=C$ and $V_{i+1}:=V_i^{\operatorname{op}}$, for $i\geq 0$. We define $V:=\bigoplus_{i\geq0} V_i$, which is also a coalgebra via the induced pointwise operations. Considering the tensor algebra $T(V)$, we have a coalgebra map $S: V \rightarrow V^{\operatorname{op}}$ given by $(x_0,x_1,\ldots) \mapsto (0,x_0,x_1,\ldots)$, which induces a bialgebra map $S: T(V)\rightarrow T(V)^{\operatorname{op}}$. Now, let
		\begin{equation*}
			I=\langle x_{(1)}S(x_{(2)})-\varepsilon(x)1,S(x_{(1)})x_{(2)}-\varepsilon(x)1 : x\in V\rangle.
		\end{equation*}
		One can check that $I$ is in fact a Hopf ideal of $T(V)$, and therefore $H(C):=T(V)/I$ is a Hopf algebra with antipode induced by $S$.
		
		A particular case of the above is when $C=(M_n(\Bbbk))^*$, where $M_n(\Bbbk)$ denotes the algebra of $n\times n$ matrices over $\Bbbk$. In such case, $H(C)$ is denoted by $H(n)$ and corresponds to the $\Bbbk$-algebra generated by $\{u_{ij}^{(\alpha)} : 1\leq i,j\leq n, \, \alpha\in \mathbb{N} \}$ satisfying the relations
		\begin{equation*}
			(u^{(\alpha)})^t u^{(\alpha+1)}=I_n=u^{(\alpha +1)}(u^{(\alpha)})^t,
		\end{equation*}
		where $u^{(\alpha)}$ is the $n\times n$ matrix $(u_{ij}^{(\alpha)})$ \cite[Theorem 3.1]{DW}. More generally, for $m,n\geq 1$ we define $H(m,n)$ as the algebra generated by $\{u_{ij}^{(\alpha)} : 1\leq i\leq m, \, 1\leq j \leq n, \, \alpha\in \mathbb{N} \}$ together with the relations
		\begin{equation*}
			(u^{(\alpha)})^t u^{(\alpha+1)}=I_m \quad \mbox{and} \quad u^{(\alpha +1)}t(u^{(\alpha)})=I_n,
		\end{equation*}
		where $u^{(\alpha)}$ is the $n\times m$ matrix $(u_{ij}^{(\alpha)})$. For $m,n\geq 2$, $(H(m),H(n),H(m,n),H(n,m))$ is a Hopf Galois system \cite[Proposition 5.2]{Bic}.
	\end{example}

	\subsubsection{More on Hopf Galois systems}\label{sec2.11.2}
	
	We explore further the connection between Hopf Galois systems and quantum torsors, by mentioning some results of Grunspan \cite{Gru2,Gru}. Recall that any quantum $K$-torsor $T$ has an associated map $\mu:T\rightarrow T\otimes T^{\operatorname{op}} \otimes T$, which is denoted by $\mu(x)=x^{(1)} \otimes x^{(2)} \otimes x^{(3)}$, and a Grunspan map $\theta: T\rightarrow T$ satisfying
	\begin{equation*}
		\theta(x)=x^{(1)}x^{(2)(3)} x^{(2)(2)} x^{(2)(1)} x^{(3)}.
	\end{equation*}
	
	\begin{theorem}[{\cite[Theorem 4.2]{Gru2}}]\label{t13}
		Let $(A,B,Z,T)$ be a $K$-Hopf Galois system. Then the map $\mu: Z\rightarrow Z\otimes Z^{\operatorname{op}} \otimes Z$ given by
		$\mu = (\operatorname{id}_Z \otimes S \otimes \operatorname{id}_Z)(\gamma \otimes \operatorname{id}_Z)\alpha_Z$ makes $Z$ into a quantum $K$-torsor.
	\end{theorem}

To see the converse of Theorem~\ref{t13}, recall the construction of the Hopf algebra $H_l(T)$ (resp. $H_r(T)$) defined in \eqref{e108} (resp. \eqref{e109}) as certain subalgebra of $T\otimes T^{\operatorname{op}}$ (resp. $T^{\operatorname{op}}\otimes T$). Summarizing Theorem~\ref{t11}, we have that the simple elements of $H_l(T)$ are of the form $x_i\otimes y_i \in T\otimes T^{\operatorname{op}}$ satisfying
\begin{equation*}
	x_i^{(1)} \otimes x_i^{(2)} \otimes \theta(x_i^{(3)}) \otimes y_i = x_i \otimes y_i^{(3)} \otimes y_i^{(2)} \otimes y_i^{(1)}.
\end{equation*}
Moreover, the comultiplication, counit and antipode on $H_l(T)$ are given by
\begin{align*}
	\Delta_{H_l(T)}(x_i\otimes y_i) = &\ x_i^{(1)}\otimes x_i^{(2)} \otimes x_i^{(3)} \otimes y_i,\\
	u_T \varepsilon_{H_l(T)}(x_i \otimes y_i) = &\ x_iy_i,\\
	S_{H_l(T)}(x_i \otimes y_i) = &\ y_i \otimes \theta(x_i).
\end{align*}
Similarly, the elements of $H_r(T)$ are of the form $x_i \otimes y_i \in T^{\operatorname{op}} \otimes T$ satisfying
\begin{equation*}
	x_i \otimes \theta(y_i^{(1)}) \otimes y_i^{(2)} \otimes y_i^{(3)} = x_i^{(3)} \otimes x_i^{(2)} \otimes x_i^{(1)} \otimes y_i,
\end{equation*}
and the comultiplication, counit and antipode are given by
\begin{align*}
	\Delta_{H_r(T)}(x_i\otimes y_i) = &\ x_i \otimes y_i^{(1)} \otimes y_i^{(2)} \otimes y_i^{(3)},\\
	u_T \varepsilon_{H_r(T)}(x_i \otimes y_i) = &\ x_iy_i,\\
	S_{H_r(T)}(x_i \otimes y_i) = &\ \theta(y_i) \otimes x_i.
\end{align*}

\begin{theorem}[{\cite[Theorem 2]{Gru}}]\label{t14}
	Let $T$ be a faithfully flat quantum $K$-torsor with associated map $\mu$ and Grunspan map $\theta$. Consider $A=H_r(T)$, $B=H_l(T)$, and $\rho_r: T\rightarrow T\otimes A$ and $\rho_l: T\rightarrow B\otimes T$ given by
	\begin{gather*}
	\rho_r(x)=x^{(1)}\otimes x^{(2)} \otimes x^{(3)} \in T \otimes A \subset T \otimes T^{\operatorname{op}} \otimes T,\\
	\rho_l(x)=x^{(1)} \otimes x^{(2)} \otimes x^{(3)} \in B\otimes T \subset T \otimes T^{\operatorname{op}} \otimes T,
	\end{gather*}
	for all $x\in T$, which define a structure of $(B,A)$-biGalois object on $T$. Then there is a natural structure of $K$-algebra on $Z:=(A\otimes T)^{\operatorname{co}A}=(T\otimes B)^{\operatorname{co}B}$ as subalgebra of $T^{\operatorname{op}} \otimes T \otimes T^{\operatorname{op}}$, so that $Z$ becomes an $(A,B)$-biGalois object (in particular, an $(A,B)$-bicomodule algebra) via the morphisms $\alpha_Z=\Delta_A \otimes \operatorname{id}_T: Z\rightarrow A\otimes Z$ and $\beta_Z=\operatorname{id}_Z\otimes \Delta_B : Z \rightarrow Z \otimes B$. Moreover, the algebra morphisms $\gamma: A \to Z\otimes T$ and $\delta: B \to T\otimes Z$ given by 
	\begin{gather*}
		\gamma(h)=x_i \otimes y_i^{(1)} \otimes y_i^{(2)} \otimes y_i^{(3)}, \qquad \mbox{for } h=x_i\otimes y_i \in A \subset T^{\operatorname{op}} \otimes T,\\
		\delta(h)=x_i^{(1)} \otimes x_i^{(2)}\otimes x_i^{(3)} \otimes y_i , \qquad \mbox{for } h=x_i\otimes y_i \in B \subset T \otimes T^{\operatorname{op}},
	\end{gather*}
	and the $K$-linear map $S_T: T\rightarrow Z$ given by
	\begin{equation*}
		S_T(x)=[(\theta \otimes \operatorname{id}_T \otimes \theta)\mu^{\operatorname{op}}](x), \qquad \mbox {for all} x\in T,
	\end{equation*}
	make $(A,B,Z,T)$ into a $K$-Hopf Galois system. Furthermore, the quantum torsor associated to this Hopf Galois system by Theorem~\ref{t13} is isomorphic to $T$.
\end{theorem}

In fact, the Hopf Galois system obtained here satisfies some additional conditions of symmetry that make it \emph{total} (see \cite[Definition 3.1]{Gru}).

Therefore, by adding Theorem~\ref{t13} (Hopf Galois systems $\Rightarrow$ quantum torsors) and Theorem~\ref{t14} (quantum torsors $\Rightarrow$ (total) Hopf Galois systems) to former results, we have explicitly closed diagram \eqref{e110}:
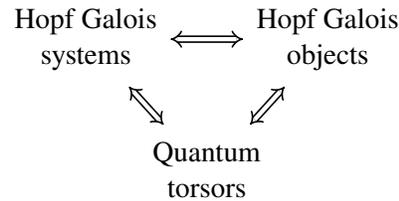
 \begin{equation*}
	\begin{tikzcd}[row sep=scriptsize,column sep=-4mm]
		\begin{matrix}
			\mbox{Hopf Galois}\\
			\mbox{systems}
		\end{matrix} \arrow[dr,Leftrightarrow]  && \begin{matrix}
			\mbox{Hopf Galois}\\
			\mbox{objects}
		\end{matrix} \arrow[ll,Leftrightarrow]\\
		& \begin{matrix}
			\mbox{Quantum}\\
			\mbox{torsors}
		\end{matrix} \arrow[ur,Leftrightarrow]
	\end{tikzcd}
\end{equation*}
	
	\section{Families of noncommutative rings}\label{ch2}
	
	In the last century, noncommutative rings and algebras have appeared in almost every subject of research, and not only mathematical contexts, but also in theoretical physics. Therefore, the study of certain algebras given by their generators and relations has become useful. However, within a more practical approach, some general families of noncommutative rings have been defined and studied along the years.
	
	Although these collections do not cover, in general, every remarkable example, most of such families contain a notorious amount of distinguished algebras, even having the case that one object can be endowed with two or more different structures, as we shall discuss in the examples. In this section, we will address several of such families of rings, most of them having a polynomial behavior.
	
	Popular for describing a broad number of algebras and for having a pioneering role in the systematic research on noncommutative rings, in Section~\ref{s21} we address skew polynomial rings. Section~\ref{s22} is dedicated to PBW extensions, which comprehend rings with the PBW basis property, while Section~\ref{s23} reviews a generalization of that setup. Finally, in Section~\ref{s24} we present a family of algebras that generalizes enveloping universal algebras of Lie algebras.
	
	\subsection{Skew polynomial rings}\label{s21}
	
	Introduced by Ore \cite{Ore}, skew polynomial rings (also known as Ore extensions) are distinguished by their elements, which have a polynomial aspect but not necessarily the variable is assumed to commute with coefficients. For such ``commutation'' to take place, certain rule involving an endomorphism and a derivation of the ground ring is established.
	
	Let $R$ be a ring and $\sigma: R\rightarrow R$ a ring endomorphism. An additive map $\delta:R\rightarrow R$ is called a \emph{$\sigma$-derivation of $R$} if $\delta(r s)=\sigma(r) \delta(s) + \delta(r)s$, for all $r,s\in R$. Notice that, in particular, $\delta(1)=\delta(1\cdot 1)=\sigma(1)\delta(1)+\delta(1)1=2\delta(1)$, whence $\delta(1)=0$.
	
	\begin{definition}[Skew polynomial ring] \label{d8}
		Let $R$ be a ring, $\sigma:R\rightarrow R$ a ring endomorphism and $\delta:R\rightarrow R$ a $\sigma$-derivation of $R$. A ring $A$ such that
		\begin{enumerate}[label=(O\arabic*), align=parleft, leftmargin=*]
			\item \label{O1} $A$ contains $R$ as a proper subring and $1_R=1_A$,
			\item \label{O2} There is a distinguishable element $x\in A$ such that $A$ is a left free $R$-module with basis $\{ 1, x, x^2, x^3, \ldots \}$,
			\item \label{O3} $xr=\sigma(r)x + \delta(r)$, for all $r\in R$,
		\end{enumerate}
		is called a \emph{skew polynomial ring over $R$}. In this case we write $A:=R[x;\sigma,\delta]$.
	\end{definition}
	
There are some constructive proofs showing the existence of skew polynomial rings, which verify the ring structure of $A=R[x;\sigma,\delta]$ without falling in the tedious calculations of a direct proof \cite[Proposition 2.3]{GW}. Moreover, such constructions guarantee that, given any ring $R$, any ring endomorphism $\sigma$ of $R$ and any $\sigma$-derivation of $R$, the skew polynomial ring $R[x;\sigma,\delta]$ always exists.

	Unless otherwise stated, for the remainder of the section we let $A:=R[x;\sigma,\delta]$ be a skew polynomial ring over a fixed ring $R$. From~\ref{O3} it is natural to ask for a general formula that allows us to express $x^ir$ ($i\in \mathbb{N}$ and $r\in R$) as a \emph{polynomial} with \emph{left coefficients}. Nevertheless, those calculations can be tricky. For instance, with $i=3$ we end up with
	\begin{equation*}
		x^3r= \sigma^3(r)x^3 + [\delta\sigma^2(r)+\sigma\delta\sigma(r)+\sigma^2\delta(r)]x^2   + [\delta^2\sigma(r)+\delta\sigma\delta(r)+\sigma\delta^2(r)]x + \delta^3(r).
	\end{equation*}
	However, using an inductive argument, it can be shown that the \emph{multiplication rule} can be written as follows. Given $r\in R$ and $i,k\in \mathbb{N}$, we denote by $W[\delta^k\sigma^{i-k}](r)$ the evaluation of $r$ in the function given by the sum of all possible \emph{words} that can be constructed with the alphabet formed by $k$-times the symbol $\delta$ and $(i-k)$-times the symbol $\sigma$, where the concatenation is understood as the composition of functions. For instance, if $i=5$ and $k=2$ we get
	\begin{align*}
		W[\delta^2\sigma^3](r) &= \delta^2 \sigma^3 (r) + \delta\sigma\delta\sigma^2 (r) + \delta \sigma^2\delta\sigma(r) + \delta \sigma^3 \delta (r) + \sigma \delta^2 \sigma^2 (r)\\ &\qquad + \sigma\delta\sigma\delta\sigma(r) + \sigma\delta\sigma^2 \delta (r) + \sigma^2 \delta^2 \sigma (r) + \sigma^2 \delta \sigma \delta (r) + \sigma^3 \delta^2 (r).
	\end{align*}
	In the general case, if $r\in R$ and $i\in \mathbb{N}$, the following formula holds:
	\begin{equation}\label{e24}
		x^i r = \sum_{k=0}^i W[\delta^k \sigma^{i-k}](r)x^{i-k}.
	\end{equation}
	Moreover, if $r, s\in R$ and $i,j \in \mathbb{N}$, then:
	\begin{equation}\label{e25}
		(rx^i)(sx^j) = r\sum_{k=0}^i W[\delta^k \sigma^{i-k}](s) x^{i+j-k}.
	\end{equation}
	This discussion shows that in $A$ the product of two \emph{terms} $rx^n$ and $sx^m$ is not necessarily another term, yet in general it will be a polynomial. However, by~\ref{O2}, it is quite clear that every element $p \in A$ can be uniquely written as
	\begin{equation*}
		p=\sum_{i=0}^n r_ix^i = r_0 + r_1 x + r_2 x^2 + \cdots + r_{n-1}x^{n-1}+r_nx^n, \quad r_i \in R \mbox{ and } 0\leq i \leq n.
	\end{equation*}
	The element $p$ is usually denoted $p(x)$ to emphasize the \emph{indeterminate} $x$. Following the classical terminology, the $r_i$ are called the \emph{coefficients} of $p(x)$. Hence, we conclude that the elements of $A$ have a polynomial expression, which justifies the name given to this construction.
	
	\begin{definition}
		If $p(x)=\sum_{i=0}^n r_ix^i$ is an element of $A$ such that $r_n \neq 0$, we define:
		\begin{enumerate}[label=\normalfont(\roman*)]
			\item  $\operatorname{dg}(p(x)):=n$ as the \emph{degree of $p(x)$},
			\item $\operatorname{lc}(p(x)):=r_n$ as the \emph{leading coefficient},
			\item $\operatorname{lm}(p(x)):=x^n$ as the \emph{leading monomial},
			\item $\operatorname{lt}(p(x)):=\operatorname{lc}(p(x))\operatorname{lm}(p(x))=r_nx^n$ as the \emph{leading term}.
		\end{enumerate}
		If all coefficient of $p(x)$ are zero, we say that $p(x):=0$ is the \emph{zero polynomial} and in this case $\operatorname{lc}(0):=0$, $\operatorname{lm}(0):=0$ and $\operatorname{lt}(0):=0$.
	\end{definition}
	
	\begin{lemma}
		If $p(x), q(x)\in A$ and $p(x),q(x)\neq 0$, then
		\begin{enumerate}[label=\normalfont(\roman*)]
			\item $\operatorname{dg}(p(x))\geq 0$,
			\item $\operatorname{dg}(p(x)+q(x)) \leq \max\left\{\operatorname{dg}(p(x)),\operatorname{dg}(q(x))\right\}$,
			\item $\operatorname{dg}(p(x)q(x)) \leq \operatorname{dg}(p(x))+\operatorname{dg}(q(x))$.
		\end{enumerate}
	\end{lemma}
	
		Notice that no degree was defined for the zero polynomial. However, some authors put $\operatorname{dg}(0):=-\infty$, so (ii) and (iii) holds for every $p,q\in R[x;\sigma,\delta]$ (see e.g. \cite[p. 37]{GW}).
	
	Since our work concerns algebras, the next result establishes when $A$ has an algebra structure induced by the ring of coefficients $R$. We were not able to find a proof of it in the literature.
	
	\begin{lemma}\label{lemmanew}
		Let $R$ be a $K$-algebra. $A=R[x;\sigma,\delta]$ is a $K$-algebra having $R$ as subalgebra if and only if $\sigma$ and $\delta$ are $K$-linear maps.
	\end{lemma} 
	
	\begin{proof}
		Suppose first that $A$ is a $K$-algebra. Hence, for a given $k\in K$ we must have $(k1_R)x=x(k1_R)$, but by~\ref{O3}, $x(k1_R)=\sigma(k1_R)x+\delta(k1_R)$. Comparing and using~\ref{O2} we get $\sigma(k1_R)=k1_R$ and $\delta(k1_R)=0$.
		
		Conversely, if $\sigma$ is a $K$-linear map and $\delta(k1_R)=0$ for every $k\in K$, we must guarantee a ring morphism $\phi:K\rightarrow A$ such that $\operatorname{Im}(\phi) \subseteq Z(A)$. Since $1_A=1_R$ and $R$ is already a $K$-algebra, define $\phi(k)=k1_R$, for all $k\in K$. Obviously, $\phi$ is a (unitary) ring morphism. Moreover, by~\ref{O3}, $x\phi(k)=\sigma(\phi(k))x + \delta(\phi(k))=\sigma(k1_R)x+\delta(k1_R)=(k1_R)x=\phi(k)x$. Hence, $\operatorname{Im}(\phi)\subset Z(A)$.
	\end{proof}
	
	For the remainder of this document, every time we have the hypothesis that $R$ is a $K$-algebra, we will automatically assume that $\sigma$ and $\delta$ are $K$-linear, so $A$ is also a $K$-algebra.
	
	The following result is known as the universal property of skew polynomial rings.
	
	\begin{theorem}[e.g. {\cite[Proposition 2.4]{GW}}]\label{t1}
		Assume that $B$ is a ring such that the following assertions hold:
		\begin{enumerate}[label=\normalfont(\roman*)]
			\item There is a ring morphism $\phi:R \rightarrow B$,
			\item There is a distinguished element $y\in B$ such that $y\phi(r)=\phi(\sigma(r))y+\phi(\delta(r))$ for all $r\in R$.
		\end{enumerate}
		Then there is a unique ring morphism $\psi: A \rightarrow B$ such that $\psi(x)=y$ and $\psi|_R=\phi$. The last relation can be represented by the following commutative diagram:
		\begin{equation}\label{e26}
			\begin{tikzcd}
				R \arrow[d,"\phi"'] \arrow[r,"\iota"] & A\arrow[dl,dashed,"\psi"]\\
				B
			\end{tikzcd}
		\end{equation}
		Here $\iota:R \rightarrow A$ is the natural inclusion $\iota(r):=r$ for all $r\in R$. Moreover, if $R$ and $B$ are $K$-algebras and $\phi$ is a $K$-algebra morphism, then $\psi$ is also a $K$-algebra morphism.
	\end{theorem}
	
	\begin{proof}
		$B$ has a right $R$-module structure via $r\cdot b := \phi(r)b$, for all $r\in R$ and $b\in B$. On the other hand, by~\ref{O2}, $A$ is a right free $R$-module with basis $\{ x^i : i\geq 0 \}$. Hence we can define a morphism of left $R$-modules $\psi: A \rightarrow B$ via $\psi(x^i)=y^i$, for all $i\geq 0$. Then $\psi$ is given by
		\begin{align*}
			\psi(r_0+r_1x+\cdots+r_nx^n) &= r_0\cdot \psi(1) + r_1 \cdot \psi(x) + \cdots r_n \cdot \psi(x^n)\\
			&= \phi(r_0)+\phi(r_1)y+\cdots+\phi(r_n)y^n.
		\end{align*}
		By definition, the diagram \eqref{e26} is commutative. Furthermore, $\psi$ preserves the unity element since $\psi(1)=\psi(1x^0)=\phi(1)y^0=1$. We want $\psi$ to be a ring morphism, so the only thing left to check is that
		\begin{equation}\label{e27}
			\psi(rx^nsx^m)=\psi(rx^n)\psi(sx^m),
		\end{equation}
		for all $r,s\in R$ and $n,m\in\mathbb{N}$. This is done by induction over $n$. The case $n=0$ is trivial. If $n=1$, we have
		\begin{align*}
			\psi(rxsx^m) &= \psi(r[\sigma(s)x + \delta(s)]x^m) = \psi(r\sigma(s)x^{m+1} + r\delta(s)x^m) \\ &= \psi(r\sigma(s)x^{m+1}) + \psi(r\delta(s)x^m) = \phi(r\sigma(s))y^{m+1} + \phi(r\delta(s))y^m\\&= \phi(r)[\phi(\sigma(s))y+\phi(\delta(s))]y^m = \phi(r)y\phi(s)y^m
			\\ &= \psi(rx)\psi(sx^m).
		\end{align*}
		Assume now that \eqref{e27} holds for a fixed $n$. Then
		\begin{align*}
			\psi(rx^{n+1}sx^m) &= \psi(rx^nxsx^m) = \psi(rx^n[\sigma(s)x+\delta(s)]x^m) \\&= \psi(rx^n \sigma(s)x^{m+1}) + \psi(rx^n\delta(s)x^m) \\
			&= \psi(rx^n)\psi(\sigma(s)x^{m+1}) + \psi(rx^n)\psi(\delta(s) x^m) \\&=\psi(rx^n)[\phi(\sigma(s))y+\phi(\delta(s))]y^m \\
			&=\psi(rx^n)y\phi(s)y^m = \psi(rx^{n+1})\psi(sx^m).
		\end{align*}
		Therefore $\psi$ is a ring morphism. By construction, $\psi$ is uniquely determined.
		
		Finally, suppose that $R$ and $B$ are $K$-algebras, and that $\phi$ is an algebra morphism. Then, $\psi( krx^i )=\phi(kr)y^i=k\phi(r)y^i=k\psi(rx^i)$, for all $r\in R$, $k\in K$ and $i\geq 0$. This guarantees that $\psi$ is $K$-linear and, since it is already a ring morphism, we have shown that it is an algebra map.
	\end{proof}
	
	\begin{corollary}[e.g. {\cite[Corollary 2.5]{GW}}]
	Assume that $B$ is a ring such that the following assertions hold:
		\begin{enumerate}[label=\normalfont(\roman*)]
			\item There is a ring morphism $\phi:R \rightarrow B$,
			\item There is a distinguished element $y\in B$ such that $y\phi(r)=\phi(\sigma(r))y+\phi(\delta(r))$ for all $r\in R$,
			\item $B$ satisfies the universal property of Theorem~\ref{t1}.
		\end{enumerate}
		Then there exists a ring isomorphism between $B$ and $A$. Moreover, if $R$ and $B$ are $K$-algebras, then the isomorphism is an algebra map.
	\end{corollary}

	\begin{proof}
		Since $A$ satisfies the universal property and the condition (ii) of Theorem~\ref{t1} holds for $B$, there is an uniquely ring morphism $\phi: A \rightarrow B$ such that $\phi(x)=y$ and the diagram
		\begin{equation*}
			\begin{tikzcd}
				R \arrow[d,"\phi"'] \arrow[r,"\iota"] & A \arrow[dl,dashed,"\psi"]\\
				B
			\end{tikzcd}
		\end{equation*}
		is commutative. Similarly, since $B$ satisfies the universal property and the relation
		\begin{equation*}
			xr=\sigma(r)x+\delta(r)=\iota(\sigma(r))x+\iota(\delta(r)), \qquad \mbox{for all } r\in R,
		\end{equation*}
		holds in $A$, then there exists an uniquely ring morphism  $\varphi:B \rightarrow A$ such that $\varphi(y)=x$ and the diagram
		\begin{equation*}
			\begin{tikzcd}
				R \arrow[d,"\iota"'] \arrow[r,"\phi"] & B \arrow[dl,dashed,"\varphi"]\\
				A
			\end{tikzcd}
		\end{equation*}
		is commutative. Moreover, using the respective universal properties of $A$ and $B$ with themselves, we get two additional commutative diagrams:
		\begin{equation*}
			\begin{tikzcd}
				R \arrow[d,"\iota"'] \arrow[r,"\iota"] & A \arrow[dl,dashed,"\operatorname{id}_{A}"]\\
				A
			\end{tikzcd} \qquad \begin{tikzcd}
				R \arrow[d,"\phi"'] \arrow[r,"\phi"] & B \arrow[dl,dashed,"\operatorname{id}_{B}"]\\
				B
			\end{tikzcd}
		\end{equation*}
		Since $\varphi\psi(x)=x$ and $\varphi\psi\iota=\iota$, by uniqueness $\varphi\psi=\operatorname{id}_{A}$. Similarly, $\psi\varphi(y)=y$ and $\psi\varphi\phi=\phi$, so $\psi\varphi=\operatorname{id}_B$. Hence, we conclude $B \cong A$ as ring. The last claim is clear from the fact that the universal property lifts algebra maps.
	\end{proof}
	
	Now we list some basic ring-theoretical properties of skew polynomial rings.
	
	\begin{proposition}\label{p13}
		Assume that $\sigma$ is injective. If $R$ is a domain, then $A$ is also a domain.
	\end{proposition}
	
	\begin{proof}
		Let $p(x)=p_0+p_1x+\cdots+p_nx^n \neq 0$ and $q(x)=q_0+q_1x+\cdots+q_mx^m \neq 0$ be two elements of $A$ such that $p_n,q_m\neq 0$. Then $\operatorname{lt}(pq)=p_n\sigma^n(q_m)x^{n+m}\neq 0$, by the injectivity of $\sigma$. Hence, $pq\neq 0$.
	\end{proof}
	
	It is clear that under these conditions, $\operatorname{gr}(pq)=\operatorname{gr}(p)+\operatorname{gr}(q)$ for all $p,q\in A-\{0\}$. Moreover, the units of $A$ coincide with the units of $R$.
	
	Recall that $R$ is said to be \emph{left Noetherian} if any ascending chain of left ideals stabilizes. We mention a result of huge relevance which generalizes the well known Hilbert's Basis Theorem.
	
	\begin{theorem}[Hilbert's Basis Theorem for skew polynomial rings, e.g. {\cite[Theorem 2.6]{GW}}]\label{t2}
		If $R$ is a left (resp. right) Noetherian ring and $\sigma$ is bijective, then $A$ is also a left (resp. right) Noetherian ring.
	\end{theorem}
	
	The construction of skew polynomial rings can be applied several times to obtain an \emph{iterated skew polynomial ring} of the form $R[x_1;\sigma_1,\delta_1]\cdots[x_n;\sigma_n,\delta_n]$. For $1\leq i \leq n$, notice that $\sigma_i$ and $\delta_i$ must be defined as
	\begin{equation*}
		\sigma_i,\delta_i : R[x_1;\sigma_1,\delta_1]\cdots[x_{i-1};\sigma_{i-1},\delta_{i-1}] \longrightarrow R[x_1;\sigma_1,\delta_1]\cdots[x_{i-1};\sigma_{i-1},\delta_{i-1}].
	\end{equation*}
	For iterated skew polynomial rings an explicit basis over the original base ring is given.
	
	\begin{lemma}\label{l8}
		If $A=R[x_1;\sigma_1,\delta_1]\cdots[x_n;\sigma_n,\delta_n]$ is an iterated skew polynomial ring over $R$, then the set $\operatorname{Mon}(x_1,\ldots,x_n):=\left\{ x_1^{\alpha_1} \cdots x_n^{\alpha_n} :  (\alpha_1,\ldots,\alpha_n)\in\mathbb{N}^n \right\}$ is a left $R$-basis of $A$.
	\end{lemma}
	
	\begin{proof}
		We proceed by induction over $n$, the number of variables. We denote
		\begin{equation*}
			A_i:= R[x_1;\sigma_1,\delta_1]\cdots[x_i;\sigma_i,\delta_i], \qquad \mbox{for all } 1\leq i \leq n.
		\end{equation*}
		Since for $n=1$ the statement reduces to~\ref{O2}, there is nothing to prove.
		
		Let $n=2$. Then, again by~\ref{O2}, the powers of $x_2$ form an left basis of $A_2$ over $A_1$, meaning that every element $p\in A_2$ can be written as
		\begin{equation*}
			p= p_0(x_1) + p_1(x_1)x_2 + p_2(x_1)x_2^2+\cdots+p_m(x_1)x_2^m,
		\end{equation*}
		with all $p_j(x_1) \in A_1$, $0\leq j \leq m$. Since every $p_j(x_1)$ can be generated by powers of $x_1$, by distributivity, it is clear that $p$ is generated by $\{ x_1^{\alpha_1}x_2^{\alpha_2} :  \alpha_1,\alpha_2 \in \mathbb{N}  \}$. Now, suppose that $0=\sum_{i,j=0}^{n,m} r_{ij}x_1^i x_2^j$, for some $r_{ij}\in R$. By associativity,
		\begin{equation*}
			\sum_{i,j=0}^{n,m} r_{ij}x_1^i x_2^j =\sum_{j=0}^m \left( \sum_{i=0}^n r_{ij}x_1^i \right)x_2^j,
		\end{equation*}
		so by linearly independence of the powers of $x_2$ on $A$, $\sum_{i=0}^n r_{ij}x_1^i=0$, for every $1\leq j \leq m$. But now by the linearly independence of the powers of $x_1$ on $A_1$, every $r_{ij}=0$.
		
		Now, assume that $\operatorname{Mon}(x_1,\ldots,x_{n-1})$ is a left basis for $A_{n-1}$ over $R$. As the previous case, every $p\in A$ can be written as
		\begin{align*}
			p&= p_0(x_1,\ldots,x_{n-1}) + p_1(x_1,\ldots,x_{n-1})x_n\\
			& \qquad  + p_2(x_1,\ldots,x_{n-1})x_{n-1}^2+\cdots+p_m(x_1,\ldots,x_{n-1})x_{n-1}^m,
		\end{align*}
		with all $p_j \in A_{n-1}$, $0\leq j \leq m$. Using the induction hypothesis for every $p_j$, it is clear that $\operatorname{Mon}(x_1,\ldots,x_n)$ generates $p$. Now, suppose that
		\begin{equation*}
			0=\sum_{\substack{
					0\leq \alpha_i \leq m_i\\
					1\leq i \leq n}} r_{\alpha}x_1^{\alpha_1}\cdots x_n^{\alpha_n}, \qquad \mbox{for some } r_{\alpha}\in R, \, \alpha:=(\alpha_1,\ldots,\alpha_n)\in\mathbb{N}^n.
		\end{equation*}
		Associating,
		\begin{equation*}
			\sum_{\substack{
					0\leq \alpha_i \leq m_i\\
					1\leq i \leq n}} r_{\alpha}x_1^{\alpha_1}\cdots x_n^{\alpha_n} =  \sum_{\alpha_n=0}^n \left( \sum_{\substack{
					0\leq \alpha_i \leq m_i\\
					1\leq i \leq n-1}} r_{\alpha}x_1^{\alpha_1}\cdots x_{n-1}^{\alpha_{n-1}} \right) x_n^{\alpha_n}.
		\end{equation*}
		Since the powers of $x_n$ are a left basis for $A$ over $A_{n-1}$, we must have
		\begin{equation*}
			\sum_{\substack{
					0\leq \alpha_i \leq m_i\\
					1\leq i \leq n-1}} r_{\alpha}x_1^{\alpha_1}\cdots x_{n-1}^{\alpha_{n-1}}=0,
		\end{equation*}
		but again by induction hypothesis, that only happens if and only if every $r_\alpha=0$, which shows the linearly independence of $\operatorname{Mon}(x_1,\ldots,x_n)$.
	\end{proof}
	
	We end this section with some examples, adapted from \cite{GW,MR}, which illustrate that skew polynomial rings are indeed a generalization of more particular well-known cases.
	
	\begin{example}[Classical polynomial ring]\label{ex15}
		Take $\sigma=\operatorname{id}_R$ and $\delta=0$. Therefore~\ref{O3} reduces to $xr=rx$, for all $r\in R$. This is simply the \emph{classical univariate polynomial ring over $R$}, and we write $R[x;\operatorname{id}_R,0]=R[x]$. Moreover the formula \eqref{e25} corresponds to usual multiplication of monomials. In this case, Theorem~\ref{t2} becomes the classical Hilbert's Basis Theorem. Furthermore, we consider the \emph{classical multivariate polynomial ring over $R$}, $R[x_1,\ldots,x_n]$, as an iterated skew polynomial ring over $R$, where $\sigma_i=\operatorname{id}$ and $\delta_i=0$, for all $1\leq i\leq n$.
	\end{example}
	
	\begin{example}[Polynomial ring of endomorphism type]\label{ex11}
		Take $\delta=0$. Then~\ref{O3} becomes $xr=\sigma(r)x$, for all $r\in R$. In this case we write $R[x;\sigma,0]=R[x;\sigma]$. The formula \eqref{e25} reduces to $(rx^n)(sx^m)=r\sigma^n(s)x^{n+m}$, for all $r,s\in R$ and $n,m\in\mathbb{N}$. A widely studied, particular case is when $\sigma$ is an automorphism of $R$.
	\end{example}
	
	\begin{example}[Polynomial ring of derivation type]\label{ex12}
		Take $\sigma=\operatorname{id}_R$. Then~\ref{O3} becomes $xr=rx+\delta(r)$, for all $r\in R$. In this case we write $R[x;\operatorname{id}_R,\delta]=R[x;\delta]$. Moreover, formula~\eqref{e25} simplifies to
		\begin{equation*}
			(rx^n)(sx^m)=r\sum_{k=0}^n \binom{n}{k} \delta^k (s) x^{n+m-k}, \qquad \mbox{for all } r,s\in R \mbox{ and } n,m\in\mathbb{N}.
		\end{equation*}
	\end{example}
	
	The generalizations of Examples~\ref{ex11} and~\ref{ex12} to several variables (i.e., iterated skew polynomial rings) is straightforward and therefore omitted; a particular case of the later is discussed in Example~\ref{ex43}.
	
	\begin{example}[Ore algebras]\label{ex37}\label{ex38}
		Since the setup in general iterated skew polynomial rings can be cumbersome, usually additional conditions may be imposed:
		\begin{align}
			\sigma_i(x_j) = &\ x_j, \qquad j<i, \label{e28}\\
			\delta_i(x_j) = &\ 0, \qquad j<i,\\
			\sigma_i\sigma_j = &\ \sigma_j\sigma_i, \qquad 1\leq i \leq n,\\
			\delta_i\delta_j = &\ \delta_j\delta_i, \qquad 1\leq i \leq n,\label{e29}
		\end{align}
		where the two last relations are understood to be restricted to $R$. Iterated skew polynomial rings satisfying these relations are common, but it does not seem to exist a standard name in the literature. In the case of one single variable (i.e., no iteration) the relations trivialize.
		
		It can be shown that \eqref{e28}-\eqref{e29} are equivalent to the following relations:
		\begin{gather}
			x_ix_j = x_jx_i, \qquad 1\leq i,j\leq n,\label{e200}\\
			\sigma_i(R),\delta_i(R)\subseteq R, \qquad 1\leq i \leq n.\label{e201}
		\end{gather}
		Therefore, under these conditions the maps $\sigma_i,\delta_i$ can be seen as maps $\sigma_i,\delta_i:R\rightarrow R$.
		
		We mention a particular case of the above, distinguished by its well behavior on computational implementations (see e.g. \cite{KJJ}). Let $\Bbbk[t_1,\ldots,t_n]$ be a classical multivariate polynomial ring over $\Bbbk$. If $A=\Bbbk[t_1,\ldots,t_n][x_1;\sigma_n,\delta]\cdots[x_n;\sigma_n,\delta_n]$ is an iterated Ore extension satisfying \eqref{e28}-\eqref{e29}, then $A$ is called an \emph{Ore algebra}.
	\end{example}
	
	We present concrete cases of the above.
	
	\begin{example}[Enveloping universal algebra of $\mathfrak{sl}_2(\Bbbk)$]\label{ex36}
		Recall from Example~\ref{ex35} that a $\Bbbk$-basis for $\mathfrak{sl}_2(\Bbbk)$ is formed by
		\begin{equation*}
			x=\left(\begin{array}{c c}
				0 & 1\\
				0 & 0
			\end{array}\right), \qquad
			y=\left(\begin{array}{c c}
				0 & 0\\
				1 & 0
			\end{array}\right), \qquad
			h=\left(\begin{array}{c c}
				1 & 0\\
				0 & -1
			\end{array}\right),
		\end{equation*}
		and thus $U:=U(\mathfrak{sl}_2(\Bbbk))$ can be seen as a the $\Bbbk$-algebra generated by $x,y,h$ subject to the relation $[x,y]=h$, $[h,x]=2x$ and $[h,y]=-2y$. It is possible to show that $U$ is isomorphic to either of the following iterated polynomial rings:
		\begin{gather*}
			\Bbbk[x][h;\delta_1][y;\sigma_2,\delta_2] \cong \Bbbk[h][x;\sigma_1][y;\sigma_2,\delta_2],
		\end{gather*}
		where
		\begin{align*}
			\delta_1&=2x\frac{d}{dx},&\sigma_1(h)&=h-2, &\sigma_2(x)&=x, \\
			\sigma_2(h)&=h+2, & \delta_2(x) &= -h, & \delta_2(h)&=0.
		\end{align*}
		By Theorem~\ref{t1} and Proposition~\ref{p13}, $U$ is a Noetherian domain.
	\end{example}
	
	\begin{example}[Quantum enveloping algebra of $\mathfrak{sl}_2(\Bbbk)$]\label{ex39}
		Recall from Example~\ref{ex32} that for $q\in \Bbbk$ an invertible element such that $q \neq \pm 1$, $U_q:=U_q(\mathfrak{sl}_2(\Bbbk))$ is the $\Bbbk$-algebra generated by $e,f,k,k^{-1}$ subject to the relations
		\begin{gather}
			kk^{-1}=k^{-1}k=1,\label{e88}\\
			kek^{-1}=q^2e,\label{e89}\\
			kfk^{-1}=q^{-2}f,\label{e90}\\
			\left[e,f\right]=ef-fe=\frac{k-k^{-1}}{q-q^{-1}}\label{e91}.
		\end{gather}
		We realized that this algebra is in fact a Hopf algebra. In this example we will show that it can also be seen as an iterated skew polynomial ring.
		
		Let $A_0:=\Bbbk[k,k^{-1}]$ be the Laurent polynomial ring in the variable $k$, in which \eqref{e88} is satisfied. Notice  that $A_0$ is a Noetherian domain and that $\{ k^l \}_{l\in \mathbb{Z}}$ is a $\Bbbk$-basis of $A_0$. Now, consider the automorphism $\sigma_1$ of $A_0$ given by $\sigma_1(k):=q^2k$ and the respective Ore extension $A_1:=A_0[f;\sigma_1]$. Then, using a similar argument to the one given in the proof of Lemma~\ref{l8}, we can prove that a $\Bbbk$-basis for $A_1$ is $\{ f^jk^l : j\in \mathbb{N},\, l\in\mathbb{Z} \}$. Moreover, by Theorem~\ref{t2}, $A_1$ is a Noetherian domain. We have $fk=\sigma_1(k)f=q^2kf$ which corresponds to the relation \eqref{e90}. By the universal property of free algebras and Theorem~\ref{t1}, $A_1$ is isomorphic to the free algebra generated by $f,k,k^{-1}$ subject to the relations \eqref{e88} and \eqref{e90}.
		
		Now we construct $A_2:=A_1[e;\sigma_2,\delta]$. Let
		\begin{equation}\label{e92}
			\sigma_2(f^jk^l):=q^{-2l}f^jk^l, \qquad j\in \mathbb{N}, \, l\in \mathbb{Z}.
		\end{equation}
		Then $\sigma_2$ is an automorphism of $A_1$. If we denote by $\delta(f)(k)$ the Laurent polynomial $\frac{k-k^{-1}}{q-q^{-1}}$, let
		\begin{equation}\label{e93}
			\delta(k^l):=0, \qquad \delta(f^jk^l):=\sum_{i=0}^{j-1} f^{j-1} \delta(f)(q^{-2i}k) k^l.
		\end{equation}
		We must verify that $\delta$ is a $\sigma_2$-derivation of $A_1$. For that, it suffices to check that for every $j,m\in \mathbb{N}$ and $l,n \in\mathbb{Z}$ we have
		\begin{equation}\label{e94}
			\delta(f^jk^lf^mk^n)=\sigma_2(f^jk^l)\delta(f^mk^n)+\delta(f^jk^l)f^mk^n.
		\end{equation}
		Indeed, starting from the right side of \eqref{e94} and using \eqref{e90}, \eqref{e92} and \eqref{e93}, we have
		\begin{align*}
			&\sigma_2(f^jk^l)\delta(f^mk^n)+\delta(f^jk^l)f^mk^n\\ &= \sum_{i=0}^{m-1} q^{-2l}f^jk^l f^{m-1}\delta(f)(q^{-2i}k) k^n + \sum_{i=0}^{j-1} f^{j-1}\delta(f)(q^{-2i}k)k^l f^m k^n\\
			&= \sum_{i=0}^{m-1} q^{-2l-2l(m-1)}f^{j+m-1} \delta(f)(q^{-2i}k) k^{l+n}+ \sum_{i=0}^{j-1}q^{-2lm}f^{m+j-1}\delta(f)(q^{-2i-2m}k)k^{l+n}\\
			&= \sum_{i=0}^{m-1}q^{-2lm}f^{m+j-1}\delta(f)(q^{-2i}k)k^{l+n} + \sum_{i=m}^{j+m-1}q^{-2lm}f^{m+j-1}\delta(f)(q^{-2i}k)k^{l+n}\\
			&=q^{-2lm}\left( \sum_{i=0}^{j+m-1} f^{j+m-1} \delta(f)(q^{-2i}k) k^{l+n} \right)= q^{-2lm}\delta(f^{j+m}k^{l+n}) = \delta(f^jk^lf^mk^n).
		\end{align*}
		Thus, in particular $\delta(f)=\frac{k-k^{-1}}{q-q^{-1}}$ and $ \delta(k)=0$, whence $ek=\sigma_2(k)e+\delta(k)=q^{-2}ke$, which corresponds to \eqref{e89}, and $ef=\sigma_2(f)e+\delta(f)=fe+\frac{k-k^{-1}}{q-q^{-1}} $, which is \eqref{e91}.
		
		Therefore, $U_q$ is isomorphic to $\Bbbk[k,k^{-1}][f;\sigma_1][e;\sigma_2,\delta]$ and hence, it is a Noetherian domain with $\Bbbk$-basis $\{ e^i f^j k^l : i,j\in\mathbb{N}, l\in\mathbb{Z} \}$.
	\end{example}
	
	\begin{example}[The algebra of shift operators]\label{ex16}
		Let $\Bbbk[t]$ the classical univariate polynomial ring over a field $\Bbbk$. If $\sigma_h:\Bbbk[t] \rightarrow \Bbbk[t]$ is the endomorphism defined by $\sigma_h(p(t))=p(t-h)$, with $p(t)\in \Bbbk[t]$, then the skew polynomial ring $S_h=\Bbbk[t][x_h;\sigma_h]$ over $\Bbbk[t]$ is known as the \emph{algebra of shift operators}. If $p(t),q(t) \in \Bbbk[t]$, the formula \eqref{e25} becomes
		\begin{equation*}
			p(t)x_h^nq(t)x_h^m=p(t)q(t-nh)x_h^{n+m}, \qquad \mbox{for all } n,m\in\mathbb{N}.
		\end{equation*}
		Notice that $S_h$ is an Ore algebra. When $\Bbbk=\mathbb{R}$ and $h>0$, it is used to model time-delays systems \cite{CQR}.
	\end{example}
	
	\begin{example}[Weyl algebra]\label{ex14}
		Let $\Bbbk[t]$ be as in Example~\ref{ex16} and denote by $\frac{d}{dt}$ the derivate operator with respect to $t$. The skew polynomial ring $A_1(\Bbbk)=\Bbbk[t][x;\frac{d}{dt}]$ over $\Bbbk[t]$ is known as the \emph{first Weyl algebra}. If $p(t),q(t) \in \Bbbk[t]$ the formula \eqref{e25} becomes
		\begin{equation*}
			p(t)x^nq(t)x^m=p(t) \sum_{k=0}^n \binom{n}{k} q^{(k)}(t) x^{n+m-k}, \qquad \mbox{for all }m,n\in \mathbb{N}.
		\end{equation*}
		Here, $q^{(k)}(t)$ is the $k$-th derivate of $q(t)$ with respect to $t$. The \emph{$n$-th Weyl algebra} ($n\geq 1$) is defined as the Ore algebra $A_n(\Bbbk):=\Bbbk[t_1,\ldots,t_n][x_1;\frac{\partial}{\partial t_1}]\cdots[x_n;\frac{\partial}{\partial t_n}] $. One of the main applications of Weyl algebras is the theory of $D$-modules (see e.g. \cite{Cou}).
	\end{example}
	
	\begin{example}[The mixed algebra]\label{ex40}
		For every $h\in \Bbbk$, we define the \emph{mixed algebra} (also known as the \emph{algebra of delayed differential operators} \cite{CQR}) as $D_h:=\Bbbk[t][x;\frac{d}{dt}][x_h;\sigma_h]$, where $\sigma_h$ is as in Example~\ref{ex16}. We have $D_h=A_1(\Bbbk)[x_h;\delta_h]$ and hence it is an Ore algebra.
	\end{example}
	
	\begin{example}[The algebra for multidimensional discrete linear systems]\label{ex41}
		The Ore algebra defined as $D:=\Bbbk[t_1,\ldots,t_n][x_1;\sigma_1]\cdots[x_n;\sigma_n]$, where
		\begin{equation*}
			\sigma_i(p(t_1,\ldots,t_n))=p(t_1,\ldots,t_{i-1},t_i+1,t_{i+1},\ldots,t_n), \qquad \mbox{for } 1\leq i \leq n,
		\end{equation*}
		is known as the \emph{algebra for multidimensional discrete linear systems} \cite{CQR}.
	\end{example}
	
	More properties and examples of skew polynomial rings can be found in the literature (e.g. \cite{GW,MR}).
	
	\subsection{PBW extensions}\label{s22}
	
	Although (iterated) skew polynomial rings describe a large amount of noncommutative algebras, these do not cover some remarkable examples, such as the generalized differential operator ring or the enveloping algebra of a finite dimensional Lie algebra. Hence, Bell and Goodearl defined a new family of rings that cover those having the PBW property and polynomial aspect \cite{BG}.
	
	\begin{definition}[PBW extension]
		Let $R$ and $A$ be two rings. We say that \emph{$A$ is a Poincar\'e-Birkhoff-Witt (PBW) extension of $R$} if the following conditions hold:
		\begin{enumerate}[label=(PBW\arabic*), align=parleft, leftmargin=*]
			\item\label{C1}  $A$ contains $R$ as a proper subring and $1_R=1_A$,
			\item\label{C2} (\emph{PBW property}) There exist finitely many elements $x_1,\ldots,x_n\in A$ such that $A$ is a free left $R$-module with basis
			\begin{equation*}
				\operatorname{Mon}(A):=\left\{ x_1^{\alpha_1} \cdots x_n^{\alpha_n} : \alpha:=(\alpha_1,\ldots,\alpha_n)\in\mathbb{N}^n \right\},
			\end{equation*}
			\item\label{C3} For each $r\in R$ and every $1\leq i \leq n$, $x_ir-rx_i \in R$,
			\item\label{C4} For every $1\leq i,j\leq n $, $x_ix_j-x_jx_i\in R+Rx_1+\cdots+Rx_n$.
		\end{enumerate}
		Under these conditions, we write $A=R\langle x_1,\ldots,x_n\rangle$, and $R$ will be called the \emph{ring of coefficients} of the extension $A$.
	\end{definition}
	
	The basis $\operatorname{Mon}(A)$ is usually called \emph{the set of standard monomials} (of $A$) and also denoted by $\operatorname{Mon}(x_1,\ldots,x_n)$. Inspired by the PBW Theorem (see Example~\ref{ex17}), $\operatorname{Mon}(A)$ is called a \emph{PBW basis} for $A$. In general, the elements $x_i$ and $x_j$ do not commute when $i\neq j$. If only~\ref{C1} and~\ref{C2} hold, we say that \emph{$A$ is a ring of left polynomial type over $R$} with respect to $\{x_1,\ldots,x_n\}$.
	
	Before giving some properties, we review a few examples of PBW extensions adapted from \cite{BG}.
	
	\begin{example}[Ore extensions of derivation type]\label{ex43}
		Let $R$ be a ring and let $$A:=R[x_1;\sigma_1,\delta_1]\cdots[x_n;\sigma_n,\delta_n]$$ be an iterated skew polynomial ring of $R$ satisfying \eqref{e28}-\eqref{e29} (or equivalently, \eqref{e200}-\eqref{e201}). We say that $A$ is \emph{an (iterated) Ore extension of derivation type} if $\sigma_i=\operatorname{id}_R$, for all $1\leq i \leq n$. These extensions are all PBW extension, since for every $r\in R$ and $1\leq i,j \leq n$ we have $x_ir - rx_i = \delta_i(r)$ and $
		x_ix_j-x_jx_i = 0$, proving~\ref{C3} and~\ref{C4}. Condition~\ref{C1} is trivial and~\ref{C2} is Lemma~\ref{l8}. In particular, the classical multivariate polynomial ring (see Example~\ref{ex15}) and  Weyl algebras (see Example~\ref{ex14}) are PBW extensions.
	\end{example}
	
	Nevertheless, not every (iterated) skew polynomial ring is a PBW extension. Indeed, by taking $A=R[x;\sigma,\delta]$ with $\sigma\neq\operatorname{id}_R$, condition~\ref{C3} does not hold. A particular example of this is the algebra of shift operators (see Example~\ref{ex16}). The other inclusion is also not true, as the next example shows.
	
	\begin{example}[Universal enveloping algebra of a finite dimensional Lie algebra]
		Let $\mathfrak{g}$ be a finite dimensional $\Bbbk$-Lie algebra with ordered basis $X=\{x_1,\ldots,x_n\}$. The PBW Theorem for the $U(\mathfrak{g})$ (see Example~\ref{ex17}) guarantees that~\ref{C1} and~\ref{C2} are satisfied when we take $R=\Bbbk$. With this, it is immediate that $U(\mathfrak{g})$ is a PBW extension of $\Bbbk$, since for all $k\in \Bbbk$ and $x_i,x_j\in X$, $x_ik - kx_i = 0 \in \Bbbk$, and $x_ix_j-x_jx_i = [x_i,x_j] \in \mathfrak{g} = \Bbbk x_1 + \cdots + \Bbbk x_n \subseteq \Bbbk + \Bbbk x_1 + \cdots + \Bbbk x_n$, which are precisely~\ref{C3} and~\ref{C4}. However, in general, $U(\mathfrak{g})$ is not necessarily an iterated skew polynomial ring, since in the expansion of the product $x_ix_j$, the variables $x_k$ (with $k>j$) may appear. Nonetheless, for some particular Lie algebras, the enveloping algebra can be seen both as PBW extension and as iterated skew polynomial ring (e.g. Example~\ref{ex36}).
	\end{example}
	
	We end this section by giving two additional examples of PBW extensions involving the algebra $U(\mathfrak{g})$.
	
	\begin{example}[Tensor product with the universal enveloping algebra of a finite-dimensional Lie algebra]\label{tensorproductLiealgebra}
		Let $\mathfrak{g}$ be a $\Bbbk$-Lie algebra with basis $X=\{x_i\}_i$ and let $R$ be an arbitrary $\Bbbk$-algebra. The $\Bbbk$-algebra $R\otimes U(\mathfrak{g})$ is also a left $R$-module via the multiplication by elements of $R$.
		
		If $W:=\{ x_{i_1}^{\alpha_1}\cdots x_{i_t}^{\alpha_t} : x_{i_j}\in X, \, \alpha_i \geq 0, \, t\geq 1 \}$ is the $\Bbbk$-basis for $U(\mathfrak{g})$ given by the PBW Theorem, then $1\otimes W:=\{ 1\otimes z : z \in W \}$ is an $R$-basis for $R \otimes U(\mathfrak{g})$. Indeed, if $M$ is a left $R$-module and $f:1\otimes W \rightarrow M$ is any function, then we can induce a $\Bbbk$-bilinear map $\overline{f}:R\times U(\mathfrak{g})\rightarrow M$ given by
		\begin{equation*}
			\left( r, \sum_i \lambda_iX_i \right) \mapsto r \cdot \sum_i \lambda_i f(1\otimes X_i),
		\end{equation*}
		with $X_i\in W$ and $\lambda_i\in \Bbbk$. Hence, by the universal property of tensor products, we can uniquely induce a $\Bbbk$-linear map $f': R \otimes U(\mathfrak{g})\rightarrow M$ such that the diagram
		\begin{equation*}
			\begin{tikzcd}
				R\times U(\mathfrak{g}) \arrow[d,"\overline{f}"'] \arrow[r,"\iota"] & R\otimes U(\mathfrak{g}) \arrow[dl,dashed,"f'"]\\
				M
			\end{tikzcd}
		\end{equation*}
		is commutative, where $\iota$ is the canonical map. In fact, $f'$ is a $R$-morphism since for every $s\in R$,
		\begin{align*}
			f'\left( s \cdot \left( r \otimes \sum_i \lambda_i X_i \right) \right) &= f'\left( sr\otimes  \sum_i \lambda_i X_i  \right) = \overline{f}\left( sr,  \sum_i \lambda_i X_i  \right) \\&= (sr) \cdot \sum_i \lambda_i f(1\otimes X_i) = s \cdot \left( r \cdot \sum_i \lambda_i f(1\otimes X_i) \right) \\&= s\cdot \overline{f}\left( r,  \sum_i \lambda_i X_i \right) = s\cdot f'\left( r\otimes  \sum_i \lambda_i X_i \right).
		\end{align*}
		Moreover, it is clear that the diagram
		\begin{equation*}
			\begin{tikzcd}
				1\otimes W \arrow[d,"f"'] \arrow[r,"j"] & R\otimes U(\mathfrak{g}) \arrow[dl,dashed,"f'"]\\
				M
			\end{tikzcd}
		\end{equation*}
		is commutative, where $j$ is the inclusion map. Additionally, by the uniqueness of $\overline{f}$, the map $f'$ is the only one satisfying such commutativity. Hence, since every function from $1\otimes W$ to an arbitrary left module of $R$ can be extended to a $R$-morphism from $R \otimes U(\mathfrak{g})$ to such module, $1\otimes W$ is indeed an $R$-basis.
		
		Notice that $R \hookrightarrow R \otimes U(\mathfrak{g})$ via $r\mapsto r\otimes 1 = r \cdot (1\otimes 1)$, which corresponds to~\ref{C1}. If $\mathfrak{g}$ is finite-dimensional with $X=\{x_1,\ldots,x_n\}$, then we just proved that $$1\otimes W=\{ (1\otimes x_1)^{\alpha_1}\cdots (1\otimes x_n)^{\alpha_n} : \alpha=(\alpha_1,\ldots,\alpha_n)\in \mathbb{N}^n \}=\operatorname{Mon}(1\otimes x_1,\ldots,1\otimes x_n)$$ is an $R$-basis for $R\otimes U(\mathfrak{g})$, which is~\ref{C2}. Furthermore,~\ref{C3} and~\ref{C4} hold, for if $r\in R$ and $1\leq i,j\leq n$, then
		\begin{equation*}
			(r\otimes 1)(1\otimes x_i)-(1\otimes x_i)(r\otimes 1)=r\otimes x_i - r\otimes x_i =0 \in R,
		\end{equation*}
		\begin{align*}
			(1\otimes x_i)(1\otimes x_j)-(1\otimes x_j)(1\otimes x_i)&=1\otimes x_ix_j - x_jx_i = 1\otimes[x_i,x_j] \\&\in R+ R(1\otimes x_1)+\cdots +R(1\otimes x_n).
		\end{align*}
		Thus $R\otimes U(\mathfrak{g})$ is a PBW extension of $R$.
	\end{example}
	
	Example \ref{tensorproductLiealgebra} holds if we change $U(\mathfrak{g})$ to any PBW extension with a finite number of indeterminates.
	
	\begin{example}[Crossed product with the universal enveloping algebra of a finite-dimensional Lie algebra]
		Let $\mathfrak{g}$ be a $\Bbbk$-Lie algebra with basis $X=\{x_i\}_i$ and let $R$ be an arbitrary $\Bbbk$-algebra. We say that a $\Bbbk$-algebra $S$ is a \emph{crossed product} of $R$ by $U(\mathfrak{g})$ if the following conditions hold:
		\begin{enumerate}[label=\normalfont(\roman*)]
			\item $S$ contains $R$ as a proper subalgebra,
			\item There exists an injective $\Bbbk$-algebra morphism $\mathfrak{g}\rightarrow S$, denoted by $x\mapsto \overline{x}$,
			\item $\overline{x}r-r\overline{x}\in R$ and $r\mapsto \overline{x}r-r\overline{x}$ is a $\Bbbk$-derivation of $R$, for all $r\in R$,
			\item $\overline{x}\overline{y}-\overline{y}\overline{x} \in \overline{[x,y]}+R$, for all $x,y\in \mathfrak{g}$,
			\item $S$ is a free right left $R$-module with the standard monomials over $\{\overline{x_i}\}$ as a basis.
		\end{enumerate}
		In such case, we write $S=R * U(\mathfrak{g})$. By definition if $X$ is finite (that is, $\mathfrak{g}$ is finite-dimensional), then $R*U(\mathfrak{g})$ is a PBW extension of $R$. Particular examples of crossed products with the universal enveloping algebra of a Lie algebra can be found in \cite[1.7.13]{MR}.
	\end{example}
	
	\subsection{Skew PBW extensions}\label{s23}
	
	We saw in the previous section that if a skew polynomial ring $A=R[x;\sigma,\delta]$ is such that $\sigma\neq\operatorname{id}_R$, then $A$ is not a PBW extension of $R$. In order to solve this incompatibility without losing the polynomial behavior, Gallego and Lezama introduced \emph{skew PBW extensions} as a generalization of PBW extensions \cite{GL} and Ore extensions \cite{Ore}. Since then, several authors have studied algebraic and geometrical properties of these objects \cite{Art,GL,AHK,AHK2,HKG,Lez20,GL3,LezG,LezR,LezV,Sua,TRS,Zam}. As a matter of fact, a book containing research results about these extensions has been recently published (see \cite{LezBook}).

	\begin{definition}[Skew PBW extension]\label{r20}
		Let $R$ and $A$ be two rings. We say that $A$ is a \emph{skew PBW extension of $R$} (also called $\sigma$-PBW extension) if the following conditions hold:
		\begin{enumerate}[label=(SPBW\arabic*), align=parleft, leftmargin=*]
			\item\label{SPBW1} $A$ contains $R$ as a proper subring and $1_R=1_A$,
			\item\label{SPBW2} There exist finitely many elements $x_1,\ldots,x_n\in A$ such that $A$ is a free left $R$-module with basis
			\begin{equation*}
				\operatorname{Mon}(A):=\operatorname{Mon}(x_1,\ldots,x_n)=\left\{ x_1^{\alpha_1} \cdots x_n^{\alpha_n} : \alpha:=(\alpha_1,\ldots,\alpha_n)\in\mathbb{N}^n \right\},
			\end{equation*}
			\item\label{SPBW3} For each $r\in R-\{0\}$ and every $1\leq i \leq n$, there exists $c_{i,r} \in R-\{0\}$ such that $x_ir- c_{i,r} x_i \in R$,
			\item\label{SPBW4} For every $1\leq i,j\leq n $, there exists $c_{i,j}\in R-\{0\}$ such that we have the relationship $x_ix_j- c_{i,j} x_jx_i\in R+Rx_1+\cdots+Rx_n$.
		\end{enumerate}
		Under these conditions we write $A=\sigma(R)\langle x_1,\ldots,x_n\rangle$, and $R$ is called the \emph{ring of coefficients} of the extension.
	\end{definition}
		
	\begin{remark}
		Several facts can be immediately deduced from Definition~\ref{r20}.
		\begin{enumerate}[label=\normalfont(\roman*)]
			\item By~\ref{SPBW2}, the elements $c_{i,r}$ and $c_{i,j}$ of~\ref{SPBW3} and~\ref{SPBW4} are unique.
			\item For $i=j$, in~\ref{SPBW4}, $c_{i,i}=1$. Indeed, since $x_i^2-c_{i,i}x_i^2 =0$, then $1-c_{i,i}=0$. If $r=0$, we define $c_{i,0}=0$.
			\item Every $c_{i,j}\in R$, with $1\leq i < j \leq n$, is left invertible. Indeed, $c_{i,j}$ and $c_{j,i}$ are such that
			\begin{gather*}
				x_ix_j- c_{i,j} x_jx_i\in R+Rx_1+\cdots+Rx_n,\\
				x_jx_i-c_{j,i}x_ix_j \in R+Rx_1+\cdots+Rx_n.
			\end{gather*}
			Since $\operatorname{Mon}(A)$ is an $R$-basis then $1=c_{i,j}c_{j,i}$.
			\item We denote the elements of $\operatorname{Mon}(A)$ as $x^{\alpha}$ when it is important to highlight the exponents $\alpha=(\alpha_1,\ldots,\alpha_n)\in \mathbb{N}^n$. An alternative notation for an arbitrary element of $\operatorname{Mon}(A)$ is using the capital letter $X$. By~\ref{SPBW2}, each element $f\in A-\{0\}$ has a unique representation in the form $f=c_1X_1+\ldots+c_tX_t$, with $c_{i}\in R-\{0\}$ and $X_i\in\operatorname{Mon}(A)$, for every $1\leq i \leq t$.
			\item It is clear that the verification of~\ref{SPBW2} in most cases can be cumbersome. There are several techniques for that purpose, including Lemma~\ref{l8} for skew polynomial rings, computation of Gr\"{o}bner bases of two-sided ideals for free algebras \cite{LezBook}, the Bergman's Diamond Lemma \cite{Ber} and the existence theorem for skew PBW extensions \cite{LA}.
		\end{enumerate}
	\end{remark}
	
	The following result justifies the notation for skew PBW extensions.
	
	\begin{proposition}[{\cite[Proposition 3]{GL}}]
		Let $A$ be a skew PBW extension of $R$. Then for each $1\leq i \leq n$ there exist an injective ring endomorphism $\sigma_i: R \rightarrow R$ and a $\sigma_i$-derivation $\delta_i: R\rightarrow R$ such that $x_ir=\sigma_i(r)x_i+\delta_i(r)$, for every $r\in R$.
	\end{proposition}
	
	\begin{proof}
		By~\ref{SPBW3}, for each $1\leq i \leq n$ and every $r\in R$, there exist elements $c_{i,r},r_i \in R$ such that $x_ir=c_{i,r}x_i+r_i$. Since $\operatorname{Mon}(A)$ is a $R$-basis of $A$, these elements are unique for $r$, so we can define the maps $\sigma_i,\delta_i:R\rightarrow R$ by $\sigma_i(r):=c_{i,r}$ and $\delta_i(r):=r_i$. Moreover, it is clear that if $r\neq 0$, then $c_{i,r}\neq 0$, which means that $\sigma_i$ is injective. It is easy to check that $\sigma_i$ is an endomorphism and that $\delta_i$ is a $\sigma_i$-derivation.
	\end{proof}
	
	We mention some particular cases.
	
	\begin{definition}[Quasi-commutative skew PBW extension, bijective skew PBW extension] Let $A$ be a skew PBW extension of $R$.
		\begin{enumerate}[label=\normalfont(\roman*)]
			\item A is said to be \emph{quasi-commutative} if~\ref{SPBW3} and~\ref{SPBW4} are	replaced by
			\begin{enumerate}[label=(SPBW'\arabic*), align=parleft, leftmargin=*]
				\item[(SPBW3')]\label{SPBW5} For every $1\leq i \leq n$ and $r\in R-\{0\}$, there exists $c_{i,r} \in R-\{0\}$ such that $x_i r = c_{i,r}x_i$.
				\item[(SPBW4')]\label{SPBW6} For every $1\leq i,j\leq n$, there exists $c_{i,j}\in R-\{0\}$ such that $x_jx_i=c_{i,j}x_ix_j$.
			\end{enumerate}
			\item $A$ is said to be \emph{bijective} if $\sigma_i$ is bijective, for every $1\leq i \leq n$, and each $c_{i,j}$ is invertible, for any $1\leq i,j\leq n$.
		\end{enumerate}
	\end{definition}

\begin{remark}\label{comparison}
	If $A$ is a quasi-commutative skew PBW extension of $R$, then $A$ is isomorphic to an iterated skew polynomial ring of endomorphism type \cite[Theorem 2.3]{LezR}. Nevertheless, skew PBW extensions of endomorphism type (i.e., with all derivations $\delta_i$ zero) are indeed more general than iterated skew polynomial rings of the same type. To clearly illustrate this, we consider the situations of only two and three indeterminates.
	
	For a skew polynomial ring $R[x;\sigma_x][y;\sigma_y]$ of endomorphism type we have the relations
	\begin{equation*}
		xr = \sigma_x(r)x, \quad yr = \sigma_y(r)y, \quad yx = \sigma_y(x)y,
	\end{equation*}
	for any $r\in R$. On the other hand, for a skew PBW extension $\sigma(R)\langle x, y\rangle$ of endomorphism type we deduce from Definition~\ref{r20} the equations
	\begin{equation*}
		xr=\sigma_1(r)x, \quad yr=\sigma_2(r)y, \quad yx = d_{1,2}xy + r_0 + r_1x + r_2y,
	\end{equation*}
	for some elements $d_{1,2}, r_0, r_1,r_2 \in R$. When we compare the defining relations of both algebraic structures, it is clear that the former is more general.
	
	Similarly, for an iterated skew polynomial ring $R[x;\sigma_x][y;\sigma_y][z;\sigma_z]$ of endomorphism type we have
	\begin{gather*}
		xr = \sigma_x(r)x, \quad yr = \sigma_y(r)y, \quad zr = \sigma_z(r)z,\\
		yx = \sigma_y(x)y, \quad zx = \sigma_z(x)z, \quad zy = \sigma_z(y)z
	\end{gather*}
	for any $r\in R$. On the other hand, for a skew PBW extension $\sigma(R)\langle x, y, z\rangle$ of endomorphism type we deduce
	\begin{gather*}
		xr=\sigma_1(r)x, \quad yr=\sigma_2(r)y, \quad zr = \sigma_3(r)z,\\
		yx = d_{1,2}xy + r_0 + r_1x + r_2y + r_3z, \quad zx = d_{1,3}xz + r_0' + r_1'x + r_2'y + r_3'z\\
		zy = d_{2,3}yz + r_0'' + r_1''x + r_2''y + r_3''z,
	\end{gather*}
	for some elements $d_{1,2}, d_{1,3}, d_{2,3}, r_0, r_0', r_0'', r_1, r_1', r_1'', r_2, r_2', r_2'', r_3, r_3', r_3'' \in R$. As we can see, as the number of indeterminates increases the generality of skew PBW extensions of endomorphism type becomes more notorious.
\end{remark}
	
	We mention two remarkable properties of skew PBW extensions that are similar to those of classical polynomial rings and skew polynomial rings.
	
	By~\ref{SPBW4}, for every $1\leq i,j\leq n $, we know that there exist a unique finite set of constants $c_{i,j},d_{i,j},a_{ij}^k\in R-\{0\}$ such that $x_ix_j=c_{i,j} x_jx_i + a^{(1)}_{ij} x_1 +\cdots + a^{(n)}_{ij} x_n + d_{ij}$. Such constants, together with the coefficient ring $R$, the number of variables $n$, the injective endomorphism $\sigma_k$ and the $\sigma_k$-derivations $\delta_k$ are known as the \emph{parameters} of the extension.
	
	\begin{theorem}[Universal property of skew PBW extensions, {\cite[Theorem 3.1]{LA}}]
		Let $A=\sigma(R)\langle x_1,\ldots,x_n \rangle $ be a skew PBW extension of $R$ with corresponding parameters $R$, $n$, $\sigma_k$, $\delta_k$, $c_{ij}$, $d_{ij}$, $a_{ij}^{(k)}$, for $1\leq i,j\leq n$ and $1\leq k \leq n$. Let $B$ a ring together with a ring morphism $\phi: R\rightarrow B$ and elements $y_1,\ldots,y_n$ such that:
		\begin{enumerate}[label=\normalfont(\roman*)]
			\item $y_k \phi(r)=\phi(\sigma_k(r))y_k+\phi(\delta_k(r))$, for every $r\in R$ and $1\leq k \leq n$,
			\item $y_jy_i=\phi(c_{ij})y_iy_j + \phi(a_{ij}^{(1)})y_1+\cdots + \phi(a_{ij}^{n})y_n + \phi(d_{ij})$, for every $1\leq i,j\leq n$.
		\end{enumerate}
		Then, there exists a unique ring morphism $\psi: A\rightarrow B$ such that $\psi(x_i)=y_i$, for $1\leq i \leq n$, and the following diagram
		\begin{equation}
			\begin{tikzcd}
				R \arrow[d,"\phi"'] \arrow[r,"\iota"] & \sigma(R)\langle x_1,\ldots,x_n \rangle \arrow[dl,dashed,"\psi"]\\
				B
			\end{tikzcd}
		\end{equation}
		is commutative, where $\iota$ is the inclusion map.
	\end{theorem}
	
	\begin{theorem}[Hilbert's Basis Theorem for skew PBW extensions, {\cite[Corollary 2.4]{LezR}}]
		 Let $A=\sigma(R)\langle x_1,\ldots,x_n \rangle$ be a bijective skew PBW extension of $R$. If $R$ is a left (resp. right) Noetherian ring then $A$ is also a left (resp. right) Noetherian ring.
	\end{theorem}
	
	The proof of this last result uses techniques of graduation-filtration, since the graded associated ring of $A$ is an iterated skew polynomial ring of endomorphism type (see e.g. \cite[Section 3.1]{LezBook}).
	
	We end this section by mentioning some examples of skew PBW extensions, which were adapted from \cite{LezBook,GL,LezR}.
	
	\begin{example}[PBW extensions]
		Any PBW extension is a bijective skew PBW extension since in that case $\sigma_i=\operatorname{id}_R$ ($1\leq i \leq n$) and $c_{i,j}=1$ ($1\leq i,j\leq n$).
	\end{example}
	
	\begin{example}[Skew polynomial rings of injective type]\label{ex42}
		Any skew polynomial ring $A=R[x; \sigma, \delta]$ with $\sigma$ injective is a skew PBW extension, $R[x; \sigma, \delta] = \sigma(R)\langle x \rangle$. If additionally $\delta=0$, then $R[x; \sigma]$ is quasi-commutative.
		
		Moreover, an iterated skew polynomial ring $A=R[x_1; \sigma_1, \delta_1] \cdots [x_n; \sigma_n, \delta_n]$ is a skew PBW extension of $R$ if the following conditions hold:
		\begin{enumerate}[label=\normalfont(\roman*)]
			\item $\sigma_i$ is injective, for $1\leq i \leq n$.
			\item $\sigma_i(R),\delta_i(R) \subseteq R$, for $1\leq i \leq n$.
			\item There exist $c_i, d_i\in R$ such that $c_i$ is left invertible and $\sigma_j(x_i) = c_ix_i+d_i$, for $i<j$.
			\item[(iv)] $\delta_j(x_i) \in R+Rx_1+\cdots+Rx_n$, for $i<j$.
		\end{enumerate}
		Under these conditions, we have $A=R[x_1; \sigma_1, \delta_1] \cdots [x_n; \sigma_n, \delta_n]=\sigma(R)\langle x_1,\ldots , x_n \rangle$ and $A$ is called \emph{of injective type}.
		
		A particular case of such situation is given by iterated skew polynomial rings satisfying \eqref{e28}-\eqref{e29} with each $\sigma_i$ being injective. If specifically $R=\Bbbk[t_1,\ldots,t_n]$, then we have an Ore algebra (cf. Example~\ref{ex38}), and
		\begin{equation*}
			\Bbbk[t_1,\ldots,t_n][x_1;\sigma_n,\delta]\cdots[x_n;\sigma_n,\delta_n]=\sigma(\Bbbk[t_1,\ldots,t_n])\langle x_1,\ldots,x_n \rangle.
		\end{equation*}
		Hence, concrete examples are the algebra of shift operators $S_h$ (Example~\ref{ex16}), the Weyl algebras $A_n(\Bbbk)$ (Example~\ref{ex14}), the mixed algebra $D_h$ ( Example~\ref{ex40}) and the algebra for multidimensional discrete linear systems $D$ (Example~\ref{ex41}). Observe that all of these examples are not PBW extensions.
	\end{example}
	
	\begin{example}[Additive analogue of the Weyl algebra]
		Given elements $q_1,\ldots,q_n\in \Bbbk-\{0\}$, let $A_n(q_1,\ldots,q_n)$ be the algebra generated by $x_1,\ldots,x_n,y_1,\ldots,y_n$ together with the relations
		\begin{gather*}
			x_jx_i=x_ix_j, \qquad y_jy_i=y_iy_j, \qquad \mbox{for } 1\leq i,j\leq n,\\
			y_ix_j=x_jy_i, \qquad \mbox{for } i\neq j,\\
			y_ix_i=q_ix_iy_i+1, \qquad \mbox{for } 1\leq i \leq n.
		\end{gather*}
		$A_n(q_1,\ldots,q_n)$ is known as the \emph{additive analogue of the Weyl algebra} \cite{Kury} and it is isomorphic to $\Bbbk[x_1,\ldots,x_n][y_1;\sigma_1,\delta_1] \cdots [y_n;\sigma_n,\delta_n]$ over $\Bbbk[x_1,\ldots,x_n]$, where
		\begin{gather*}
			\sigma_j(y_i)=y_i, \qquad \delta_j(y_i)=0, \qquad \mbox{for } 1\leq i <j \leq n,\\
			\sigma_i(x_j)=x_j,\qquad \delta_i(x_j)=0, \qquad \mbox{for } i\neq j,\\
			\sigma_i(x_i)=q_ix_i,\qquad \delta_i(x_i)=1, \qquad \mbox{for } 1\leq i \leq n.
		\end{gather*}
		Since $A_n(q_1,\ldots,q_n)$  is an iterated Ore extension of injective type, it is also a skew PBW extension of $\Bbbk[x_1,\ldots,x_n]$. Moreover, it is bijective and
		\begin{equation*}
			A_n(q_1,\ldots,q_n)=\sigma(\Bbbk[x_1,\ldots,x_n])\langle y_1,\ldots,y_n \rangle.
		\end{equation*}
		Nonetheless, notice that $A_n(q_1,\ldots,q_n)$ can also be viewed as a skew PBW extension of $\Bbbk$, by putting $A_n(q_1,\ldots,q_n)=\sigma(\Bbbk)\langle x_1,\ldots,x_n,y_1,\ldots,y_n \rangle$. If $q_i = q \neq 0$, for all $1\leq i \leq n$, then $A_n(q_1,\ldots,q_n)$ becomes the \emph{algebra of $q$-differential operators} \cite{JBS}.
	\end{example}
	
	\begin{example}[Multiplicative analogue of the Weyl algebra]
		Given $\lambda_{ij}\in \Bbbk-\{0\}$, with $1\leq i<j\leq n$, let $\mathcal{O}_n(\lambda_{ij})$ be the algebra generated by $x_1,\ldots,x_n$ and subject to the relations
		\begin{equation*}
			x_jx_i=\lambda_{ij}x_ix_j, \qquad \mbox{for } 1\leq i < j \leq n.
		\end{equation*}
		 The algebra $\mathcal{O}_n(\lambda_{ij})$ is known as the \emph{multiplicative analogue of the Weyl algebra} \cite{Jate} and it is isomorphic to the iterated skew polynomial ring $\Bbbk[x_1][x_2;\sigma_2] \cdots [x_n;\sigma_n]$ over $\Bbbk[x_1]$, where $\sigma_j(x_i)=\lambda_{ij}x_i$, for $1\leq i < j \leq n$. Since $\mathcal{O}_n(\lambda_{ij})$ satisfies conditions (i)-(iv) of Example~\ref{ex42}, it is also a skew PBW extension of $K[x_1]$ and hence $\mathcal{O}_n(\lambda_{ij})=\sigma(K[x_1])\langle x_2,\ldots,x_n \rangle$. Notice that $\mathcal{O}_n(\lambda_{ij})$ is quasi-commutative and bijective, and can also be viewed as a skew PBW extension of $\Bbbk$ by putting $\mathcal{O}_n(\lambda_{ij})=\sigma(\Bbbk)\langle x_1,\ldots,x_n \rangle$. $\mathcal{O}_n(\lambda_{ij})$ is also called the \emph{homogeneous solvable polynomial algebra}. If $n = 2$, then $\mathcal{O}_2(\lambda_{12})$ is the \emph{quantum plane} (e.g. \cite{Manin}). If all $\lambda_{ij} = q^{-2} \neq0$, for some $q \in \Bbbk - \{0\}$, then $\mathcal{O}_n(\lambda_{ij})$ becomes the well-known \emph{coordinate ring of the the quantum affine $n$-space} \cite{Smi}. 
	\end{example}
	
	\begin{example}[$q$-Heisenberg algebra]
		Given $q\in\Bbbk-\{0\}$, let $h_n(q)$ be the algebra generated by $x_1,\ldots,x_n,y_1,\ldots,y_n,z_1,\ldots,z_n$ together with the relations
		\begin{gather*}
			x_j x_i = x_ix_j , \qquad z_j z_i = z_iz_j , \qquad y_j y_i = y_iy_j ,\qquad \mbox{for } 1 \leq i, j \leq n,\\
			z_j y_i = y_iz_j ,\qquad z_j x_i = x_iz_j , \qquad y_j x_i = x_iy_j , \qquad \mbox{for } i\neq j,\\
			z_iy_i = qy_iz_i, \qquad z_ix_i = q^{-1} x_iz_i + y_i, \qquad y_ix_i = qx_iy_i, \qquad \mbox{for } 1 \leq i \leq n.
		\end{gather*}
		$h_n(q)$ is known as the \emph{$q$-Heisenberg algebra} \cite{Berg} and it is isomorphic to the iterated skew polynomial ring $\Bbbk[x_1,\ldots,x_n][y_1;\sigma_1]\cdots [y_n;\sigma_n][z_1;\theta_1,\delta_1] \cdots [z_n;\theta_n,\delta_n]$ over $\Bbbk[x_1,\ldots,x_n]$, where
		\begin{gather*}
			\theta_j(z_i) = z_i, \quad  \delta_j (z_i) = 0, \quad \sigma_j (y_i) = y_i, \quad \mbox{for } 1 \leq i < j \leq n,\\
			\theta_j(y_i) = y_i, \quad \delta_j (y_i) = 0, \quad \theta_j (x_i) = x_i,\quad \delta_j (x_i) = 0, \quad \sigma_j (x_i) = x_i, \quad \mbox{for } i\neq j,\\
			\theta_i(y_i) = qy_i, \quad  \delta_i(y_i) = 0,  \quad \theta_i(x_i) = q^{-1} x_i,  \quad \delta_i(x_i) = y_i,  \quad \sigma_i(x_i) = qx_i, \quad  1 \leq i \leq n.
		\end{gather*}
		Since $\delta_i(x_i)=y_i\notin \Bbbk[x_1,\ldots,x_n]$, considering the extension over $\Bbbk[x_1,\ldots,x_n]$, $h_n(q)$ does not satisfy the condition (iii) of Example~\ref{ex42}. However, if the base ring is $\Bbbk$, it does satisfy conditions (i)-(iv) and hence $h_n(q)$ is a bijective skew PBW extension of $\Bbbk$, that is, $h_n(q)=\sigma(\Bbbk)\langle x_1,\ldots,x_n,y_1,\ldots,y_n,z_1,\ldots,z_n \rangle$. This algebra has its roots in the study of $q$-calculus \cite{Wall}.
	\end{example}

	\begin{example}[Dispin algebra]\label{ex44}
	Let $U(\mathfrak{osp}(1,2))$ be the algebra generated by $x,y,z$ together with the relations
	\begin{equation*}
		yz - zy = z, \quad  zx + xz = y \quad \mbox{and} \quad xy - yx = x.
	\end{equation*}
	Then $U(\mathfrak{osp}(1,2)) = \sigma(\Bbbk)\langle x, y, z \rangle$. As was pointed to us by David A. Jordan, this algebra may also be seen as the skew polynomial ring $U(\mathfrak{osp}(1,2)) \cong \Bbbk[y][z;\sigma][x;\alpha,\delta]$ over $\Bbbk[y][z;\sigma]$, where $\sigma(y)=y-1$, $\alpha(z)=-z$. $\alpha(y)=y+1$ (it restricts to $\sigma^{-1}$ on $\Bbbk[y]$), $\delta(y)=0$, and $\delta(z)=y$ (see \cite[Example 1.2.(ii)]{Jor1} and \cite[Example 1.3]{Jor2}). As an algebra, $U(\mathfrak{osp}(1,2))$ corresponds to the universal enveloping algebra of the Lie superalgebra $\mathfrak{osp}(1,2)$ (e.g. \cite[C4.1]{Ros}).
\end{example}
	
	We already discussed in Remark~\ref{comparison} the relation of skew PBW extensions with iterated skew polynomial rings when no derivations are considered. However, from these examples one could think that in the general case skew PBW extensions coincide with (iterated) skew polynomial rings of injective type. However, that is also not the case as the following examples show.
	
	\begin{example}[Quantum algebra $U'_q(\mathfrak{so}_3)$]
		Given $q\in\Bbbk-\{0\}$, let $U'_q(\mathfrak{so}_3)$ be the algebra generated by $I_1,I_2,I_3$ subject to the relations
		\begin{equation*}
			I_2I_1 - qI_1I_2 = -q^{1/2} I_3, \quad I_3I_1 - q^{-1} I_1I_3 = q^{-1/2}I_2 \quad \mbox{and} \quad I_3I_2 - qI_2I_3 = -q^{1/2}I_1.
		\end{equation*}
		 This algebra is a skew PBW extension of $\Bbbk$, $U'_q(\mathfrak{so}_3) = \sigma(\Bbbk)\langle I_1,I_2,I_3 \rangle$. Moreover, from the relations it is clear that it cannot be expressed as a skew polynomial ring over $\Bbbk$, since the commutation rule of two variables involves the third. This algebra was introduced by Gavrilik and Klimyk \cite{GaK} and it is a nonstandard $q$-deformation of the universal enveloping algebra $U(\mathfrak{so}_3)$ of the Lie algebra $\mathfrak{so}_3$ \cite{HKP}.
	\end{example}
	
	\begin{example}[Hayashi algebra]
		Given $q \in \Bbbk-\{0\}$, let $W_q(J)$ be the algebra generated by $x_1,\ldots,x_n,y_1,\ldots,y_n,z_1,\ldots,z_n$ together with the relations
		\begin{gather*}
			x_j x_i = x_ix_j , \qquad z_j z_i = z_iz_j , \qquad y_j y_i = y_iy_j ,\qquad \mbox{for } 1 \leq i, j \leq n,\\
			z_j y_i = y_iz_j ,\qquad z_j x_i = x_iz_j , \qquad y_j x_i = x_iy_j , \qquad \mbox{for } i\neq j,\\
			z_iy_i = qy_iz_i, \qquad y_ix_i = qx_iy_i, \qquad \mbox{for } 1 \leq i \leq n,\\
			(z_ix_i - qx_iz_i)y_i = 1 = y_i(z_ix_i - qx_iz_i), \qquad \mbox{for } 1\leq i \leq n.
		\end{gather*}
		$W_q(J)$ is known as the \emph{Hayashi algebra} \cite{Hay}. Notice that $W_q(J)$ is a skew PBW extension of the multivariate Laurent polynomial ring $\Bbbk[y_1^{\pm 1},\ldots,y_n^{\pm 1}]$, since
		\begin{gather*}
			x_iy_j^{-1} = y_j^{-1}x_i, \qquad z_iy_j^{-1} = y_j^{-1}z_i, \qquad y_jy_j^{-1}= y_j^{-1}y_j = 1,\qquad \mbox{for } 1\leq i,j\leq n,\\
			z_ix_i = qx_iz_i + y_i^{-1}, \qquad \mbox{for } 1\leq i\leq n.
		\end{gather*}
		One can check that $W_q(J)=\sigma( \Bbbk[y_1^{\pm 1},\ldots,y_n^{\pm 1}] )\langle x_1, \ldots, x_n, z_1, \dotsc, z_n \rangle$.
	\end{example}

	\begin{example}[Diffusion algebras]
	A \emph{diffusion algebra} with \emph{parameters} $a_{ij}\in \mathbb{C} - \{0\}$ ($1\le i, j\le n$) is a $\mathbb{C}$-algebra $A$ generated by variables $x_1,\ldots,x_n$, subject to relations
	\begin{equation*}
		a_{ij}x_ix_j-b_{ij}x_jx_i=r_jx_i-r_ix_j, \quad i<j, b_{ij}, r_i\in \mathbb{C},
	\end{equation*}
	and such that the indeterminates $x$'s form a $\mathbb{C}$-basis of the algebra $A$. Diffusion algebras arise in physics as a way of understand a large class of $1$-dimensional stochastic process \cite{IPR}. In such applications, the parameters $a_{ij}$ are strictly positive reals and the parameters $b_{ij}$ are positive reals as they represent unnormalised measures of probability. These algebras are not skew polynomial rings over $\mathbb{C}[x_1,\dotsc, x_n]$ but are skew PBW extensions of it \cite{ReyS2}.	
\end{example}
	
	In the literature it has been shown that skew PBW extensions also generalize several families of noncommutative rings such as the \emph{almost normalizing extensions} defined by McConnell and Robson \cite{MR}, \emph{ambiskew polynomial rings} introduced by Jordan \cite{Jor1, Jor2}, \emph{solvable polynomial rings} introduced by Kandri-Rody and Weispfenning \cite{KRWeisp1990}, and others. As we saw in Definition~\ref{r20}, the advantage of skew PBW extensions is that they do not require the coefficients to commute with the variables and, moreover, those coefficients need not come from a field. In fact, skew PBW extensions contain well-known classes of algebras such as some types of Auslander-Gorenstein rings, several Calabi-Yau and skew Calabi-Yau algebras, certain Artin-Schelter regular algebras, some Koszul algebras, quantum polynomials, some quantum universal enveloping algebras, several examples of $G$-algebras, and various skew graded Clifford algebras. Several connections between skew PBW extensions and other algebras with PBW bases (such as PBW rings \cite{Buesoetal}) can be found in \cite{LezBook, Lez20, LezR, LL, Lev, MR}. Ring-theoretical properties and more examples of skew PBW extensions have been studied in \cite{ReyL,NRR,Rey19,ReyR,ReyS2}.

	We end this section by adressing \emph{3-dimensional skew polynomial algebras}, which is a family of rings included in the class of PBW extensions. Some remarkable examples are the universal enveloping algebra $U(\mathfrak{sl}(2,\Bbbk))$ (Example~\ref{ex35}), the Dispin algebra (Example~\ref{ex44}) $U(\mathfrak{osp}(1,2))$ and the Woronowicz's algebra $W_{\nu}(\mathfrak{sl}(2,\Bbbk))$ \cite{Ros,Wor}. These algebras were introduced by Bell and Smith, and are very important in noncommutative algebraic geometry (see e.g. \cite[Section C.4.3]{Ros}).
	
	\begin{definition}[3-dimensional skew polynomial algebra]\label{3dimensionaldimension}
		A \emph{3-dimensional skew polynomial algebra} $A$ is a $\Bbbk$-algebra generated by the variables $x,y,z$, subject to relations
		\begin{equation*}
			yz-\alpha zy=\lambda, \quad zx-\beta xz=\mu, \quad xy-\gamma yx=\nu,
		\end{equation*}
		and such that
		\begin{enumerate}[label=\normalfont(\roman*)]
			\item $\lambda, \mu, \nu\in \Bbbk+\Bbbk x+\Bbbk y+\Bbbk z$, and $\alpha, \beta, \gamma \in \Bbbk - \{0\}$,
			\item The set of standard monomials $\{x^iy^jz^l: i,j,l\ge 0\}$ is a $\Bbbk$-basis of the algebra.
		\end{enumerate}
	\end{definition}

	From Definition~\ref{3dimensionaldimension}, it is clear that a 3-dimensional skew polynomial algebra $A$ is a skew PBW extensions over $\Bbbk$, $A\cong \sigma(\Bbbk)\langle x,y,z\rangle$ \cite{RS}. These algebras can, in fact, be classified.

	\begin{proposition}[{\cite[Theorem C.4.3.1]{Ros}}]\label{3-dimensionalClassification}
		If $A$ is a 3-dimensional skew polynomial algebra, then $A$
		is one of the following algebras:
		\begin{enumerate}[label=\normalfont(\alph*)]
			\item If $|\{\alpha, \beta, \gamma\}|=3$, then $A$ is defined by $yz-\alpha zy=0$, $zx-\beta xz=0$, $xy-\gamma yx=0$.
			\item if $|\{\alpha, \beta, \gamma\}|=2$ and $\beta\neq \alpha =\gamma =1$, then $A$ is one of the following algebras:
			\begin{enumerate}[label=\normalfont(\roman*)]
				\item $yz-zy=z$, $zx-\beta xz=y$, $xy-yx=x$,
				\item $yz-zy=z$, $zx-\beta xz=b$, $xy-yx=x$,
				\item $yz-zy=0$, $zx-\beta xz=y$, $xy-yx=0$,
				\item $yz-zy=0$, $zx-\beta xz=b$, $xy-yx=0$,
				\item $yz-zy=az$, $zx-\beta xz=0$, $xy-yx=x$,
				\item $yz-zy=z$, $zx-\beta xz=0$, $xy-yx=0$,
			\end{enumerate}
			where $a, b$ are any elements of $\Bbbk$. All nonzero values of $b$ give isomorphic algebras.
			\item If $|\{\alpha, \beta, \gamma\}|=2$ and $\beta\neq \alpha=\gamma\neq 1$, then $A$ is one of the following algebras:
			\begin{enumerate}[label=\normalfont(\roman*)]
				\item $yz-\alpha zy=0$, $zx-\beta xz=y+b$, $xy-\alpha yx=0$,
				\item $yz-\alpha zy=0$, $zx-\beta xz=b$, $xy-\alpha yx=0$.
			\end{enumerate}
			In this case, $b$ is an arbitrary element of $\Bbbk$. Again, all nonzero values of $b$ give isomorphic algebras.
			\item If $\alpha=\beta=\gamma\neq 1$, then $A$ is the algebra defined by the relations  $yz-\alpha zy=a_1x+b_1$, $zx-\alpha xz=a_2y+b_2$, $xy-\alpha yx=a_3z+b_3$. If $a_i=0$ (for $i=1,2,3$), then all nonzero values of $b_i$ give isomorphic
			algebras.
			\item If $\alpha=\beta=\gamma=1$, then $A$ is isomorphic to one of the following algebras:
			\begin{enumerate}[label=\normalfont(\roman*)]
				\item $yz-zy=x$, $zx-xz=y$, $xy-yx=z$,
				\item $yz-zy=0$, $zx-xz=0$, $xy-yx=z$,
				\item $yz-zy=0$, $zx-xz=0$, $xy-yx=b$,
				\item $yz-zy=-y$, $zx-xz=x+y$, $xy-yx=0$,
				\item $yz-zy=az$, $zx-xz=z$, $xy-yx=0$.
			\end{enumerate}
			Parameters $a,b\in \Bbbk$ are arbitrary,  and all nonzero values of
			$b$ generate isomorphic algebras.
		\end{enumerate}
	\end{proposition}
		
	Ring and theoretical properties of 3-dimensional skew polynomial algebras and diffusion algebras have been studied in several papers \cite{Hinchcliffe2005, Jor1, Jor2, LezBook, PT02, RedmanPhD1996, Redman1999, ReyR, RS} and references therein.
	
	\subsection{Almost symmetric algebras}\label{s24}
	
	In this section, we introduce a certain class of $\mathbb{N}$-filtered algebras whose main purpose is to generalize universal enveloping algebras of Lie algebras \cite{Lod}. Several preliminaries of graded and filtered rings not included in this document will be used (see e.g. \cite{MR}). Throughout we assume that the base ring is a field $\Bbbk$. Given a $\mathbb{N}$-filtered algebra $A$ with filtration $\{F_n(A)\}_{n \in \mathbb{N}}$, we denote by $\operatorname{gr}(A)$ its associated graded algebra.
	
	\begin{definition}[Almost symmetric algebra]
		Let $A$ be a $\mathbb{N}$-filtered algebra. $A$ is said to be \emph{almost symmetric} if there exists a graded algebra isomorphism between $\operatorname{gr}(A)$ and the symmetric algebra $S(\operatorname{gr}(A)_1)$.
	\end{definition} 
	
	\begin{remark}\label{rem2}
		If $A$ is an almost symmetric algebra, then $F_0(A) \cong \Bbbk$. Indeed, since $S(\operatorname{gr}(A)_1)$ is connected we have $\operatorname{gr}(A)_0 \cong \Bbbk$. But
		\begin{equation*}
		 \Bbbk \cong \operatorname{gr}(A)_0 = F_0(A)/F_{-1}(A) = F_0(A)/0 \cong F_0(A).
		\end{equation*}
	\end{remark}
	
	In order to classify these algebras, we give some definitions. Recall that, for any $\Bbbk$-vector space $V$, a bilinear form $f:V\times V\rightarrow \Bbbk$ is said to be \emph{alternating} if $f(v,v)=0$, for all $v\in V$.
	
	\begin{definition}[2-cocycles of Lie algebras]
		Let $\mathfrak{g}$ be a Lie algebra and $f:\mathfrak{g} \times \mathfrak{g} \rightarrow \Bbbk$ be a bilinear alternating form. We say that $f$ is a \emph{2-cocycle} of $\mathfrak{g}$ if
		\begin{equation*}
			f(x,[y,z])+f(y,[z,x])+f(z,[x,y])=0, \quad \mbox{for all } x,y,z\in \mathfrak{g}.
		\end{equation*}
		The set of 2-cocycles of $\mathfrak{g}$ is denoted by $Z^2(\mathfrak{g},\Bbbk)$.
	\end{definition}

	We introduce Sridharan enveloping algebras \cite[Definition 2.1]{Sri} and our next goal is to prove that those coincide with almost symmetric algebras.
	
	\begin{definition}[Sridharan enveloping algebra]
		Let $\mathfrak{g}$ be a Lie algebra and $f\in Z^2(\mathfrak{g},\Bbbk)$. If $T(\mathfrak{g})$ is the tensor algebra over $\mathfrak{g}$ and
		\begin{equation*}
			I_f:=\langle x\otimes y - y\otimes x - [x,y] - f(x,y) : x,y\in \mathfrak{g} \rangle,
		\end{equation*}
		the (associative) algebra $U_f(\mathfrak{g}):=T(\mathfrak{g})/I_f$ is called a \emph{$f$-Sridharan enveloping algebra} of $\mathfrak{g}$.
	\end{definition}

 Notice that Sridharan enveloping algebras are a generalization of universal enveloping algebras of Lie algebras (see Example~\ref{ex17}). 
	
	\begin{lemma}\label{l10}
		If $\mathfrak{g}$ is a Lie algebra and $f\in Z^2(\mathfrak{g},\Bbbk)$, then $U_f(\mathfrak{g})$ is $\mathbb{N}$-filtered.
	\end{lemma}
	
	\begin{proof}
		Since $T(\mathfrak{g})$ is $\mathbb{N}$-graded by $\{ \mathfrak{g}^{\otimes i}\}_{p\in\mathbb{N}}$, the family $\{\bigoplus_{i\leq p} \mathfrak{g}^{\otimes i}\}_{p\in\mathbb{N}}$ becomes a $\mathbb{N}$-filtration. Hence the quotient $U_f(\mathfrak{g})=T(\mathfrak{g})/I_f$ is also $\mathbb{N}$-filtered. Explicitly,
		\begin{equation*}
			F_p(U_f(\mathfrak{g})):= \eta_f \left( \bigoplus_{i\leq p} \mathfrak{g}^{\otimes i} \right) , \qquad \mbox{for all } p\in\mathbb{N},
		\end{equation*}
		where $\eta_f : T(\mathfrak{g})\rightarrow U_f(\mathfrak{g})=T(\mathfrak{g})/I_f$ is the canonical algebra map, i.e., $\eta_f(z)=\overline{z}:=z+I_f$, for every $z\in T(\mathfrak{g})$.
	\end{proof}
	
	The restriction of $\eta_f :T(\mathfrak{g})\rightarrow U_f(\mathfrak{g})$ to $\mathfrak{g}$ induces a $\Bbbk$-linear map $i_f:\mathfrak{g} \rightarrow U_f(\mathfrak{g})$ which, for every $x,y\in \mathfrak{g}$, satisfies
	\begin{align}
		i_f(x)i_f(y)-i_f(y)i_f(x)&=\overline{x}\overline{y}-\overline{y}\overline{x} =\overline{x\otimes y-y\otimes x} \nonumber \\&=\overline{[x,y]+f(x,y)} = i_f([x,y])+f(x,y)\cdot i_f(1).\label{e41}
	\end{align}
	
	\begin{lemma}[{\cite[Lemma 2.4]{Sri}}]\label{l5}
		Let $\mathfrak{g}$ be a $\Bbbk$-Lie algebra and $f\in Z^2(\mathfrak{g},\Bbbk)$. If $x_1,\ldots,x_p \in \mathfrak{g}$ and $\sigma$ is a permutation of $(1,\ldots,p)$, then
		\begin{equation*}
			i_f(x_1)\cdots i_f(x_p) - i_f(x_{\sigma(1)})\cdots i_f(x_{\sigma(p)}) \in F_{p-1}(U_f(\mathfrak{g})).
		\end{equation*}
	\end{lemma}
	
	\begin{proof}
		Decomposing the permutation as a product of transpositions, it suffices to consider the case of a transposition interchanging two consecutive indexes $j$ and $j+1$. In this case, relation \eqref{e41} gives
		\begin{equation*}
			i_f(x_j)i_f(x_{j+1})-i_f(x_{j+1})i_f(x_j) = i_f([x_j,x_{j+1}])+f(x_j,x_{j+1})\cdot i_f(1).
		\end{equation*}
	Since $[x_j,x_{j+1}]\in \mathfrak{g}$ and $f(x_j,x_{j+1}) \in \Bbbk$, then $i_f([x_j,x_{j+1}])+f(x_j,x_{j+1})\cdot i_f(1) \in F_2(U_f(\mathfrak{g}))$, as required.
	\end{proof}
	
	\begin{proposition}[{\cite[Proposition 2.3]{Sri}}]
		Let $\mathfrak{g}$ be a $\Bbbk$-Lie algebra and $f\in Z^2(\mathfrak{g},\Bbbk)$. Then $\operatorname{gr}(U_f(\mathfrak{g}))$ is a commutative algebra.
	\end{proposition}
	
	\begin{proof}
		The set $\{ i_f(x) : x\in \mathfrak{g} \} \cup \{ i_f(1) \}$ generates $U_f(\mathfrak{g})$ as an algebra. By Lemma~\ref{l5} those generators commute in the associated graded algebra $\operatorname{gr}(U_f(\mathfrak{g}))$ and hence it must be commutative.
	\end{proof}
	
	Let $X=\{x_i\}_{i\in J}$ be a $\Bbbk$-basis for $\mathfrak{g}$ and $\leq$ be a total order in $J$. We write $y_i:=i_f(x_i) \in U_f(\mathfrak{g})$. Sridharan showed that the set containing 1 and all \emph{standard monomials} of the form
	\begin{equation*}
		y_{i_1}y_{i_2}\cdots y_{i_n}, \qquad \mbox{with }  i_1\leq i_2 \leq \cdots \leq i_n,
	\end{equation*}
	is a $\Bbbk$-basis of $U_f(\mathfrak{g})$ \cite[Theorem 2.6]{Sri}, i.e., the PBW Theorem holds for Sridharan enveloping algebras. This fact implies that $i_f$ is injective \cite[Corollary 2.8]{Sri} and the following result.
	
	\begin{theorem}[{\cite[Theorem 2.5]{Sri}}]\label{t4}
		Let $\mathfrak{g}$ be a $\Bbbk$-Lie algebra and $f\in Z^2(\mathfrak{g},\Bbbk)$. Then
		$\operatorname{gr}(U_f(\mathfrak{g})) \cong S(\mathfrak{g})$ as graded algebras.
	\end{theorem}

	\begin{proof}
		With the notation above, we denote a standard monomial of $ U_f(\mathfrak{g})$ by $y:=y_{i_1}y_{i_2}\cdots y_{i_n}$ (with $i_1\leq i_2 \leq \cdots \leq i_n$ in the index set $J$). Similarly, since we have the identification $\mathfrak{g} \subset T(\mathfrak{g})$, we write $z_i:= \overline{x_i} \in S(\mathfrak{g})$ for every $i\in J$, and $z:=z_{i_1}z_{i_2}\cdots z_{i_n}$. Hence, PBW theorem for Sridharan enveloping algebras guarantees the existence of an unique $\Bbbk$-linear map $\psi': U_f(\mathfrak{g}) \to S(\mathfrak{g})$ such that $\psi'(y)=z$. With the canonical $\mathbb{N}$-filtration on $S(\mathfrak{g})$ (which comes from its graduation), it is clear that $\psi'$ is filtered. Then we can induce a graded map $\psi:=\operatorname{gr}(\psi') : 	\operatorname{gr}(U_f(\mathfrak{g})) \to S(\mathfrak{g})$.
		
		On the other hand, the filtered algebra map $\eta_f :T(\mathfrak{g})\rightarrow U_f(\mathfrak{g})$ mentioned above induces a graded algebra morphism $\operatorname{gr}(\eta_f) :T(\mathfrak{g})\rightarrow \operatorname{gr}(U_f(\mathfrak{g}))$. Using the universal property of $S(\mathfrak{g})$, we have a graded algebra morphism $\phi: S(\mathfrak{g}) \to \operatorname{gr}(U_f(\mathfrak{g}))$. Notice that, by construction, $\phi (z) = \overline{y} \in \operatorname{gr}(U_f(\mathfrak{g}))$. Thus, it is clear that $\psi$ is the inverse of $\phi$, and then 		$\operatorname{gr}(U_f(\mathfrak{g})) \cong S(\mathfrak{g})$ as graded algebras.
	\end{proof}
	
	Recall that, for any Lie algebra $\mathfrak{g}$, a $\Bbbk$-subspace $\mathfrak{I}$ is said to be a \emph{Lie ideal} if
	\begin{equation*}
		[\mathfrak{I},\mathfrak{g}]:=\operatorname{span}_\Bbbk \{ [x,y] : x\in\mathfrak{I}, y\in \mathfrak{g}  \} \subseteq \mathfrak{I}.
	\end{equation*}
	In this case, the quotient space $\mathfrak{g}/\mathfrak{I}$ has Lie algebra structure given by
	\begin{equation*}
		[\overline{x},\overline{y}]:=\overline{[x,y]}, \qquad \mbox{for all } x,y\in \mathfrak{g}.
	\end{equation*}
	
	Before the main theorem of this section, we recall a result on graded algebras.
	
	\begin{proposition}\label{p7}
		Let $A,B$ be two $\mathbb{N}$-filtered rings (resp. $\Bbbk$-algebras) with respective filtration $\{ F_p(A) \}_{p\in\mathbb{N}}$ and $\{ F_p(B) \}_{p\in\mathbb{N}}$. If $f: A\rightarrow B$ is a filtered map such that $\operatorname{gr}(f):\operatorname{gr}(A)\rightarrow \operatorname{gr}(B)$ is injective (resp. surjective, bijective), then $f$ is also injective (resp. surjective, bijective).
	\end{proposition}
	
	\begin{proof}
		For every $p\in \mathbb{N}$, we denote by $F_p(f): F_p(A)\rightarrow F_p(B)$ the restriction of $f$ to $F_p(A)$. Notice that $F_p(f)$ is a group morphism, for every $p\in \mathbb{N}$. Moreover, $F_0(f)$ is injective (resp. surjective, bijective). Indeed, since $\operatorname{gr}(f)$ is injective (resp. surjective, bijective), then every $\operatorname{gr}(f)_p$ is also injective (resp. surjective, bijective). In particular, the morphism  $\operatorname{gr}(f)_0: \operatorname{gr}(A)_0=F_0(A)\rightarrow \operatorname{gr}(B)=F_0(B)$ is injective (resp. surjective, bijective), which by construction coincides with $F_0(f)$.
		
		We assume by induction that $F_{p-1}(f)$ is injective (resp. surjective, bijective). Hence we can consider the following commutative diagram:
		\begin{equation*}
			\begin{tikzcd}
				0 \arrow[r] & F_{p-1}(A) \arrow[r,hook] \arrow[d,"F_{p-1}(f)"] & F_p(A) \arrow[r,two heads] \arrow[d,"F_p(f)"] & \operatorname{gr}(A)_p \arrow[r] \arrow[d," \operatorname{gr}(f)_p"] & 0\\
				0 \arrow[r] & F_{p-1}(B) \arrow[r,hook] & F_p(B) \arrow[r,two heads] & \operatorname{gr}(B)_p \arrow[r] & 0
			\end{tikzcd}
		\end{equation*}
		Since $F_{p-1}(f)$ and $\operatorname{gr}(f)_p$ are injective (resp. surjective, bijective) and each row is exact, the Short Five Lemma implies that the map $F_p(f)$ is also injective (resp. surjective, bijective). Since $p\in\mathbb{N}$ was arbitrary, the assertion has been proved.
	\end{proof}
	
	The next result gives a complete characterization of almost symmetric algebras.
	
	\begin{theorem}[Sridharan's classification, {\cite[Section 3]{Sri}}]\label{t3}
		Let $A$ be an almost symmetric algebra. Then there exist a Lie algebra $\mathfrak{g}$ and a 2-cocycle $f:\mathfrak{g}\otimes \mathfrak{g}\rightarrow \Bbbk$ such that $A\cong U_f(\mathfrak{g})$ as $\mathbb{N}$-filtered algebras.
	\end{theorem}
	
	\begin{proof}
		By definition $\operatorname{gr}(A)$ is isomorphic to a symmetric algebra, so it must be commutative and hence, for every $x,y\in F_1(A)$, we have $\overline{xy-yx}=\overline{0}$ in $\operatorname{gr}(A)_2 = F_2(A)/F_1(A)$. Thus $xy-yx\in F_1(A)$ and the $\Bbbk$-space $F_1(A)$ acquires a structure of Lie algebra given by
		\begin{equation*}
			[x,y]:=xy-yx, \qquad \mbox{for all }x,y\in F_1(A).
		\end{equation*}
		 Also, $F_0(A)=\Bbbk$ is a Lie ideal of $F_1(A)$. Indeed, $[k,x]=kx-xk=0$, for all $k\in\Bbbk$ and $x\in F_1(x)$. Hence $\mathfrak{g}:=F_1(A)/\Bbbk$ becomes a Lie algebra with induced Lie bracket given by
		\begin{equation*}
			[\overline{x},\overline{y}]:=\overline{[x,y]}=\overline{xy-yx}, \qquad \mbox{for all } x,y\in F_1(A).
		\end{equation*}
		Here $\overline{x}=x + \Bbbk \in \mathfrak{g}$. Then we have the exact short sequence
		\begin{equation*}
			\begin{tikzcd}
				0 \arrow[r] & \Bbbk \arrow [r,"\iota"] & F_1(A) \arrow[r,"j"] & \mathfrak{g} \arrow[r] & 0,
			\end{tikzcd}
		\end{equation*}
		where $\iota$ is the inclusion and $j$ is the quotient map. Since this sequence is made of $\Bbbk$-vector spaces, it splits and hence there exist a $\Bbbk$-linear map $t: \mathfrak{g}\rightarrow F_1(A)$ such that $jt=\operatorname{id}_{\mathfrak{g}}$. Define the $\Bbbk$-bilinear map $f: \mathfrak{g} \times \mathfrak{g} \rightarrow \Bbbk$ by
		\begin{equation*}
			f(\overline{x},\overline{y}):=[ t(x),t(y) ] - t([x,y]),\qquad \mbox{for all }x,y\in F_1(A).
		\end{equation*}
		A quick computation shows that $f$ is, in fact, a 2-cocycle, and then we can consider $U_f(\mathfrak{g})$.
		
		By definition, $\operatorname{gr}(A)\cong S(\mathfrak{g})$, i.e., there is a graded algebra isomorphism $\alpha: \operatorname{gr}(A) \to S(\mathfrak{g})$. This induces a filtered algebra map $\alpha' : A \to S(\mathfrak{g})$ such that $\operatorname{gr}(\alpha')=\alpha$. On the other hand, in the proof of Theorem~\ref{t4} we define a filtered algebra map $\psi': U_f(\mathfrak{g}) \to S(\mathfrak{g}) $ that was lifted up to a graded algebra isomorphism $\psi : \operatorname{gr}(U_f(\mathfrak{g})) \to S(\mathfrak{g})$ such that $\operatorname{gr}(\psi')=\psi$. Notice that $\phi':=\psi^{-1} |_{ \operatorname{Im}(\psi')}$ can be seen as a map from $S(\mathfrak{g})$ to $U_f(\mathfrak{g})$, and by construction it is filtered. This yields the filtered composition map $\phi' \alpha' : A \to U_f(\mathfrak{g})$. Moreover, since $\operatorname{gr}$ is a functor, we have the graded algebra isomorphism $\operatorname{gr}(\phi' \alpha') = \psi^{-1} \alpha$. By Proposition~\ref{p7}, we get $A\cong U_f(\mathfrak{g})$.
	\end{proof}
	
	This classification is used to endow almost any symmetric algebra with a comodule structure and thus obtain new examples of Hopf Galois extensions.
	
	\begin{theorem}[{\cite[Proposition 6.5]{JS}}]\label{t5}
		Let $A$ be an almost symmetric $\Bbbk$-algebra. Then there exists a Lie algebra $\mathfrak{g}$ such that $A$ is an $U(\mathfrak{g})$-Galois object.
	\end{theorem}
	
	\begin{proof}
		By Theorem~\ref{t3}, there exist a Lie algebra $\mathfrak{g}$ and $f\in Z^2(\mathfrak{g},\Bbbk)$ such that $A\cong U_f(\mathfrak{g})$. Let $h: T(\mathfrak{g})\rightarrow U_f(\mathfrak{g})\otimes U(\mathfrak{g})$ be the $\Bbbk$-linear map induced by $x \mapsto \overline{x} \otimes 1 + \overline{1}\otimes x$, for all $x\in \mathfrak{g}$. This map factorizes through an algebra map $\rho: U_f(\mathfrak{g})\rightarrow U_f(\mathfrak{g}) \otimes U(\mathfrak{g})$ and hence $U_f(\mathfrak{g})$ is a $U(\mathfrak{g})$-comodule algebra such that $\rho(\overline{x})=\overline{x}\otimes 1 + \overline{1}\otimes x$, for all $x\in \mathfrak{g}$. Let $X=\{x_i\}_{i\in J}$ be a $\Bbbk$-basis for $\mathfrak{g}$ and $\leq$ a total order in $J$. By the PBW Theorem there exists an unique $\Bbbk$-linear map $\theta: U(\mathfrak{g})\rightarrow U_f(\mathfrak{g})$ such that $\theta(1)=1$ and
		\begin{equation*}
			\theta(x_{i_1}x_{i_2}\cdots x_{i_n})= \overline{x_{i_1}} \, \overline{x_{i_2}} \cdots \overline{x_{i_n}}, \qquad \mbox{for every } i_1\leq i_2 \leq \cdots \leq i_n.
		\end{equation*}
		One can easily check that $\theta$ is in fact a $U(\mathfrak{g})$-comodule morphism. Moreover, $\theta$ is bijective and hence $U_f(\mathfrak{g})^{co U(\mathfrak{g})}=\Bbbk$.
		
		A result from Bell \cite[Proposition 1.5]{Bell} states that all faithfully flat $U(\mathfrak{g})$-Galois extensions $B^{\operatorname{co}H}\subset B$, with $B$ a $U(\mathfrak{g})$-comodule algebra, are characterized by maps $\lambda:\mathfrak{g}\rightarrow B$ such that $\rho(\lambda(x))=\lambda(x)\otimes 1 + 1\otimes x$, for all $x\in \mathfrak{g}$. Then, by taking $B:=U_f(\mathfrak{g})$ and $\lambda:=i_f$, it immediately follows that $\Bbbk \subset U_f(\mathfrak{g})$ is an $U(\mathfrak{g})$-Galois object.
	\end{proof}
	
	We end this section mentioning a classification for Sridharan enveloping algebras (and thus for almost symmetric algebras) when the associated Lie algebra is of dimension three over an algebraic closed field $\Bbbk$ of characteristic 0.
	
	\begin{theorem}[{\cite[Theorem 1.3]{Nuss}}]
	 Let $\mathfrak{g}$ be a Lie algebra such that $\dim_\Bbbk(\mathfrak{g})=3$, and $f\in Z^2(\mathfrak{g},\Bbbk)$. Then the Sridharan enveloping algebra $U_f(\mathfrak{g})$ is isomorphic to one of the ten algebras presented in Table~\ref{tab1}.
	\end{theorem}

	\begin{table}[h!]
		\centering
		\begin{tabular}{|c|c|c|c|}
			\hline
			Type & $xy-yx=$ & $yz-zy=$   & $zx-xz=$ \\ \hline
			1    & 0        & 0          & 0        \\ \hline
			2    & 0        & $x$        & 0        \\ \hline
			3    & $x$      & 0          & 0        \\ \hline
			4    & 0        & $\alpha y$ & $-x$     \\ \hline
			5    & 0        & $-y$       & $-(x+y)$ \\ \hline
			6    & $z$      & $-2y$      & $-2x$    \\ \hline
			7    & 1        & 0          & 0        \\ \hline
			8    & 1        & $x$        & 0        \\ \hline
			9    & $x$      & 1          & 0        \\ \hline
			10   & 1        & $y$        & $x$      \\ \hline
		\end{tabular}
	\caption{Each row corresponds to an algebra generated by $x$, $y$ and $z$, together with the commutating relations presented. $\alpha$ denotes a non-zero scalar.}\label{tab1}
	\end{table}
	
	Although some of the algebras presented in Table~\ref{tab1} are iterated skew polynomial rings over $\Bbbk$ (e.g. type 1, 7 or 8), not all of them are so (e.g., type 6).  Nevertheless, all these are skew PBW extensions of $\Bbbk$, that is, $U_f(\mathfrak{g}) \cong \sigma(\Bbbk)\langle x,y,z\rangle$.
	
	\subsection{Some interactions with Hopf Galois theory}\label{s25}
	
	In this last section, we will review relations between Hopf Galois extensions defined in Part~\ref{ch1} and some families (and particular examples) discussed in Part~\ref{ch2}. Specifically, in Section~\ref{inter1} we describe coactions over skew polynomial rings. Section~\ref{inter2} relates almost symmetric algebras with Hopf Galois systems and Section~\ref{inter3} endows Kashiwara algebras with a quantum torsor structure. We remark that the focus here is on Hopf Galois theory, but relations between noncommutative rings and Hopf algebras in general have been also studied (see e.g. \cite{BOZZ,Hua,Pan,Sal}). Throughout, $H$ will denote an arbitrary $K$-Hopf algebra (faithfully flat, if needed).
	
	\subsubsection{Coactions over skew polynomial rings}\label{inter1}
	
	Our goal here is to describe those coactions of an arbitrary Hopf algebra $H$ over a skew polynomial ring induced by the algebra of coefficients. For that, we develop some preliminary facts. All the results of this section are probably new.
	
	\begin{lemma}\label{l9}
		Let $R,B$ be two $K$-algebras. If $A=R[x;\sigma,\delta]$ is a skew polynomial ring over $R$, then
		\begin{gather*}
			A\otimes_K B \cong (R\otimes_K B)[z;\sigma \otimes \operatorname{id}_B, \delta\otimes \operatorname{id}_B],\\
			B\otimes_K A \cong (B\otimes_K R)[z;\operatorname{id}_B \otimes \sigma,\operatorname{id}_B \otimes \delta]
		\end{gather*}
		as $K$-algebras.
	\end{lemma}
	
	\begin{proof}
		We shall prove the first isomorphism since the argument for the second one is quite similar. Provided that $\sigma$ is a $K$-algebra morphism, $\sigma \otimes \operatorname{id}_B$ is also of the same type. Also, since $\delta$ is additive and $\delta(k1)=0$, for all $k\in K$, it follows that $\delta\otimes \operatorname{id}_B$ is also additive and $(\delta\otimes \operatorname{id}_B)(k1\otimes 1)=0$. Furthermore, for all $r,s\in R$ and $b,c\in B$, we have
		\begin{align*}
			&(\delta\otimes\operatorname{id}_B)[(r\otimes b)(s\otimes c)]\\
			&\quad =(\delta\otimes\operatorname{id}_B)(rs\otimes bc)=\delta(rs)\otimes bc= (\sigma(r)\delta(s)+\delta(r)s) \otimes bc\\
			&\quad= \sigma(r)\delta(s) \otimes bc + \delta(r)s\otimes bc = (\sigma(r)\otimes b)(\delta(s) \otimes c) + (\delta(r)\otimes b)(s\otimes c)\\
			&\quad= [(\delta \otimes \operatorname{id}_B)(r\otimes b)][(\delta\otimes\operatorname{id}_B)(s\otimes c)] + [(\delta\otimes\operatorname{id}_B)(r\otimes b)](s\otimes c).
		\end{align*}
		Thus, $\delta\otimes \operatorname{id}_B$ is a $(\sigma\otimes \operatorname{id}_B)$-derivation of $R\otimes B$. Then we can consider the $K$-algebra $(R\otimes_K B)[z;\sigma \otimes \operatorname{id}_B, \delta\otimes\operatorname{id}_B]$.
		
		Now, since the map $R\times B\rightarrow A \otimes B$ given by $(r,b)\mapsto r\otimes b$ is $K$-bilinear, by the universal property of the tensor product, there exists a $K$-linear map $\phi: R\otimes B \rightarrow A\otimes B$ given by $\phi(r\otimes b)=r\otimes b$, for all $r\in R \subset A$ and $b\in B$. In fact, $\phi$ is a $K$-algebra morphism. Now, since for all $r\in R$,
		\begin{align*}
			(x\otimes 1)\phi(r \otimes b) &= (x\otimes 1)(r\otimes b)=xr\otimes b = (\sigma(r)x+\delta(r)) \otimes b\\&= \sigma(r)x \otimes b + \delta(r) \otimes b = \phi(\sigma(r) \otimes b)(x\otimes 1) + \sigma(\delta(b)\otimes b) \\ &= \phi[(\sigma \otimes \operatorname{id}_B)(r\otimes b)](x\otimes 1) + \phi[(\delta\otimes\operatorname{id}_B)(r\otimes b)],
		\end{align*}
		by Theorem~\ref{t1}, there exists an uniquely $K$-algebra morphism $$\psi: (R\otimes B)[z; \sigma\otimes \operatorname{id}_B, \delta \otimes\operatorname{id}_B]\rightarrow A\otimes B$$ such that $\psi(z)=x\otimes 1$ and $\phi|_{R\otimes B}=\phi$. Explicitly, $\psi$ is given by
		\begin{equation*}
			\psi\left( \sum_{i=0}^n (r_i\otimes b_i)z^i \right) = \sum_{i=0}^n \phi(r_i\otimes b_i)(x\otimes 1)^i = \sum_{i=0}^n (r_i \otimes b_i)(x\otimes 1)^i = \sum_{i=0}^n r_ix^i \otimes b_i.
		\end{equation*}
		 Conversely, since the map $A\times B \rightarrow (R\otimes B)[z; \sigma\otimes \operatorname{id}_B, \delta \otimes\operatorname{id}_B]$ given by $$\left( \sum_{i=0}^n r_ix^i,b \right) \mapsto \sum_{i=0}^n (r_i \otimes b)z^i$$ is $K$-bilinear, by the universal property of the tensor product, there exists a $K$-linear map $\varphi: A \otimes B \rightarrow (R\otimes B)[z; \sigma\otimes \operatorname{id}_B, \delta \otimes\operatorname{id}_B]$ given by
		\begin{equation*}
			\varphi\left( \sum_{i=0}^n r_ix^i \otimes b \right)=\sum_{i=0}^n (r_i \otimes b)z^i.
		\end{equation*}
		Now, the calculation
		\begin{align*}
			\psi\varphi\left( \sum_{i=0}^n r_ix^i \otimes b \right) &= \psi\left( \sum_{i=0}^n (r_i \otimes b)z^i \right) = \sum_{i=0}^n r_i x^i \otimes b,\\
			\varphi\psi\left( \sum_{i=0}^n (r_i \otimes b_i)z^i \right) &= \varphi\left( \sum_{i=0}^n r_ix^i \otimes b_i \right) = \sum_{i=0}^n \varphi( r_ix^i \otimes b_i ) = \sum_{i=0}^n (r_i \otimes b_i)z^i,
		\end{align*}
		proves that these maps are inverse of each other.
	\end{proof}
	
	Via the isomorphism, the indeterminate $z$ in $(R\otimes B)[z;\sigma \otimes \operatorname{id}_B, \delta\otimes\operatorname{id}_B]$ can be identified with the element $x\otimes 1$ of $R[x;\sigma,\delta] \otimes B$, and hence we write $(R\otimes B)[x\otimes 1;\sigma \otimes \operatorname{id}_B, \delta\otimes\operatorname{id}_B]$.
	
	\begin{proposition}\label{p14}
		Let $H$ be a $K$-Hopf algebra. Suppose that $R$ is a right $H$-comodule algebra with structure map $\rho_R: R \rightarrow R\otimes H$, and let $A=R[x;\sigma,\delta]$ be a skew polynomial ring over $R$ such that $\sigma$ and $\delta$ are $H$-comodule morphisms.
		Then $A$ is also a right $H$-comodule algebra with induced structure map $\rho_A: A \rightarrow A\otimes H$ given by
		\begin{equation*}
			\rho_A\left( \sum_{i=0}^n r_ix^i \right) = \sum_{i=0}^n (r_i)_{(0)} x^i \otimes (r_i)_{(1)} , \qquad \mbox{with } r_i \in R, \, 0\leq i \leq n.
		\end{equation*}
		Moreover, $A^{\operatorname{co}H}= R^{\operatorname{co}H}[x;\sigma,\delta]$, where $\sigma$ and $\delta$ are considered restricted to $R^{\operatorname{co}H}$.
	\end{proposition}
	
	\begin{proof}
		By Lemma~\ref{l9} we have $A\otimes H \cong (R\otimes H)[x\otimes 1;\sigma \otimes \operatorname{id}_H, \delta\otimes\operatorname{id}_H]$ as algebras. Hence, since $R$ is a right comodule algebra, $\rho_R: R \rightarrow R \otimes H \subset (R\otimes H)[x\otimes 1;\sigma \otimes \operatorname{id}_H, \delta\otimes\operatorname{id}_H]$ is an algebra morphism. Moreover, for every $r\in R$, since $\sigma$ and $\delta$ are comodule morphisms, we have
		\begin{align*}
			\rho_R(\sigma(r))(x\otimes 1)+\rho_R(\delta(r)) &= (\sigma(r)_{(0)} \otimes \sigma(r)_{(1)})(x\otimes 1) + \delta(r)_{(0)} \otimes \delta(r)_{(1)}\\
			&= (\sigma(r_{(0)}) \otimes r_{(1)})(x\otimes 1) + (\delta(r_{(0)}) \otimes r_{(1)})\\
			&= [(\sigma\otimes \operatorname{id}_H) (r_{(0)} \otimes r_{(1)})] (x\otimes 1) + (\delta\otimes \operatorname{id}_H)(r_{(0)}\otimes r_{(1)})\\
			&= (x\otimes 1)(r_{(0)}\otimes r_{(1)}) = (x\otimes 1)\rho_R(r).
		\end{align*}
		Hence, by Theorem~\ref{t1}, there exists an algebra morphism $$\overline{\rho}: A \rightarrow (R\otimes H)[x\otimes 1;\sigma \otimes \operatorname{id}_H, \delta\otimes\operatorname{id}_H]$$ such that $\overline{\rho}(x)=x\otimes 1$ and $\overline{\rho}|_R = \rho_R$. Explicitly,
		\begin{equation*}
			\overline{\rho}\left( \sum_{i=0}^n r_ix^i \right) = \sum_{i=0}^n \rho_R(r_i)(x\otimes 1)^i = \sum_{i=0}^n ((r_i)_{(0)} \otimes (r_i)_{(1)})(x\otimes 1)^i.
		\end{equation*}
		Now, we use the isomorphism $\psi$ of the proof of Lemma~\ref{l9} to define the algebra morphism $\rho_A : A\rightarrow A \otimes H$ as $\rho_A:=\psi\overline{\rho}$. Then, we have
		\begin{align*}
			\rho_A\left( \sum_{i=0}^n r_ix^i \right) &= \psi\overline{\rho}\left( \sum_{i=0}^n r_ix^i \right) \\&= \psi \left( \sum_{i=0}^n ((r_i)_{(0)} \otimes (r_i)_{(1)})(x\otimes 1)^i \right) \\&= \sum_{i=0}^n (r_i)_{(0)} x^i \otimes (r_i)_{(1)}.
		\end{align*}
		Furthermore,
		\begin{align*}
			[(\operatorname{id}_A \otimes \Delta)\rho_A]\left( \sum_{i=0}^n r_ix^i\right) &= (\operatorname{id}_A \otimes \Delta) \left( \sum_{i=0}^n (r_i)_{(0)} x^i \otimes (r_i)_{(1)}\right)\\&= \sum_{i=0}^n (r_i)_{(0)} x^i \otimes (r_i)_{(1)} \otimes (r_i)_{(2)} \\
			&= (\rho_A \otimes \operatorname{id}_H) \left( \sum_{i=0}^n (r_i)_{(0)} x^i \otimes (r_i)_{(1)} \right)\\& = [(\rho_A \otimes \operatorname{id}_H)\rho_A]\left(\sum_{i=0}^n r_ix^i  \right),
		\end{align*}
	and
	\begin{align*}
			[(\operatorname{id}_A \otimes \varepsilon)\rho_A] \left( \sum_{i=0}^n r_ix^i \right) &= (\operatorname{id}_A \otimes \varepsilon) \left( \sum_{i=0}^n (r_i)_{(0)} x^i \otimes (r_i)_{(1)} \right) \\&= \sum_{i=0}^n (r_i)_{(0)} x^i \otimes \varepsilon((r_i)_{(1)})1\\
			&= \sum_{i=0}^n (r_i)_{(0)} \varepsilon((r_i)_{(1)}) x^i \otimes 1 = \sum_{i=0}^n r_i x^i \otimes 1,
		\end{align*}
		which proves that $A$ is a right $H$-comodule with structure map $\rho_A$. Moreover, since, $\rho_A$ is an algebra morphism, $A$ is indeed a right $H$-comodule algebra.
		
		Now, we define
		\begin{equation*}
			R^{\operatorname{co}H}[x;\sigma,\delta] := \left\{ \sum_{i=0}^n r_ix^i \in A :  r_i \in R^{\operatorname{co}H}, \, 0\leq i \leq n \right\} \subset A,
		\end{equation*}
		which, since $\sigma$ and $\delta$ are comodule morphisms, is indeed a subalgebra of $A$. For a given $ p(x)=\sum_{i=0}^n r_ix^i \in R^{\operatorname{co}H}[x;\sigma,\delta]$ we have
		\begin{equation*}
			\rho_A(p(x))=\rho_A \left( \sum_{i=0}^n r_ix^i \right) = \sum_{i=0}^n r_i x^i \otimes 1 = p(x) \otimes 1,
		\end{equation*}
		so $p(x)\in A^{\operatorname{co}H}$. Conversely, if $p(x)=\sum_{i=0}^n r_ix^i\in A^{\operatorname{co}H}$, then
		\begin{equation*}
			\sum_{i=0}^n (r_i)_{(0)} x^i \otimes (r_i)_{(1)} = \sum_{i=0}^n r_i x^i \otimes 1 \in A\otimes H.
		\end{equation*}
		Using the isomorphism of Lemma~\ref{l9}, this means
		\begin{equation*}
			\sum_{i=0}^n ((r_i)_{(0)} \otimes (r_i)_{(1)})(x\otimes 1)^i = \sum_{i=0}^n (r_i \otimes 1)(x\otimes 1)^i \in (R\otimes H)[x\otimes 1;\sigma \otimes \operatorname{id}_H, \delta\otimes\operatorname{id}_H].
		\end{equation*}
		By~\ref{O2} we must have $\rho_R(r_i)=r_i \otimes 1$ for all $0\leq i \leq n$, so each $r_i$ lies in $R^{\operatorname{co}H}$ and hence $p(x)$ is an element of $R^{\operatorname{co}H}[x;\sigma,\delta]$. Then $A^{\operatorname{co}H}= R^{\operatorname{co}H}[x;\sigma,\delta]$.
	\end{proof}
	
	\begin{corollary}
		Let $H$ be a $K$-Hopf algebra. Suppose that $R$ is a right $H$-comodule algebra with structure map $\rho_R: R \rightarrow R\otimes H$, and let $A=R[x_1;\sigma_1,\delta_1]\cdots [x_n;\sigma_n,\delta_n]$ be an iterated skew polynomial ring over $R$ such that each $\sigma_i$ and $\delta_i$ are $H$-comodule morphisms, for $1\leq i \leq n$. Then $A$ is also a right $H$-comodule algebra with induced structure map $\rho_A: A \rightarrow A\otimes H$ given by
		\begin{equation*}
			\rho_A\left( \sum_{i=0}^n r_i X_i \right) = \sum_{i=0}^n (r_i)_{(0)} X_i \otimes (r_i)_{(1)} , \qquad \mbox{with } r_i \in R \mbox{ and } X_i\in\operatorname{Mon}(A), \, 0\leq i \leq n.
		\end{equation*}
		Moreover, $A^{\operatorname{co}H}= R^{\operatorname{co}H}[x_1;\sigma_1,\delta_1]\cdots[x_n;\sigma_n,\delta_n]$, where $\sigma_i$ and $\delta_i$ are considered restricted to $R^{\operatorname{co}H}[x_1;\sigma_1,\delta_1] \cdots[x_{i-1};\sigma_{i-1},\delta_{i-1}]$, for every $1\leq i \leq n$.
	\end{corollary}
	
	Now, we prove that for a certain type of skew polynomial rings the Hopf Galois extension condition is preserved.
	
	\begin{theorem}\label{t15}
		Let $H$ be a $\Bbbk$-Hopf algebra, $R$ a $\Bbbk$-algebra and $A=R[x;\sigma]$ a polynomial ring of endomorphism type over $R$ such that $R$ is a right $H$-comodule algebra and $\sigma$ is an injective comodule morphism. If $R$ is a right $H$-Galois object, then $\Bbbk[x;\sigma] \subset A$ is a right $H$-Galois extension.
	\end{theorem}
	
	\begin{proof}
		$R$ being a right $H$-Galois object means that the map $\beta_R: R \otimes R \rightarrow R \otimes H$ given by
		$\beta_R(r\otimes s)=(r\otimes 1)\rho_R(s)=rs_{(0)} \otimes s_{(1)}$, for all $r,s\in R$, is bijective. For $A$, with the comodule structure induced by Proposition~\ref{p14}, we have that the Galois map $\beta_A: A \otimes_{A^{\operatorname{co}H}} A \rightarrow A \otimes H$ is given by
		\begin{align}
			\beta_A\left( \sum_{i=0}^n r_i x^i \otimes \sum_{j=0}^m s_jx^j \right) &= \left( \sum_{i=0}^n r_i x^i \otimes 1 \right) \rho_A\left( \sum_{j=0}^m s_jx^j\right) \nonumber\\ \label{e106}&= \left( \sum_{i=0}^n r_i x^i \right)\left( \sum_{j=0}^m (s_j)_{(0)}x^j \right) \otimes (s_j)_{(1)}.
		\end{align}
		 To prove that $\beta_A$ is injective it suffices to show that if for $r,s\in R$ and $i,j\in \mathbb{N}$ we have $\beta_A(rx^i \otimes sx^j)=0$, then $rx^i\otimes sx^j=0$. By \eqref{e106} we have
		\begin{equation*}
			0=\beta_A( rx^i \otimes sx^j )=(rx^i)(s_{(0)}x^j) \otimes s_{(1)} \overset{\mbox{\eqref{e25}}}{=} r\sigma^i(s_{(0)})x^{i+j} \otimes s_{(1)}.
		\end{equation*}
		Using the isomorphism of Lemma~\ref{l9}, we have $(r\sigma^i(s_{(0)})\otimes s_{(1)})(x\otimes 1)^{i+j}=0$, which by~\ref{O2} means that $r\sigma^i(s_{(0)})\otimes s_{(1)}=\beta_R(r\otimes \sigma^i(s))=0 \in R\otimes R$. By  hypothesis, it follows that $r\otimes \sigma^i(s)=0$. Since the tensor product is taken over the field $\Bbbk$ it follows that $r=0$ or $\sigma^i(s)=0$ (see e.g. \cite[Theorem 14.5]{Rom2}). Hence, by the injectivity of $\sigma$, $r=0$ or $s=0$. Either case, $rx^i \otimes sx^j=0$.
		
		Now, for the surjectivity of $\beta_A$, recall the notation of Section~\ref{se4}, i.e., we write $\beta_R^{-1}(1\otimes h)=h^{[1]}\otimes h^{[2]} \in R\otimes R$, for all $h\in H$. Then, for any $\sum_{i=0}^n r_ix^i \otimes h \in A \otimes H$, we have
		\begin{align*}
			\beta_A\left( \sum_{i=0}^n r_i\sigma^i(h^{[1]}) x^i \otimes h^{[2]} \right) &= \beta_A\left( \sum_{i=0}^n r_i x^ih^{[1]} \otimes h^{[2]} \right)\\
			&\ = \sum_{i=0}^n r_i x^ih^{[1]} {h^{[2]}}_{(0)} \otimes {h^{[2]}}_{(1)}\\
			&\ = \left( \sum_{i=0}^n r_ix^i \otimes 1 \right)( h^{[1]}{h^{[2]}}_{(0)}  \otimes {h^{[2]}}_{(1)})\\
			& \overset{\mbox{\eqref{e107}}}{=} \left( \sum_{i=0}^n r_ix^i \otimes 1 \right)(1\otimes h)\\
			&\ =\sum_{i=0}^n r_ix^i \otimes h.
		\end{align*}
		Thus $\beta$ is bijective and the extension is Galois.
	\end{proof}

	Notice that the calculations above strongly depend on the fact that the skew polynomial ring involved has no derivations. We end this section by formulating questions that, as far as the authors know, are still open:
\begin{itemize}
	\item Is the description given in Proposition~\ref{p14} unique? More precisely, do all coactions of a Hopf algebra on a skew polynomial ring arise in this manner?
	\item Under some additional compatibility conditions, can Theorem~\ref{t15} be extended to skew polynomial rings with non-zero derivations?
	\item May Theorem~\ref{t15} be extended to (skew) PBW extensions?
\end{itemize}

	\subsubsection{Almost symmetric algebras and Hopf Galois systems}\label{inter2}
	
	We saw in Section~\ref{s24} that for any almost symmetric algebra $A$, there exists a Lie algebra $\mathfrak{g}$ such that $\Bbbk \subset A$ is an $U(\mathfrak{g})$-extension. In this section, we mention an alternative path for that result using the equivalence between Hopf Galois objects and Hopf Galois systems (see diagram~\eqref{e110}).
	
	\begin{theorem}[{\cite[Theorem 3]{Gru}}]
	Let $\mathfrak{g}$ be a $\Bbbk$-Lie algebra and $f\in Z^2(\mathfrak{g},\Bbbk)$. Consider the Sridharan enveloping algebras $U_f(\mathfrak{g})$ and $U_{-f}(\mathfrak{g})$, and define $\gamma: U(\mathfrak{g})\rightarrow U_f (\mathfrak{g})\otimes U_{-f}(\mathfrak{g})$ and $\delta: U(\mathfrak{g})\rightarrow U_{-f}(\mathfrak{g}) \otimes U_f(\mathfrak{g})$ as $x \mapsto 1\otimes x + x\otimes 1$, for all $x\in\mathfrak{g}$, and $S: U_{-f}(\mathfrak{g})\rightarrow U_f(\mathfrak{g})$ as $S(x)=-x$, for all $x\in \mathfrak{g}$. Then, $(U(\mathfrak{g}), U(\mathfrak{g}), U_f(\mathfrak{g}), U_{-f}(\mathfrak{g}))$ is a $\Bbbk$-Hopf Galois system.
	\end{theorem}
	
	By the equivalence theorems of Section~\ref{sec2.11}, we have the following immediate results.
	
	\begin{corollary}[{\cite[Corollaries 2 and 3]{Gru}}]
		Let $\mathfrak{g}$ be a $\Bbbk$-Lie algebra and $f\in Z^2(\mathfrak{g},\Bbbk)$. Then the following assertions for $U_f(\mathfrak{g})$ hold:
		\begin{enumerate}[label=\normalfont(\roman*)]
			\item $U_f(\mathfrak{g})$ is a quantum $\Bbbk$-torsor with associated map $\mu(x)=x\otimes 1 - 1 \otimes x \otimes 1 + 1\otimes 1 \otimes x$, for $x\in \mathfrak{g}$, and Grunspan map $\theta=\operatorname{id}_{U_f(\mathfrak{g})}$. Moreover, $H_l(U_f(\mathfrak{g})) \cong H_r(U_f(\mathfrak{g}))\cong U(\mathfrak{g})$.
			\item $U_f(\mathfrak{g})$ is a $(U(\mathfrak{g}),U(\mathfrak{g}))$-biGalois object.
		\end{enumerate}
	\end{corollary}
	
	\subsubsection{Kashiwara algebras and quantum torsors}\label{inter3}
	
	Kashiwara \cite{Kash} defined a type of algebras useful in the study of crystal bases. In this section, we introduce the preliminaries for such algebras and then prove that these are examples of Hopf Galois systems. As a reference for basic terminology, we follow \cite{HK,Kac}.
	
	\begin{definition}[Generalized Cartan matrix]
		 A square matrix $A=[a_{ij}]_{i,j=1}^n$ with entries in $\mathbb{Z}$ is called a \emph{generalized Cartan matrix} if it satisfies the following conditions:
		 \begin{enumerate}[label=\normalfont(\roman*)]
		 	\item $a_{ii}=2$, for $1\leq i \leq n$,
		 	\item $a_{ij}\leq 0$, for $i\neq j$,
		 	\item $a_{ij}=0$ if and only if $a_{ji}=0$.
		 \end{enumerate}
		   Moreover, $A$ is said to be \emph{indecomposable} if for every pair of nonempty subsets $I_1,I_2 \subseteq I=\{1,\ldots,n\}$ with $I_1\cup I_2 = I$, there exists some $i\in I_1$ and $j\in I_2$ such that $a_{ij}\neq 0$.
	\end{definition}
	
	Throughout we will suppose that every generalized Cartan matrix is \emph{symmetrizable}, i.e., there exists a diagonal matrix $D$ with entries in $\mathbb{Z}_{>0}$ such that $DA$ is symmetric.
	
	Let $P^\vee$ be a free Abelian group of rank $2n-\operatorname{rank}(A)$ with a $\mathbb{Z}$-basis
	\begin{equation*}
		\{h_i : 1\leq i \leq n \} \cup \{ d_s : s=1,\ldots,n-\operatorname{rank}(A) \}.
	\end{equation*}
	The group $P^\vee$ is known as the \emph{dual weight lattice}. The $\Bbbk$-linear space spanned by $P^\vee$, $\mathfrak{h}:=\Bbbk\otimes_\mathbb{Z} P^\vee$, is called \emph{the Cartan subalgebra}. We also define the \emph{weight lattice} to be
	\begin{equation*}
		P:=\{ \lambda \in \mathfrak{h}^* : \lambda(P^\vee) \subset \mathbb{Z} \}.
	\end{equation*}
	
	The elements of a linear independent subset $\Pi:=\{ \alpha_i : 1\leq i \leq n \} \subset \mathfrak{h}^*$ satisfying $		\alpha_j(h_i)=a_{ij}$ and $ \alpha_j(d_s)\in\{0,1\}$, for all $1\leq i,j\leq n$ and $s=1,\ldots,n-\operatorname{rank}(A)$,	are called \emph{simple roots}. Similarly, each element of the set $\Pi^\vee:=\{ h_i : 1\leq i \leq n \}$ is called a \emph{simple coroots}.
	
	\begin{definition}[Cartan datum]
		Let $A=[a_{ij}]_{i,j=1}^n$ be a generalized Cartan matrix. The quintuple $(A,\Pi,\Pi^\vee,P,P^\vee)$ defined as above is said to form a \emph{Cartan datum} associated to $A$. 
	\end{definition}
	
	Recall that if $V$ is a $\Bbbk$-vector space, the space $\mathfrak{gl}(V)$ of all $\Bbbk$-linear maps on $V$ acquires a Lie algebra structure via the Lie bracket $[x,y]=xy-yx$, for all $x,y\in\mathfrak{gl}(V)$, and it is called the \emph{general linear Lie algebra}. If $V=\Bbbk^n$, we denote the general linear Lie algebra by $\mathfrak{gl}(n,\Bbbk)$.  Given a Lie algebra $\mathfrak{g}$ we define the Lie morphism $\operatorname{ad}:\mathfrak{g}\rightarrow\mathfrak{gl}(\mathfrak{g})$, given by $\operatorname{ad} x(y)=[x,y]$, for all $x, y\in\mathfrak{g}$, which is called \emph{the adjoint representation of $\mathfrak{g}$}.
	
	 With this, we are able to define a type of algebras of great relevance since they are considered a natural generalization of semisimple Lie algebras to the infinite dimensional case \cite[Chapter 1]{Kac}.
	
	\begin{definition}[Kac-Moody algebra]
		Let $(A,\Pi,\Pi^\vee,P,P^\vee)$ be a Cartan datum associated to a generalized Cartan matrix $A=[a_{ij}]_{i,j=1}^n$. The \emph{Kac-Moody algebra} associated to the Cartan datum is the $\Bbbk$-Lie algebra generated by the elements $e_i$, $f_i$ ($1\leq i \leq n$) and $h\in P^\vee$ subject to the following defining relations:
		\begin{enumerate}[label=(KMA\arabic*), align=parleft, leftmargin=*]
			\item\label{KMA1} $[h,h']=0$ for all $h,h'\in P^\vee$,
			\item\label{KMA2} $[e_i,f_j]=\delta_{ij}h_i$ for all $i,j\in I$,
			\item\label{KMA3} $[h,e_i]=\alpha_i(h)e_i$ for all $i\in I$ and $h\in P^\vee$,
			\item\label{KMA4} $[h,f_i]=-\alpha_i(h)f_i$ for all $i\in I$ and $h\in P^\vee$,
			\item\label{KMA5} $ (\operatorname{ad} e_i)^{1-a_{ij}} e_j =0 $ for all $i\neq j$, $i,j\in I$,
			\item\label{KMA6} $(\operatorname{ad} f_i)^{1-a_{ij}} f_j = 0$, for all $i\neq j$, $i,j\in I$.
		\end{enumerate}
	\end{definition}
	
Conditions \ref{KMA1}-\ref{KMA4} are called the \emph{Weyl relations}, while~\ref{KMA5}-\ref{KMA6} are known as the \emph{Serre relations}.
	
	Given a Kac-Moody algebra $\mathfrak{g}$ associated to the Cartan datum $(A,\Pi,\Pi^\vee,P,P^\vee)$, we define an inner product on the Cartan subalgebra $\mathfrak{h}^*$ so that
	\begin{equation*}
		(\alpha_i,\alpha_i)\in \mathbb{N} \qquad \mbox{and} \qquad \langle h_i,\lambda \rangle=2(\alpha_i,\lambda)/(\alpha_i,\alpha_i), \qquad \mbox{for all } \lambda\in \mathfrak{h}^*.
	\end{equation*}
	
	Our base ring is $K:=\mathbb{Q}[[\hbar]]$, the formal power series ring over $\mathbb{Q}$. We also set $q=\exp(\hbar)$, $q_i=q^{\langle \alpha_i,\alpha_i \rangle/2}$, $t_i=q^{h_i}$,
	\begin{equation*}
		[n]_i=\dfrac{q_i^n-q_i^{-n}}{q_i - q_i^{-1}} \qquad \mbox{and} \qquad [n]_i!=\prod_{k=1}^n [k]_i.
	\end{equation*}
	
	\begin{definition}[Kashiwara algebra]
		Given a Kac-Moody algebra $\mathfrak{g}$ associated to the Cartan datum $(A,\Pi,\Pi^\vee,P,P^\vee)$, the \emph{Kashiwara algebra} $B_q(\mathfrak{g})$ is the associative $K$-algebra generated by the elements $e'_i$, $f_i$ ($1\leq i \leq n$) and $q^h,h\in \bigoplus_{i=1}^n \mathbb{Z} h_i$ (with $h_i\in \Pi^\vee$) together with the relations:
		\begin{gather*}
			q^h e'_i q^{-h} =  q^{\langle h,\alpha_i \rangle}e_i',\\
			q^h f_i q^{-h} =  q^{-\langle h,\alpha_i \rangle}f_i,\\
			e'_i f_j =  q_i^{\langle h_i,\alpha_j \rangle}f_je'_i + \delta_{ij},\\
			\sum_{k=0}^{1-\langle h_i,\alpha_k \rangle} (-1)^k X_i^{(k)}X_j X_i^{(1-\langle h_i,\alpha_j\rangle)} =  0,
		\end{gather*}
		where $X=e',f$ and $X_i^{(n)}=X_i^n/[n]_i!$.
	\end{definition}
	
	\begin{theorem}[{\cite[Theorem 4]{Gru}}]
		Let $B:=B_q(\mathfrak{g})$ be a Kashiwara algebra. Then the map $\mu: B\rightarrow B\otimes B^{\operatorname{op}} \otimes B$ defined by
		\begin{align*}
			\mu(e_i'):= &\ 1\otimes 1 \otimes e_i' - 1 \otimes e_i' t_i \otimes t_i^{-1} + e_i'\otimes t_i \otimes t_i^{-1},\\
			\mu(f_i) = &\ 1\otimes 1 \otimes f_i - 1\otimes f_i t_i \otimes t_i^{-1}+f_i \otimes t_i \otimes t_i^{-1},\\
			\mu(q^h) = &\ q^h \otimes q^{-h} \otimes q^h,
		\end{align*}
		makes $B$ into a quantum torsor. Moreover, the Grunspan map $\theta:B\rightarrow B$ is given by
		\begin{gather*}
			\theta(e_i')=t_i^{-1} e'_i t_i, \qquad \theta(f_i) t_i^{-1}f_it_i \qquad \mbox{and} \qquad \theta(q^h)=q^h,
		\end{gather*}
	so the torsor is autonomous.
	\end{theorem}
	
	Using characterization theorems, Grunspan gave an explicit description of the two Hopf algebras $H_l(B)$ and $H_r(B)$ (which turn out to be quantum groups) that can be attached to the torsor, endowing it with a Hopf Galois system \cite[Section 4.2]{Gru}.
	
	\section{Conclusions}
	
	The investigation of properties of algebraic structures is a topic of interest for the mathematical community, and in this document we have shown several of those features, both from the point of view of Hopf Galois extensions and of noncommutative families of rings. We have presented in detail various relations between Hopf Galois extensions and different algebraic structures. Moreover, with examples, results and properties we illustrate the scope of the theory. Additionally, the study of the interactions between these algebraic contexts give rise to new open problems both in Hopf theory and in noncommutative algebra.
	
	\section{Acknowledgments}
	
	We thank Professor David A. Jordan for his remarks about skew polynomial rings of mixed type, central elements and simple rings, and Professor Zoran \v{S}koda for his comments on quantum heaps as an equivalent definition of quantum torsors. We are also grateful to the anonymous referees for their useful comments and corrections that improved the overall quality of the document.
	
	The authors were supported by the research fund of the Department of Mathematics, Universidad Nacional de Colombia - Sede Bogot\'a, Colombia, HERMES Code 52464. The first author was also supported by the Fulbright Visiting Student Researcher Program. 
\printbibliography
\clearpage
\begin{shaded*}
        {\fontsize{11}{10}\selectfont\textbf{\textcolor{myseccolor}{Algunas
        interacciones entre extensiones de Hopf Galois y anillos no conmutativos}}}

        \vspace{3mm}

        {\fontsize{11}{10}\selectfont\textbf{\textcolor{myseccolor}{Resumen:}}}
        En este artículo, nuestros objetos de interés son las extensiones de
        Hopf Galois (p. ej., álgebras de Hopf, extensiones de Galois de cuerpos
        algebraicos, álgebras fuertemente graduadas, productos cruzados,
        fibrados principales, etc.) y familias de anillos no conmutativos (p.
        ej., anillos polinomiales torcidos, extensiones PBW y extensiones PBW
        torcidas, etc.). Recopilamos y sistematizamos preguntas, problemas,
        propiedades y avances recientes en ambas teorías desarrollando
        explícitamente ejemplos y haciendo cálculos que generalmente se omiten
        en la literatura. En particular, para las extensiones de Hopf Galois
        consideramos enfoques desde el punto de vista de torsores cuánticos
        (también conocidos como "heaps" cuánticos) y sistemas de Hopf Galois,
        mientras que para algunas familias de anillos no conmutativos
        presentamos avances en la caracterización de propiedades homológicas y
        teóricas de anillos. Cada tema desarrollado se ejemplifica con
        abundantes referencias a obras clásicas y actuales, por lo que este
        trabajo sirve de referencia para los interesados en
        cualquiera de las dos teorías. A lo largo de este trabajo, se presentan
        las interacciones entre ambos.

        {\fontsize{11}{10}\selectfont\textbf{\textcolor{myseccolor}{Palabras Clave:}}} 
        álgebra de Hopf; extensión de Hopf Galois; anillo no conmutativo; extensión de Ore; extensión PBW torcida.
\end{shaded*}
\begin{shaded*}
        {\fontsize{11}{10}\selectfont\textbf{\textcolor{myseccolor}{Algumas
        interações entre extensões de Hopf Galois e anéis não comutativos}}}

        \vspace{3mm}

        {\fontsize{11}{10}\selectfont\textbf{\textcolor{myseccolor}{Resumo:}}}
        Neste artigo, nossos objetos de interesse são extensões de Hopf Galois
        (por exemplo, álgebras de Hopf, extensões de Galois de cuerpos
        algebraicos, álgebras fortemente graduadas, produtos cruzados, fibrados
        principais, etc.), e famílias de anéis não comutativos (por exemplo,
        anéis polinomiais torcidos, extensões PBW e extensões PBW torcidas,
        etc.). Coletamos e sistematizamos questões, problemas, propriedades e
        avanços recentes em ambas as teorias, desenvolvendo explicitamente
        exemplos e fazendo cálculos que geralmente são omitidos na literatura.
        Em particular, para extensões de Hopf Galois consideramos abordagens do
        ponto de vista de torsores quânticos (também conhecidos como heaps
        quânticos) e sistemas Hopf Galois, enquanto para algumas famílias de
        anéis não comutativos apresentamos avanços na caracterização de
        propriedades homológicas e teóricas de anéis. Cada tema desenvolvido é
        exemplificado com abundantes referências a obras clássicas e atuais, por
        isso este trabalho serve como referência para os interessados
        em qualquer uma das duas teorias. Ao longo desde trabalho,
        as interações entre os dois são apresentadas.

        {\fontsize{11}{10}\selectfont\textbf{\textcolor{myseccolor}{Palavras-chave:}}}
álgebra de Hopf; extensão de Hopf Galois; anel não comutativo; extensão de Ore; extensão PBW torcida.
\end{shaded*}

\clearpage

\begin{shaded*}
	{\fontsize{11}{10}\selectfont\textbf{\textcolor{myseccolor}{Fabio Calder\'on}}}
	
	Fabio Calder\'on is a PhD(c) at Universidad Nacional de Colombia (Bogot\'a, Colombia). In 2021-2022 he did a research stay at Rice University (Houston, TX, USA) funded by the Fulbright Visiting Student Researcher program. His topics of research are noncommutative algebra and (generalizations of) Hopf algebras.	
	
	ORCID: \href{https://orcid.org/0000-0003-1777-0805}{0000-0003-1777-0805}
\end{shaded*}

\vspace{5mm}
\begin{shaded*}
	{\fontsize{11}{10}\selectfont\textbf{\textcolor{myseccolor}{Armando Reyes}}}
	
	Armando Reyes is Associate Professor (Exclusive Dedication) at the Department of Mathematics, Universidad Nacional de Colombia - Sede Bogot\'a. He is the leader of the research group LEMAGNOC (M\'etodos algebraicos, anal\'iticos y geom\'etricos en geometr\'ia no conmutativa). He is \textit{Investigador Senior} by Minciencias.
	
	ORCID: \href{https://orcid.org/0000-0002-5774-0822}{0000-0002-5774-0822}
\end{shaded*}
\end{document}